\pdfoutput=1

\documentclass[a4paper]{amsart}

\usepackage{picins}
\usepackage{textcomp}
\usepackage{epsfig,amsmath}
\usepackage{xspace}
\usepackage[psamsfonts]{amssymb}
\usepackage[utf8]{inputenc}
\usepackage{float}\restylefloat{figure}
\usepackage{color}
\usepackage{nomencl}
\usepackage{gensymb}

\usepackage{enumerate}
\usepackage[pdftex]{hyperref}
\usepackage{tikz}
\usetikzlibrary{arrows}
\graphicspath{{img/}}
\usepackage[all]{xy}
\usepackage{placeins} 
\usepackage{cleveref}

\newtheorem*{theorem*}{Theorem}
\newtheorem*{mainthm}{Main Theorem}
\newtheorem*{assertion*}{Assertion}

\newtheorem{lemma}{Lemma}
\newtheorem{definition}[lemma]{Definition}
\newtheorem{theorem}[lemma]{Theorem}
\newtheorem{corollary}[lemma]{Corollary}
\newtheorem{proposition}[lemma]{Proposition}
\newtheorem{assertion}[lemma]{Assertion}
\newtheorem{complem}[lemma]{Complement}

\def\remark{\vskip.2cm \noindent{\bf Remark. }}
\def\endremark{\par\medskip}

\renewcommand{\epsilon}{\varepsilon}
\renewcommand{\emptyset}{\varnothing}
\renewcommand{\Re}{{\on{Re}\,}}
\renewcommand{\Im}{{\on{Im}\,}}

\newcommand{\im}{\Im}

\newcommand{\ov}{\overline}
\newcommand{\on}{\operatorname}
\newcommand{\ds}{\displaystyle}
\newcommand{\tend}{\longrightarrow}
\newcommand{\wh}{\widehat}
\newcommand{\entete}[1]{\medskip\textsf{#1}}
\newcommand{\keyw}[1]{\emph{#1}\xspace}

\newcommand{\cal}[1]{\mathcal{#1}}
\newcommand{\wt}[1]{\widetilde{#1}}
\newcommand{\setof}[2]{\big\{{#1}\,\big|\,{#2}\big\}}

\newcounter{rememberItem}

\def\bEA{\begin{eqnarray*}}
\def\eEA{\end{eqnarray*}}
\def\bEAn{\begin{eqnarray}}
\def\eEAn{\end{eqnarray}}

\def\C{{\mathbb C}}
\def\D{{\mathbb D}}
\def\H{{\mathbb H}}
\def\N{{\mathbb N}}

\def\R{{\mathbb R}}

\def\Z{{\mathbb Z}}

\newcommand{\st}{\mathrel{\ooalign{\hss$\supset$\hss\cr$\longrightarrow$}}}
\newcommand{\ag}[1]{\langle{#1}\rangle}

\newcommand{\hl}[2]{\operatorname{hlen}_{#1}(#2)}
\newcommand{\len}[2]{\operatorname{len}_{#1}(#2)}
\newcommand{\extent}[2]{\operatorname{extent}_{#1}(#2)}

\newcommand{\sub}[2]{{#1}\circledcirc{#2}}
\newcommand{\esub}[2]{{#1}\Vdash{#2}}

\newcommand{\at}{{\on{attr}}}
\newcommand{\rep}{{\on{rep}}}

\newcommand{\sclass}{{\cal S}}
\newcommand{\sch}{{\cal{SL}}}
\newcommand{\dom}{\on{Dom}}

\makenomenclature
\begin{document}

\nomenclature[Wz1]{$\Phi_{\at}$}{attracting Fatou coordinates; normalized and extended except at the beginning of \Cref{subsec:parabopt}; normalized by the expansion at infinity in \Cref{sec:pf}}
\nomenclature[Wz2]{$\Psi_{\rep}$}{repelling inverse Fatou coordinates; same remarks as for $\Phi_\at$ apply}
\nomenclature[aa]{$\cdots[f]$}{used to emphasize the dependence on $f$ of a given object}
\nomenclature[Dom]{$\dom(f)$}{domain of definition of the map $f$}

\title[Near parabolic renormalization for unicritical holomorphic maps]{Near parabolic renormalization for unicritical holomorphic maps}
\author{Arnaud Chéritat}
\address{Centre National de la Recherche Scientifique\\
Institut de Mathématiques de Toulouse, UMR 5219\\
Université de Toulouse}
\email{arnaud.cheritat@math.univ-toulouse.fr}
\thanks{This research was partially funded by the Agence Nationale de la Recherche, Grant ABC (At the Boundary of Chaos) ARN--08--JCJC--0002--01}

\abstract Inou and Shishikura provided a class of maps that is invariant by near-parabolic renormalization, and that has proved extremely useful in the study of the dynamics of quadratic polynomials. We provide here another construction, using more general arguments. This will allow to extend the range of applications to unicritical polynomials of all degrees. 
\endabstract

\maketitle

\tableofcontents

Notations: $\D$\nomenclature[D]{$\D$}{the open unit disk in $\C$} refers to the unit disk in the complex plane: $\D=\setof{z\in\C}{|z|<1}$ and $\H$\nomenclature[H]{$\H$}{the upper half plane in $\C$} to the upper half plane: $\H=\setof{z\in\C}{\Im z>0}$. The translation by $1$ in $\C$ is denoted by $T_1: z\mapsto z+1$.\nomenclature[T1]{$T_1$}{$z\mapsto z+1$}
By convention, $\N$ includes $0$ and we will denote $\N^*$ the set of positive integers.
Beyond its usual meaning as Archimedes' constant, the symbol $\pi$ will often refer to the canonical projection $\pi : \C \to \C/\Z$. We will often make use of the map $E(z)=e^{2i\pi z}$.\nomenclature[E]{$E$}{$E(z)=e^{2i\pi z}$}
We adopt the following convention for open and semi-open intervals: $]a,b[\,$, $[a,b[\,$, $]a,b]$. An upper half plane means a half plane (usually open) bounded by a horizontal line and sitting above it.
The restriction of a map $f$ to the set $A$ is denoted $f\big|_A$. The floor of $x\in\R$, i.e.\ the greatest relative integer $n\in\Z$ such that $n\leq x$, is denoted $\lfloor x\rfloor$.
The notation $\sch$ refers to the class of Schlicht maps, i.e.\ holomorphic injective maps $\phi:\D\to\C$ with $\phi(0)=0$ and $\phi'(0)=1$.\nomenclature[SL]{$\sch$}{the class of Schlicht maps\nomrefpage}
There are a lot of more specific notations in this article, and a (partial) summary of symbols has been added near the end.

Conventions: The hyperbolic metric on $\D$ is chosen to be $\frac{|dz|}{1-|z|^2}$, and the hyperbolic metric on open strict subsets $U$ of $\C$ is normalized according to this convention, i.e.\ it is the image of the metric of the disk by its identification with the universal cover of $U$. With that convention, the hyperbolic metric on $\H$ takes the form $|dz|/2\Im z$.
(Some authors prefer using $\frac{2|dz|}{1-|z|^2}$ on $\D$ so that one gets $|dz|/\Im z$ on $\H$.)

\section{Introduction}

This article has a long introduction and the main theorem appears only on page~\pageref{subsec:main}.

\subsection{Structural equivalence}\label{subsec:structeq}

In the breakthrough by Inou and Shishikura \cite{IS}, they make use of a class of maps defined as follows (notations and details may differ):
$\cal F_\textrm{IS}$ is the set of maps of the form $f=P\circ\phi^{-1}$ where $\phi$ varies among the univalent maps on $V$ such that $\phi(z)=z+\cal O(z^2)$ at the origin. Here $P(z)=z(1+z)^2$ and $V$ is a specific open subset of $\C$ containing $0$ defined in their article. The set $\cal F_\textrm{IS}$ is better thought of as the set of maps that cover the plane in a specific way, and with $f(z)=z+\cal O(z^2)$. They are not covers because they have ramification points. And they are not even ramified covers, because the cardinality of the preimage of a point is not constant, even when counted with multiplicity. So they are a sort of partial ramified covers.
This class $\cal F_\textrm{IS}$ comes in fact from another class of maps, invariant by parabolic renormalization (defined later in this section), with a much richer ramified cover structure, but which was too rigid for their purposes, which was to have a class invariant by \emph{near} parabolic renormalization.
They extracted a carefully chosen subset of this structure to define their class $\cal F_\textrm{IS}$.

We are going to use the same idea, but we will keep more of the original ramified cover structure. Let us formalize the notion of structure:

\begin{definition}
Let $X_1$, $X_2$, $Y$ be dimension one analytic manifolds\footnote{They are usually called \keyw{Riemann surfaces} when they are connected.}. Consider an index set $I$, and two collections of marked points $a_i\in X_1$ and $b_i \in X_2$ indexed by $i\in I$. Consider also two analytic maps which are nowhere locally constant $f_1:X_1\to Y$ and $f_2:X_2\to Y$. We will say that the pairs $(a,f_1)$ and $(b,f_2)$ are structurally equivalent if there exists an analytic isomorphism $\phi:X_1\to X_2$ such that $f_1=f_2\circ \phi$ and $b=\phi\circ a$ i.e.\ such that the following diagram commutes
\[\xymatrix@R=20pt@C=5pt{& I \ar[dl]_{a} \ar[dr]^{b} & \\X_1 \ar@{.>}[rr]^\phi \ar[dr]_{f_1} && X_2 \ar[dl]^{f_2} \\ & Y & }\]
i.e.\ such that $\phi$ sends the marked point $a_i$ to $b_i$ and such that it sends the fiber $f_1^{-1}(y)$ in the fiber $f_2^{-1}(y)$ for all $y\in Y$. Note that this requires that $f_2\circ b = f_1 \circ a$. Structural equivalence is an equivalence relation, which depends on $I$ and $Y$. To specify them, we will sometimes use the terminology $(I,Y)$-structurally equivalent or structurally equivalent over $Y$ with marker $I$.
The equivalence classes will be called \keyw{structures} (or $(I,Y)$-structures).
\end{definition}

We could also call this a \keyw{marked analytic partial ramified cover equivalence class} but it would be a long name for a simply defined notion.

The restriction of partial covers (without losing marked points) induces a preorder on structures as follows:
\begin{definition} With the same definition as above, but assuming $\phi$ analytic injective instead of analytic isomorphism (thus dropping the surjectivity assumption), we will say that the structure of $(a,f_1)$ is a \keyw{sub-structure} of that of $(b,f_2)$: this is indeed independent of the choice of representatives in their equivalence classes. 
We will also say that $(b,f_2)$ has \emph{at least} the structure of $(a,f_1)$.
It is equivalent to the following: $(a,f_1)$ is structurally equivalent to $(b,g_2)$ where $g_2$ is a restriction of $f_2$ to a set containing the image of $b$.
In other words sub-structures of $(b,f_2)$ are equivalence classes of restrictions of $f_2$ to open sets containing the marked points.
\end{definition}

This preorder is not always an order: for instance if $I=\emptyset$ , and the sets $X_1 \subset \C$ defined by $\Re(z)>0$ and $X_2$ defined by $\Re(z)>1/2$ are both mapped to $\C/\Z$ using the canonical projection from $\C$ to the quotient, then each has at least the structure of the other (take $\phi_1(z)=z+1$ and $\phi_2(z)=z$), while they are not equivalent.

However, on the subclass of structures with connected $X$ and at least one marked point, this preorder is an order:

\begin{proof} Assume each of $(a,f_1)$ and $(b,f_2)$ has at least the structure of the other and assume that both $X_i$ are connected and $I\neq\emptyset$. Call $\phi_1:X_1\to X_2$ and $\phi_2: X_2\to X_1$ the two analytic injections. We have to prove that $(a,f_1)$ is structurally equivalent to $(b,f_2)$. It is sufficient to prove that $\phi_2$ is surjective (the inverse of an analytic bijection is analytic). Call $\zeta=\phi_2\circ\phi_1$. It is injective, satisfies $f_1 \circ \zeta = f_1$ and fixes the marked points of $f_1$. The map $f_1$ being not locally constant at the marked points, each marked point has a neighborhood on which some iterate $\zeta^m$ of the map $\zeta$ is the identity, where $m$ is the local degree of $f_1$ at the marked point. Since there is at least one marked point and since $X_1$ is connected, $\zeta^m =\on{id}$ holds everywhere by analytic continuation. Hence $\phi_2$ is surjective. The proof is analogous for $\phi_1$.
\end{proof}

\subsection{Parabolic points}\label{subsec:parabopt}

The present section is given mainly to fix notations. The reader that does not already know the theory of parabolic fixed points of one dimensional holomorphic dynamical systems will have hard times understanding the article, we recommend learning it in any of the classic books introducing holomorphic dynamics, or in \cite{D,Z}. The article \cite{BE} is also instructive and very well illustrated. There is no claim that any of the statements given in this section is due to the author.

What is understood under the terminology \keyw{parabolic point} has  variations, according to whether or not linearizable maps are allowed, and according to whether the allowed values of multiplier should be $1$ or any root of unity. So here we will try and avoid solely mentioning parabolic points and use instead the following terminology:
\begin{itemize}
\item \keyw{Tangent to identity}: fixed point whose multiplier is equal to $1$.
\item \keyw{Rationally indifferent}: fixed (or periodic) whose multiplier is a root of unity.
\item \keyw{non-linearizable parabolic point}: irrationally indifferent fixed (or periodic) point which is not linearizable.
\end{itemize}
Non-linearizability is the condition to have petals. A parabolic point \keyw{with petals} will thus be a synonym for a non-linearizable parabolic point. 

Consider a holomorphic dynamical system, $f : \dom(f)\subset X\to X$.
Assume it contains a non-linearizable parabolic point of period one, rotation number $0$ and with one attracting petal, i.e.\ in some chart $f$ has expression $f(z)=z+a_2 z^2 +\cal O(z^3)$ with $a_2\neq 0$.
To this are associated attracting Fatou coordinates $\Phi_\at$ and repelling Fatou coordinates $\Phi_\rep$ and the local conjugacy invariant called horn maps. Let us quickly recall what these are.

\newcommand{\Pa}{\cal P_\at}
\newcommand{\Pe}{\cal P_\rep}

\begin{figure}
\begin{tikzpicture}[>=latex]
\node at (0,0) {\includegraphics[width=6cm]{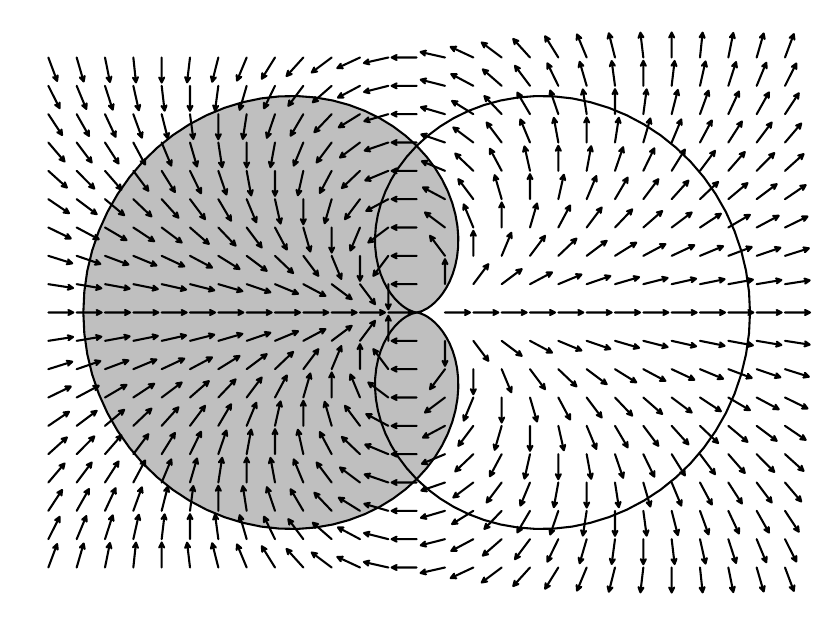}};
\node at (6,0) {\includegraphics[width=6cm]{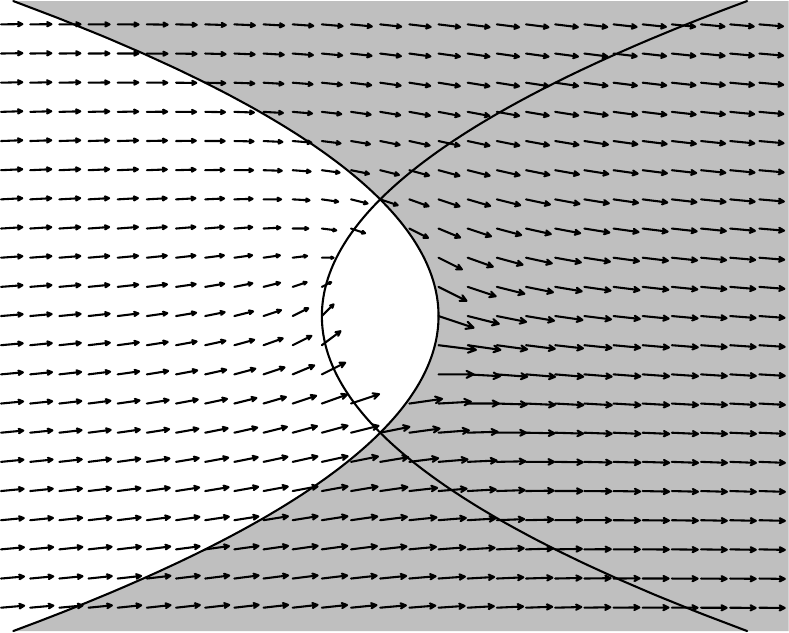}};
\node at (0,-5) {\includegraphics[width=6cm]{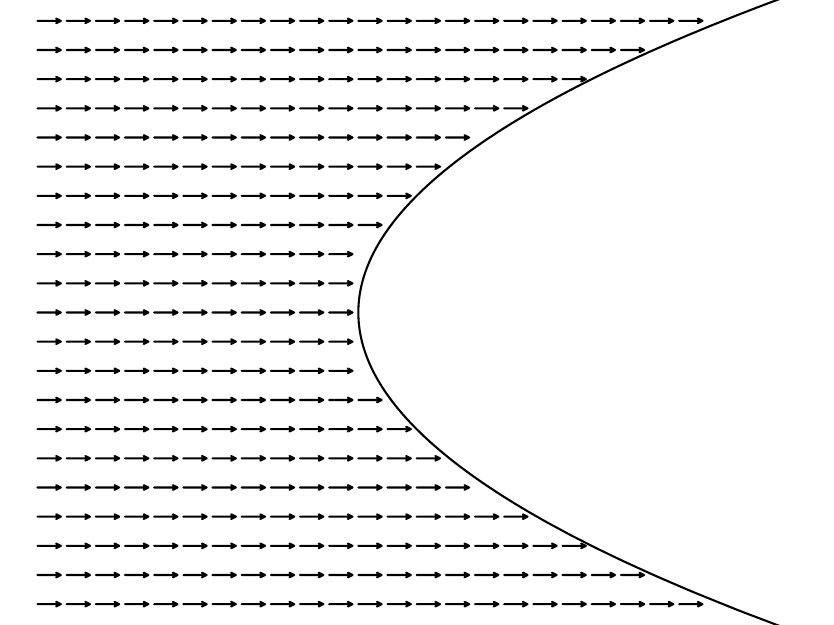}};
\node at (6,-5) {\includegraphics[width=6cm]{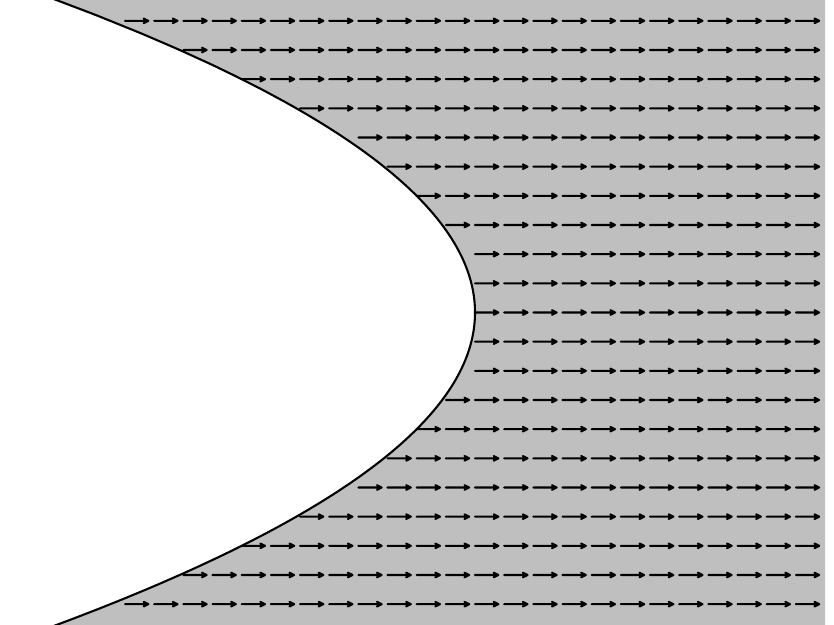}};
\end{tikzpicture}
\caption{Illustration of petals and Fatou coordinates (schematic). Parabolic petals may look like the upper left picture. An attracting petal has been colored in gray. The arrows indicate the direction but not the length (for more readability) of the vector $f(z)-z$: the latter gets quite small near the fixed point. Upper right: The map $z\mapsto -1/a_2z$ will send the petals to the two regions bounded by parabolic-like curves (and exterior to these curves). Lower left and right: the repelling and attracting Fatou coordinates conjugate $f$ to a translation and map both petals to two other regions that look very much like the upper right picture, but cannot be drawn on a same complex plane in a compatible way: this is precisely the  horn maps that tell how they glue.}
\label{fig:petals}
\end{figure}

\entete{Petals and Fatou coordinates:} There are various domains on which the Fatou coordinates are usually defined by different authors, but the following is certainly true: there exists a domain $\Pa$ that we will call \keyw{(wide) attracting petal} and a function $\Phi_\at:\Pa\to\C$ called \keyw{attracting Fatou coordinate} such that 
\begin{itemize}
\item $\Pa$ is open, non-empty, connected and simply connected,
\item $f(\Pa)\subset \Pa$,
\item $\forall z\in \Pa$, $f^n(z)\tend 0$ as $n\tend+\infty$,
\item conversely all orbits tending to $0$ eventually fall in $\Pa$ or on $0$, 
\item $\Phi_\at$ is injective on $\Pa$,
\item (wide) its image $\Phi_\at(\Pa)$ contains big sectors as follows: $\forall \epsilon$, $\exists R>0$,\\ $\setof{z\in\C}{|z|>R\text{ and }|\arg(z)|<\pi-\epsilon}\subset\Phi_\at(\Pa)$,
\item the following form of conjugacy holds:\\ $\forall z\in\Pa$, $\Phi_\at\circ f(z) = T_1\circ\Phi_\at(z)$ with $T_1(z)=z+1$, 
\item $\Phi_\at(z) \sim -1/a_2z$ as $z\tend 0$.
\end{itemize}
The repelling version of the Fatou coordinate is similar to the attracting version for a branch of $f^{-1}$ fixing the origin, with the difference that we ask $\Phi_\rep$ to be the composition of an attracting
Fatou coordinate for $f^{-1}$ followed by $z\mapsto -z$, so that it still conjugates $f$ to the translation by $1$.
Also $\Phi_\rep(z) \sim -1/a_2z$ as $z\tend 0$, exactly as $\Phi_\at$. The inverses $\Phi_\at^{-1}(z)$ and $\Phi_\rep^{-1}(z)$ also satisfy this equivalent, but as $z\tend \infty$. See \Cref{fig:petals} for an illustration.
The petals are not canonically defined: many variants exist satisfying the above conditions, many others exists satisfying other conditions, and it is not clear which definition should be preferred.

\entete{Normalization:} For all $c,c'\in\C$ the maps $\Phi_\at+c$ and $\Phi_\rep+c'$ satisfy the same properties.
Conversely there is a form of uniqueness:
Assume $U_1$ and $U_2$ are open sets, contained in $\dom(f)$,  $f(U_i)\subset U_i$, all points in $U_i$ have their orbit tending to $0$ and every orbit in $X$ tending to $0$ is eventually either equal to $0$ or contained in $U_i$. Assume that there are holomorphic functions (not assumed injective) $\Phi_i : U_i\to \C$ such that $\Phi_i(f(z)) = \Phi_i(z)+1$ holds on $U_i$. Then $U_1\cap U_2$ satisfies the same assumptions, in particular it is non-empty, and there exists a constant $c\in\C$ such that $\Phi_2(z) = \Phi_1(z) +c$ holds on $U_1\cap U_2$. In particular, if one takes $\Phi_2=\Phi_\at$ and $U_2=\Pa$, we see that $\Phi_1$ must be equal to $c+\Phi_\at$ on the non-empty set $U_1 \cap \Pa$.
A similar statement holds for the repelling Fatou coordinate. So in some sense the Fatou coordinates are \emph{unique up to addition of a constant}. 
So Fatou coordinates come in classes parameterized by a complex number. The choice of an element in a class is called a \keyw{normalization}.

\entete{Extension:} There exists a unique extension of the attracting Fatou coordinate $\Phi_\at$ on the basin of the parabolic point, such that
\[\Phi_\at\circ f = T_1\circ \Phi_\at.\]
Here, we mean in particular that the two compositions have the same domain of definition, which is possible iff the domain of $\Phi_\at$ is the whole basin.
The extension can be defined as follows: let $\cal P_\at$ be an attracting petal on which a Fatou coordinate $\Phi$ is defined. For all $z$ such that there exists $n\in\N$ with $f^n(z)\in \cal P_\at$, the quantity $\Phi(f^n(z))-n$ is independent of $n$ and we define $\Phi_\at(z)$ to be this complex number.
The extended attracting Fatou coordinate plays the role of a greatest element in the set of attracting coordinates\footnote{For this to be correct we in fact set up an order relation on classes of Fatou coordinates, where equivalence is up to addition of a constant, define the order as inclusion of the domain of definition, and define Fatou coordinates as maps satisfying the weak assumptions given in the Normalization paragraph: i.e.\ we have at least to drop the injectivity assumption, as the greatest element will usually not satisfy it, and may also drop the big sectors assumption, though it is one that the greatest element does satisfy.}. It is not necessarily injective. In the cases we will look at, it will not be. If so, the relation above is not a conjugacy but a semi-conjugacy.

There is no similar maximal element for the repelling Fatou coordinates. Instead, there exists a unique extension of the reciprocal $\Psi_\rep=\Phi_\rep^{-1}$ such that
\[\Psi_\rep\circ  T_1 = f \circ \Psi_\rep.\]
Again, we want the domains of both compositions to be equal.
The definition is similar: let $\cal P_\rep$ be a repelling petal. For all $z\in\C$, there exists $n\in\N$ such that $z-n\in \Phi_\rep(\cal P_\rep)$. The existence and the value of the quantity $f^n(\Phi_\rep^{-1}(z-n))$ is independent of $n\geq 0$, and this defines $\Psi_\rep(z)$.
It is again holomorphic and not necessarily injective.

If $f$ is a global map (a map whose orbits are all defined for all times, like a polynomial, an entire map, a rational map, \ldots) then $\Psi_\rep$ is defined on the whole complex plane $\C$.

\entete{extended horn maps and parabolic renormalization:}\nomenclature[R]{$\cal R[f]$}{the upper parabolic renormalization of $f$\nomrefpage}\nomenclature[Hf]{$\cal H[f]$}{the horn map of $f$, semi-conjugatedby $E$\nomrefpage} The \keyw{extended horn map} is the composition
\[h = \Phi_\at\circ\Psi_\rep\]
of these extensions. Changing the normalizations of the Fatou coordinates replaces $h$ with its pre composition and post composition with two unrelated translations.

To define a renormalization, we proceed as follows.
This definition does not pretend to be the best one, it is well suited to our purposes. 
The map $h$ commutes with $T_1$ and its domain of definition is $T_1$-invariant and contains an upper and a lower half plane. There is thus a quotient map $\dom(h)/\Z \to \C/\Z$. Conjugate it by $E:z\mapsto e^{2i\pi z}$ to a map defined on an open subset of $\C^*$ containing a neighborhood of $0$ and $\infty$. With the properties of Fatou coordinates one proves that it can be continuously (and thus holomorphically) extended at these points, and that the extension fixes $0$ and $\infty$. We will denote $\cal H$ this extension, or $\cal H[f]$ to emphasize its dependence on $f$.
For the \keyw{upper parabolic renormalization} of $f$, denoted $\cal R[f]$, consider the restriction of this extension to the connected component of its domain of definition that contains $0$, and possibly pre and post compose it with two linear maps ($z\mapsto az$ and $z\mapsto bz$) to be chosen according to conventions.
For the lower parabolic renormalization of $f$, conjugate first the extension by $z\mapsto 1/z$, then restrict it to the connected component of the domain of definition containing $0$ and finally compose with linear maps. 
The reason why we allow for these linear maps is that we will find it convenient later to use a different normalization for parabolic renormalization than for Fatou coordinates and the associated horn map.

\entete{Another point of view on extended horn maps, and parabolic renormalization:}
Since $\Phi_\at$ and $\Psi_\rep$ are defined beyond the petal $\cal P_\at$ and beyond $\Phi_\rep(\cal P_\rep)$ by using iteration of $f$, the definition of $h[f]$ can be reformulated as follows:
\begin{itemize}
\item for $\zeta\in\dom(h[f])$, there exists $n\in\N$ such that $\zeta-n \in \Phi_\rep(\cal P_\rep)$,
\item $\zeta-n=\Phi_\rep(z)$ for a unique $z\in\cal P_\rep$, 
\item there exists $m\in\N$ such that $f^m(z)\in\cal P_\at$,
\item $h(\zeta)=\Phi_\at(f^m(z))-m+n$.
\end{itemize}
We have illustrated a possible orbit on \Cref{fig:decomp}.

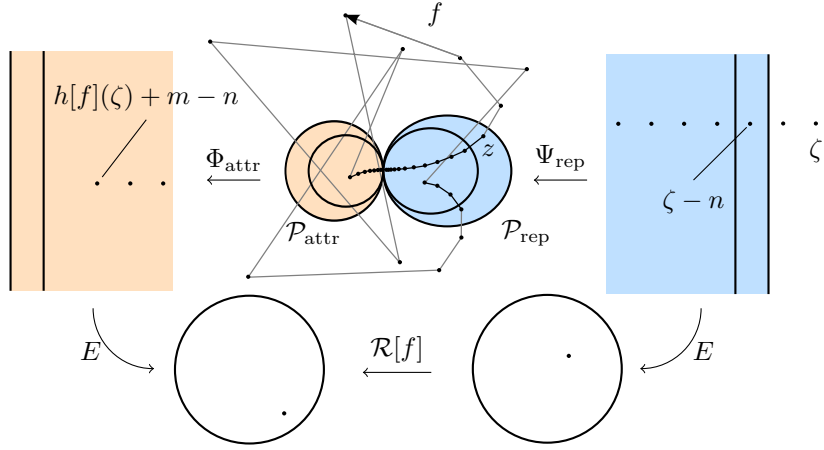
\begin{figure}
\begin{tikzpicture}
\node at (0,-0.12) {\includegraphics[width=10cm]{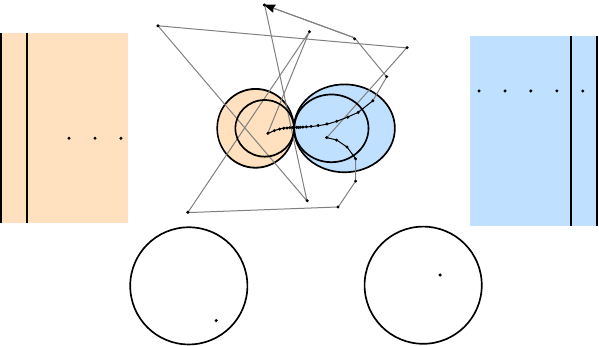}};
\draw[->] (0.55,-2.05) -- node[above] {$\cal H[f]$} (-0.35,-2.05);
\draw[->] (4.1,-1.2) to[out=-90,in=0] node[below,right] {$E$} (3.3,-2);
\draw[->] (-3.9,-1.2) to[out=-90,in=180] node[below,left] {$E$} (-3.1,-2);
\draw[->] (2.6,0.5) -- node[above] {$\Psi_\rep$} (1.9,0.5);
\draw[->] (-1.7,0.5) -- node[above] {$\Phi_\at$} (-2.4,0.5);
\node at (0.6,2.7) {$f$};
\node at (1.3,0.9) {$z$};
\filldraw (5.61,1.25) circle [radius=0.02];
\filldraw (5.18,1.25) circle [radius=0.02];
\node at (5.61,0.9) {$\zeta$};
\draw[-] (4.0,0.5) node[below] {$\zeta-n$}  -- (4.7,1.2);
\node at (1.8,-0.2) {$\cal P_\rep$};
\node at (-1,-0.2) {$\cal P_\at$};
\draw[-] (-3.2,1.3) node [above] {$h[f](\zeta)+m-n$} -- (-3.78,0.54);
\end{tikzpicture}
\caption{Decomposing $h[f]$ and $\cal H[f]$. For convenience, we have chosen petals $\cal P_\at$ and $\cal P_\rep$ whose image in Fatou coordinates are right and left half planes. Note that the orbit may visit the repelling petal more than one time, and does not necessarily enter the attracting petal by its leftmost part (the crescent shaped fundamental domain).}
\label{fig:decomp}
\end{figure}

\entete{The iterative residue:} Let
\[ f(z)  =  z +a_2z^2+a_3 z^3 +\ldots \]
be the power series expansion of $f$. The iterative residue of $f$ is the quantity $\gamma = 1- \frac{a_3}{a_2^2}$. It is related to the residue at $0$ of the meromorphic form $\frac{dz}{f(z)-z}$ by the following formula: $\frac{1}{2\pi i}\oint \frac{dz}{f(z)-z} = \gamma-1$. In fact the (multivalued near the origin) primitive $\int \frac{dz}{f(z)-z}+\frac{dz}{z}$ turns out to be an interesting approximation of the Fatou coordinates, as their expansions share the same first two terms: as $z$ tends to $0$ within a closed sector avoiding the repelling axis for $\Phi=\Phi_\at$ or the attracting axis for $\Phi=\Phi_\rep$:
\[ \Phi(z) = \frac{-1}{a_2 z} - \gamma \log z + \on{constant} + o(1).
\]
Another characterization is in terms of the horn map: there are expansions
\[\begin{array}{rcll}
h(z)&=&z+a_{\on{up}}+o(1)&\text{as }\im(z)\to+\infty\\
h(z)&=&z+a_{\on{down}}+o(1)&\text{as }\im(z)\to-\infty
\end{array}\]
The constants $a_{\on{up}}$ and $a_{\on{down}}$ depend on the normalization of Fatou coordinates, but not the quantity $a_{\on{up}}-a_{\on{down}}$. It turns out that \[ a_{\on{up}}-a_{\on{down}} = -2\pi i \gamma
. \]
Interestingly, if we consider the horn map with the normalization number~\ref{item:nor:2} presented below, then $a_{\on{up}}=-\pi i\gamma$ and $a_{\on{down}}=\pi i\gamma$.

\entete{Some normalizations:} 
We will give here three examples of normalizations for the upper parabolic renormalization $\cal R[f]$ of $f$. The first two work well for germs,\footnote{We use the word \keyw{germ} in the following meaning: an equivalence class of holomorphic maps defined near the origin, with $f\sim g$ if they coincide in some neighborhood of $0$. This is equivalent to $f$ and $g$ having the same power series expansion at the origin.} the third makes strong structural assumptions on $f$. Recall $\cal H[f]$ denotes the semi-conjugate of the horn map by the map $E:z\mapsto e^{2\pi i z}$, extended at $0$ and $\infty$ by fixing them, and that
\[\cal R[f] = A\circ \cal H^r[f] \circ B (z)\]
for some linear maps
\bEA
A: z & \mapsto & az \text{ and}
\\ B: & z\mapsto & bz,
\eEA
where $\cal H^r[f]$ denotes the restriction of\footnote{For lower renormalization instead of upper, replace $\cal H[f]$ with $s\circ \cal H[f]\circ s$ where $s(z)=1/z$.} $\cal H[f]$ to the component containing $0$ of its domain.
Let
\bEA  f(z) & = & z +a_2z^2+a_3 z^3 +\ldots
\\ \cal H[f](z) & = & b_1 z +b_2z^2+\ldots
\\ \cal R[f](z) & = & b'_1 z +b'_2z^2+\ldots
\eEA
be their power series expansions. We have $b_1\in \C^*$, and $b_2\in \C$. The constants $b_1'$ and $b_2'$ can are expressed from $a,b,b_1$ and $b_2$ as follows $b_1'=abb_1$ and $b_2'=ab^2b_2$.
Here are our examples of normalizations:
\begin{enumerate}
\item By imposing $b'_1=1$ and $b'_2=1$: this first approach is easier but assumes that $b_2\neq 0$. Then there is a unique pair of linear maps $A,B$ such that $\cal R[f](z) = z + z^2 + \cal O(z^3)$.
\item\label{item:nor:2} By normalizing the expansion of the Fatou coordinates and taking $\cal R = \cal H$: Fatou coordinates are unique up to addition of a constant. Moreover, the following limited expansion is valid (even though there is not a \emph{convergent} power series expansion in general): on all closed sectors avoiding respectively the repelling and the attracting axis, we have, as $z\to 0$:
\bEA
\Phi_\at(z) &=& \frac{-1}{a_2 z} - \gamma \log_p \frac{-1}{a_2 z} + \on{constant} + o(1)
\\
\Phi_\rep(z) &=& \frac{-1}{a_2 z} - \gamma \log_p \frac{1}{a_2 z} + \on{constant} + o(1)
\eEA
where $\log_p$ denotes the principal branch of the logarithm. The normalization just consists in adding constants to both Fatou coordinates so as to cancel the two constants in the above expansions. This normalizes $h=\Phi_\at\circ \Psi_\rep$ and we then choose  $\cal R [f] = \cal H[f]$.
Note that with this normalization, 
\bEA
& & h(z) = z - i\pi \gamma + o(1)\text{ as } \im z\to+\infty\text{ and} 
\\ & & h(z) = z + i\pi \gamma + o(1)\text{ as }\im z\to-\infty.
\eEA
where $\gamma$ is the iterative residue.
\item\label{item:nor:3} By the singular value: we will meet later in this article a class of maps whose renormalizations have a unique critical value.\footnote{or a unique non-zero singular value}
Fix a preferred complex number $v\in\C^*$. We then choose the linear map $A$ so as to place the critical value of $\cal R[f] = A\circ \cal H^r[f] \circ B$ at $v$ and then $B$ so that $A\circ \cal R[f] \circ B$ has derivative $1$ at the origin.
\end{enumerate}
Each of these conventions has its own advantages. Conventions number~1 and~3 have the property that $\cal R [g\circ f\circ g^{-1}] = \cal R[f]$ in a neighborhood of $0$ for all holomorphic maps $g$ fixing the origin with $g'(0)\neq 0$. They also give back a germ $\cal R[f]$ tangent to the identity. Number~2 does not necessarily, but it is defined for all $f$. We will work with a class of maps satisfying number~3. Our choice in most of the article will be to normalize Fatou coordinates, the horn map and $\cal H[f]$ according to number~2, and the parabolic renormalization $\cal R[f]$ according to number~3.

\subsection{What are horn maps good for?}

Horn maps occur in at least two ways: 
\begin{itemize}
\item First as local conjugacy invariants. A complete local conjugacy invariant of a non-linearizable parabolic germ with one attracting petal is more or less given by the data of the pair of germs of its horn maps at both ends of the cylinder (see \cite{V} for precise statements; \cite{MR} gives an interesting equivalent point of view).
\item Second as limits of cylinder renormalization. If a sequence of maps $f_n$ tends to $f$ and fix the origin with multiplier $\lambda_n$ and if $2\pi i/(\lambda_n-1) = N_n + a +o(1)$ with $N_n\in\Z$, $N_n\tend \pm\infty$ and $a\in\C$, under some mild supplementary assumptions, the fixed point of $f$ at the origin is the limit of a pair of fixed points of $f_n$, the origin and another one, and is possible to draw crescent shaped domains with tips at these two fixed points delimited by a curve $C_n$ and its image $f_n(C_n)$. The quotient of this domain by identifying $z\in C_n$ with $f_n(z)$ is isomorphic as a Riemann surface to the cylinder $\C/\Z$. The first return map from the cylinder to itself then tends, as $n\tend +\infty$, to the horn map (up to pre and post composition with translations). See \cite{L,D,S2,S,IS}.
\end{itemize}

The second point justifies why it makes sense to iterate horn maps.

A very important application comes from Lavaurs' theorem: let $\sigma\in\C$ and let the Lavaurs map $g_\sigma$ be defined as
\[g_{\sigma} = \Psi_\rep\circ T_\sigma \circ \Phi_\at\]
where $T_\sigma(z)=z+\sigma$.
Then under the same assumptions as above, $f_n^{N_n}\tend g_{\sigma}$ for some value of $\sigma$ that depends on $a$ (and on the chosen normalizations of the Fatou coordinates). This is why the Lavaurs maps are also called \keyw{geometric limits} by analogy with the field of Kleinian groups.
Application of Lavaurs's theorem include parabolic enrichments (understanding the Hausdorff limits of Julia sets of a sequence of polynomials tending to one with a non-linearizable parabolic point), non local connectedness of some bifurcation loci, and several discontinuity theorems.

Now horn maps are closely related to Lavaurs maps because each are semi-conjugate of the other. More precisely, consider the following non-commuting diagram:
\begin{center}\begin{tikzpicture}[>=latex]
\node (u) at (0,1.1) {$\C$};
\node (l) at (-.7,0) {$\C$};
\node (r) at (.7,0) {$\C$};
\draw[->] (r) -- node[below] {$T_\sigma$} (l);
\draw[->] (l) -- node[left] {$\Psi_\rep$} (u);
\draw[->] (u) -- node[right] {$\Phi_\at$} (r);
\end{tikzpicture}\end{center}
The map $g_\sigma$ is the composition obtained by starting from the top node, and following the arrows in a loop back to the starting node. The map $h_\sigma := T_\sigma\circ h$ is the same but starting from the lower left corner.

\begin{center}\begin{tikzpicture}[>=latex]
\node (u) at (0,1.8) {$\C$};
\node (l) at (-1.1,0) {$\C$};
\node (r) at (1.1,0) {$\C$};
\draw[->,style=dashed,rounded corners=20pt] (0.07,1.4) -- (0.8,0.15) --  (-0.8,0.15) -- (-0.07,1.4);
\draw[->,style=dashed,rounded corners=30pt] (-1.2,0.4) -- (0,2.6) -- (1.6,-0.3) -- (-0.8,-0.3);
\node at (0.05,0.5) {$g_\sigma$};
\node at (0.8,1.9) {$h_\sigma$};
\draw[->] (r) -- node[below] {$T_\sigma$} (l);
\draw[->] (l) -- node[left] {$\Psi_\rep$} (u);
\draw[->] (u) -- node[right] {$\Phi_\at$} (r);
\end{tikzpicture}
\end{center}

Following one resp.\ two arrows from one corner to another gives a semi-conjugacy from $h_\sigma$ to $g_\sigma$ resp.\ from $g_\sigma$ to $h_\sigma$. 
The first advantage of horn maps over Lavaurs maps is that they are easier to understand and have better covering properties in many applications (the best is to project the extended horn maps, they commute with $T_1$, down to a dynamical system on $\C/\Z$). From this stems a second advantage:  the invariance under parabolic renormalization of some classes of maps, as explained in the following sections.

\subsection{A reminder about singular values of maps}

Let $f : X \to Y$ be a holomorphic map where $X$ and $Y$ are Riemann surfaces. 
Let us recall that a \keyw{singular value} of $f$, as a map from $X$ to $Y$, is an element $z\in Y$ which has no open neighborhood over which $f$ is a cover\footnote{I.e.\ there is no open subset $V$ of $Y$ containing $z$ s.t.\ $f$ is a cover from $f^{-1}(V)$ to $V$. The definition is equivalent if we consider only neighborhoods $V$ of $z$ homeomorphic to disks.}.
Every critical value is singular, as is every asymptotic value\footnote{A point $z\in Y$ is an asymptotic value whenever there exists a continuous path  $\gamma : [0,t[ \to X$ that leaves every compact of $X$ and whose image by $f$ tends to $z$.}, and it is a simple yet very useful theorem that the set of singular values is the closure of the set of all critical and asymptotic values (see for instance\footnote{The language is slightly different in \cite{Er} since he calls singular values the critical or asymptotical ones. But his Proposition~1 amounts our claim.} \cite{Er} or Corollary~2.7 in \cite{RS}).

It shall be noted that restricting the domain of a map will likely introduce a lot of singular values: if $U\subset X$, every point in $f(\partial U)$ will be a singular value of $f$ as a map from $U$ to $Y$. In fact:
\begin{lemma}[folk.]\label{lem:singr}
Let $F:X\to Y$ be holomorphic and denote $S$ its set of singular values. Assume $A\subset X$ and $B\subset Y$ are open $X$ and that $f(A)\subset B$. Then the set of singular values of the restriction $\tilde f : A\to B$ of $f$ contains $B\cap f(\partial A)$ and is contained in $f(\partial A)\cup S$.
\end{lemma}
\begin{proof}
First inclusion: The set of points in $\partial A$ that are accessible from $A$ is dense in $\partial A$. If $b=f(a)\in B$ with $a\in\partial A$ then $a$ is the limit of $a_n\in\partial A$ which is accessible and $f(a_n)$ is an asymptotic value of $\tilde f$ and tends to $b$, hence $b$ is a singular value.

Second inclusion: it is enough to prove it for critical values and critical points, since the set of singular values of $\tilde f$ is the closure of their union.
All critical values of $\tilde f$ are of course critical values of $f$, hence in $S$. Consider an asymptotic value $b$ of $\tilde f$ and let $\gamma:{[0,1[}\to A$ with $f\circ\gamma(t)\underset{t\to 1}\tend b$ and $\gamma(t)$ leaves every compact of $A$. If the set of accumulation points of $\gamma$ in $X$ contains more than one point, then $f$ must be constant on the connected component of $\dom f$ containing $\gamma$, and then $b$ is a singular value. Otherwise either $\gamma$ leaves every compact of $X$, and then $b$ is an asymptotic value of $f$ hence in $S$, or $\gamma$ converges to a point $a\in \partial A$, whence $b=f(a)\in f(\partial A)$.
\end{proof}

Similarly, enlarging the range $Y$ of $f:X\to Y$ will introduce singular values at boundary points.

As a corollary of \Cref{lem:singr}, if we restrict $f$ to a parabolic immediate basin, we do not introduce new singular values:

\begin{lemma}[folk.]\label{lem:singra}
If $f:U\subset X\to X$ is a holomorphic map with a non-linearizable parabolic fixed point $p$, and if $A$ denotes the union of a cycle of immediate basins of $p$, then the set of singular values of $f|_A : A\to A$ is contained in the intersection of $A$ with the set of singular values of $f$.
\end{lemma}
\begin{proof}
Indeed $f(\partial A)\cap A = \emptyset$ (here $\partial$ is relative to $U$). \end{proof}

%
%

\subsection{Universality and maps with all ``the'' structure}\label{subsec:all}

For $d \geq 2$ an integer,  let
\[B_d(z) = \left(\frac{z+a}{1+a z}\right)^d\text{ with }a = a_d = \frac{d-1}{d+1}.\]
Let\nomenclature[Bd]{$B_d$}{a unicritical Blaschke product with a parabolic point at $z=1$\nomrefpage}
\[B_\infty(z) =\exp\left(2\frac{z-1}{z+1}\right).\]
They induce unisingular self maps of $\D$ with a unique singular value $z=0$ in $\D$ and they have a non-linearizable parabolic fixed point on the boundary at $z=1$ with two attracting petals. 
Interestingly:
\[B_d\underset{d\to+\infty}\tend B_\infty\]
uniformly on compact subsets of $\D$.

The unit disk is the (immediate) basin of one of the two petals. The inverse of the unit disk is the basin of the other.
We let $\Phi_\at[B_d]: \D\to\C$ be the extended attracting Fatou coordinate for the first petal. The map has also two repelling petals, with vertical axes. We choose the one on the top and let $\Psi_\rep[B_d]$ denote the corresponding extended repelling Fatou coordinate. We let $h[B_d]=\Phi_\at\circ\Psi_\rep$. It is defined on an upper half plane.

The following theorem is an interpretation of a classical theorem of Fatou.

\begin{theorem}[Fatou$+$folk]\label{thm:fat}
Let $f : U\subset \wh{\C} \to\wh {\C}$ a holomorphic map with a non-linearizable parabolic fixed point. Let $A$ be a cycle of immediate  parabolic basins associated to this fixed point. Then 
\begin{itemize}
\item either $U=\wh{\C}$ and $f$ is a homography
\item or there is a singular value in $A$ of the restriction of $f$ to $A$.
\end{itemize}
\end{theorem}
\begin{proof}(Sketch)
Assume that the second point does not hold.
Since the set of singular values is the closure of the union of the set of critical values of and asymptotic values, it follows that $A$ contains no asymptotic nor critical values of $f|_A$. In particular there is no critical point of $f$ in $A$ so $\Phi_\at$ has no critical point either.

From the absence of singular value one get the following path-lifting property:
Given a close-ended path $\gamma:[0,1]\to A$ and a point $a\in A$ such that $f(a) =\gamma(0)$, there exists a lift $\tilde \gamma$ by $f$ that starts from $a$ and such that $\tilde\gamma([0,1])\subset A$.

Let $\Phi_\at: A\to\C$ be the attracting Fatou coordinate extended to $A$. One can then prove that $\Phi_\at$ also has a path lifting property: Given a path $\gamma:[0,1]\to \C$ and a point $a\in A$ such that $\Phi_\at(a) =\gamma(0)$, there exists a lift $\tilde \gamma$ by $\Phi_\at$ that starts from $a$ and such that $\tilde\gamma([0,1])\subset A$.
The proof consists in finding $n\geq 0$ such that $f^n(a)$ is in an attracting  petal $\cal P$ and $n+\gamma$ is contained in $\Phi_\at(\cal P)$, then applying the first path lifting proprery $n$ times.

One deduces that $\Phi_\at$ has no asymptotic values. Since $\Phi_\at$ cannot have critical values either, it has no singular values and hence it is a covering from $A$ to $\C$. Since $\C$ is simply connected, each connected component of $A$ is conformally isomorphic to $\C$. So each component is $\wh{\C}$ minus one point, so there is only one such component and the parabolic point $p$ is the missing point, so $U=\wh{\C}$. The isomorphism $\Phi_\at$ is a homography, and $f$ restricted to $\C\setminus\{p\}$ is the conjugate of a translation by this homography, hence $f$ is a homography of $\wh{\C}$.
\end{proof}

\begin{complem}
In the second case of \Cref{thm:fat}, at least one of the two statements below is true
\begin{itemize}
\item there is a critical point of $f$ in $A$,
\item or there is an open ended path in $\gamma:[0,1[\to A$ which leaves every compact of $U$ and whose image $f\circ \gamma$ tends to a point of $A$.
\end{itemize}
\end{complem}
\begin{proof}
Indeed the set of singular values of $f|_U$ is the closure of the set of its critical values and of its asymptotic values. In the case of an asymptotic value $v$, we will repeat here the analysis done in the proof of \Cref{lem:singr}: there is a path $\gamma:{[0,1[}\to A$ that leaves every compact of $A$ and with $f\circ \gamma(t)\tend v$ as $t\to 1$. Such a path must also leave every compact of $U$ for otherwise:

-- either $\gamma(t)$ converges in $U$ to a point $a$ that must then be in $\partial A$ but also must satisfy $f(a)\in A$, hence a whole neighborhood of $a$ in $U$ is contained in the basin, but it also contains points of $A$ so $a\in A$, leading to a contradiction.

-- or $\gamma(t)$ has an accumulation set that is bigger than one point. But since $f$ is holomorphic and nowhere constant, this would contradict that $f\circ \gamma(t)$ converges.
\end{proof}

The following theorem treats the case when there is only one such singular value.

\begin{theorem}[folk]\label{thm:s1}
Let $f$ be as in \Cref{thm:fat} and $A$ a cycle of immediate basins of its parabolic fixed point. Denote by $p$ the period of $A$.
Assume that one and only one singular value of the restriction of $f$ to $A$ lies in $A$. Then the restriction of $f^p$ to any connected component of $A$ is analytically conjugated to the restriction of $B_d$ to $\D$ for some $d\in \{2, 3, \ldots\}\cup\{\infty\}$.
\end{theorem}

See for instance \cite{DH}, exposé IX for a similar statement. We will be mainly interested by the case $p=1$. This has the following consequences, discovered by several authors, including Shishikura (see \cite{S2}), and Lanford and Yampolsky (see \cite{LY}). See also Part~3 of \cite{Che1}.

\begin{corollary}[S., L.-Y.]\label{cor:univ}
With the same notations, call $\zeta: A\to \D$ the conjugacy from $f$ to $B_d$. Then there exists a constant $\tau\in\C$ (which depends on the normalizations of the Fatou coordinates) such that $\Phi_\at[B_d] \circ \zeta = \tau+\Phi_\at[f]$, where the right hand side is restricted to $A$.
\end{corollary}

Thus in particular, using the terminology introduced here, $\tau+\Phi_\at[f]$ restricted to $A$ is structurally equivalent to $\Phi_\at[B_d]$ over $\C$.
This is illustrated on \Cref{fig:univ_rabbit,fig:univ_blaf,fig:univ_exp_1,fig:univ_exp_2}, using a widespread visualization technique explained in \Cref{sec:viz}.

\medskip

Below is a theorem specifying some structure of renormalizations of maps of \Cref{thm:s1}. For simplicity we restrict to the case with only one attracting petal. For this statement and the subsequent one, we will denote $\sclass_d$ the set of maps $f$ as in \Cref{thm:s1} that satisfy its conclusion with this value of $d$ and that have only one attracting petal.\nomenclature[Sd]{$\sclass$}{parabolic renormalization invariant class with the full structure\nomrefpage}
In other words:

\begin{definition}\label{def:sclass}
Let $d\in \{2, 3, \ldots\}\cup\{\infty\}$. We denote $\sclass_d$ the set of holomorphic maps $f : U\subset \wh{\C} \to\wh {\C}$ with a non-linearizable parabolic fixed point with only one attracting petal, such that if we denote $A$ its immediate parabolic basin, then there is one and only one singular value in $A$ of the restriction of $f$ to $A$, and such that:
\begin{itemize}
\item (if $d=\infty$) this singular value is an asymptotic value of $f|_A$ or
\item (if $d<\infty$) there is a critical point of $f|_A$ of degree $d$.
\end{itemize}
\end{definition}

\noindent According to \Cref{thm:s1} these cases are mutually exclusive.

There is no specification in the statement below of which normalization is chosen for the parabolic renomalization $\cal R[f]$ of $f$.

\begin{theorem}[S., L.-Y.]\label{thm:shi2a}
Consider a map $f\in \sclass_d$.
Then the upper or lower renormalization $\cal R[f] : V \to \wh{\C}$ is defined on a simply connected subset of $\C$ and has exactly $3$ singular values: the asymptotic values $z=0$, $z=\infty$ and a critical value $z=v\in\C^*$ if $d$ is finite or an asymptotic value $z=v\in\C^*$ otherwise. We have $\cal R[f]^{-1}(\{0\})=\{0\}$, $\cal R[f]^{-1}(\{\infty\})$ is empty. 
\begin{itemize}
\item If $d<\infty$, the set $\cal R[f]^{-1}(\{v\})$ consists in regular points and critical points of degree $d$, and $v$ is not an asymptotic value.
\item If $d=\infty$, the set $\cal R[f]^{-1}(\{v\})$ consists only in regular points, and $v$ is an asymptotic value.
\end{itemize}
\end{theorem}

The following lemma is not optimal but it will be convenient for future reference.

\begin{lemma}\label{lem:sufsd}
If $v\in\C^*$ and $f:U\to\hat \C$ is such that:
\begin{itemize}
\item The set of singular values of $f$ is equal to $\{0,v,\infty\}$,
\item $f(0)=0$ and $f'(0)=1$,
\item $U\subset \C$.
\end{itemize} 
Then $f\in\cal S_d$ for some $d\in \{2, 3, \ldots\}\cup\{\infty\}$.
\end{lemma}
\begin{proof}
We first prove by contradiction that $f$ cannot be the identity near $0$.
Otherwise it would be the identity on the component of $U$ containing $0$ by analytic continuation.
Then $\partial U$ would be contained in the set of singular values of $f$.
Hence $\partial U\subset \{v,\infty\}$. But then $0$ cannot be a singular value, leading to a contradiction.
Hence $f$ it has a non-linearizable parabolic point at $0$: it has petals.

Since $f'(0)=1$ the petals have period one. We then prove that $f$ can have only one attracting petal. Let $A$ be the immediate basin of an attracting petal.
By \Cref{lem:singra}, the restriction of $\tilde f:A\to A$ has a set of singular values contained in that of $f$. But an immediate basin of $f$ cannot contain $0$, which is fixed, nor $\infty$, which is not in the domain of $f$, so only $v$ is available as a singular value. Moreover different immediate basins being disjoint, the singular values that each contains must be distinct. Hence there can only be one immediate basin, and it contains $v$ but not $0$ nor $\infty$.
\end{proof}

This applies in particular to the map $\cal R[f]$ of \Cref{thm:shi2b}. Hence we get for each value of $d$ a class that is stable by parabolic renormalization: 

\begin{corollary}[S., L.-Y.]\label{cor:F}
Consider a parabolic renormalization $\cal R[f]$ of a map $f\in \sclass_d$. If $\cal R[f]$ is normalized so that $\cal R[f](0)=1$ then $\cal R[f]\in \sclass_d$ too.
\end{corollary}

The following statement complements this.

\begin{theorem}[S., L.-Y.]\label{thm:shi2b}
Choose and fix any $v\in\C^*$.
All the maps $f\in\sclass_d$ have structurally equivalent upper renormalizations $\cal R[f]$, when the latter is normalized so that $v$ is a singular value ($0$ and $\infty$ necessarily are). More precisely they are $(I,Y)$-structurally equivalent with $Y=\C$, $I$ being a singleton and the marked point being the origin.
The same holds for the lower renormalization, and the upper one is structurally equivalent to the conjugate of the lower one by the reflection $z\mapsto 1/\ov{z}$.
\end{theorem}

This is illustrated in \Cref{fig:th2:1,fig:th2:2}.

By \Cref{cor:F} the two structural equivalence classes mentioned in \Cref{thm:shi2b} are stable by upper, resp.\ lower, parabolic renormalization. 
For what we are concerned with in this article, this is the base of everything. 

\medskip

For later reference, let us mention the following universality statement for the repelling extended Fatou coordinates: 
\begin{complem}[S., L.--Y.]
Recall $h=\Phi_\at\circ\Psi_\rep$.
Let $f$ be a holomorphic map as in \Cref{thm:shi2b}. Let $U[f]$ denote the component of the domain of $h[f]$ that contains an upper (resp.\ a lower) half plane. (Up to a complex rescaling, resp.\ an inversion and a complex rescaling, the image of $U[f]$ by $E:z\mapsto e^{2\pi iz}$ is the domain of the renormalization of $f$.) Then there is a conformal isomorphism $\phi_0 : U[f] \to U[B_d]$ that commutes with $T_1$ and such that $\Psi_\rep[B_d] \circ \phi_0 = \zeta \circ \Psi_\rep[f]$, where $\zeta : A[f] \to A[B_d]$ is the conjugacy of the immediate parabolic basins of the respective fixed attracting petals, mentioned in \Cref{thm:s1}.
\end{complem}

\subsection{Inou and Shishikura: giving up part of the structure to gain flexibility}

Here is the central gear in the work of Inou and Shishikura:

\begin{theorem}[Inou Shishikura]\label{thm:IS0}
There exists an explicit pair of open subsets $A,B$ of $\C$ and an explicit holomorphic map $f_0: B\to \C$
with the following properties:
\begin{enumerate}
\item $0\in A$, $A$ is compactly contained in $B$,
\item $A$ and $B$ are simply connected,
\item $f_0$ fixes the origin and has derivative $1$ there,
\item $f_0$ has exactly one critical point in $B$; this critical point has local degree two, belongs to $A$, and is mapped to $-4/27$ by $f_0$,
\item for any upper renormalization $\cal R[f]$ of a map $f\in\sclass_d$, there exists a subset $U$ of $\dom \cal R[f]$ and a holomorphic bijection $\phi : B \to U$ with $\phi(0)=0$ and $\cal R[f]\big|_U = f_0 \circ \phi^{-1}$,
\item for any univalent map $\phi : A \to \C$ with $\phi(0)=0$ and $\phi'(0)=1$, there exists a univalent map $\psi : B \to \C$ with $\psi(0)=0$ and $\psi'(0)=1$, such that the the map $f_0\circ \phi^{-1}$, which fixes the origin with multiplier one, has an upper renormalization which has a restriction of the form $f_0 \circ \psi^{-1}$.
\end{enumerate}
\end{theorem}

The map $f_0$ has a particularly simple expression: $f_0(z)=z(1+z)^2$.
It turns out that $f_0$ commutes with $z\mapsto \ov z$, thus the theorem holds with the same $f_0$ for lower renormalization.

The statement below is a reformulation of their theorem using the language introduced in the present article.
Given a structure $\cal B$ and a sub-structure $\cal A$, we will say that the second is a \keyw{relatively compact sub-structure}\footnote{The reader is welcome to introduce a better terminology.} of the first whenever maps in $\cal A$ are structurally equivalent to restrictions of maps in $\cal B$ to relatively compact open subsets of their domains (not just subsets).\footnote{If it holds for some representatives then it holds for all representatives in the equivalence class.}

\begin{theorem}[Inou Shishikura, reformulated]\label{thm:IS}
Let $I$ be a singleton and $Y=\C$.
There exists an explicit pair of $(I,Y)$-structures $\cal A$ and $\cal B$ with the following properties:
\begin{enumerate}
\item\label{item:IS:2} $\cal A$ is a relatively compact sub-structure of $\cal B$ and $\cal B$ is a sub-structure of the universal structure of \Cref{thm:shi2b},
\item $\forall (a,f)\in\cal A$, the map $f$ is defined on a connected and simply connected Riemann surface and has exactly one critical point, of local degree two; the same holds for $\cal B$.
\item For any map in $\cal A$ whose domain of definition is a subset of $\C$ and that fixes the marked point with multiplier one, its (suitably normalized) upper parabolic renormalization has at least structure $\cal B$.
\end{enumerate}
\end{theorem}

It is also the case for the \emph{lower} parabolic renormalization, with the same structures $\cal A$, $\cal B$.

\begin{definition}[High type numbers] For $N\in\N^*$, let $\on{HT}_N$ be the set of irrationals whose modified continued fraction satisfies $|a_n|\geq N$, $\forall n\in\N$. In settings where $N$ has been fixed, the set $\on{HT}_N$ is often called the set of \keyw{high type numbers}. We will call it here the set of numbers of \keyw{type $\geq N$}.
\end{definition}

To keep it short, the following corollary, also by Inou and Shishikura, is stated here with some imprecision concerning the renormalization:

\begin{corollary}[Perturbation]\label{cor:cor} There exists $N>0$ such that the class of maps defined in an open subset of $\C$, with structure $\cal A$ and fixing the marked point with a rotation number $\theta$ of type $\geq N$, is invariant under a cylinder renormalization operator (called the \keyw{near-parabolic renormalization}). 
\end{corollary}

They prove more: thanks to the compact inclusion of structure $\cal A$ in $\cal B$, there is a form of contraction.
Cylinder renormalization was introduced by Yampolsky in the study of analytic circle homeomorphisms with a critical point.

Consequences of this corollary are numerous and are still being harvested. Its main quality is that it allows a fine control on the post-critical set of quadratic polynomials with high type rotation numbers.
For instance, Shishikura and Yang Fei proved that in this case the boundary of the Siegel disk is a Jordan curve \cite{SY}. It allows to study the hedgehogs and the size of Siegel disks. In \cite{CC} we proved the Marmi Moussa Yoccoz conjecture restricted to high type numbers.
Cheraghi has given many other applications \cite{Chera2,Chera4,AC,Chera3} and \cite{CS} which is a progresses on the MLC conjecture.
It was also used in~\cite{mespos} to prove the existence of quadratic polynomials with a Julia set of positive Lebesgue measure.
We believe that it can also give a new approach to the results of McMullen \cite{McMu} on the self similarity of Siegel disks whose rotation number has an eventually periodic continued fraction expansion\footnote{these rotation numbers are the quadratic irrationals} at the critical point. McMullen used Ghys' quasiconformal surgery procedure as a first step in his proofs, to transfer some properties that are easier to prove for circle maps. It would be nice to have a more direct proof, that would adapt to situation, like the exponential maps $z\mapsto e^z+c$, where a quasiconformal surgery does not exist but where self similarity still seems to occur.

\subsection{Main Theorem}\label{subsec:main}

In this article, we prove the following extension of Inou and Shishikura's Theorem.

\begin{mainthm}\label{thm:main}
Let $I$ be a singleton and $Y=\C$.
For all $d\in\N$ with $d\geq 2$, there exists $(I,Y)$-structures $\cal A$ and $\cal B$ with the following properties:
\begin{enumerate}
\item $\cal A$ is a relatively compact sub-structure of $\cal B$ and $\cal B$ is a sub-structure of the universal structure of \Cref{thm:shi2b},
\item every map in $\cal A$ or in $\cal B$ is defined on a connected and simply connected Riemann surface,
\item every map in $\cal A$ or in $\cal B$ has exactly one critical value, and all critical points have local degree $d$,
\item for any map in $\cal A$ whose domain of definition is a subset of $\C$ and that fixes the marked point with multiplier one, the upper parabolic renormalization has a at least structure $\cal B$, and the lower parabolic renormalization has at least structure the conjugate of $\cal B$ by $z\mapsto \bar z$, for appropriate normalizations of the renormalizations.
\end{enumerate}
\end{mainthm}

The structures $\cal A$ and $\cal B$ are obtained by retaining most of the universal structure (call it $\cal U$) of \Cref{thm:shi2b}. More precisely we choose for $\cal B$ the restriction of a map in $\cal U$ to a subset of its domain $U$ defined as points having $U$-hyperbolic distance $\leq L$ to the point marked by $I$, and we prove in \Cref{sec:pf} that for $L$ big enough, there is a relatively compact sub-structure $\cal A$ of $\cal B$ such that the main theorem holds.

\begin{remark}It should be noted that for $d=2$, our theorem can be considered as weaker than Inou and Shishikura's. For one thing, maps in our class have much more structure, so our class is smaller. Second they have many critical points (though only one critical value), whereas there is only one in Inou and Shishikura's. This should not prevent our class, though, to be applied to $z^d+c$ as we explain now. Note that a similar situation occurs for the IS class: a polynomial $z^2+c$ with and indifferent fixed point of multiplier close to $1$ never has a restriction that belongs to the IS class, but its first cylinder renormalization has some as soon as the multiplier is close to $0$. Here it is the same: a map of the form $z^d+c$ never has structure $\cal A$ or more, but its first cylinder renormalization does if the rotation number is close enough to $0$. \end{remark}

It should be easy to check that the analog of \Cref{cor:cor} also holds.
We believe that many of its consequences for quadratic maps therefore carry over to unicritical polynomials.

\bigskip\noindent\textsl{About unisingular maps:}
\medskip

We wonder if one can extend the above work to the case $d=+\infty$. There are some subtleties occurring here.

We do not believe that one can take for $\cal B$ a substructure of $g\in\cal R[\sclass_d]$ defined by a restriction on a compact subset of the domain of $g$, like we did in the case $d<+\infty$, that would yield a invariant class. The natural idea is to keep instead a whole connected preimage of a neighborhood of the singular value, which adds a subset of $\dom{g}$ that is at least as tangent to its boundary as a horocycle. Unfortunately, we realized that this does not provide an invariant class either.

It shall be noted that some consequences of Inou and Shishikura's invariant class for $d<+\infty$ won't hold anymore for $d=+\infty$: for instance there are unisingular maps for which the boundary of the Siegel disk is not a Jordan curve.
This includes the exponential $z\mapsto \lambda(\exp(z)-1)$ (or equivalently $z\mapsto e^z+\kappa$) when it has an indifferent periodic point of rotation number in Herman's class\footnote{By \cite{H} the Siegel disk is unbounded and then by \cite{BW} it is not locally connected. Thanks to Lasse Rempe for pointing this out to me.}.
Interestingly, there are some other maps with only one active singular value, with $d=+\infty$, and for which the Siegel disk seems to be more often locally connected (always): for instance the semi-conjugate of $z\mapsto e^{i\theta/2}\tan z$ by $z\mapsto z^2$, i.e.\ $z\mapsto e^{i\theta}(\tan \sqrt{z})^2$.

It is to be noted that thought the two (essentially) unisingular families $\lambda(e^z-1)$ and $\lambda (\tan t{z})^2$ have very different Siegel disks for $\theta=$ the golden mean, computer experiments weakly hint at a possible identical asymptotic limit when zooming at the singular value: there might exist a cylinder renormalization operator with a fixed point capturing both maps. 

\begin{figure}
\begin{tikzpicture}
\node at (0,0) {\includegraphics[height=12cm,angle=90]{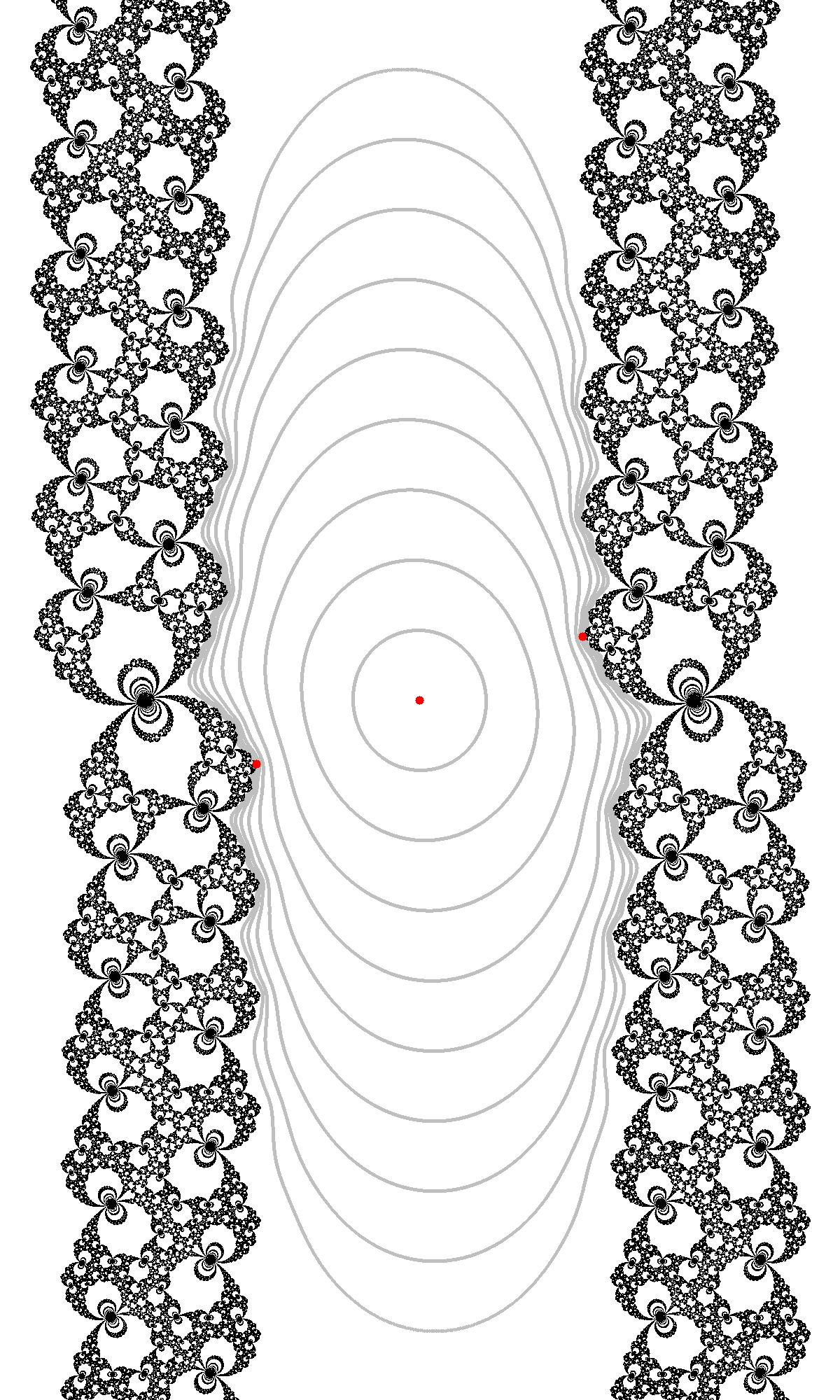}};
\end{tikzpicture}
\caption{Rotated by 90\textdegree, the Julia set of the map $z\mapsto\lambda \tan z$ with $\lambda$ so that the origin is indifferent with rotation number $\theta/2$ and $\theta=(\sqrt{5}-1)/2$ is the golden mean. The Julia set is periodic of period $\pi$, we drew only two periods. There also seems to be an asymptotic similarity by some imaginary translation. There are red points at the origin and at the two (symmetric) asymptotic values. A few orbits inside the Siegel disk have been drawn. The Siegel disk seems to be bounded by a Jordan curve (but not a quasicircle: there must be a dense set of cusps). The rotation number is $\theta/2$ but the picture has a symmetry of order $2$ and quotienting out, i.e.\ semi-conjugating by $z\mapsto z^2$, gives a transcendental meromorphic map $z\mapsto \lambda^2 (\tan \sqrt{z})^2$ with rotation number $\theta$ at $0$, with infinitely many critical points but that all map to $0$, and with only one asymptotic value $-\lambda^2$.}
\label{fig:tanz}
\end{figure}

\begin{figure}
\includegraphics[width=9.5cm]{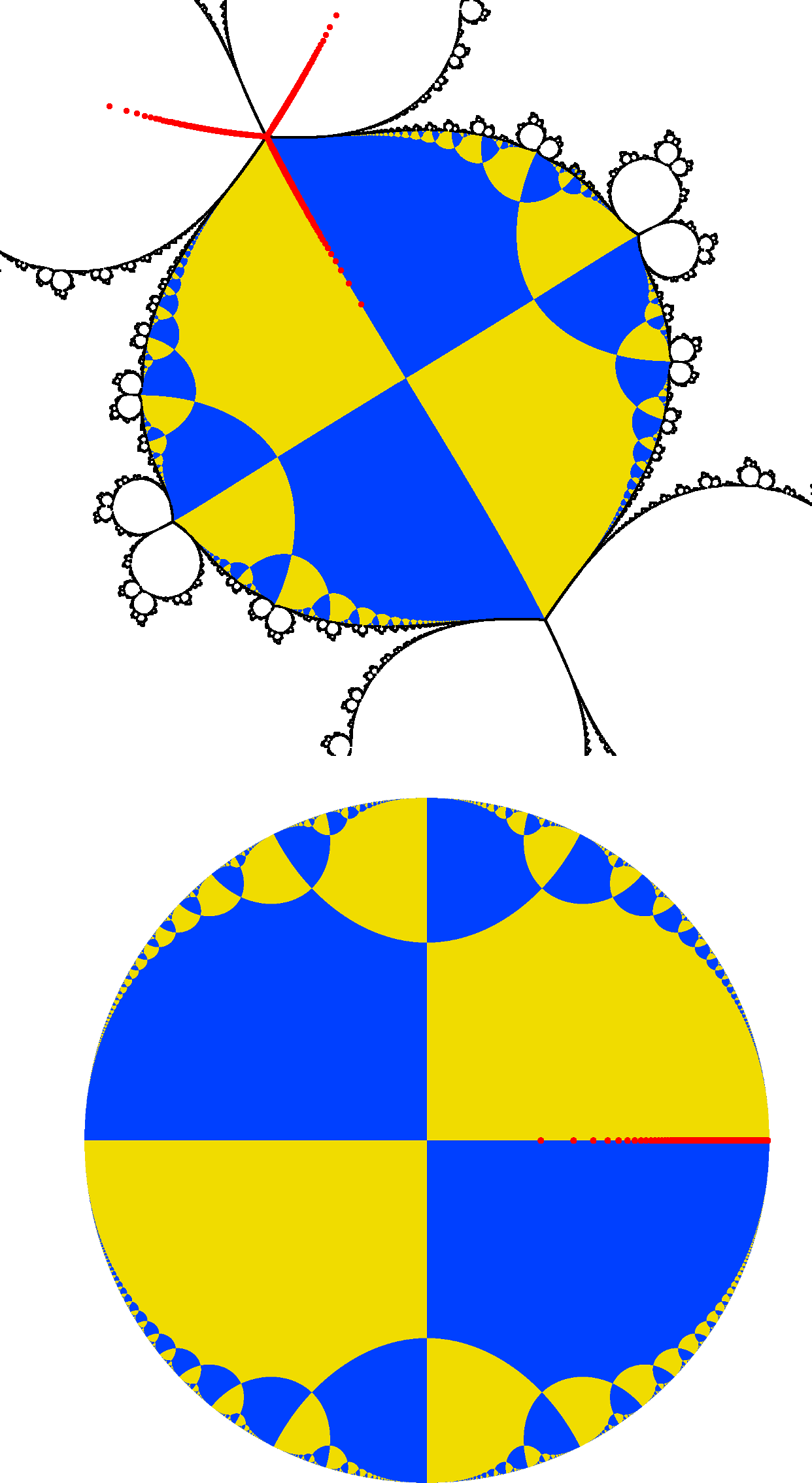}
\caption{Illustration of \Cref{thm:s1}. Above: zoom on the fat Douady rabbit, the Julia set of the quadratic map $P=e^{2i\pi/3}z+z^2$, which has a parabolic fixed point with three attracting petals, and acts transitively on them. The Fatou component $U$ that contains the finite critical point has been colored with the parabolic chessboard, whose definition is recalled later in the present article. In red, the orbit of the critical value. On $U$, $P^3$ satisfies the hypotheses of the theorem. According to the conclusion, $P^3$ is conjugated on $U$ to the Blaschke product $B_2(z)=\big(\frac{z+1/3}{1+z/3}\big)^2$ 
Below: the chessboard of $\mu_2\circ B_2\circ \mu_2^{-1}$, with $\mu(z)=(z+1/3)/(1+z/3)$, which is conjugated to $B_2$ and hence to $P^3$ on $U$. The conjugacy has to transport the chessboard and the critical orbit.}\label{fig:univ_rabbit}
\end{figure}

\begin{figure}
\begin{tikzpicture}\node at (0,0) {\includegraphics[width=10cm]{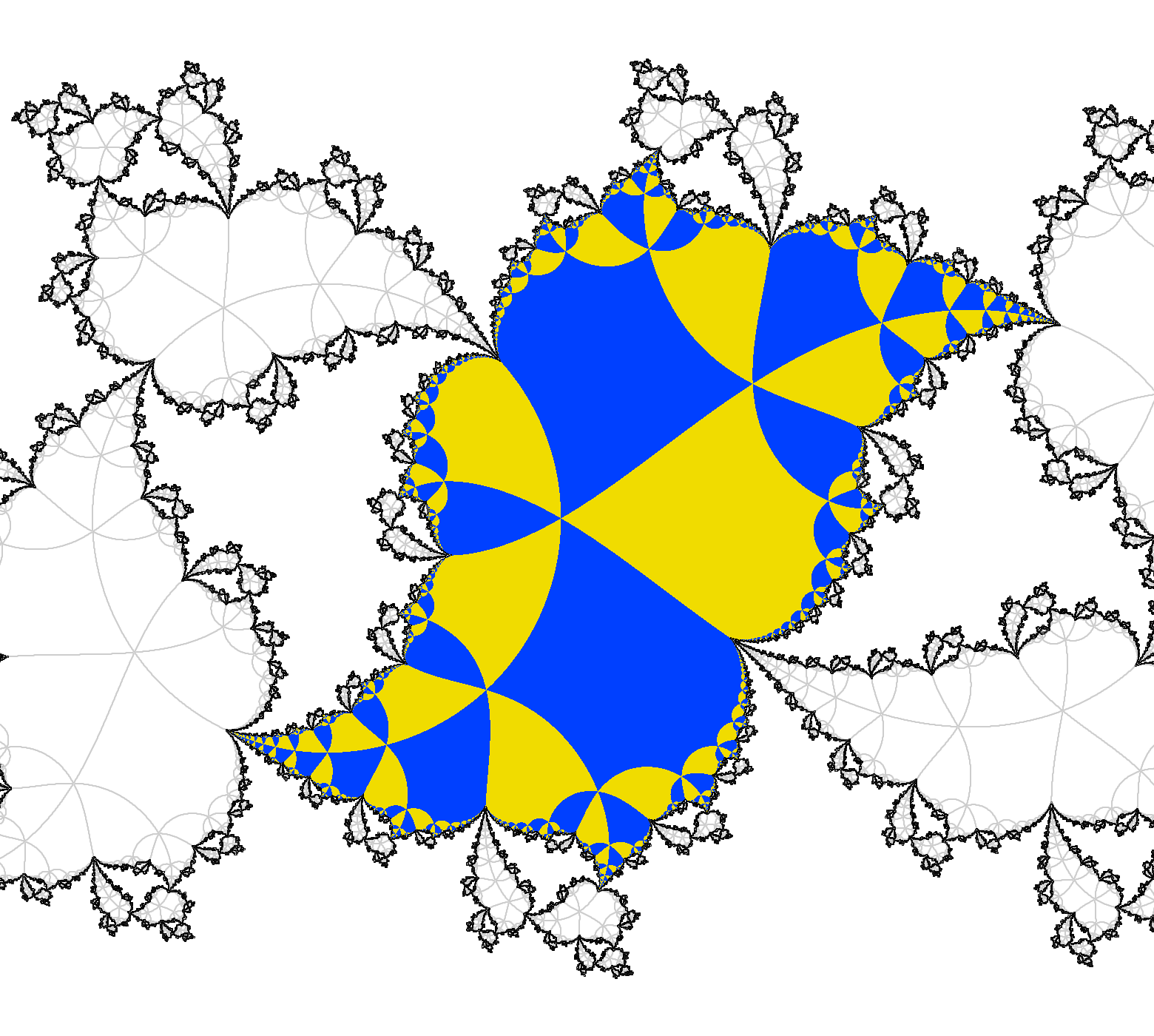}};\fill[color=red] (-0.66,1.35) circle (.7mm);\end{tikzpicture}
\\
\includegraphics[width=7cm]{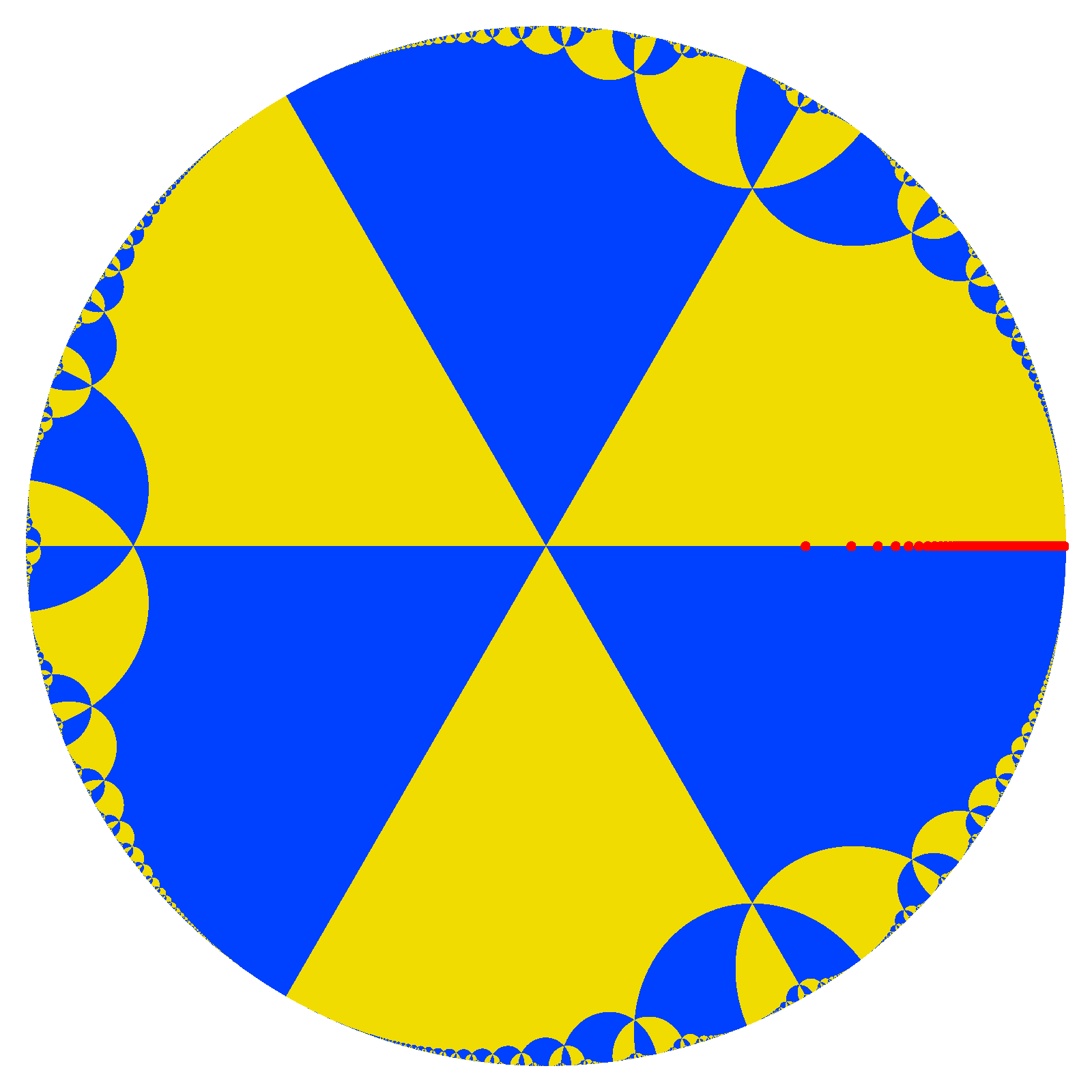}
\caption{Another illustration of \Cref{thm:s1}. This time, $d=3$. The parabolic point on the first picture is indicated by a red dot. The orbit of the critical value is indicated in red on the second picture.}\label{fig:univ_blaf}
\end{figure}

\begin{figure}
\scalebox{1}[-1]{\includegraphics[width=12.5cm]{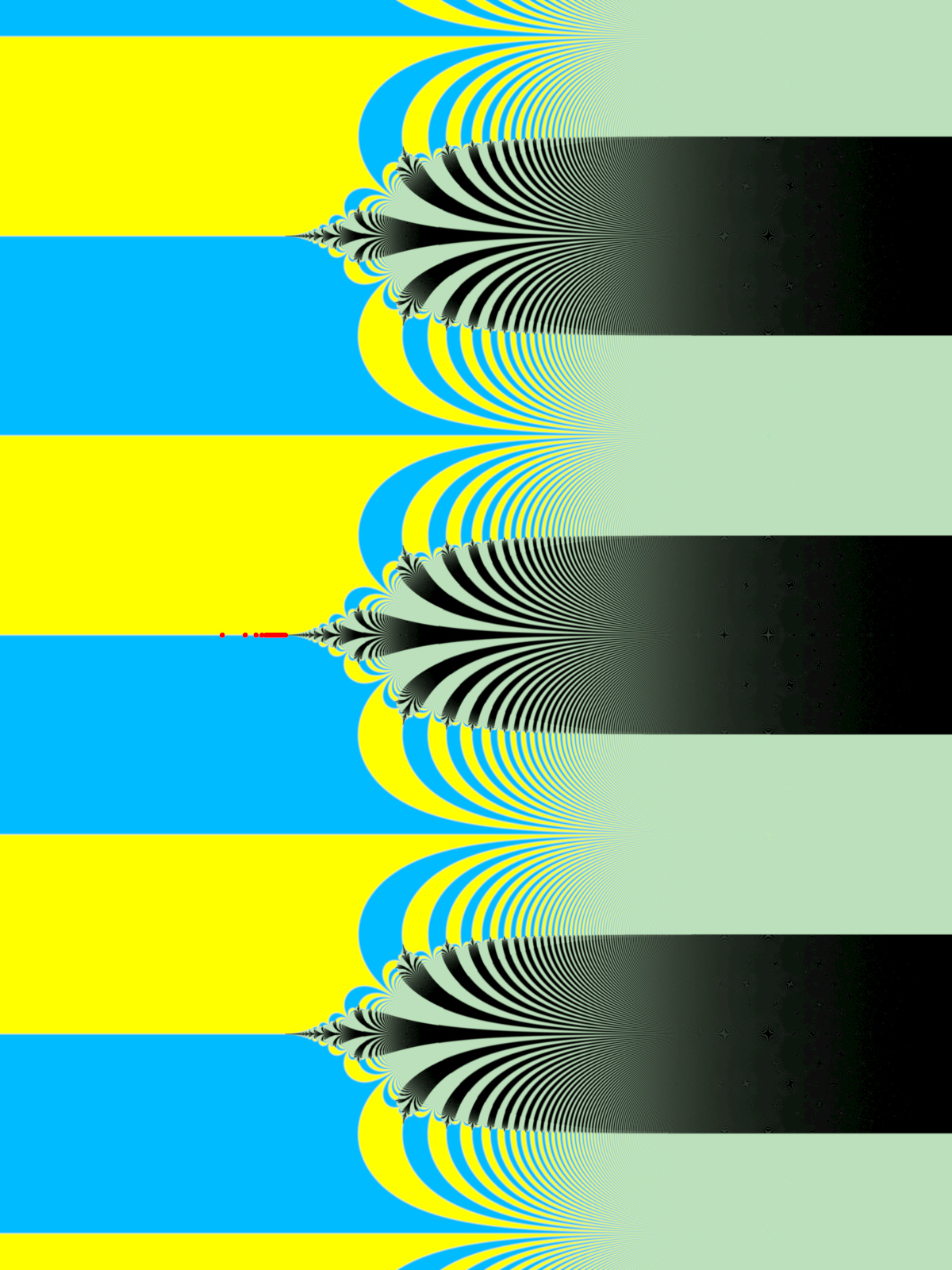}}
\caption{Illustration of \Cref{thm:s1} with, $d=\infty$. Continued on \Cref{fig:univ_exp_2}.}\label{fig:univ_exp_1}
\end{figure}

\begin{figure}
\includegraphics[width=9cm]{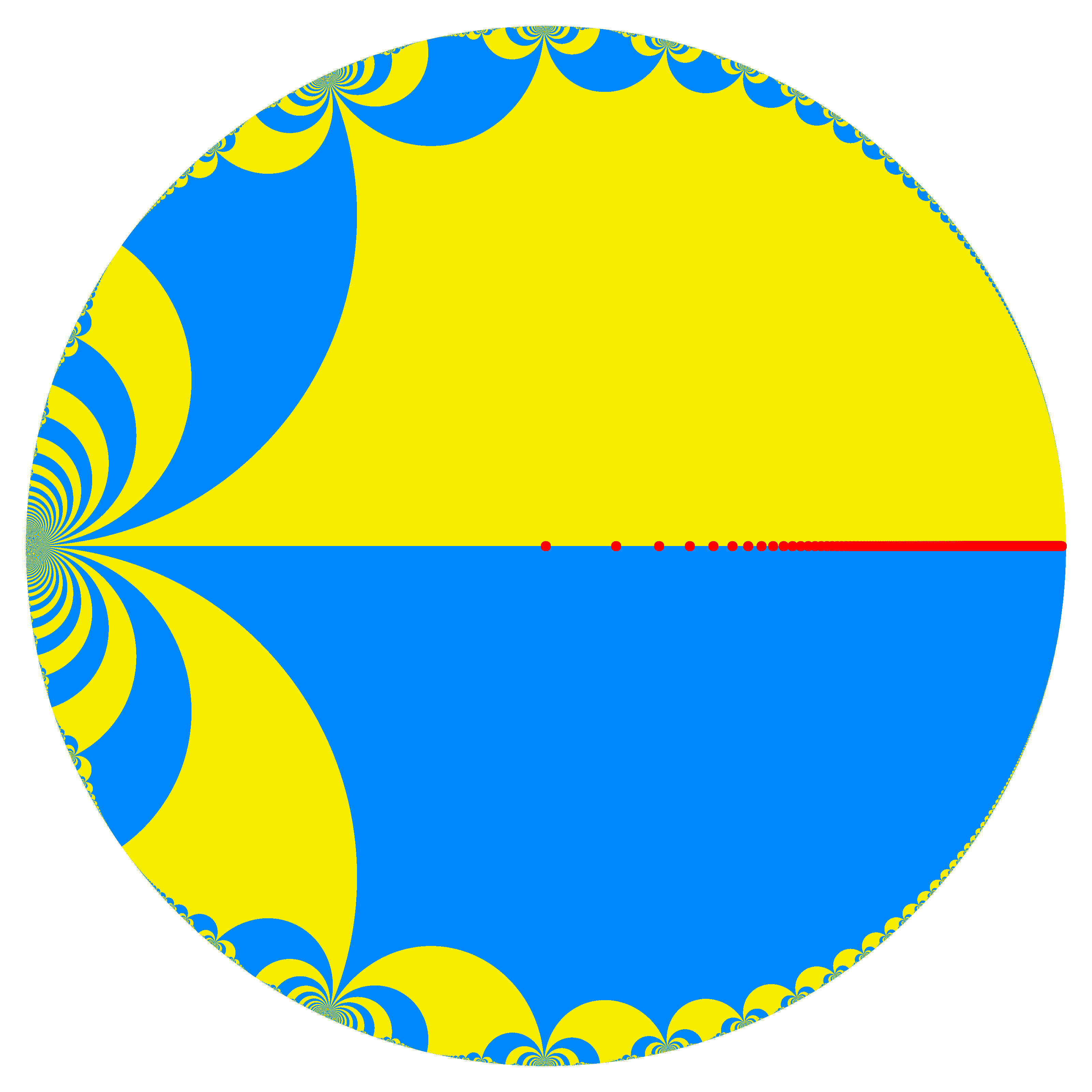}
\caption{Continuation of \Cref{fig:univ_exp_1}, for which the map is $z\mapsto e^{z}-1$, which has a non-linearizable parabolic fixed point at the origin, with one attracting petal. Its basin is connected. The Julia set is a Cantor bouquet indicated in black (see \cite{Bob2}). The orbit of the asymptotic value is drawn with red dots. The parabolic basin is painted with blue and yellow according to the chessboard partition. The green hues correspond to parts where the yellow and blue are mixed below a pixel's width. For the Julia set, the darker shades correspond to places where the Julia set, thickened by an amount comparable to a fraction of a pixel's size, gets denser. On the present figure, we drew the parabolic chessboard of $B_\infty$ and the orbit of its asymptotic value $0$. The conjugacy maps the chessboards and singular orbits of each map to that of the other. Looking at the pictures, the correspondence is not so obvious at first sight.}\label{fig:univ_exp_2}
\end{figure}

\FloatBarrier 

\section{Visualizing structures}\label{sec:viz}

An often used and very useful technique of visualization of ramified covers (and partial cover structures that are not too messy) consists in cutting the range in domains, often simply connected, along lines joining singular values, and taking the pre-image of these pieces, which gives a new set of pieces. The way they connect together and the way they map to the range gives information about the structure. 

\subsection{Changes of variables}\label{subsec:chvar}

The map $B_d$ is the composition of the automorphism $\mu : z\mapsto \frac{z+a}{1+a z}$ of the disk, with $a=(d-1)/(d+1)$,  followed by $\on{pow}: z\mapsto z^d$: $B_d=\on{pow}\circ \mu$. If we conjugate $B_d$ by $\mu$ we get the map $\mu\circ \on{pow}$:\nomenclature[Bd2]{$\wt B_d$}{another normalization of $B_d$\nomrefpage}
\[\wt B_d= \mu\circ B_d\circ \mu^{-1}: z\mapsto \frac{z^d+a}{1+a z^d}\]
which is a Blaschke product too and has its critical point at the origin, the parabolic point still being at $z=1$. 
As $d\to+\infty$, $\wt B_d$ tends to the constant $1$ uniformly on compact subsets of $\D$.

The map $B_\infty$ is the composition of $\mu: z\mapsto i\frac{1-z}{1+z}$ (mapping the disk to the upper half plane) followed by $z\mapsto \exp(2iz)$. Interestingly, if we conjugate $B_\infty$ with $\mu$ we get the trigonometric map:
\[\mu\circ B_\infty\circ \mu^{-1} :z \mapsto \tan z\]
whose non-linearizable parabolic fixed point is at the origin and which maps the upper half plane to itself, the asymptotic value in the upper half plane being $i$.

\subsection{Preferred representative}\label{subsec:preferred}

\Cref{thm:shi2b} says that all maps satisfying some assumption have structurally equivalent upper parabolic renormalizations (appropriately normalized). Their equivalence class, that depends only on $d$, has something universal. We will here choose a preferred representative in this equivalence class, and for this use the maps $B_d$. A defect of the maps $B_d$ and $\wt B_d$, seen as maps from the Riemann sphere to itself, is that their parabolic point has two attracting petals instead of one. We prefer to work with a semi-conjugate $C_d$ of $B_d$ that we introduce now. The map $B_d$ commutes with $z\mapsto 1/\ov z$ and with $z\mapsto \ov z$ hence with $z\mapsto 1/z$. It is therefore a well defined map on pairs $\{z,1/z\}$.
A first change of variables $u = (1-z)/(1+z)$ maps the unit disk to the right half plane ``$\Re(z)>0$'' sending the parabolic point to $0$ and conjugates $B_d$ to a map which can be formulated as follows:
\[u\mapsto \frac{\on{odd}\left(\left(1+\frac{u}{d}\right)^d\right)}{\on{even}\left(\left(1+\frac{u}{d}\right)^d\right)}.\]
where odd and even refer to the sum of monomials of odd and even power in $u$ in the polynomial expansion of $(1+u/d)^d$.
For $d=\infty$ we get the ratio of the odd and even parts of the exponential, a.k.a.\
\[u\mapsto \tanh(u).\]
Setting $v=-u^2 = -((1-z)/(1+z))^2$ identifies pairs $\{z,1/z\}$ with single values of $v$. There exists a map $C_d$, rational of degree $d$ if $d<\infty$, entire transcendental if $d=\infty$, such that the following diagram commutes\nomenclature[Cd]{$C_d$}{a semiconjugate of $B_d$, so that the parabolic point has only one attracting petal\nomrefpage}
\[\xymatrix{\ar[r]^{B_d} \ar[d]_{S} & \ar[d]^{S} \\ \ar[r]_{C_d} &}\]
where $S(z)=v=-u^2 =(i(1-z)/(1+z))^2$.
If $d=\infty$ we get $C_\infty(v) = (\tan\sqrt{v})^2$.
If $d<\infty$ the formula is more complicated.
The map $S$ is a bijection from the unit disk to the complement $A$ of $[0,+\infty]$ in the Riemann sphere, and sends $1$ to $0$, $-1$ to $\infty$ and $0$ to $-1$.
The map $C_d$ has a parabolic fixed point at the origin with one attracting petal, whose immediate basin is $A$. By construction $C_d$ is conjugate on $A$ to the restriction of $B_d$ to $\D$. The extended horn map of $C_d$ is defined on the complement of a horizontal line. Thus the upper and lower parabolic renormalizations of $C_d$ are defined on round disks centered on the origin.
A lengthy computation shows that
\[\gamma[C_d] = \frac{3}{20}\cdot\frac{d^2+1}{d^2-1}.\]
We have not defined in this article a general theory of Fatou coordinates, horn maps and their normalizations, for parabolic points with more that one attracting petal. However, in the particular case of $B_d$, which has two attracting petals, 
we defined in \Cref{subsec:parabopt} the objects $\Phi_\at[B_d]$, $\Psi_\rep[B_d]$ and $h[B_d]$. Let us recall this here.

The unit disk is the (immediate) basin of one of the two attracting petals of $B_d$. 
We let $\Phi_\at[B_d]: \D\to\C$ be the extended attracting Fatou coordinate for this petal. The map has also two repelling petals, with vertical axes. We choose the one on the top and let $\Psi_\rep[B_d]$ denote the corresponding extended repelling Fatou coordinate. We then let $h[B_d]=\Phi_\at\circ\Psi_\rep$. It is defined on an upper half plane.

The object for $B_d$ are related to those of $C_d$ as follows:
\bEA
T \circ \Phi_\at[B_d] & = & \Phi_\at[C_d]\circ S \big|_\D
\\
S\circ \Psi_\rep[B_d] & = & \Psi_\rep[C_d]\circ T' 
\\
T\circ h[B_d] & = & h[C_d]\big|_{H} \circ T'
\eEA
where $H$ is the upper half plane on which $h[B_d]$ is defined, $T$ and $T'$ are translations that depends on normalizations, and
$S$ is the 2:1 rational map defined a few lines above, that semi conjugates $B_d$ to $C_d$.
If we choose a normalization for the objects associated to $C_d$ this induces a normalization for the objects associated to $B_d$ by declaring that $T$ and $T'$ must be the identity:\footnote{There is a way to extend convention~\ref{item:nor:2} on page~\pageref{item:nor:2} for the normalizations, using asymptotic expansions of Fatou coordinates and the general definition of the iterative residue, see for example \cite{Che2}, Chapter~1. A remarkable fact is then that these normalizations of $B_d$ and $C_d$ are compatible: $T$ and $T'$ are automatically the identity.}
\bEA
\Phi_\at[B_d] & = & \Phi_\at[C_d]\circ S \big|_\D
\\
S\circ \Psi_\rep[B_d] & = & \Psi_\rep[C_d]
\\
h[B_d] & = & h[C_d]\big|_{H}.
\eEA

Let $\cal H[B_d]$ denote the semi-conjugate of $h$ by $E$, that we complete by $\cal H[B_d](0)=0$. Then $\cal H[B_d]$ is defined and holomorphic on the unit disk. We now define $\cal R[B_d]$ and a preferred normalization for it:
\[\cal R[B_d] = A\circ\cal H[B_d]\circ B\]
with $A$ and $B$ linear such that:
$B = \on{id}$, and $\cal R[B_d]'(0)=1$.

\subsection{Visualizations}\label{subsec:viz}

In the case of parabolic renormalizations of maps $f$ satisfying the hypotheses of \Cref{thm:shi2b}, our preferred visualization works on the cylinder coordinates $\C/\Z$ just before the conjugacy by $E:\C/\Z\to \C$, $z\mapsto e^{2\pi i z}$ and completion at $0$ and $\infty$. In fact, we will first look at a visualization before the projection from $\C$ to $\C/\Z$, i.e.\ we will look at a visualization of the horn map $h=\Phi\circ\Psi$, where $\Phi$ is a shorthand for $\Phi_\at$ and $\Psi$ for $\Psi_\rep$ (extended Fatou coordinate and parameterizations).

Let $v_f$ denote the unique singular value of $f|_A$ in the immediate parabolic basin $A$.
The set of singular values of $h$ over $\C$ is of the form $v'+\Z$ for $v'=\Phi(v_f)$. Let us cut the range along the horizontal line $v'+\R$ passing through them. To understand the shape of the preimages of this line and of the upper and lower half planes it bounds, it is useful to work first with the map $B_d$. Recall: $h$ is the horn map associated to a dynamical system $f$ with an immediate parabolic basin $A$, on which there is a conjugacy $\zeta : A \to \D$ to the map $B_d$, and $\Phi[B_d] \circ \zeta= \tau +\Phi[f]$. 
Thus the preimage $\Phi[f]^{-1}(v'+\R)$ is mapped by the isomorphism $\zeta$ to a universal shape, that depends only on $d$.
The set $\Phi[f]^{-1}(v'+\R)$ is called the \keyw{parabolic chessboard graph} of $f$ on $A$. The connected components of its complement in $A$ are called the \keyw{chessboard boxes} (in an actual chessboard they are called squares but here they have infinitely many corners and not just four).
The chessboard is the name of this decomposition of $A$ into a graph and boxes.
Since the chessboard is universal, it can be well understood by looking only at the maps $B_d$. Note that these maps have a singular orbit contained in $[0,1[$ and that they send reals to reals, thus the chessboard graph is also the union of the preimages of $[0,1[$.

\smallskip

\textbf{Case 1: $d=\infty$.} We obtain \Cref{fig:univ_exp_2}.

\smallskip

\textbf{Case 2: $d$ is finite.} Instead of showing the graph for $B_d$, whose critical value is at the origin we prefer to show it for the conjugate map $\wt B_d$ introduced earlier, for which the critical point is at the origin. The result is given for $d=2$ and $d=3$ on \Cref{fig:univ_rabbit,fig:univ_blaf}.

\smallskip

To the convergence of $B_d$ to $B_\infty$ as $d\to+\infty$, seems to echo a convergence of the chessboard decomposition: see \Cref{fig:chess_conv}.

\begin{figure}
\includegraphics[width=12.5cm]{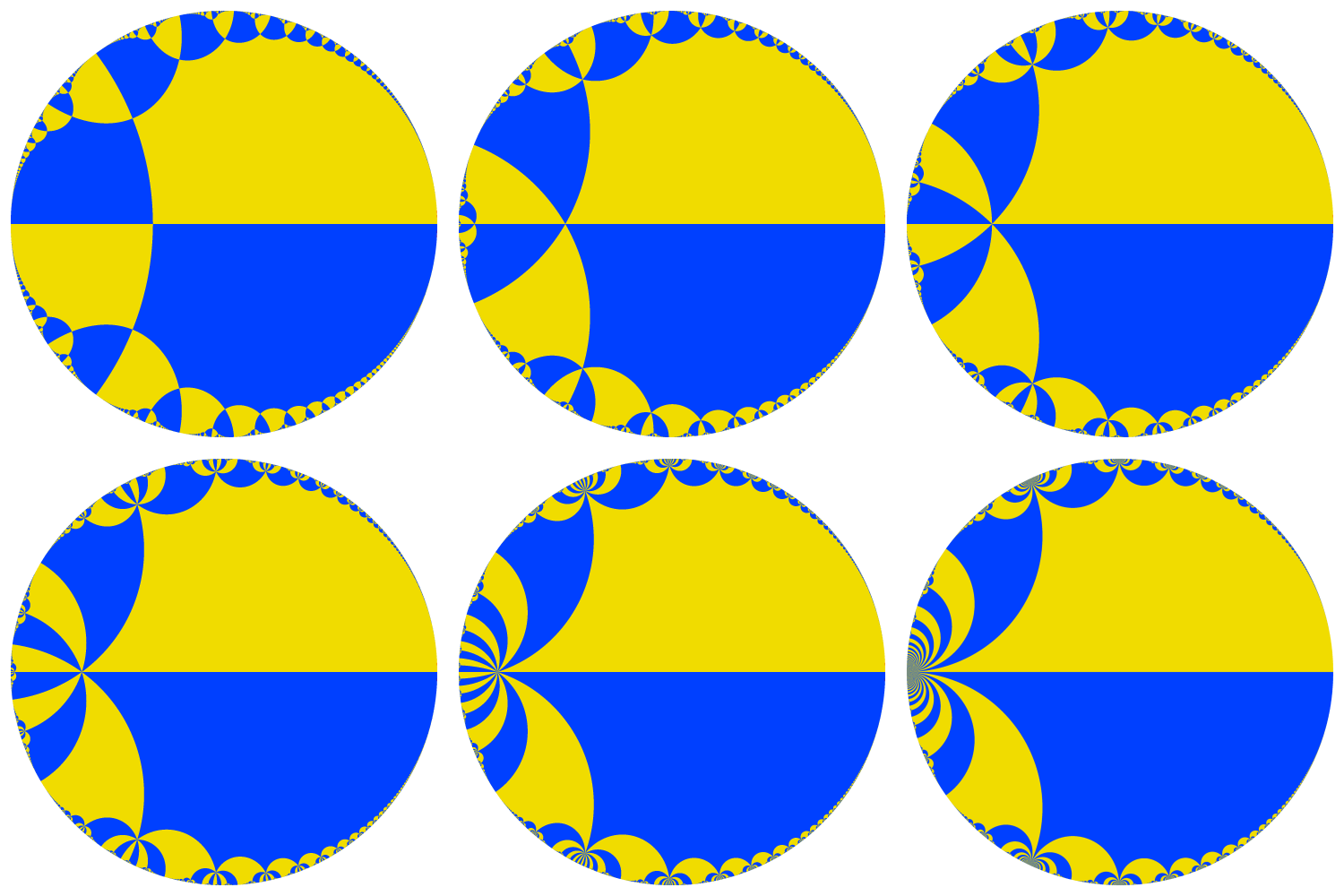}
\caption{Convergence of the chessboard as $d\to+\infty$. In reading order, the chessboard of $B_d$ for $d=2$, 3, 4, 5, 10 and $\infty$.}\label{fig:chess_conv}
\end{figure}

Each chessboard box of $f$ is mapped by $\Phi=\Phi[f]$ to the upper or the lower half plane delimited by $v'+\R$, and we can color them accordingly (we chose yellow and blue in many of the illustrations of the present article). The set of singular values of $\Phi$ is precisely $\{v'-1,v'-2,v'-3,\ldots\}$. These singular values however have also regular preimages, so these universal structures we are considering are not so simple as ramified covers.
Under the dynamics of $f$, each box is mapped to a box of the same color, and there is exactly one box of each color that is fixed by $f$: these are the ones that have the singular value in their boundaries. The Fatou coordinate $\Phi$ conjugates the dynamics of $f$ on these two fixed boxes to the dynamics of the translation by $1$ on the upper and lower half planes. The chessboard also tells us about the structure of $\Phi$ as defined in \Cref{subsec:structeq}.
In view of this, the chessboard in the immediate basin $A$ is both a dynamical object w.r.t.\ $f$ and a structural object w.r.t.\ $\Phi$.

The figures can be enhanced a little bit: let us use two shades of yellow and two shades of blue in the range of $\Phi$. Use the light shade if the floor integer part $\lfloor \Re(z-v')\rfloor$ is even, and the dark shade otherwise. Color points in $A$ according to $\Phi(z)$. Then we get \Cref{fig:shades} for $f=\wt B_2$, $\wt B_3$ and $B_\infty$. This color scheme is useful to visualize the pull-back by $\Phi[f]$ of the vertical direction. Under $f$, a light stripe is mapped to a dark stripe and vice-versa.

\begin{figure}
\includegraphics[width=4cm]{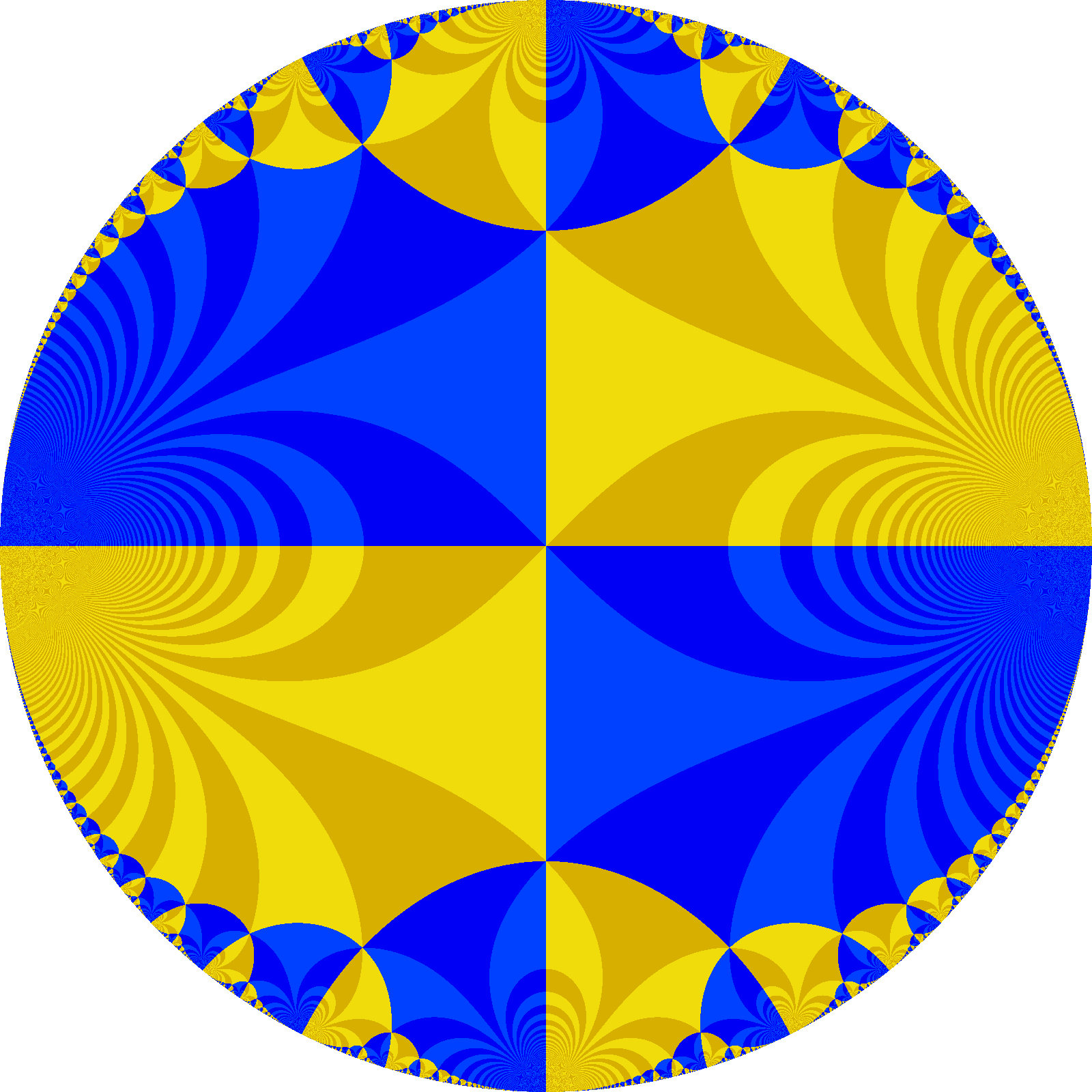}
\includegraphics[width=4cm]{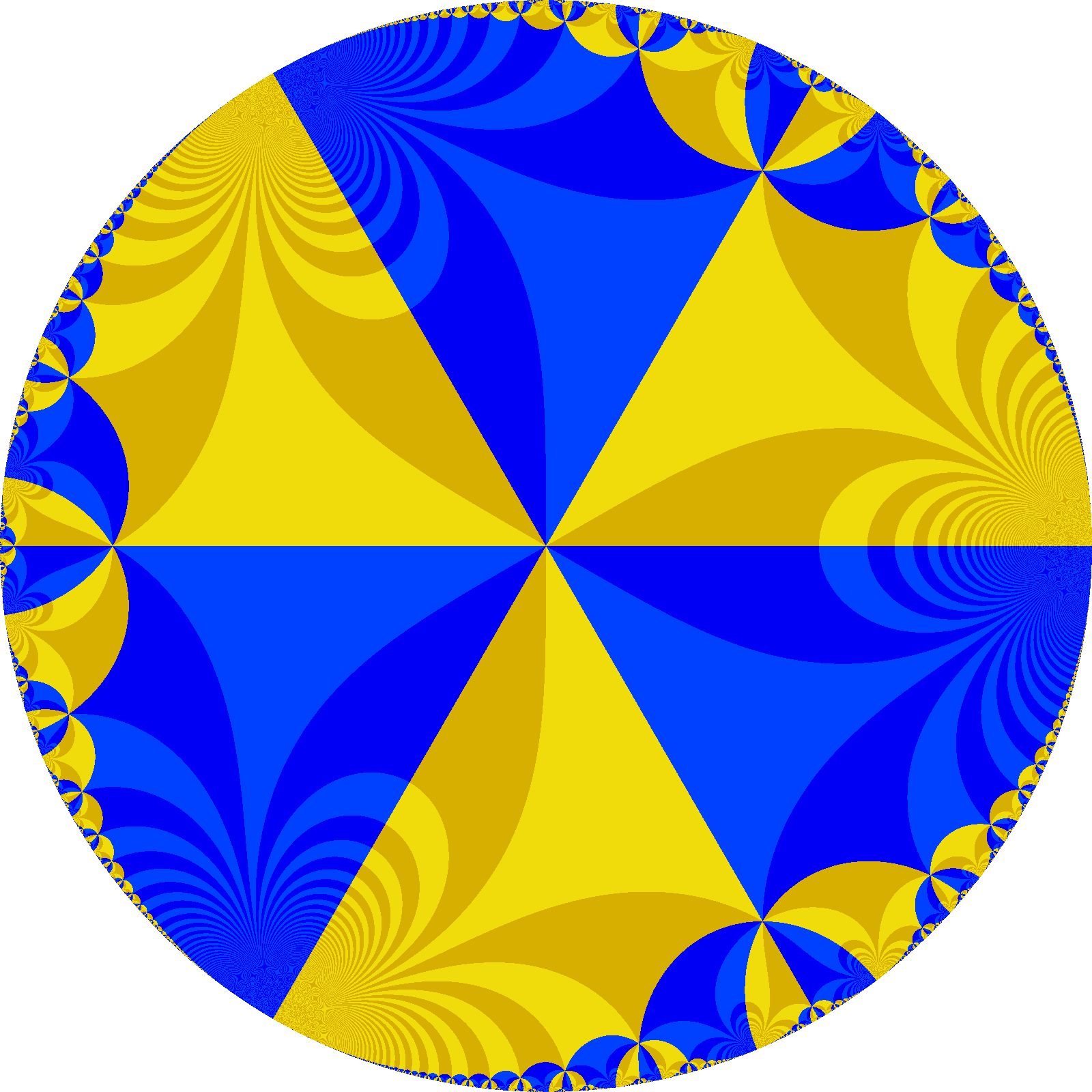}
\includegraphics[width=4cm]{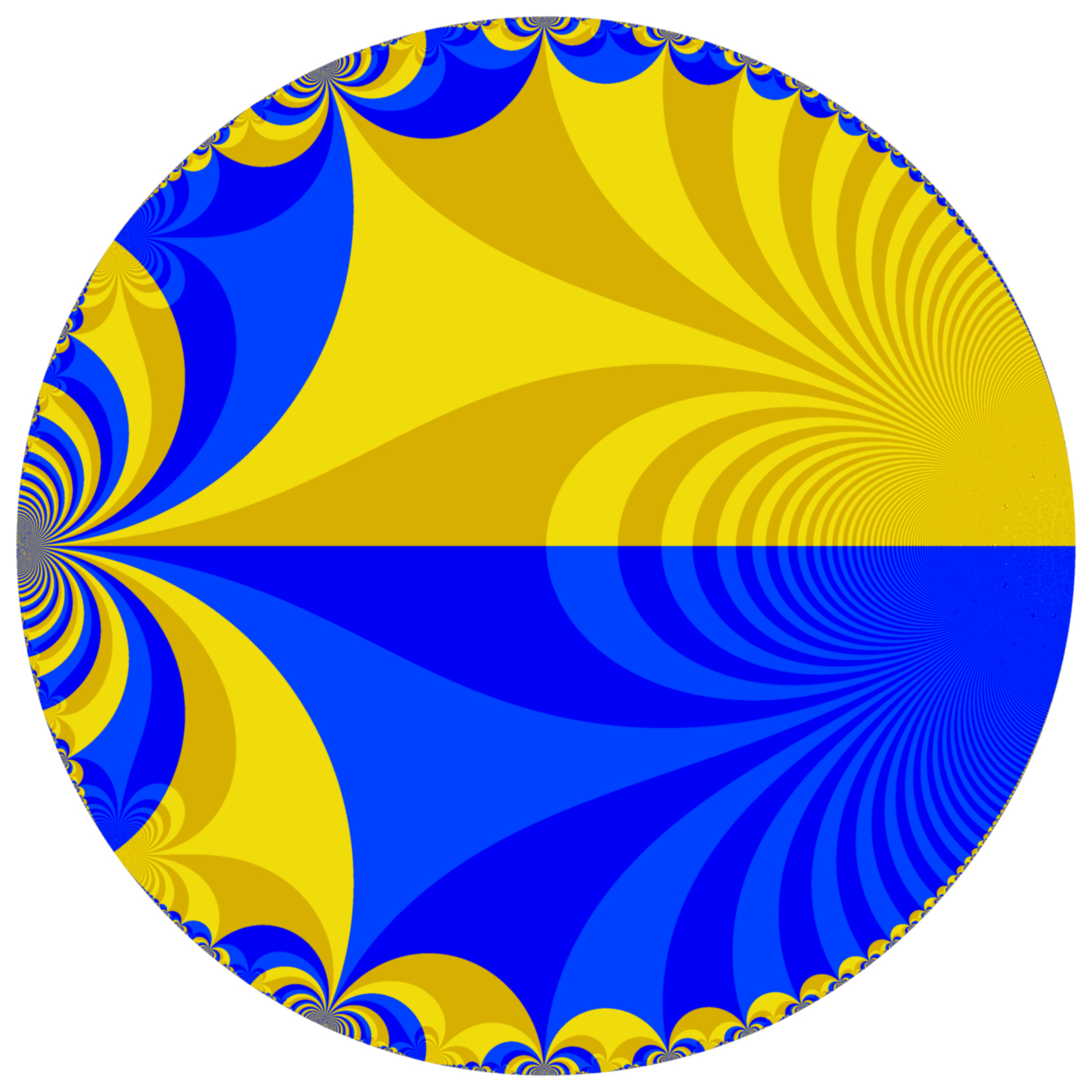}
\caption{Light and dark stripes, preimages of vertical strips of width $1$ under the extended attracting Fatou coordinate $\Phi$, for $\wt B_2$, $\wt B_3$ and $B_\infty$.}\label{fig:shades}
\end{figure}

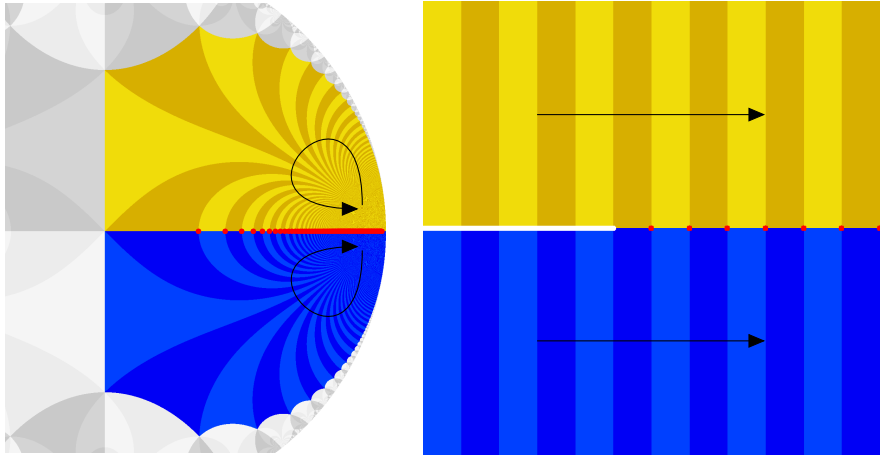
\begin{figure}
\begin{tikzpicture}[>=triangle 45]
\node at (0,0) {\includegraphics[width=5cm]{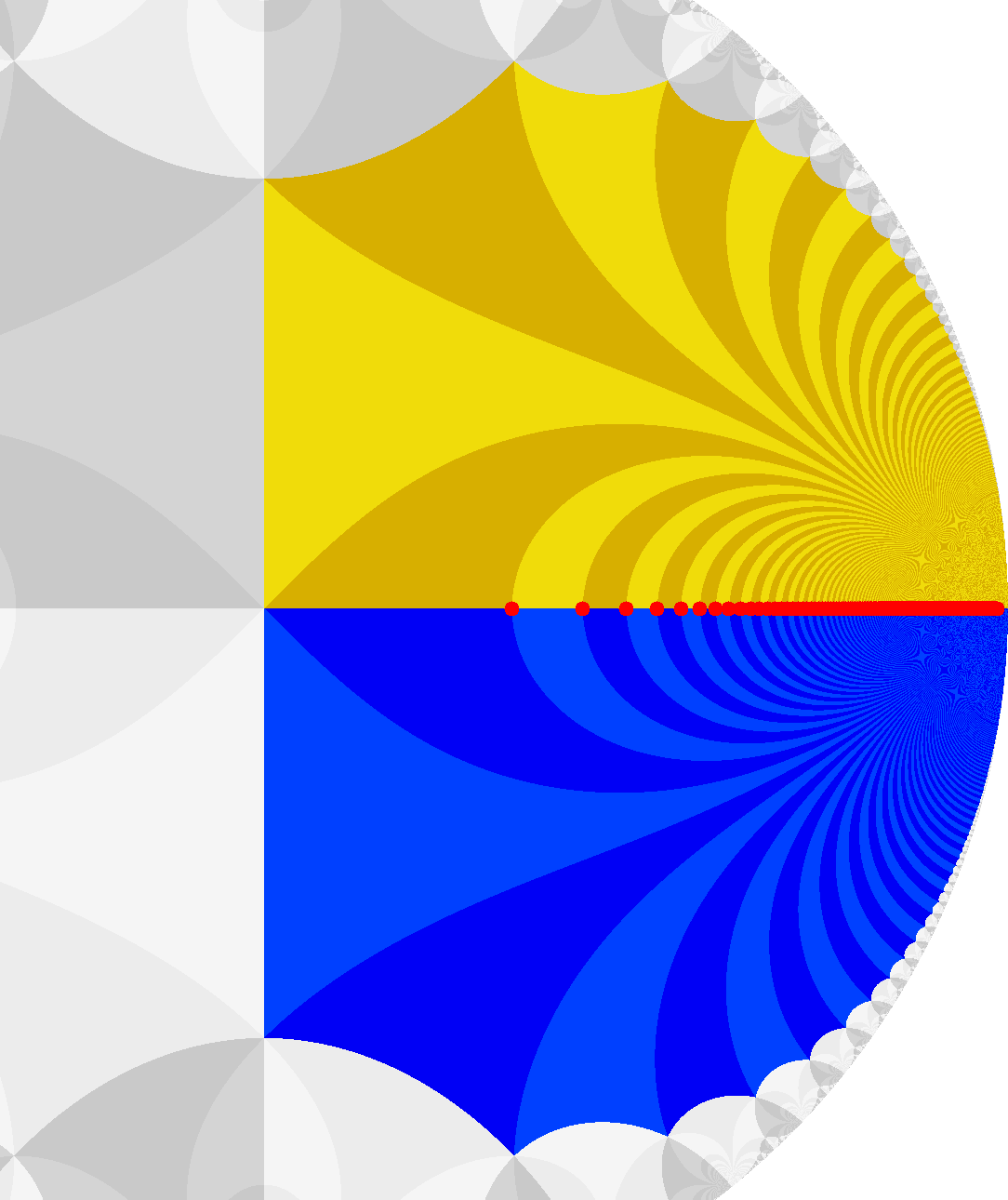}};
\node at (6,0) {\includegraphics[width=6cm]{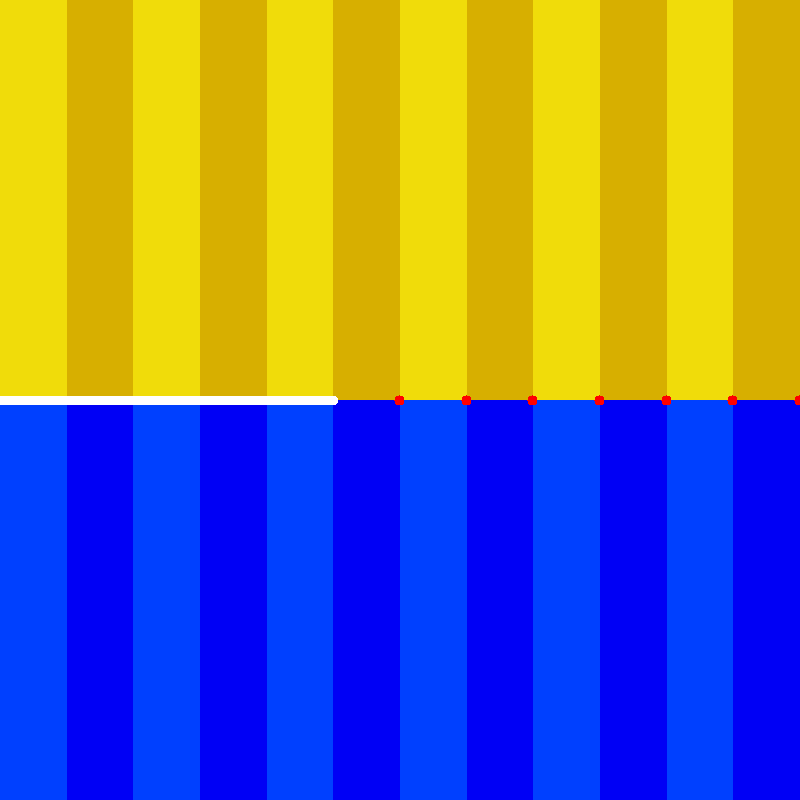}};
\draw[->] (2.2,0.3) .. controls +(up:2cm) and +(left:2cm) .. (2.15,0.25);
\draw[->] (2.2,-0.3) .. controls +(down:2cm) and +(left:2cm) .. (2.15,-0.25);
\draw[->] (4.5,1.5) -- (7.5,1.5);
\draw[->] (4.5,-1.5) -- (7.5,-1.5);
\end{tikzpicture}
\caption{The extended attracting Fatou coordinates of $\wt B_2$ conjugate the restriction of $\wt B_2$ to the two principal chessboard boxes and the segment $]0,1[$, indicated on the left, to the translation $z\mapsto z+1$ on a slit plane, as on the right.}
\label{fig:shades_img}
\end{figure}

\begin{figure}
\begin{tikzpicture}
\node at (0,0) {\includegraphics[width=10cm]{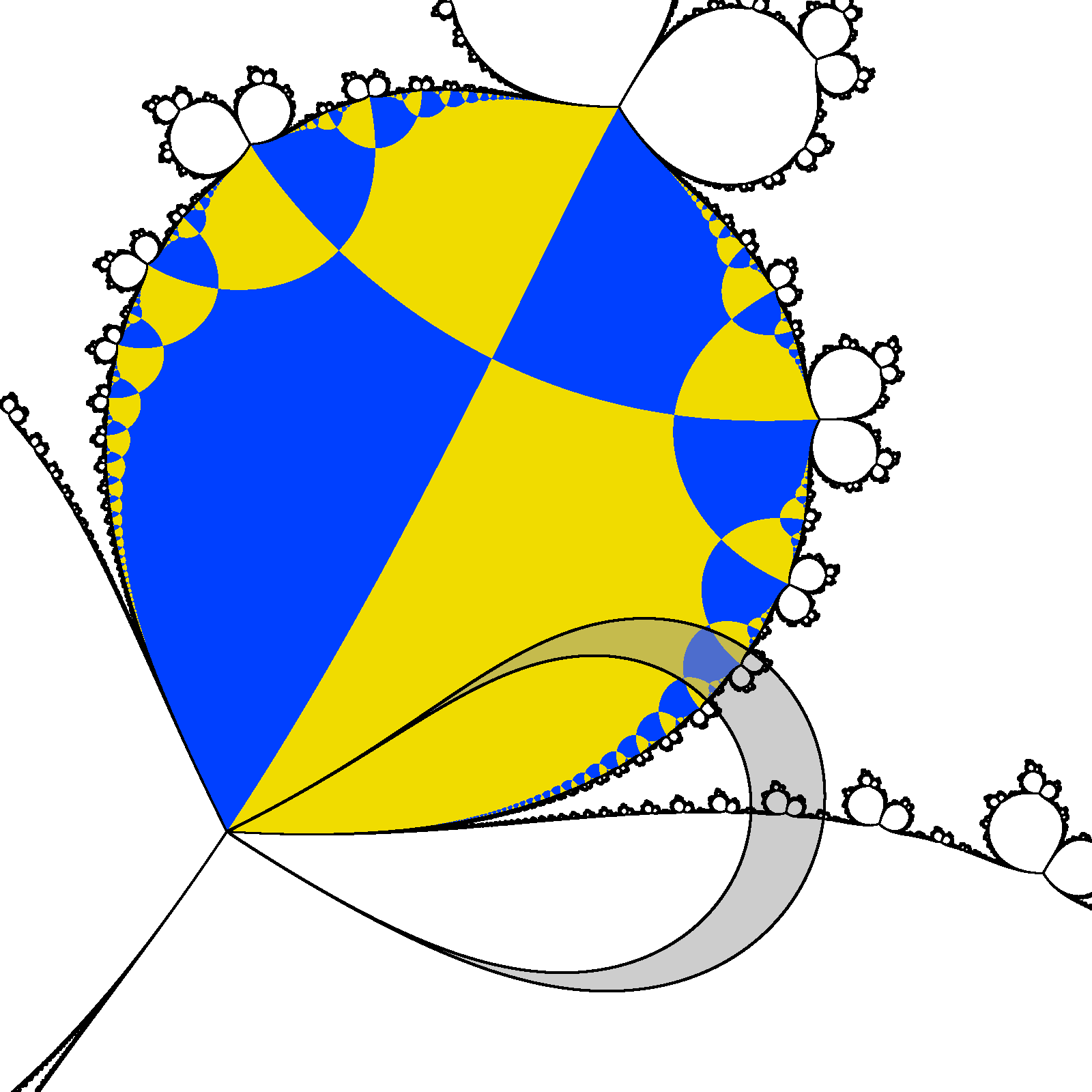}};
\end{tikzpicture}
\caption{This figure shows a grayed out fundamental domain (accurately computed) for the repelling Fatou coordinates, more precisely an open set mapped conformally by $\Phi_\rep[f]$ to a vertical strip of the form $a<\Re(z)<a+1$ for some $a\in\R$. The map $f$ is the third iterate of the fat Douady rabbit polynomial and we chose $U$ to be one of the three immediate basins of the parabolic point. This parabolic point is located at the the end of the tips of the gray set. This kind of drawings help to understand the structure of the extended horn map. See the text page~\pageref{here:rabbit}. See also \Cref{fig:r2}.}
\label{fig:fdtal}
\end{figure}

The chessboard graph has no endpoint, and it is closed in $A$ but not compact. Since we considered the chessboard graph as a subset of $\C$ endowed with its topology, not as a combinatorial object, there is an ambiguity outside branching points concerning which points are vertices of valence $2$ and which points belong to edges: the singular value is one such point w. So let us define an abstract graph with vertices at all preimages of the singular value by $\pi \circ \Phi[f]: A \to \C/\Z$, and edges as preimages of the horizontal circle through it.

\remark We will not make use of it, but it would make sense to consider some supplementary topological information on the abstract graph, like the cyclic order induced by the embedding in the plane on edges at every vertex.
\endremark

The abstract graph and the way it is embedded in $A$ tells us how are glued together pieces obtained by cutting $A$ along the preimages of the horizontal circle through the critical value of $\pi\circ\Phi[f]$. It also tells us how are glued together pieces obtained by cutting $A$ along the preimages of the vertical line through this critical value: each of these piece is a union of two light stripes or two dark stripes, glued along a segment of the graph.

\Cref{fig:shades_img} explains how the union of the edges that touch points in the orbit of the singular value form an infinite line in the graph, and how the union of this line and of the two chessboard boxes whose closure contain the line, make a domain where the dynamics is conjugated to the translation by $1$ restricted to $\C\setminus ]-\infty, 0[$. The bright and dark stripes help to figure out how things are mapped and what the dynamics is within this domain. This would work for any $d\geq 2$, including $\infty$.

Let us now introduce the chessboard associated to the horn map $h$. It is defined using the pre-image of the horizontal line through the set of singular values of $h$, and of the upper and lower half plane cut by this line. From the definitions, it follows that it is also equal to the pre-image by $\Psi$ of the chessboard of $f$ in the full parabolic basin (the union of all iterated preimages of $A$ by $f$).
This time, it is not a dynamically invariant object, but it gives information on the \keyw{structure} of $h$ as defined in \Cref{subsec:structeq}. 

\phantomsection
\Cref{fig:fdtal,fig:r2}\label{here:rabbit} explains how one can proceed to guess the shape of the domain of the horn map and its parabolic chessboard, using a crescent shaped fundamental domain for the repelling Fatou coordinates. In the picture we only looked at the immediate basin. Note that we only defined the horn map for maps with one attracting petal attached to the parabolic point, whereas \Cref{fig:fdtal} shows an example with three. Here, there are several inequivalent definitions for the horn map. Let us give one such that the domain of the horn map is the smallest still giving a parabolic renormalization with the full structure: $h=\Phi_\at\circ\Psi_\rep$ where $\Phi_\at$ is the extended attracting Fatou coordinate restricted to the immediate basin $A$ of the petal, and $\Psi_\rep$ is the repelling Fatou coordinate associated to one of the two repelling petals adjacent to $A$. For another definition giving $h$ a bigger domain of definition, see Section~9 of \cite{Che3}.

\begin{figure}
\begin{tikzpicture}
\node at (0,0) {\includegraphics[height=18.5cm]{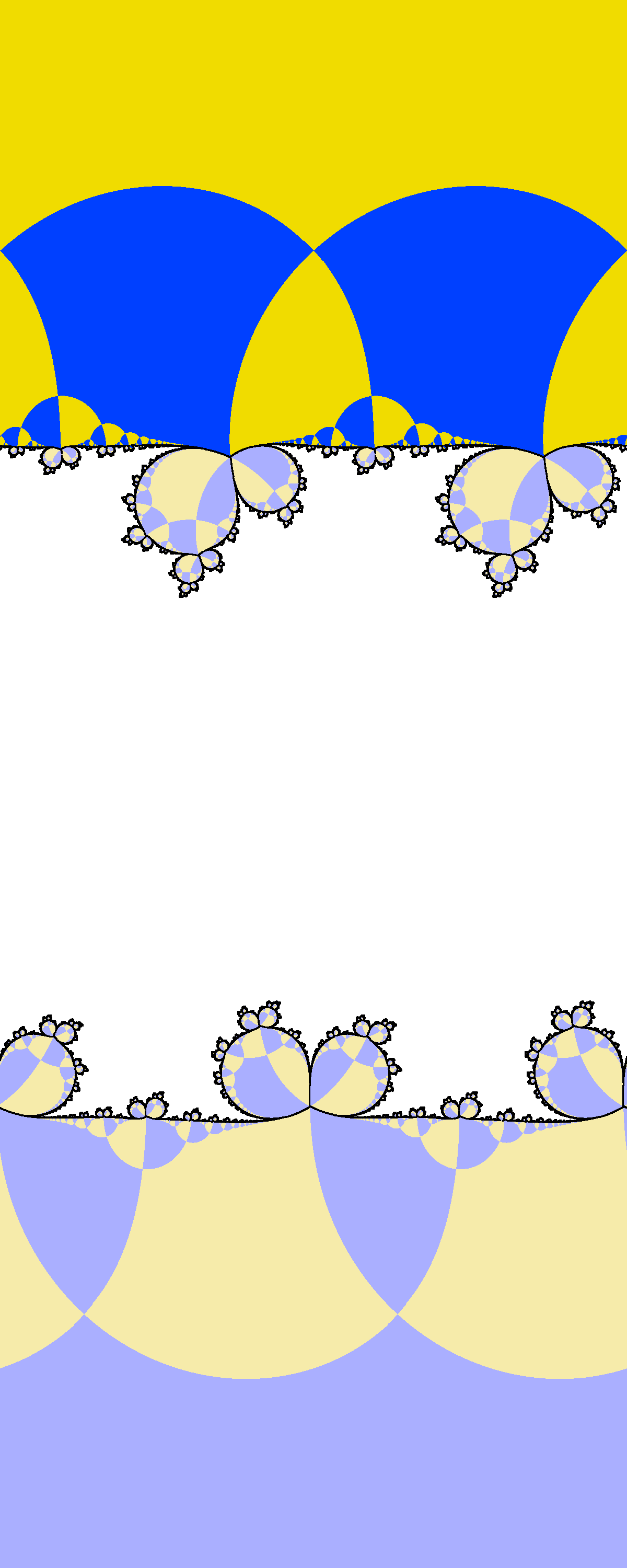}};
\end{tikzpicture}
\caption{Chessboard of the extended map $h$ for the fat Douady rabbit. We highlighted the component that contains an upper half plane. We give in the text on page~\pageref{here:rabbit} a definition of $h$ in the case with more than one attracting petal. The picture in fact illustrate the chessboard for a version of this definition giving $h$ a bigger domain, see Section~9 of \cite{Che3} for a precise definition in the particular case of a parabolic point with one cycle of attracting petals, of which the fat Douady rabbit is an instance.}
\label{fig:r2}
\end{figure}

\begin{figure}
\begin{tikzpicture}
\node at (0,5.1) {\includegraphics[width=10cm]{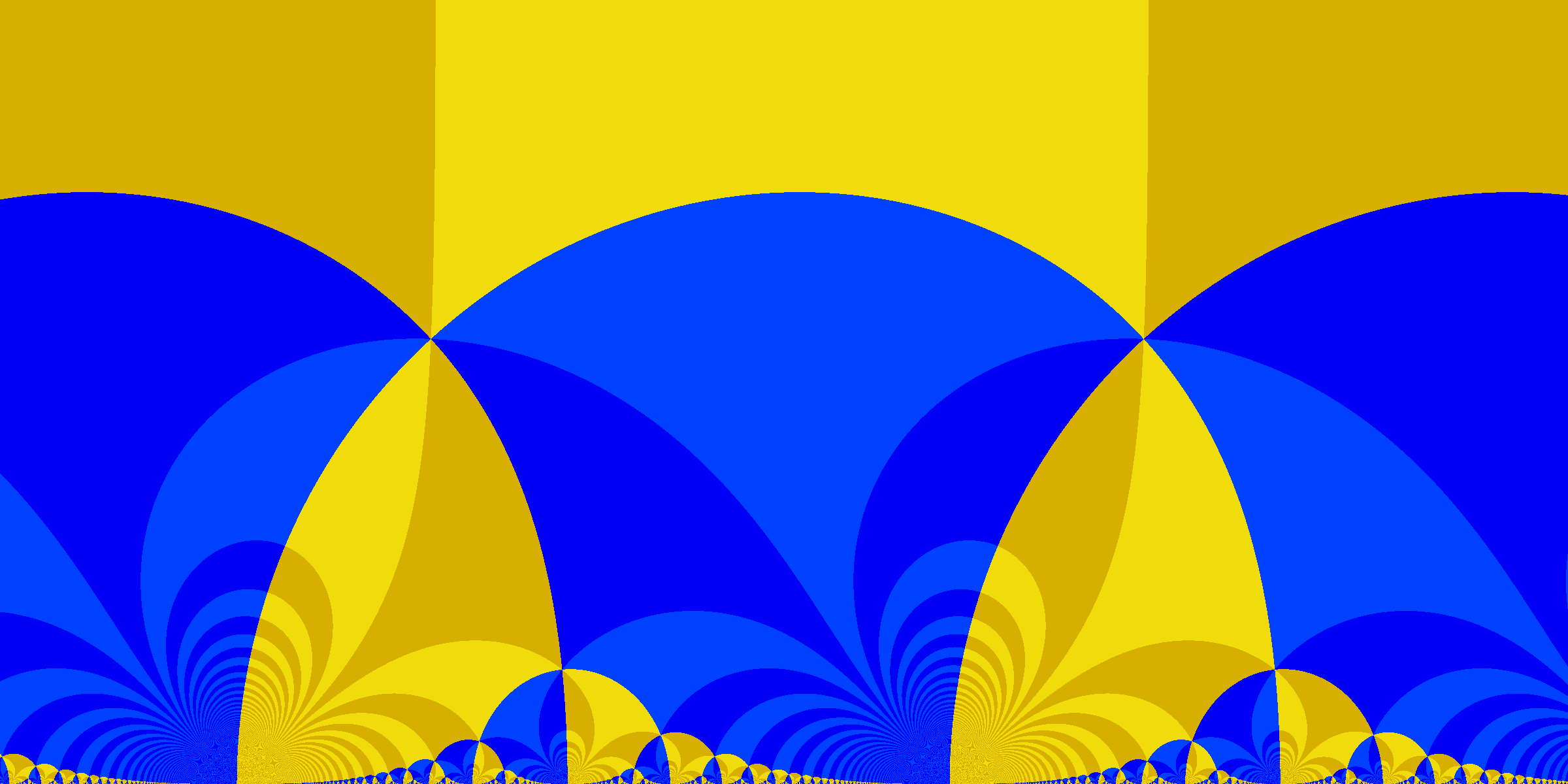}};
\node at (0,0) {\includegraphics[width=10cm]{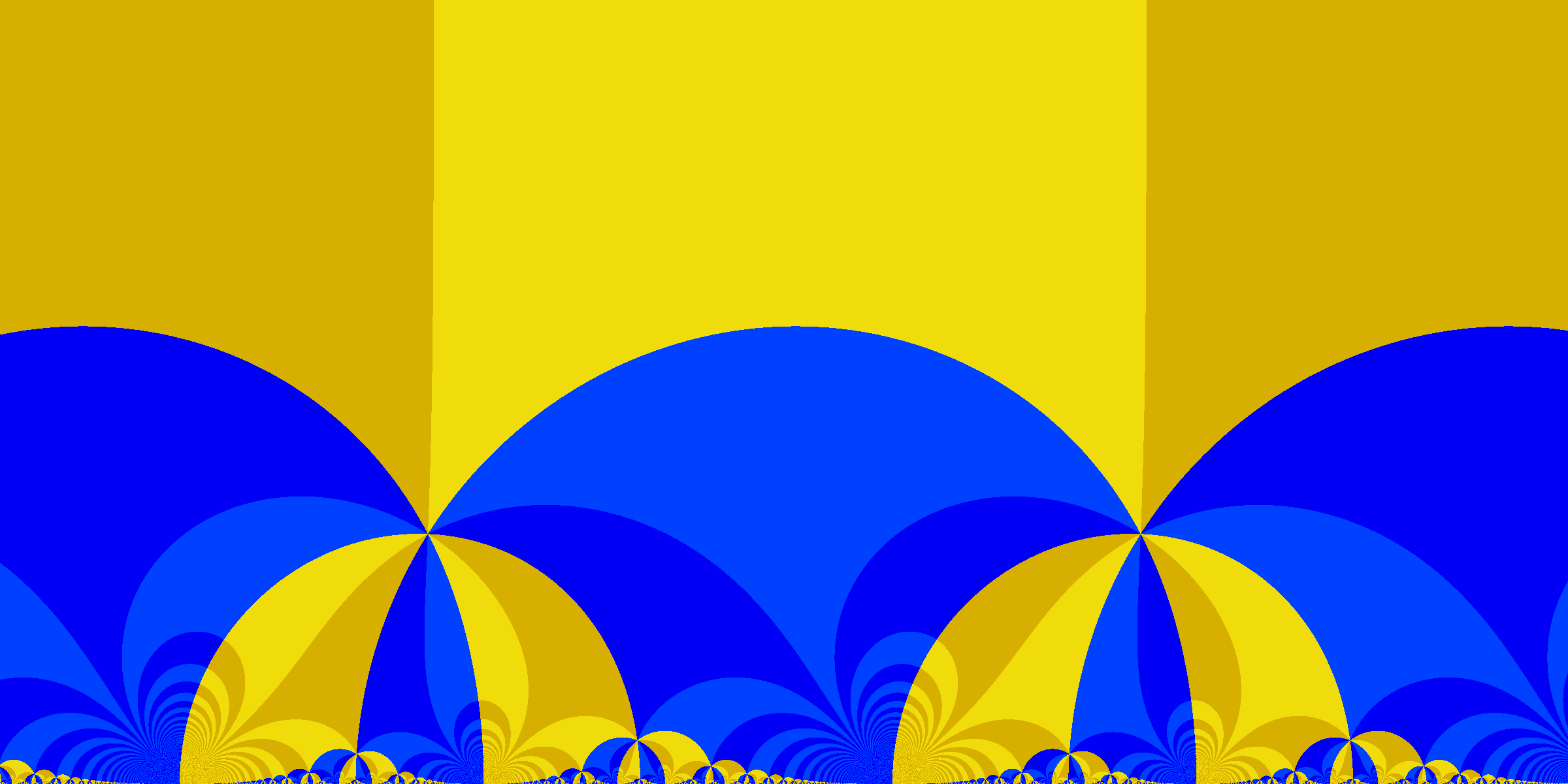}};
\node at (0,-5.1) {\includegraphics[width=10cm]{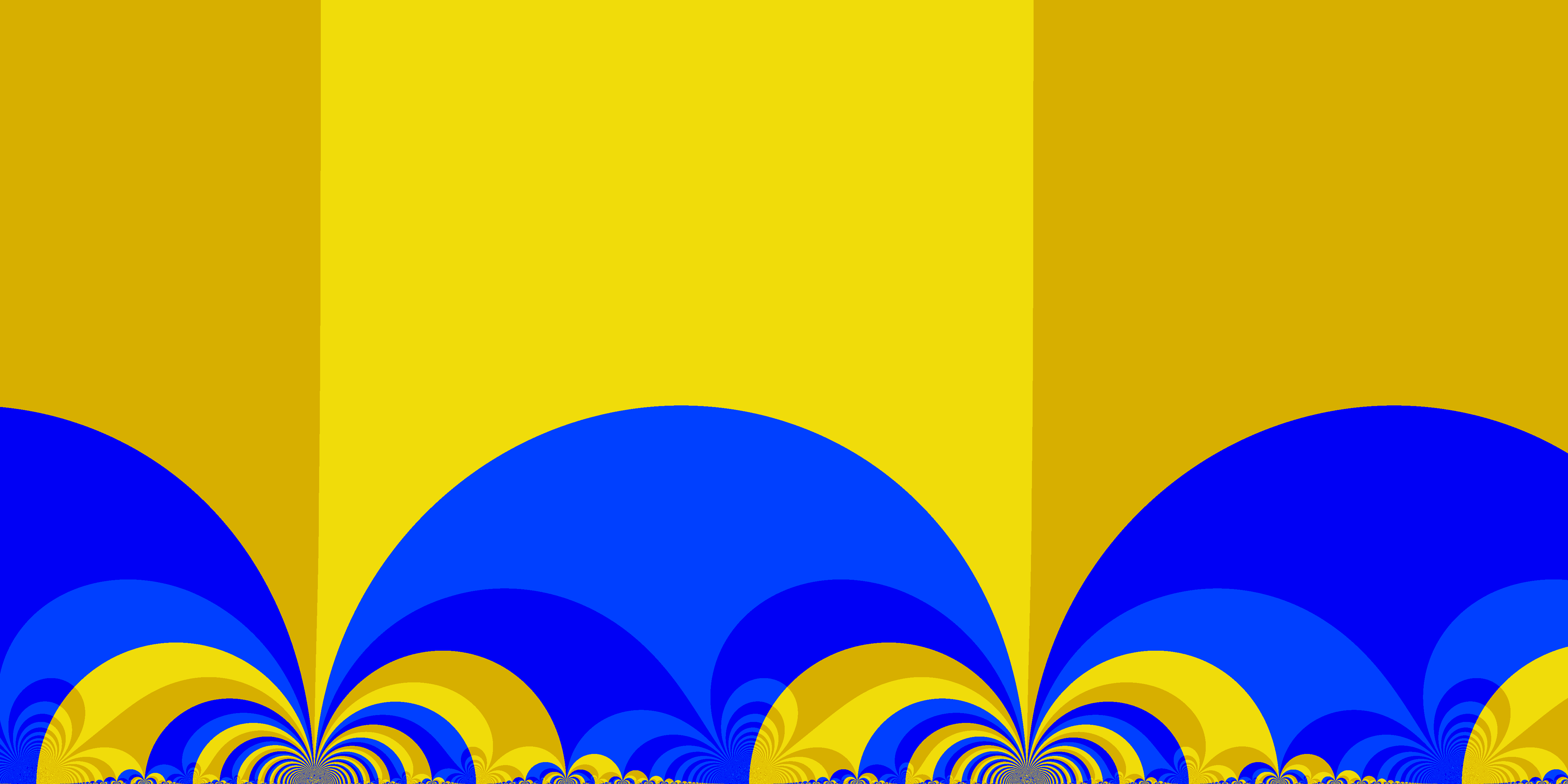}};
\end{tikzpicture}
\caption{These three pictures show the structures of the extended horn maps $h$ of respectively $B_2$, $B_3$ and $B_\infty$. They are all defined on the complement of a horizontal line; in each case, we only drew the picture above this line; the full picture is obtained by reflection through this line, permuting blue$\leftrightarrow$yellow. The same coloring conventions apply as in the previous figures: yellow boxes map by $h$ to the upper half plane delimited by the horizontal line through the singular value, blue boxes to the lower half plane. The boundaries between dark and light shades of a given color are mapped by $\pi \circ h$ to the vertical line through the critical value.}
\label{fig:h-struct-of-B}
\end{figure}

The next set of pictures, in \Cref{fig:h-struct-of-B}, shows the structure of the horn maps of $B_d$. The image of these three pictures by the exponential map $z\mapsto \exp(2\pi i z)$ is shown on \Cref{fig:struct-RB2,fig:struct-RB3,fig:struct-RBinf}, and gives us information about the structure of the upper renormalization $\cal R[B_d]$ of $B_d$.
One sees that it is also defined on a disk centered on the origin.
For the beauty of the thing, we replaced the dark and light strips by a lighting scheme that gives the illusion of a texture made of cylinders.\footnote{The trick to produce such a computer picture is called \keyw{normal mapping}, it is the same trick used to give a realistic look in 3D imaging to texture-mapped polygons subjected to a light source. Some specular reflection reinforces the feeling of relief.} A more shameful reason for this change is that the light and dark stripe scheme does not pass to the quotient. Let us explain a bit more these pictures.

Recall that $\cal H[B_d]$ denotes the semi-conjugate of $h[B_d]$ by $E: z\mapsto e^{2\pi iz}$, i.e.\ $\cal H \circ E = E\circ h$, completed by $\cal H[B_d](0)=0$, and that $\cal R[B_d] = A\circ \cal H[B_d] $ for some linear map $A$.
Recall also that there are three singular values of $\cal R[B_d]$: $0$, $\infty$ and some third point $v$. The only preimage of $0$ is $0$.

The chessboard decomposition for $\cal R[B_d]$ has two equivalent definitions. First as the preimage by $\cal R[B_d]$ of the decomposition of $\C$ into the circle of center $0$ going through $v$ and the two connected components of its complement, yellow corresponding to the inside and blue to the outside.
But is is also the image by $E$  of the chessboard decomposition of $h[B_d]$.

With our coloring convention, the box surrounding $0$ is yellow, every blue box is mapped by $\cal R[B_d]$ to the set $|z|>|v|$ as a universal cover, the yellow box containing the origin is mapped $1:1$ to $|z|<|v|$ and every other yellow box is mapped as a universal cover to $|z|<|v|$ minus the origin.

This can be generalized to the renormalization $\cal R[f]$ for any $f$ satisfying the hypotheses of \Cref{thm:shi2b}, i.e.\ $f\in \sclass_d$. It is both: on one hand the preimage of the decomposition of $\C$ into the circle of center $0$ going through the singular value $v$ of $\cal R[f]$ and the two components of its complement; on the other hand the image by $E$ of the chessboard of $h$, subject to the same restriction, rescaling and possibly inversion if we are performing lower parabolic renormalization instead of upper, as were done to pass from $h$ to $\cal R[f]$, and completed by adding $0$.
The partition of the domain of $\cal R[f]$ into two colors and the graph separating them is called the \keyw{structural chessboard of $\cal R[f]$}. It is different from the \keyw{dynamical chessboard} of $\cal R[f]$, which is defined only in the basin of its parabolic point $z=0$ and only if the normalization of $\cal R[f]$ is such that $0$ is a parabolic point of $\cal R[f]$. In particular, unlike the dynamical chessboard, the structural chessboard is \emph{not} dynamically invariant.

Recall that we chose to define $\cal R[B_d]$ with a normalization such that, in particular:
\[\dom\cal R[B_d] = \D \text{ and } \cal R[B_d]'(0)=1.\]
Then it turns out that 
$|v|$ is notably smaller than the radius $r=1$ of the unit disk, on which $\cal R[B_d]$ is defined, as was already remarked by several people before:
\[ |v| = e^{-2\pi^2 \Re(\gamma)}
\]
where $\gamma$ denotes the iterative residue of $C_d$ (see \Cref{subsec:parabopt,subsec:preferred}).
Let us quickly justify this: the horn map $h=\Phi_\at\circ\Psi_\rep$ has expansion $h(z)=z+a_{\on{up}}+o(1)$ when $\im(z)\to+\infty$ and if $\Phi_\at$ and $\Psi_\rep$ are normalized by the expansion (convention number~\ref{item:nor:2} on page~\pageref{item:nor:2}) then $a_{\on{up}}=-\pi i \gamma$.
The normalized horn map is defined on the unit disk and its singular values are real.
For the parabolic renormalization, we need first to semi-conjugate by $E$, which gives a map fixing $0$ with multiplier $e^{2\pi^2\gamma}$.
Then we pre and post compose by two linear maps $B$ and $A$ so that the derivative at the origin is $1$. The claim follows.
In fact we chose $B=\on{id}$ so we get the more precise identity
\[ v = e^{-2\pi^2 \gamma}.\]

\phantomsection
We mentioned earlier that $\gamma[B_d] = \frac{3}{20}\cdot\frac{d^2+1}{d^2-1}$. The ratio $|v|/r$ thus ranges between $\approx 1/140$ (for $d=2$) and $\approx 1/20$ (for $d=\infty$).\label{here:gamma}
In the case $d<\infty$, notice that there is a tiny loop bounding a small yellow box containing the origin and that looks like a droplet. When $d$ increases the angle at the tip of the loop decreases and the tip gets closer to the boundary of the domain of definition of the map.
In the case $d=\infty$, the droplet touches the boundary.

\begin{figure}
\includegraphics[width=10cm]{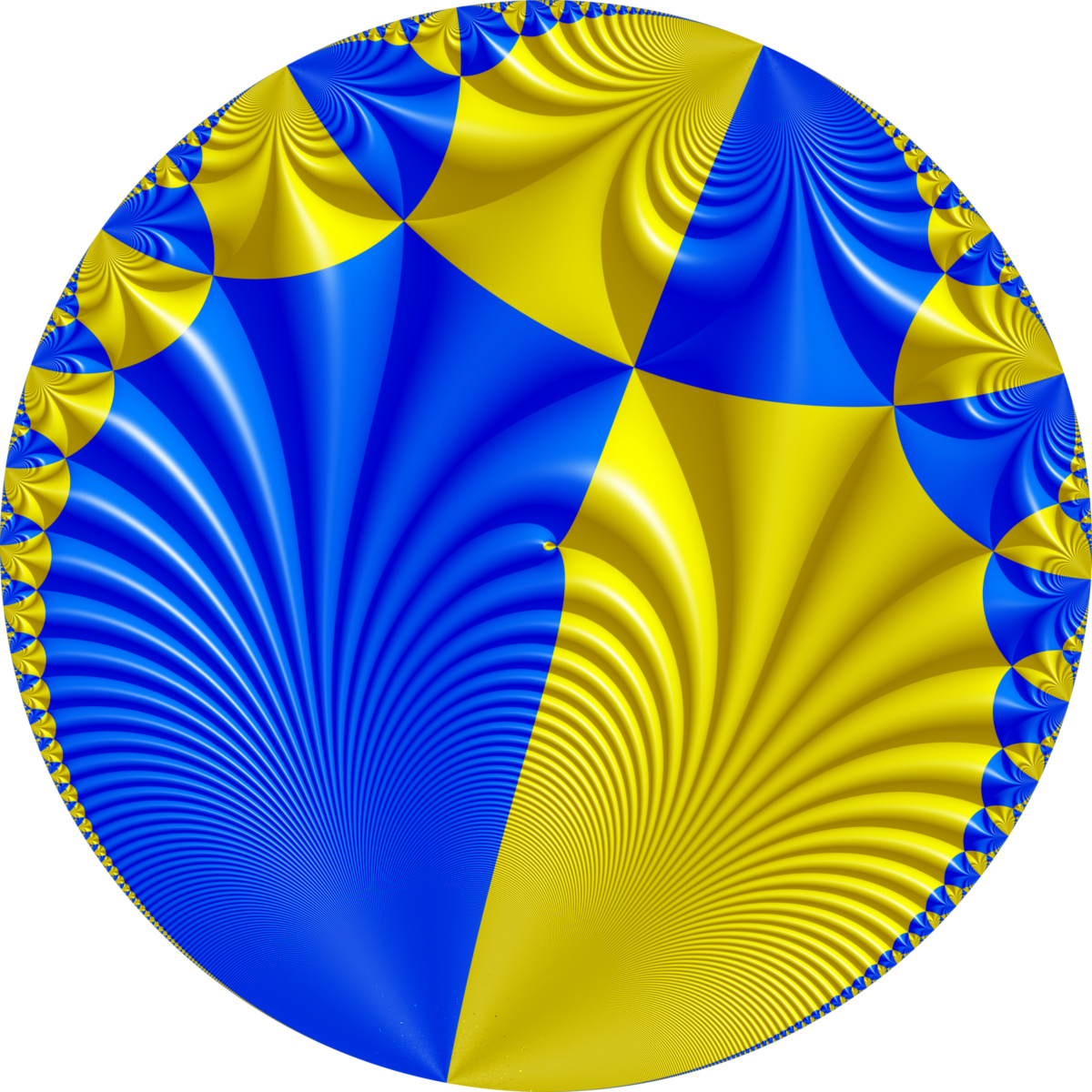}
\caption{Structure of $\cal R[B_2]$.}
\label{fig:struct-RB2}
\end{figure}

\begin{figure}
\includegraphics[width=10cm]{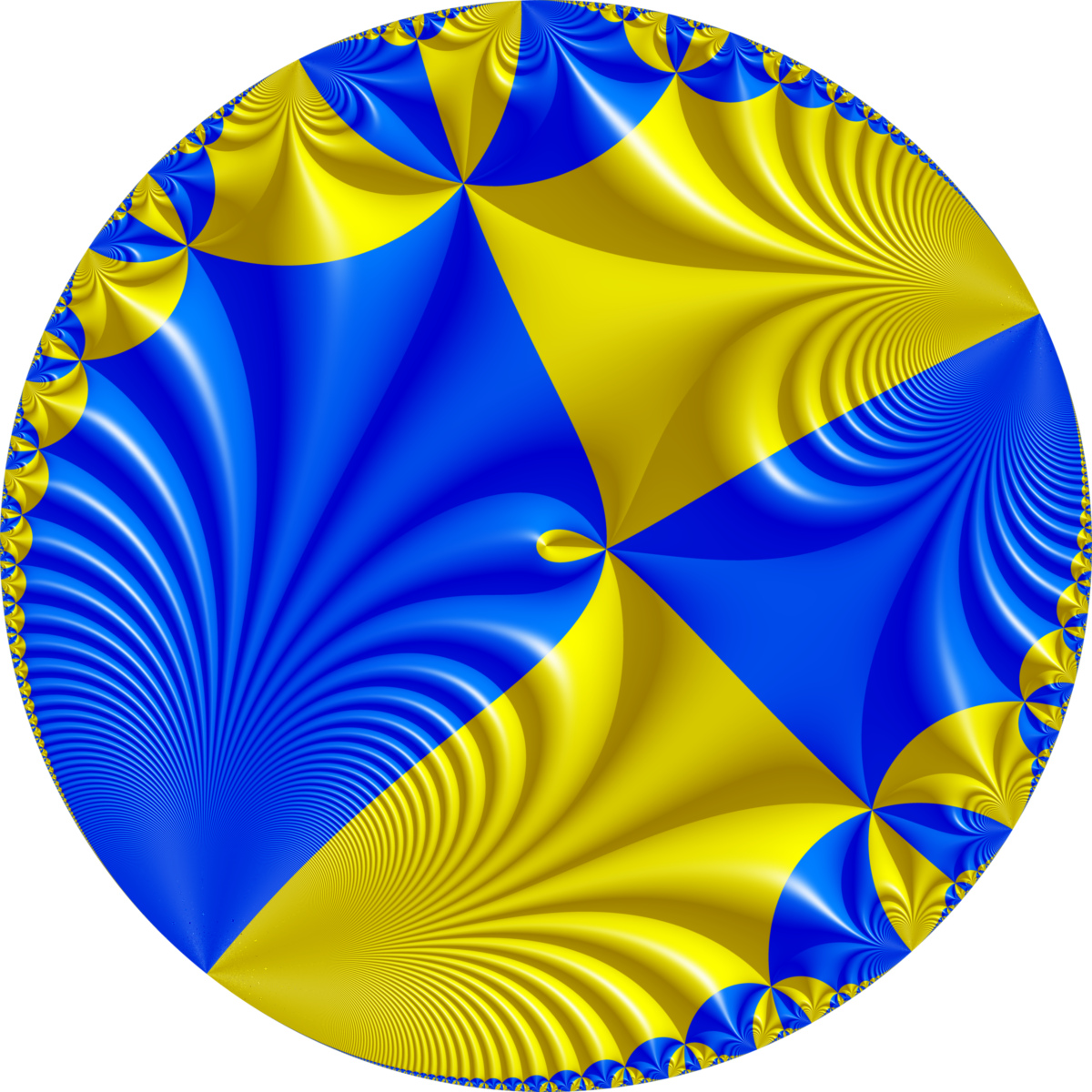}
\caption{Structure of $\cal R[B_3]$.}
\label{fig:struct-RB3}
\end{figure}

\begin{figure}
\includegraphics[width=10cm]{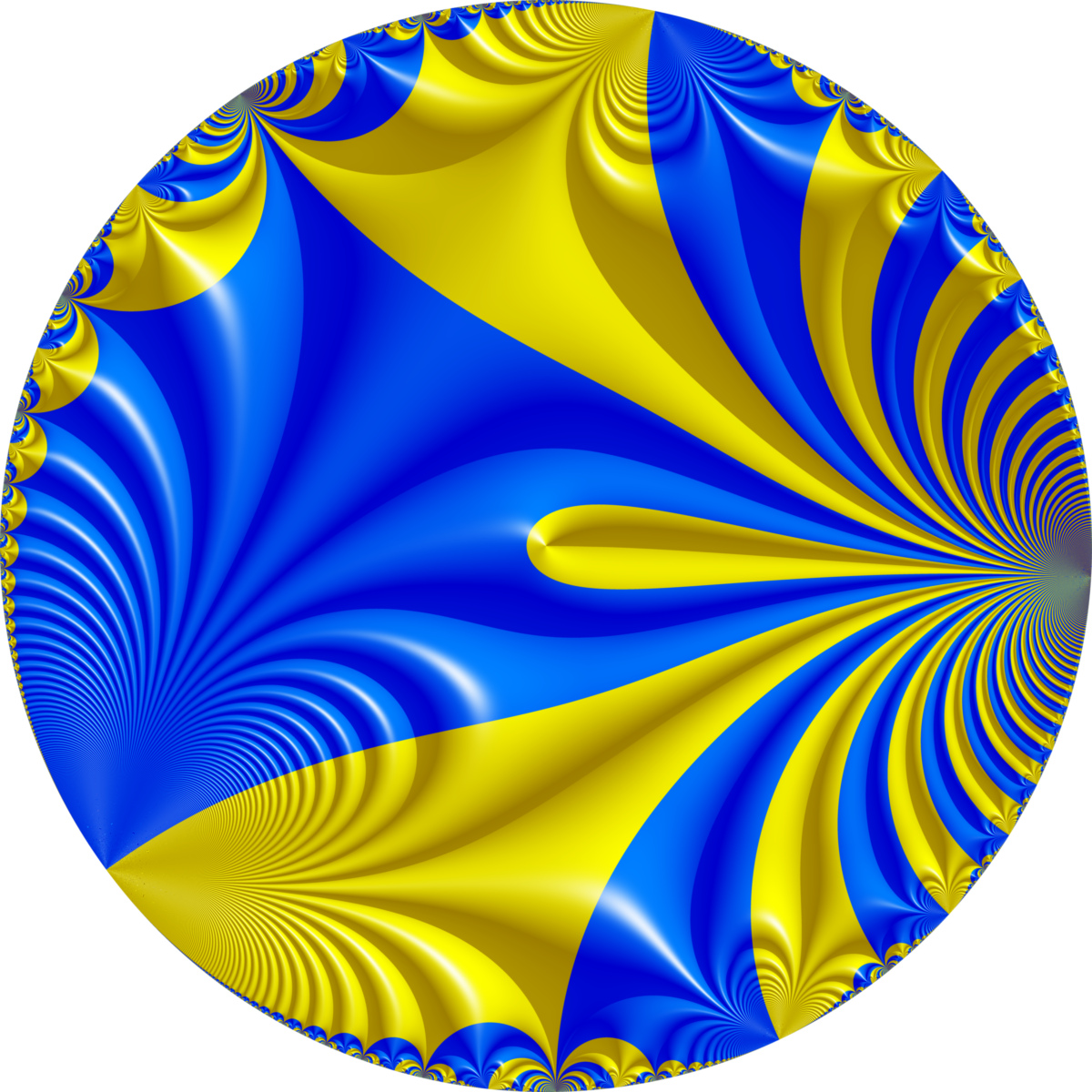}
\caption{Structure of $\cal R[B_\infty]$.}
\label{fig:struct-RBinf}
\end{figure}

\phantomsection
The next pictures illustrate \Cref{thm:shi2a,thm:shi2b}. \Cref{fig:th2:1}\label{here:atop} shows the famous case dubbed the Cauliflower: this is the Julia set of $z\mapsto z^2+1/4$.
We removed the colors and drew the boundaries between boxes and the boundaries of the definition domains.
The six images are ordered in a $2\times 3$ rectangle whose first column figures the dynamical chessboard of $f$ atop and of $B_2$ below.
The next column represents views of their chessboards in repelling Fatou coordinates or, more precisely, two periods of their preimage by $\Psi_\rep$.
The last column is the projection to $\C^*$ of the middle column by the map $b[f]^{-1}E: z\mapsto b[f]^{-1}\exp(2\pi i z)$ where $b[f]$ is some constant used in normalizing $\cal R$.
The vertical arrows are isomorphisms between the three pairs of domains, mapping graph to graph, respecting box colors (not figured here) and even better: they are structure isomorphisms for the following respective maps (properly normalized): the attracting Fatou coordinate for the first column, the horn map for the second column and the parabolic renormalization for the last one.
This diagram is also commutative if one adds the following self maps of the six sets: column~1: $f$, $B_2$, column~2: $T_1$, $T_1$, column~3: $\on{Id}$, $\on{Id}$.
From all this we can build a big commutative diagram, but we do not think that it would be much readable.
Note that the tiny loops in the last column are the images of the big unbounded square that lie above in both middle images.
The image of this square by $\Psi_\rep$ is one of the two $f$-invariant (resp.\ $B_2$-invariant) squares (they touch the fixed point), but the latter has many other preimages by $\Psi_\rep$. 

\Cref{fig:th2:2} shows the analog, but for the exponential map $z\mapsto e^z-1$. The caption of \Cref{fig:univ_exp_2} gives more explanation about the color scheme of the top row.
 One thing worth noticing in the middle top image: \emph{all} the yellow and blue components, not only the topmost, are unbounded (each has countably many arms that extend to the right, in some of the channels between hairs of the black set.) 
The image in the upper right corner, looking like a yin yang symbol, is very interesting, but we need to zoom near the center to see the details: this is done on \Cref{fig:yinyangzoom}.


\begin{figure}
\rotatebox{90}{\begin{tikzpicture}
\node at (0,0) {\includegraphics[width=4.5cm]{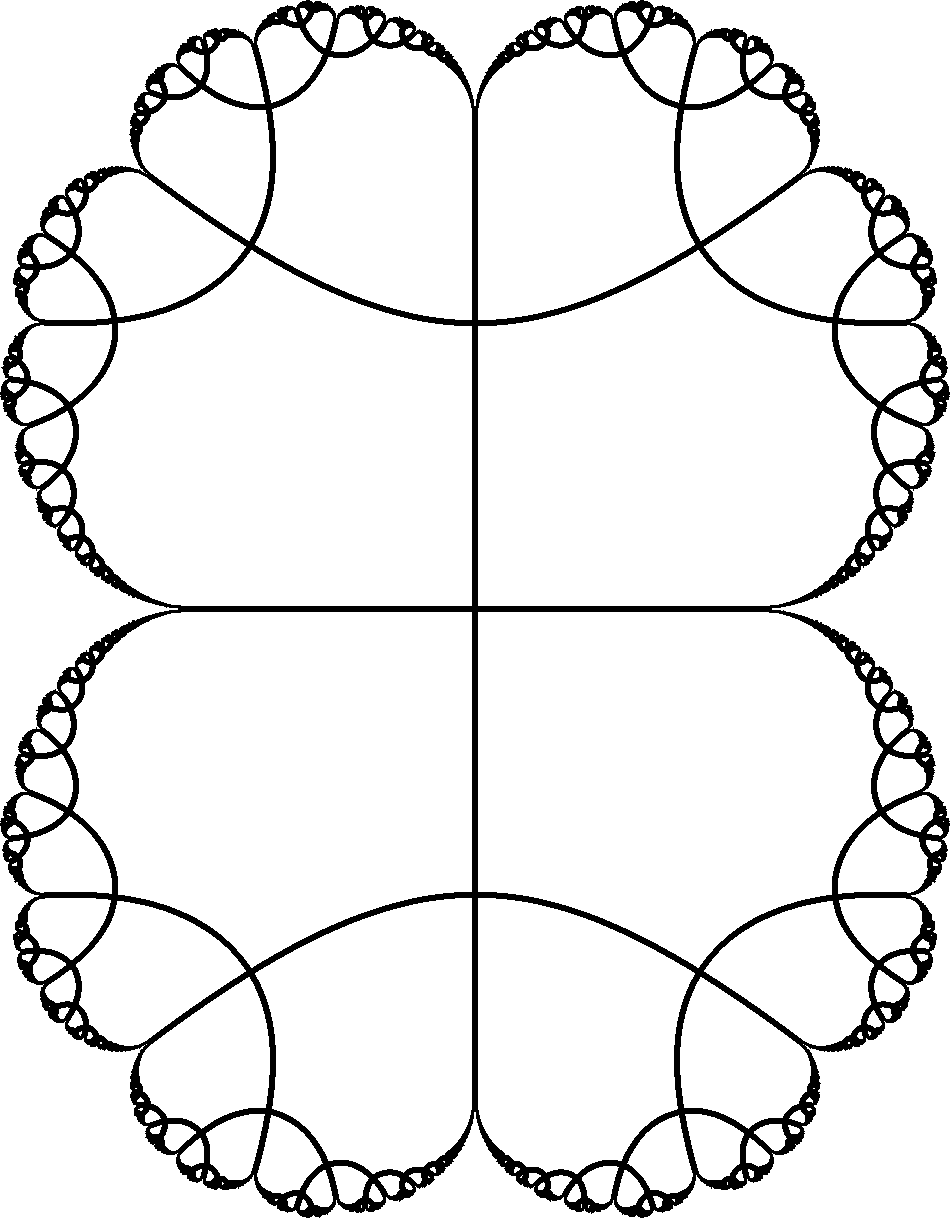}};
\node at (0,-7) {\includegraphics[width=5cm]{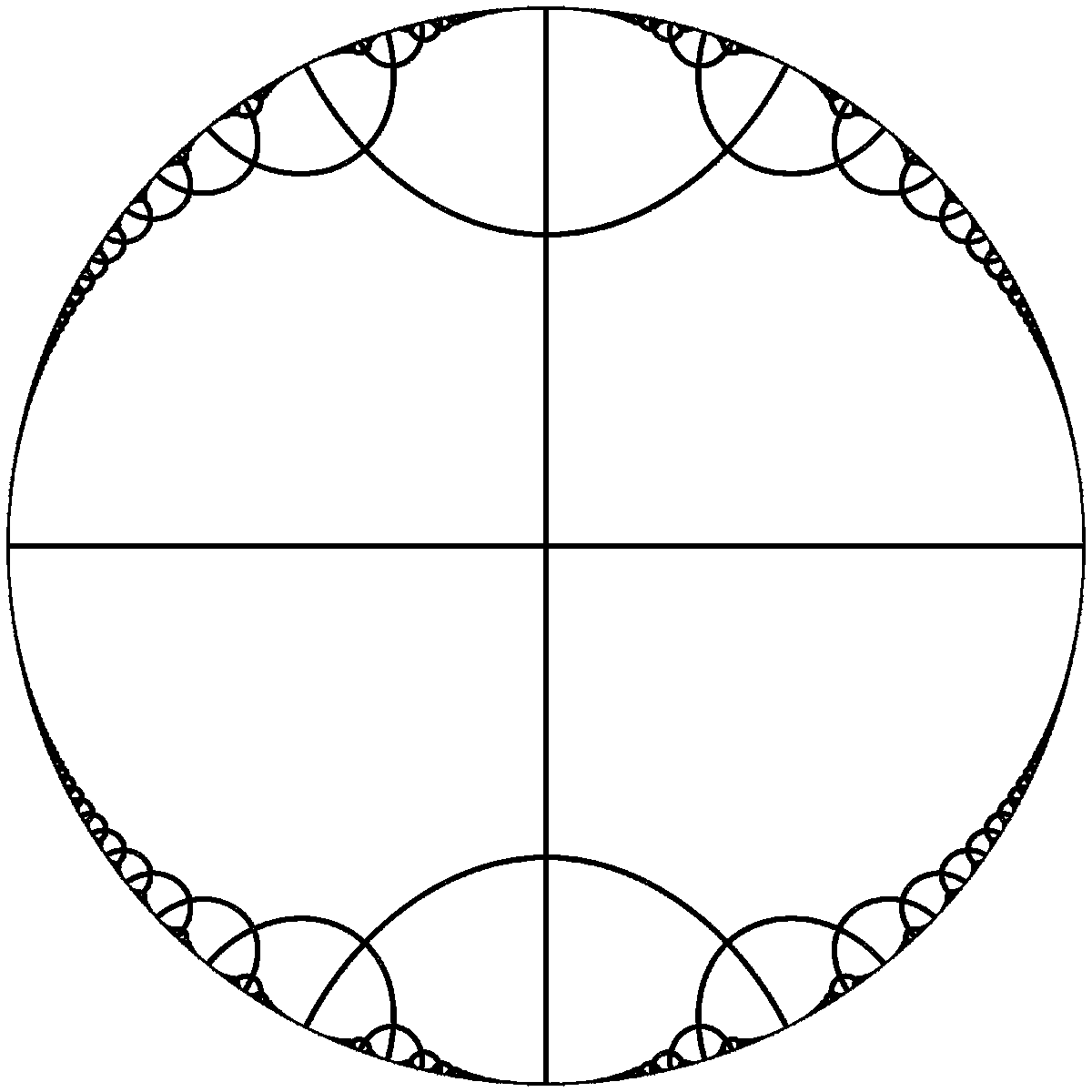}};
\node at (7,0) {\includegraphics[width=7cm]{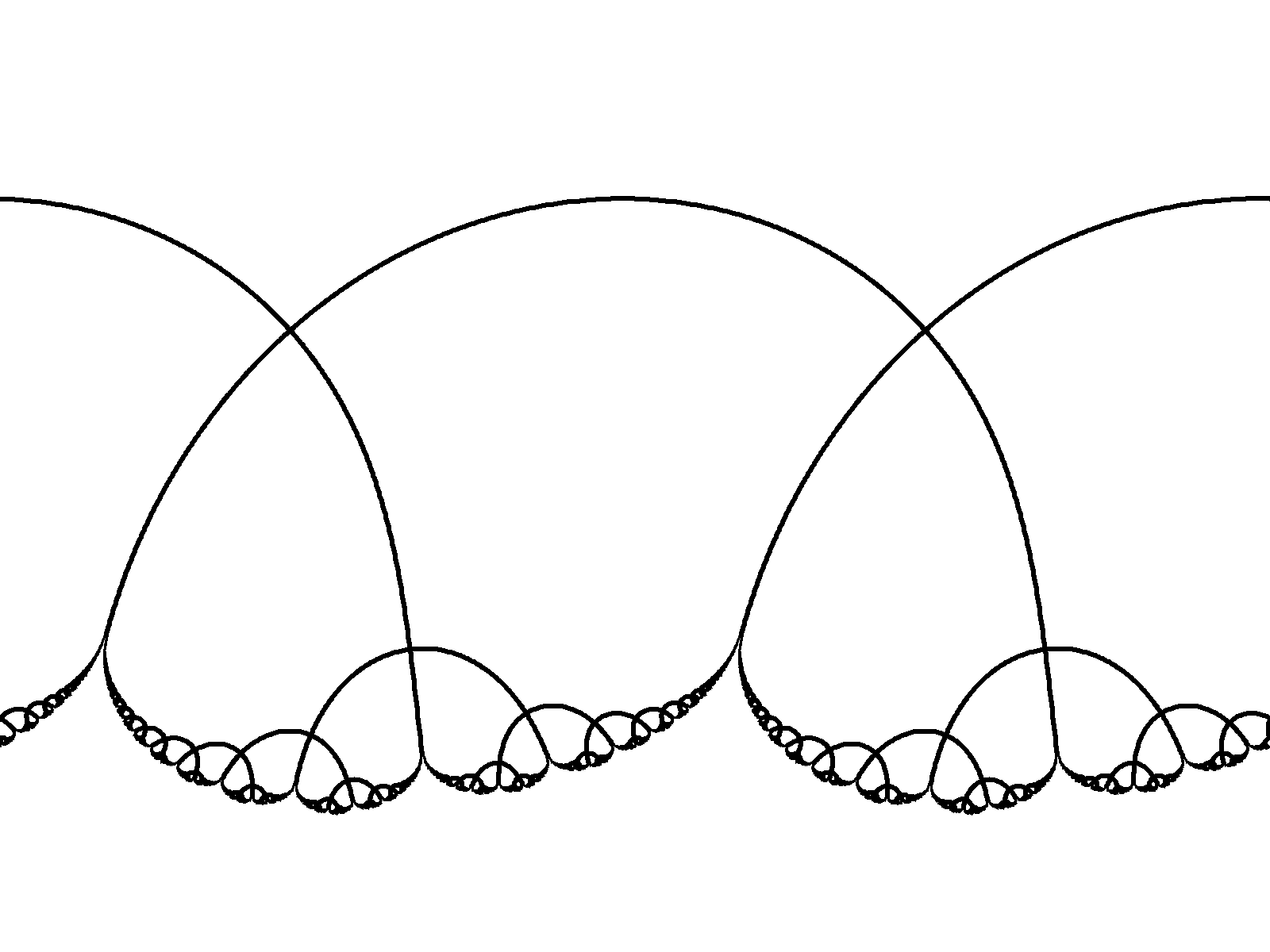}};
\node at (7,-6.5) {\includegraphics[width=7cm]{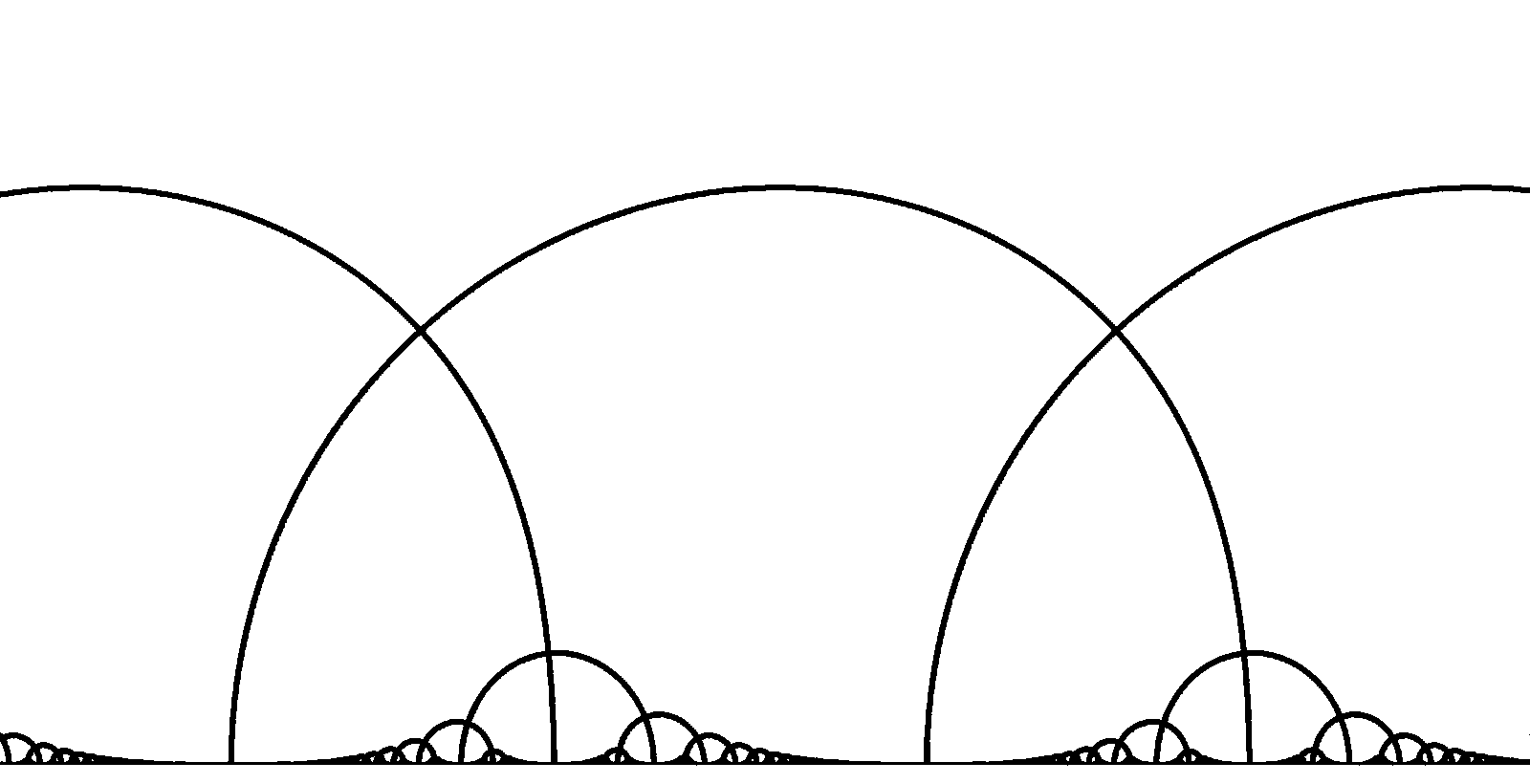}};
\node at (14,0) {\includegraphics[width=4.8cm]{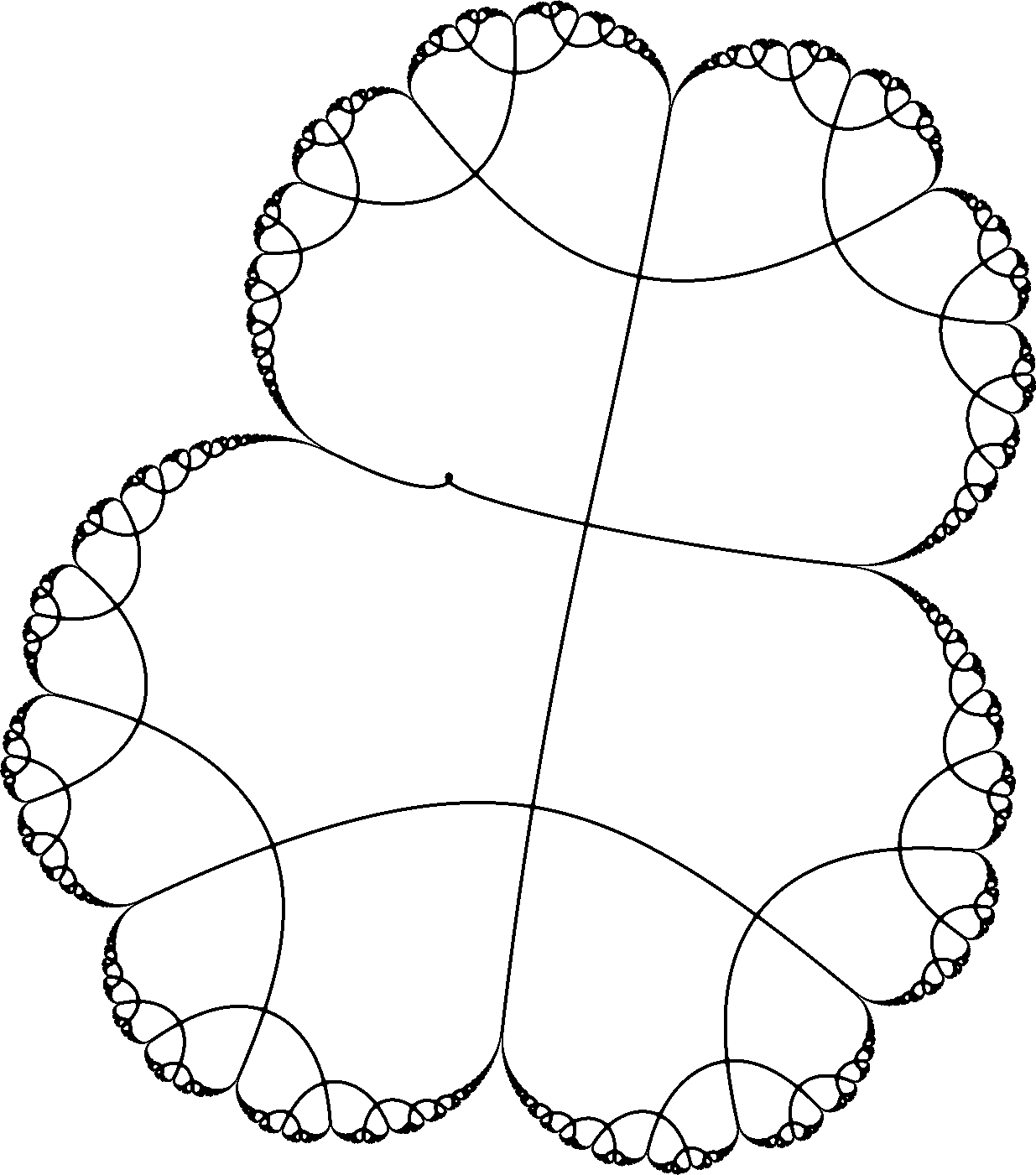}};
\node at (14,-7) {\includegraphics[width=5cm]{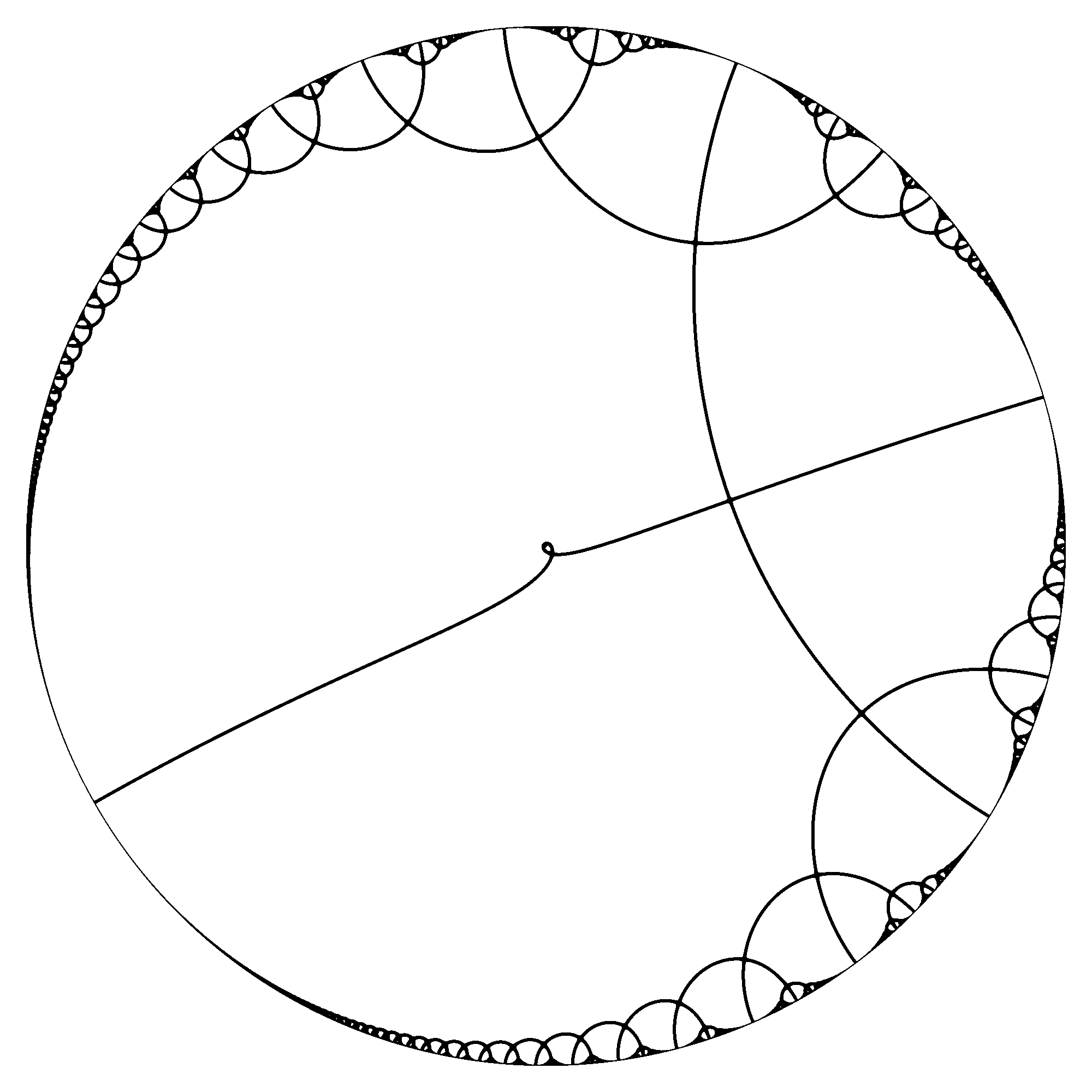}};
\draw[->] (0,-3.5) -- node[left] {$\zeta$} (0,-4.2);
\draw[->] (7,-3) -- (7,-4);
\draw[->] (14,-3.1) -- (14,-4);
\draw[->] (3.7,0) -- node[above] {$\Psi_\rep[f]$} (2.6,0);
\draw[->] (3.7,-7) -- node[above] {$\Psi_\rep[B_2]$} (2.6,-7);
\draw[<-] (11.3,0) -- node[above] {$\frac{1}{b}E$} (10.2,0);
\draw[<-] (11.3,-7) -- node[above] {$E$} (10.2,-7);
\end{tikzpicture}}
\caption{(rotated 90$^\circ$) Illustration of \Cref{thm:shi2a,thm:shi2b} for $f(z)=z^2+1/4$. See the text page~\pageref{here:atop} for a description.}
\label{fig:th2:1}
\end{figure}

\begin{figure}
\rotatebox{90}{\begin{tikzpicture}
\node at (14.2,0) {\includegraphics[width=5cm]{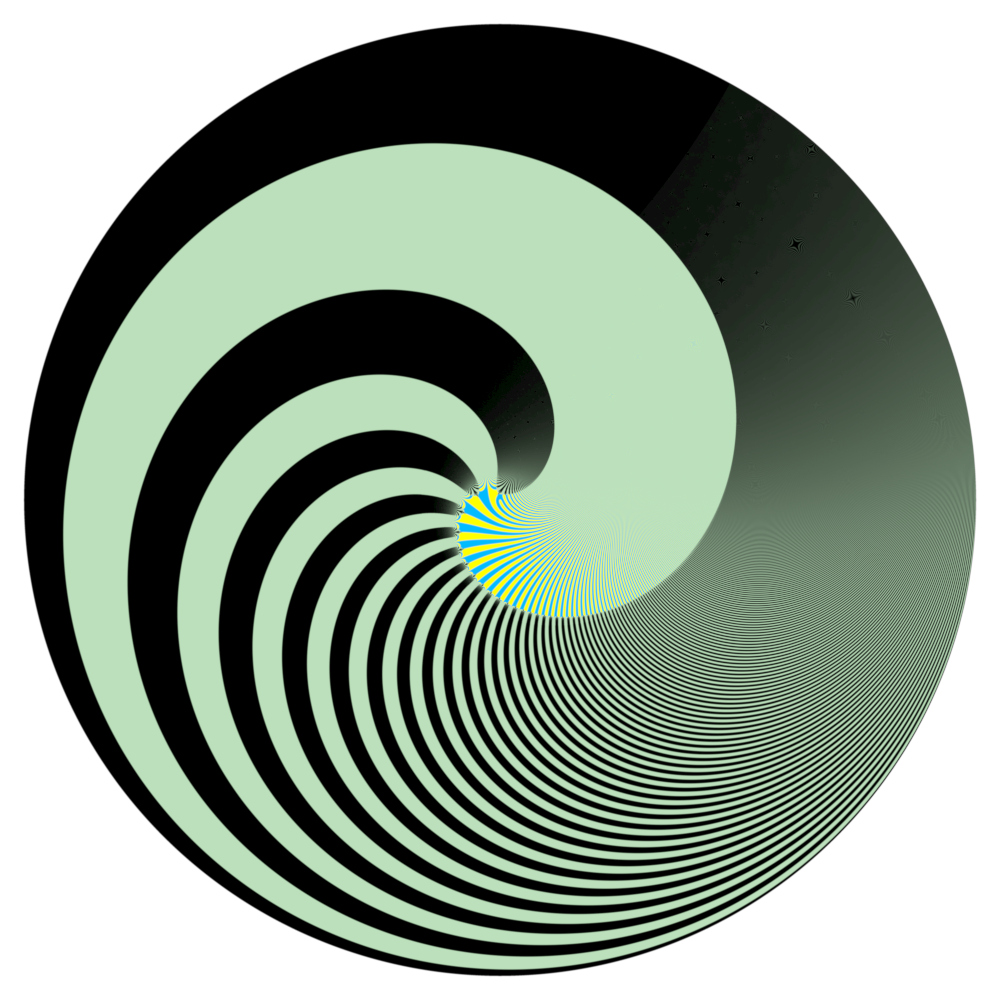}};
\node at (7,0) {\includegraphics[width=7cm]{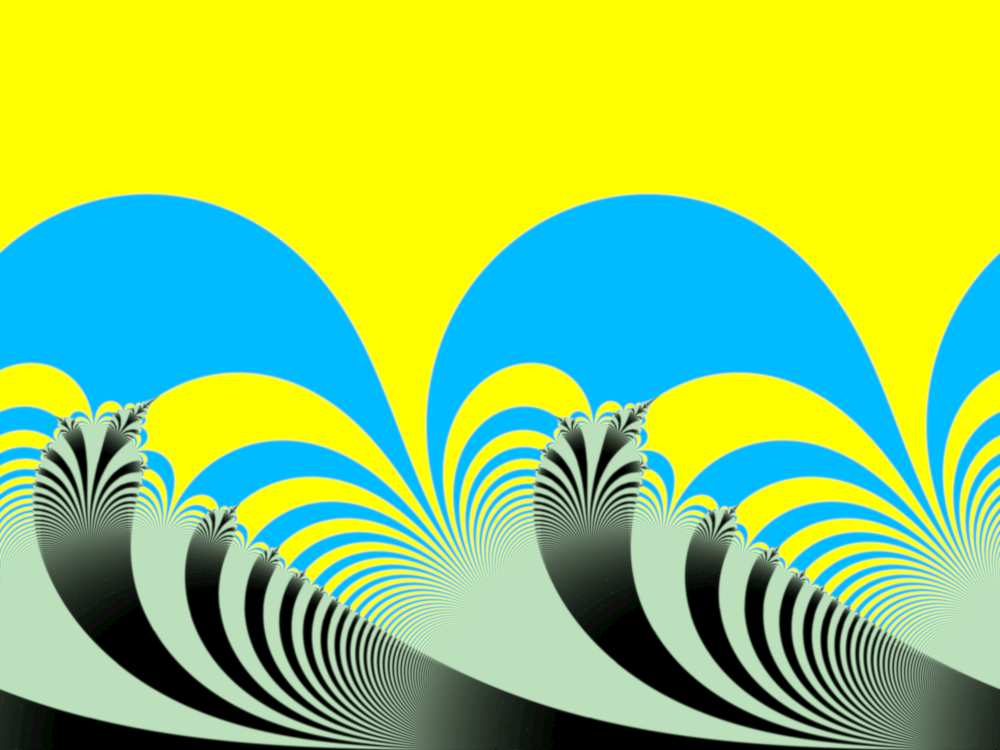}};
\node at (-0.5,0) {\includegraphics[height=4cm]{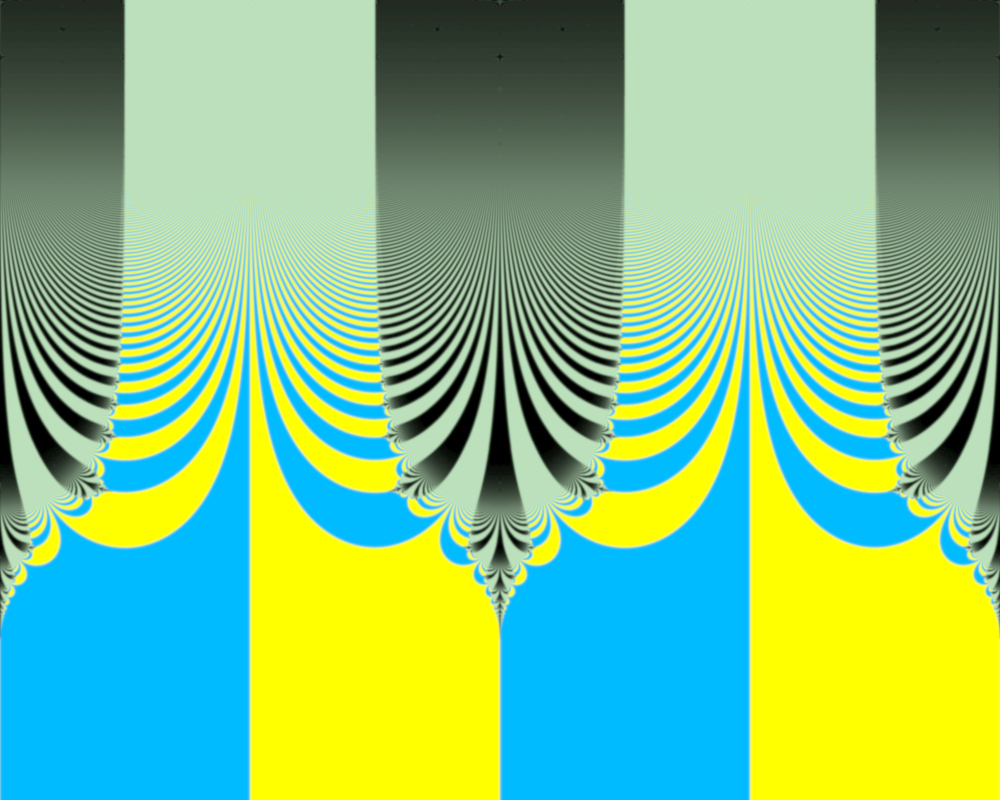}};
\node at (14.2,-7) {\includegraphics[width=5cm]{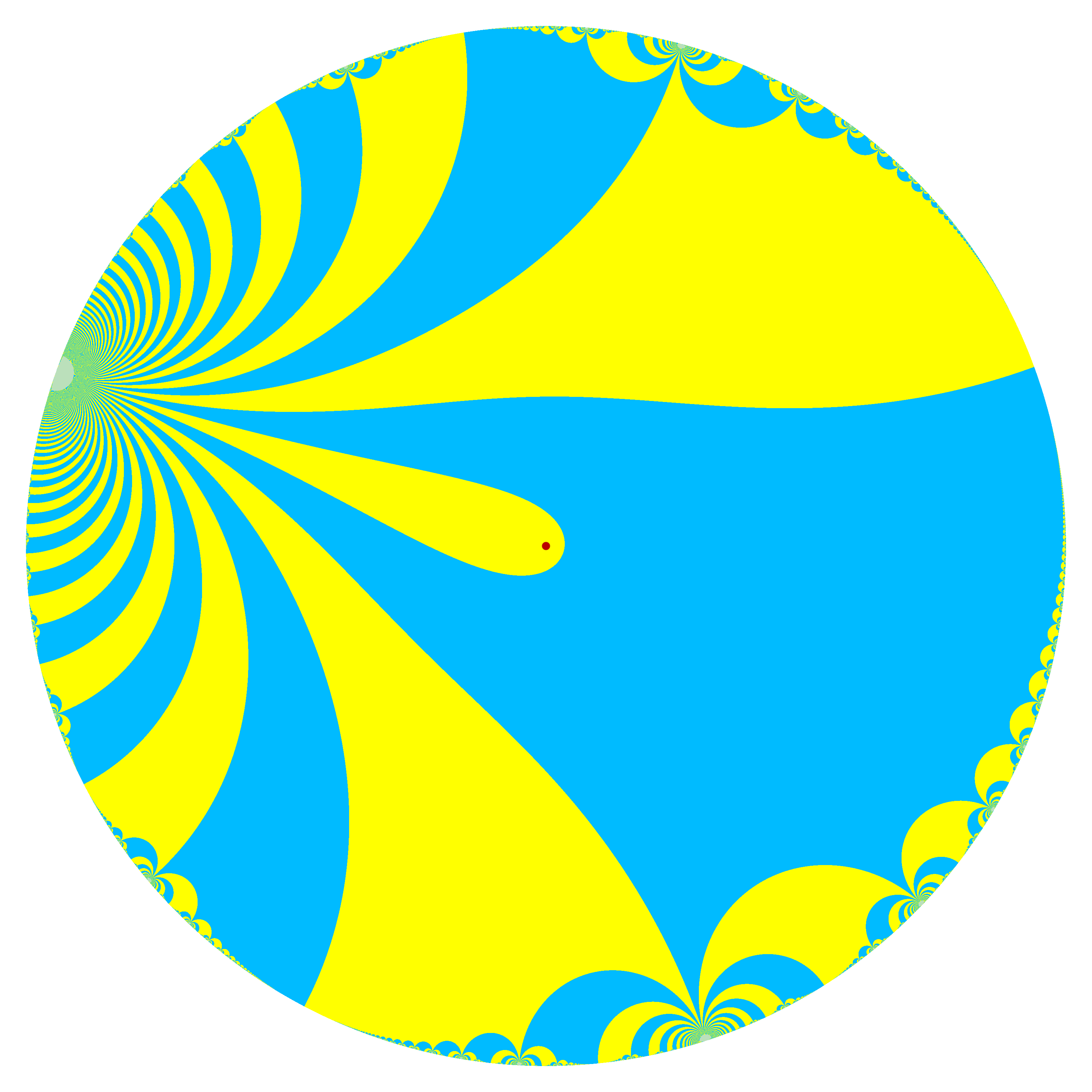}};
\node at (7,-7) {\includegraphics[width=7cm]{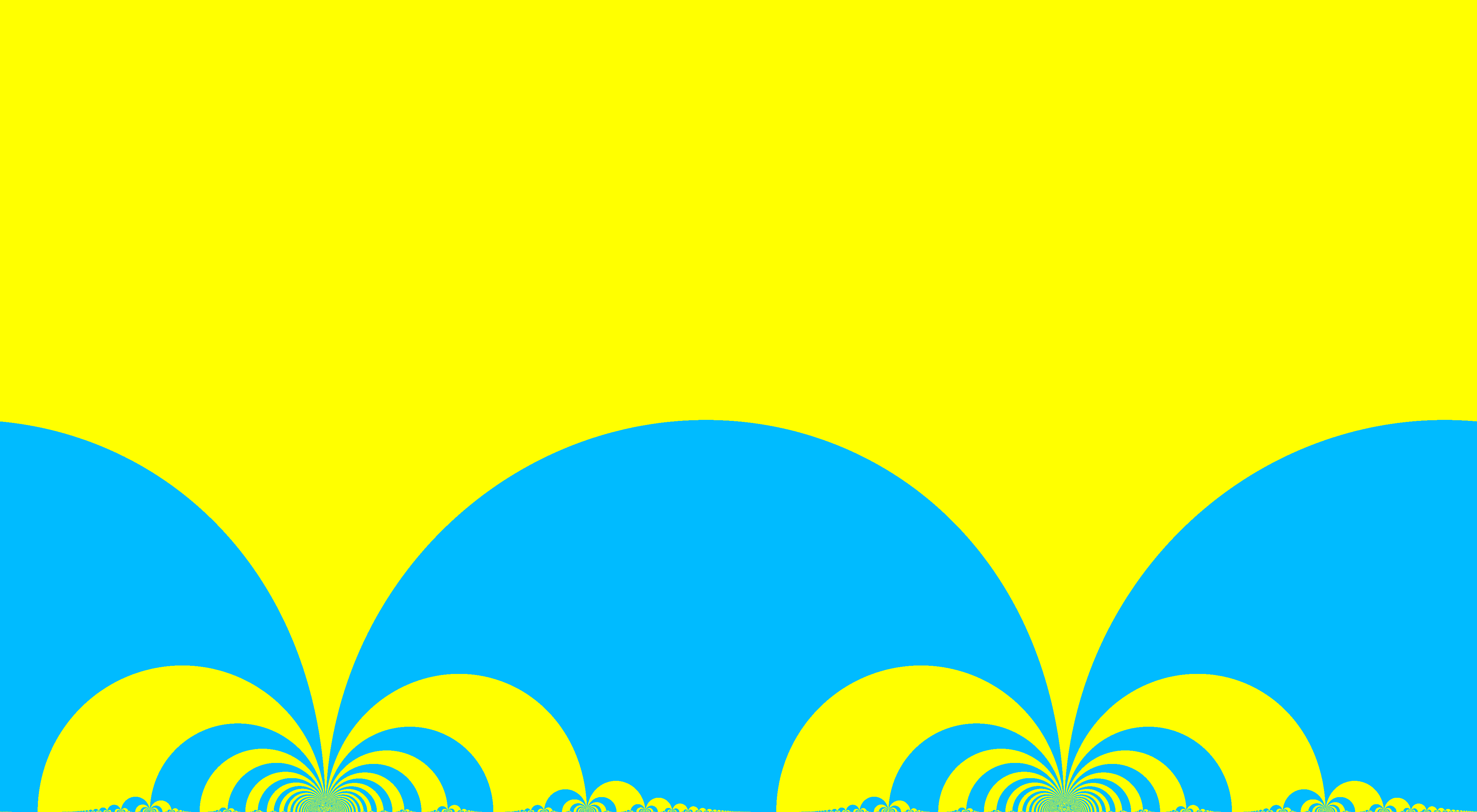}};
\node at (-0.5,-7) {\includegraphics[width=5cm]{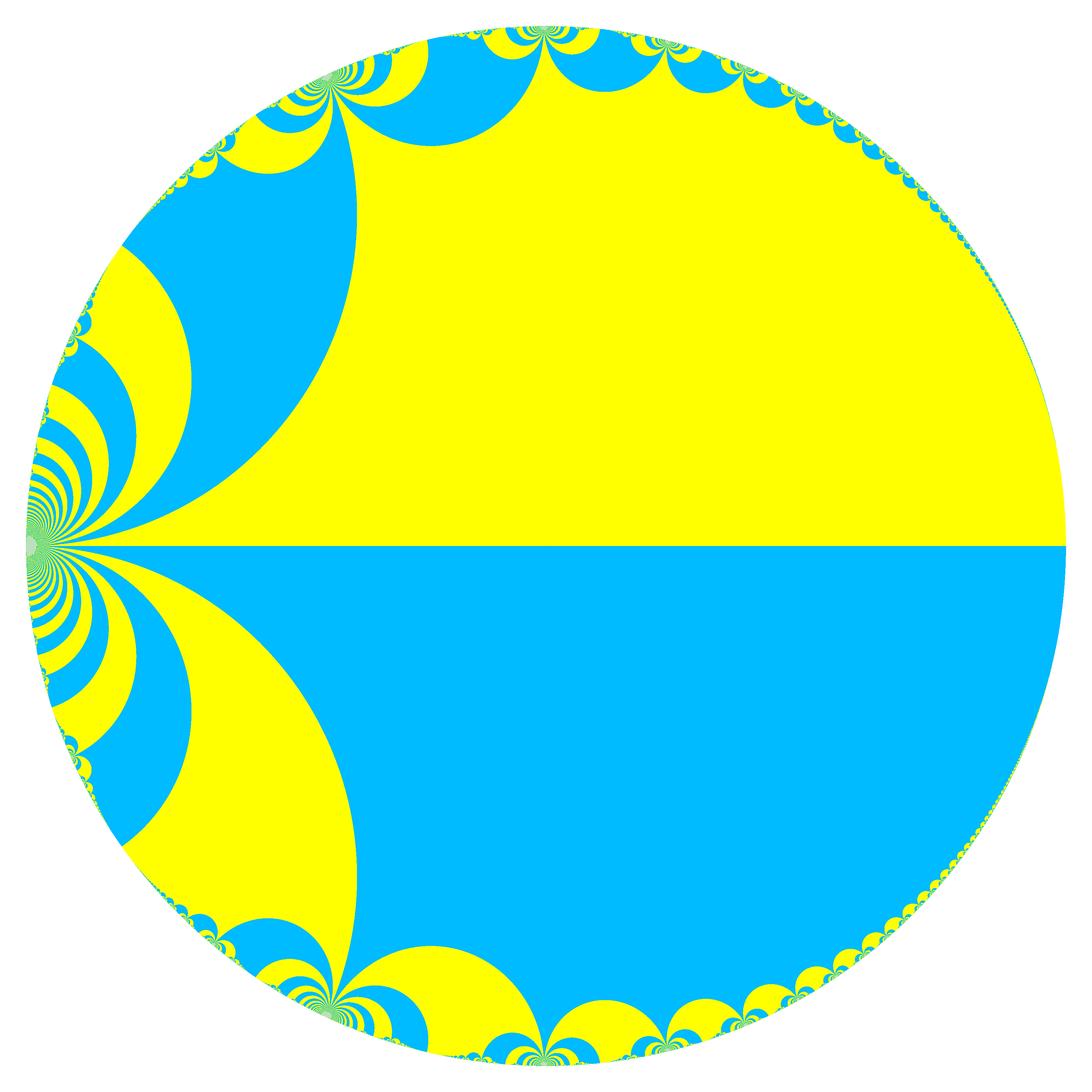}};
\draw[->] (-0.5,-3.3) -- node[left] {$\zeta$} (-0.5,-4.3);
\draw[->] (7,-3.3) -- (7,-4.4);
\draw[->] (14.2,-3.0) -- (14.2,-4.1);
\draw[->] (3.2,0) -- node[above] {$\Psi_\rep[f]$} (2.2,0);
\draw[->] (3.2,-7) -- node[above] {$\Psi_\rep[B_\infty]$} (2.2,-7);
\draw[<-] (11.6,0) -- node[above] {$\frac{1}{b}E$} (10.7,0);
\draw[<-] (11.6,-7) -- node[above] {$E$} (10.7,-7);
\end{tikzpicture}}
\caption{(rotated 90$^\circ$) Analog of \Cref{fig:th2:1} for $f(z)=e^z-1$ and $B_\infty$. See the text for a description. There are enlargements of the top right image on \Cref{fig:yinyangzoom}.}
\label{fig:th2:2}
\end{figure}

\begin{figure}
\includegraphics[width=10cm]{yinyang}
\\
\includegraphics[width=6cm]{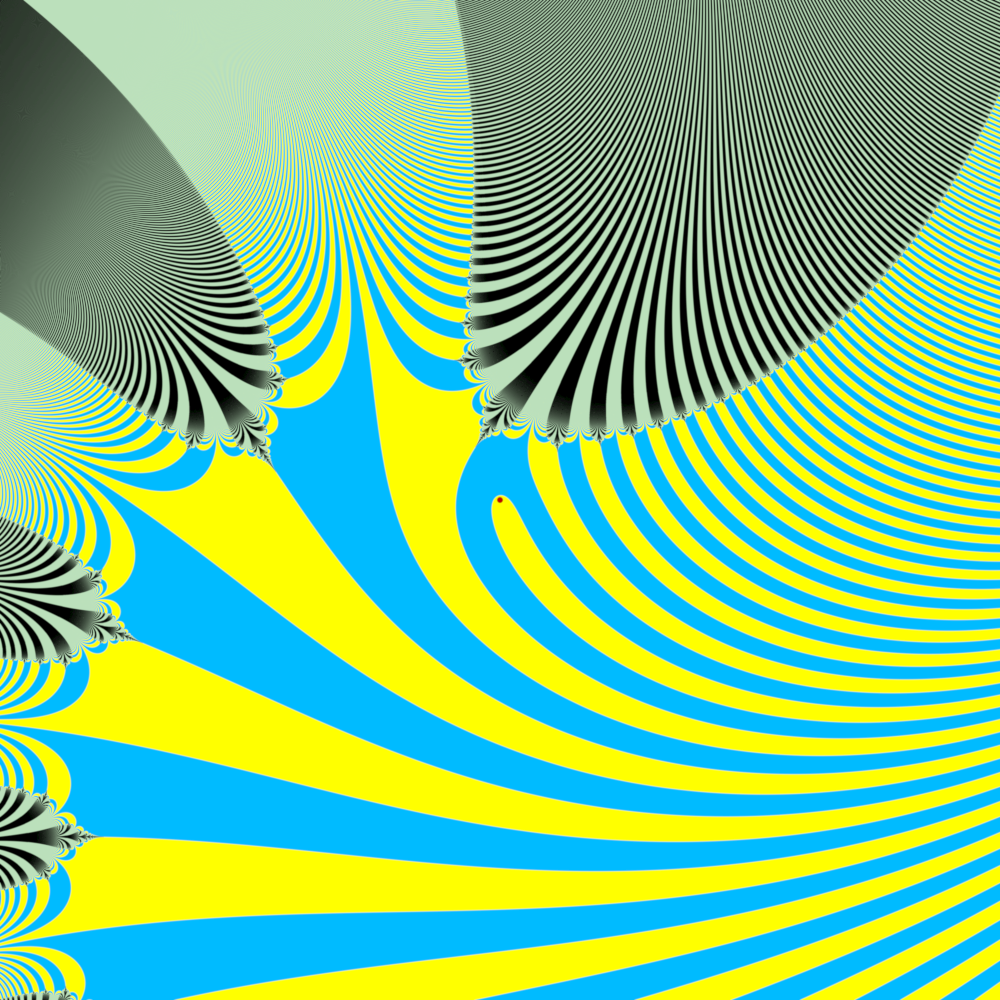}
\,
\includegraphics[width=6cm]{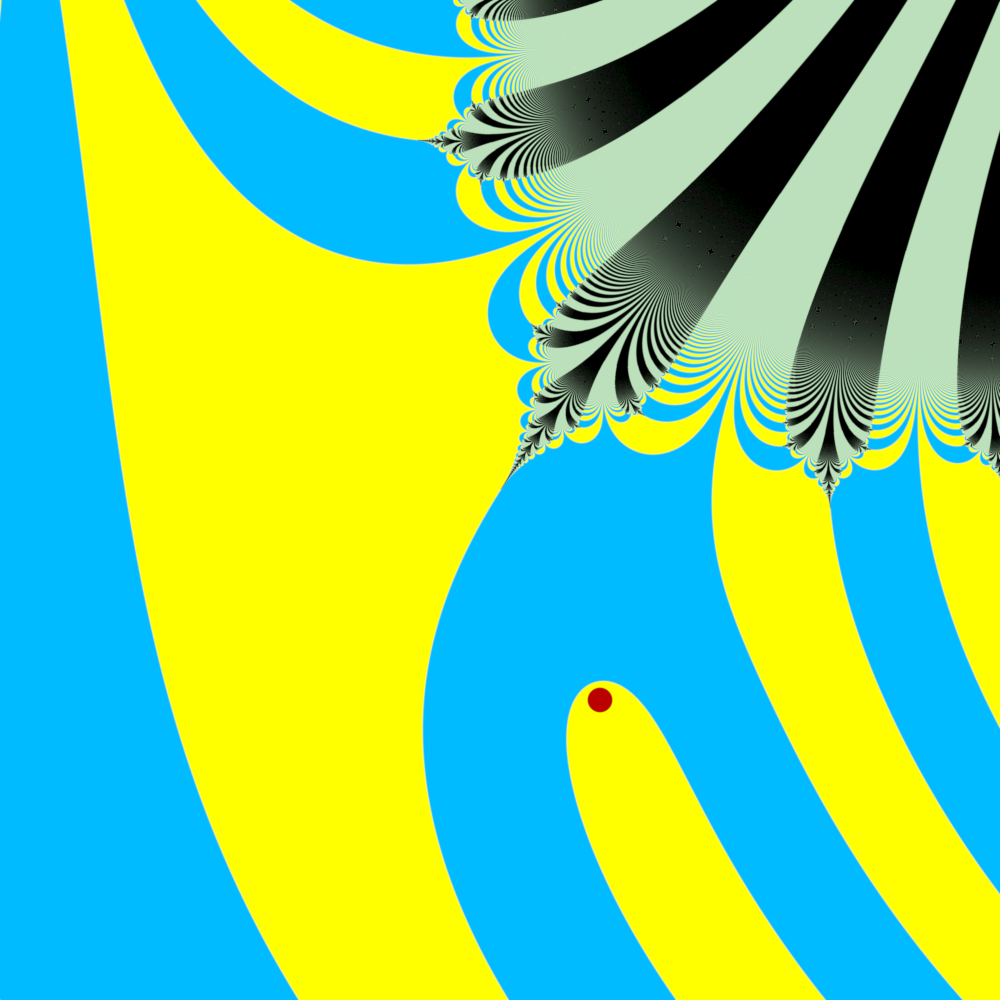}
\caption{Zoom on the origin, for the picture of the structure of the upper parabolic renormalization of the parabolic exponential map $f(z)=e^z-1$. The domain of definition of $\cal R[f]$ is the complement of the black set. The closure of this domain turns out to be a euclidean disk. The complement of this domain is sort of a Cantor bouquet that winds infinitely many times as one approaches the boundary of the disk. On the top picture, the structure is too fine to be properly seen. Below, we added two closer looks near the origin, pinpointed by a purple dot.}
\label{fig:yinyangzoom}
\end{figure}



\subsection{Inou and Shishikura's sub-structure}

To finish this visualization chapter, we present here the structure $\cal B$ of \Cref{thm:IS} (Inou-Shishikura's theorem), and how $\cal B$ and its sub-structure $\cal A$ fit as sub-structures of $\cal R[B_2]$.

\phantomsection
The first set\label{here:ISstruct} of drawings shows one of the ways Inou and Shishikura used to present it. They defined a Riemann surface with a natural projection over $\C/\Z$ as follows: cut the cylinder $\C/\Z$ so as to retain only the part where $\Im(z)>-\eta$ with $\eta=2$. (\footnote{There is some flexibility in the value of this lower bound, in \cite{IS}, they proved that their theorem holds for any real $\eta$ between $13$ and $2$ included. Here, we drew the domain only for their original value $\eta=2$.})
Slit this cylinder along the vertical segment from $0$ to $-\eta i$. To this, glue the rectangle $\Re(z)\in\,]-1,1[$ and $\Im(z) \in\,]-\eta,\eta[$, cut along the same segment. As usual with Riemann surfaces, we glue each side of the segment in one piece to the opposite side on the other piece. This is represented on the upper left part of \Cref{fig:IS1}. This method is reminiscent of the way Perez-Marco uses to build structures in his work.
Below it in the same figure, is a tentative to picture the way it projects to the cylinder $\C/\Z$, while on its right there is a planar open set isomorphic to it (conformal moduli are not respected in the figure). In the lower right corner, there is the image of the lower left by $z\mapsto \exp(2\pi i z)$ (rotated by 90 degrees). The right column is a map $f$ with structure $\cal B$ (the marked point is $z=0$).
The left part of \Cref{fig:IS2} accurately shows how $\cal B$ sits as a substructure of the structure of $\cal R[B_2]$. The right part identifies the pieces.

\newpage

\begin{figure}
\begin{tikzpicture}
\node at (0,0) {\includegraphics[height=6cm]{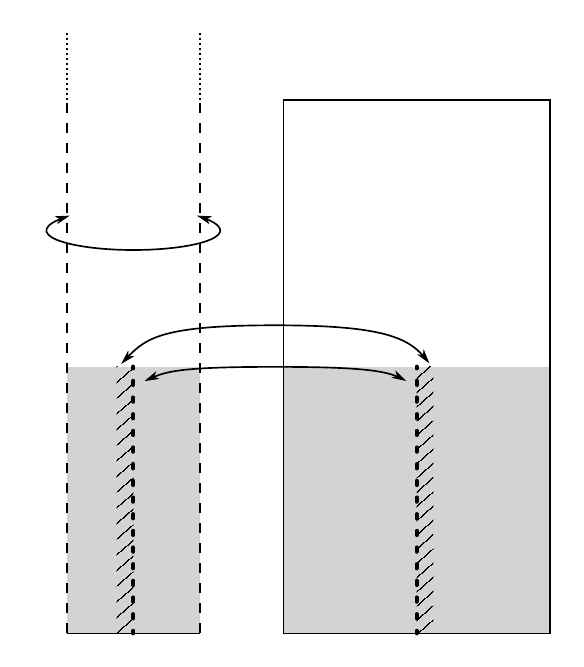}};
\node at (-0.1,-7) {\includegraphics[height=6cm]{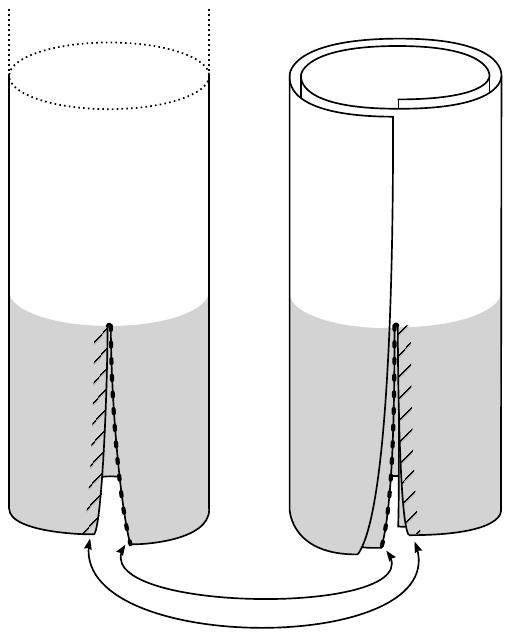}};
\draw[<->] (-1.4,-3) -- node[right]{$\simeq$} (-1.4,-4);
\draw[<->] (1.2,-3) -- node[right]{$\simeq$} (1.2,-4);

\node at (-1.4,1) {a};
\node at (-1.1,-1) {b};
\node at (-1.8,-1) {b};
\node at (1.1,1) {c};
\node at (0.5,-1) {d};
\node at (1.8,-1) {e};

\node at (5.5,-6.5) {\includegraphics[height=5cm]{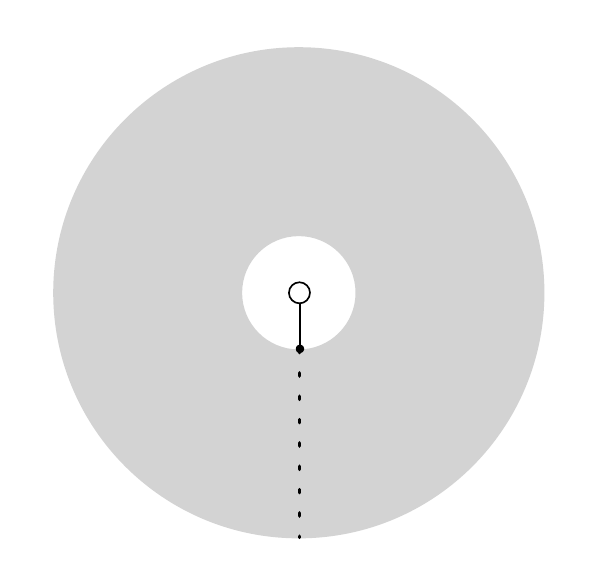}};
\node at (5.5,-0.3) {\includegraphics[height=5cm]{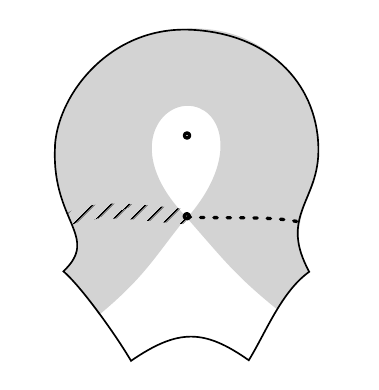}};
\draw[<->] (2.5,0) -- node[above]{$\simeq$} (3.5,0);
\draw[->] (2.4,-6.5) -- node[above]{$e^{2\pi iz}$} (3.2,-6.5);
\draw[->] (5.5,-3) -- node[right]{$f$} (5.5,-4);

\node at (5.5,0) {a};
\node at (5.5,1.3) {b};
\node at (5.5,-1.5) {c};
\node at (6.35,-1) {d};
\node at (4.7,-1) {e};

\end{tikzpicture}
\caption{Only the upper left section of this figure is conformally correct. Explanations in the text on page~\pageref{here:ISstruct}.}
\label{fig:IS1}
\end{figure}

\begin{figure}
\begin{tikzpicture}
\node at (0,0) {\includegraphics[width=6cm]{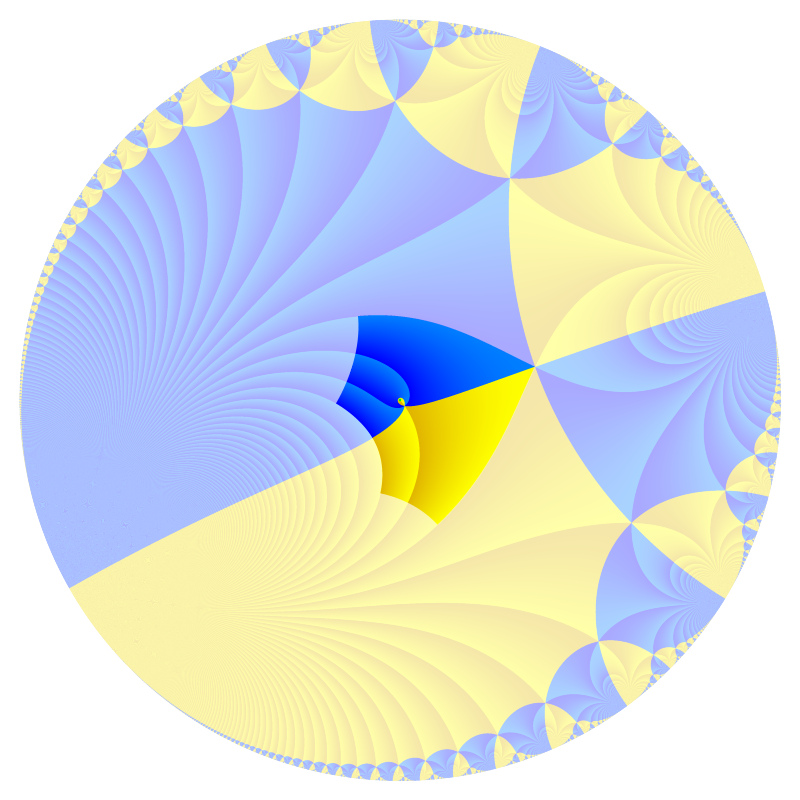}};
\node at (6,0) {\includegraphics[width=5cm]{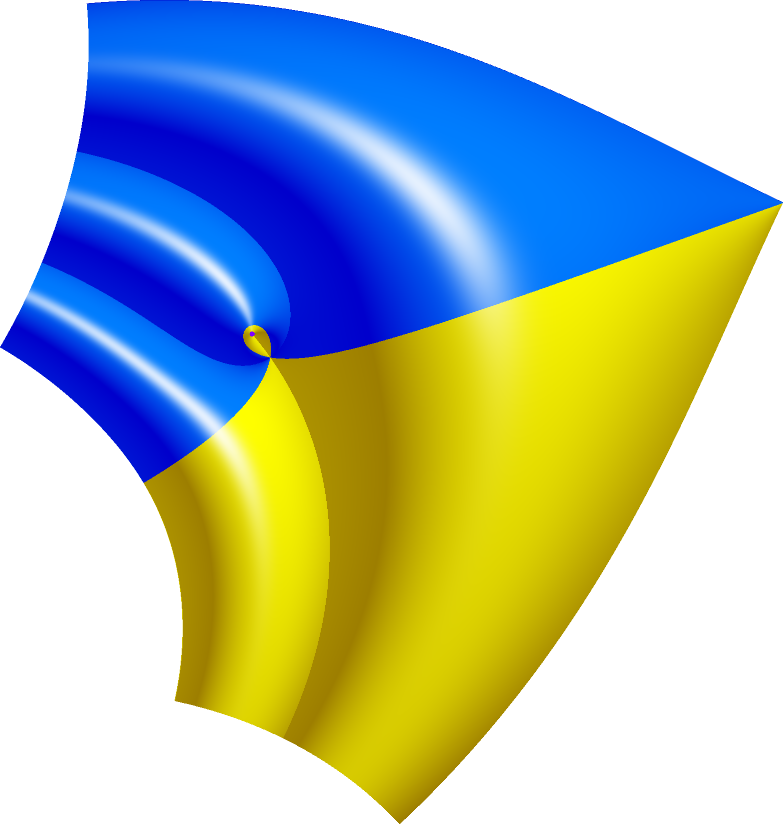}};
\node (a) at (3.8,-1.0) {a};
\draw[-] (a) -- (5.07,0.41);
\node at (3.6,1.5) {b};
\node at (5.05,-2.7) {c};
\path (4.6,-2.0) edge[-,bend right=20] (5.9,-2.7);
\node at (3.8,2.2) {d};
\node at (3.4,0.8) {e};
\end{tikzpicture}
\caption{Caption in the text. Note that compared to the upper right part of \Cref{fig:IS1}, there is a supplementary corner. The picture is accurate.}
\label{fig:IS2}
\end{figure}

\FloatBarrier

\newpage

Structure $\cal A$ is a substructure of $\cal B$, obtained by mapping conformally the domain of $f$ minus the origin to the complement of the closed unit disk and removing the interior of some specific and explicitly defined ellipse (see \cite{IS}). The result, mapped to the set of \Cref{fig:IS2}, is shown on \Cref{fig:IS3}.

Recall that $\cal R[B_d]$ is the unit disk $\D$. Let $U\Subset V\Subset \D$ be the sub-domains corresponding to respectively $\cal A$ and $\cal B$. Inou and Shishikura worked with the particular sets we just described. It is more natural, though not easy, to take for $U$ and $V$ a pair of disks centered on the origin. The objective of the present article is to prove that this works. The downside is that we lose unicriticality of maps in the class we construct. Yet, it still applies to unicritical polynomials, after taking one renormalization (they become multicritical, with only one critical value); recall that even Inou and Shishikura need to take first one iteration of renormalization of to get a map in their class from a quadratic polynomial anyway. The upside is that our approach will work for critical points of any degree.

\begin{figure}
\includegraphics[width=8cm]{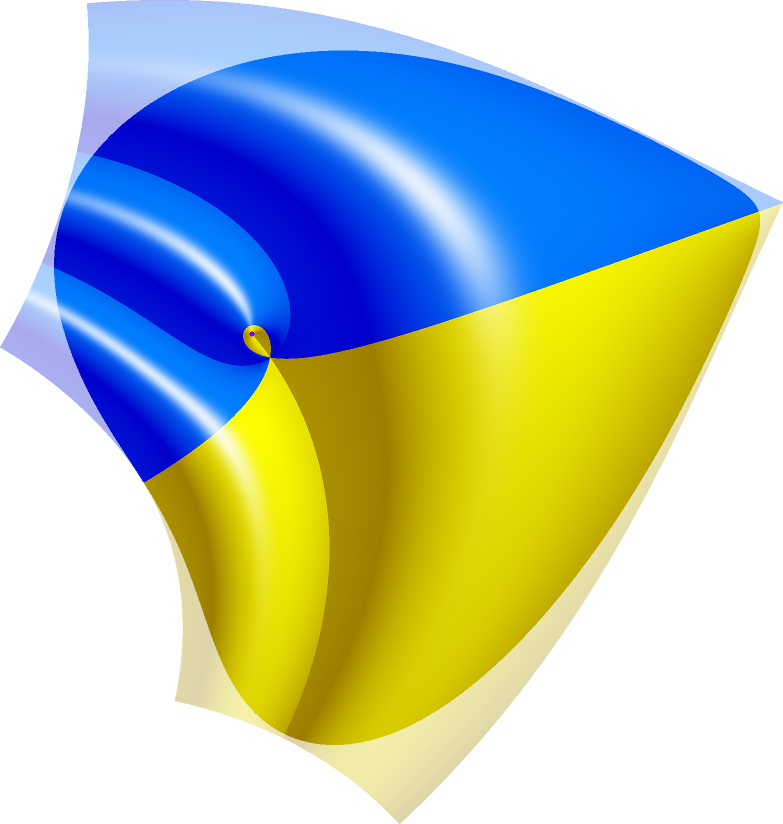}
\caption{Comparison of $\cal A$ and $\cal B$. The picture is accurate. Though it is hard to see, the boundary of the light-toned domain and the boundary of the color-saturated domain are disjoint.}
\label{fig:IS3}
\end{figure}

\FloatBarrier

\clearpage

\section{Proof}\label{sec:pf}

The element of hyperbolic metric of a connected open subset $U$ of $\C$ will be denoted by $\rho_U(z)|dz|$ and the corresponding hyperbolic distance by $d_U(z,z')$. \nomenclature[dU]{$d_U$}{hyperbolic distance w.r.t.\ $U$\nomrefpage}\nomenclature[rU]{$\rho_U$}{element of hyperbolic metric w.r.t.\ $U$\nomrefpage}

\subsection{A convenient notation}\label{subsec:morenot}

Given $r\in\,]0,1[$ and a subset $U$ of $\C$ conformally equivalent to $\D$ and containing $0$, we will denote
\[ \sub{U}{r}=\setof{z\in U}{d_U(0,z)<d_\D(0,r)}
.\]
Note that $\sub{U}{r}=\phi(B(0,r))$ where $\phi: \D \to U$ is a conformal isomorphism mapping $0$ to $0$.\nomenclature[Oo]{$\sub{}{}$}{$\sub{U}{r}$ is the set of points $z\in U$ with $d_U(0,z)<d_\D(0,r)$\nomrefpage}

Recall that we denoted $E(z)=e^{2\pi iz}$, which is a universal cover from $\C$ to $\C^*$. Given a set of the form $V=E^{-1}(U)$ where $U$ is as above, we will denote\nomenclature[Oo2]{$\esub{}{}$}{$\esub{V}{r}$ is the set of points $z\in V$ with $E(z) \in \sub{E(V)}{r}$\nomrefpage}%
\[\esub{V}{r} = E^{-1}(\sub{U}{r})
.\]

\subsection{Outline}\label{subsec:outline}

Our main theorem will be proved in two steps. Let us fix in this section some $d\in\N$ with $d\geq 2$. From now on, parabolic renormalization refers to upper parabolic renormalization.
In \Cref{subsec:preferred} we defined the objects $\Phi_\at[B_d]$, $\Psi_\rep[B_d]$, $h[B_d]$, $\cal H[B_d]$ and $\cal R[B_d]$ and adopted specific normalizations for them (except for $\cal R[B_d]$). 
In particular, we chose to define $h[B_d]$ on the upper half plane $\H$ only.
The map $\cal H[B_d]$ is defined on $\D$, maps $0$ to $0$ and satisfies $\cal H[B_d] \circ E = E \circ h[B_d]$. The map $\cal R[B_d]$ is defined as
\[\cal R[B_d] = b \cal H[B_d]\]
where $b\in\C^*$ is chosen so that $\cal R[B_d]'(0)=1$.
In \Cref{subsec:preferred} we introduced a semi-conjugate $C_d$ of $B_d$ by a 2:1 rational map, such that $C_d$ has only one attracting petal, and we gave relations between the objects for $B_d$ and the objects for $C_d$. Note that $\cal R[B_d]$ coincides with $\cal R[C_d]$ and
\[\dom \cal R[B_d] =\D.\]

Let\nomenclature[Fa]{$\cal F$}{Shishikura's invariant class\nomrefpage}
\[\cal F = \setof{\cal R[B_d] \circ \phi^{-1}}{\phi:\D\to \C \text{ is univalent and } \phi(z)=z+\cal O(z^2)}
\]
and\nomenclature[Fe]{$\cal F_\epsilon$}{a class of maps with slightly less structure\nomrefpage}
\[\cal F_\epsilon = \setof{\cal R[B_d] \circ \phi^{-1}}{\phi:B(0,1-\epsilon)\to \C \text{ is univalent and } \phi(z)=z+\cal O(z^2)}.
\]
In other words, $\cal F$ is %
the invariant class if Shishikura, Lanford, Yampolsky consisting of maps $f$ with a fixed point at the origin tangent to the identity and such that $f$ is structurally equivalent to the renormalization of the Blaschke product $B_d$, and $\cal F_\epsilon$ is a class of maps having only a subset of this structure. The smaller $\epsilon$, the more structure. Note that $\cal F=\cal F_0$.
To be more precise and to stick to the language introduced in \Cref{subsec:structeq}, let $I$ be a singleton. If we mark the origin by the unique map $I\to\{0\}$, maps in $\cal F_\epsilon$ are all $(I,\wh \C)$-structurally equivalent.

The maps $f\in \cal F$ have the same set of singular values as $\cal R[B_d]$ and they have the same nature: $\{0,v,\infty\}$ for some $v\in \C^*$ that depends only on $d$, with $0$ and $\infty$ two asymptotic values, $\cal R[f]^{-1}(\{0\})=\{0\}$, $\cal R[f]^{-1}(\{\infty\})=\emptyset$ and $v$ is a critical value that is not an asymptotic value, and $\cal R[f]^{-1}(\{v\})$ consists in regular points and critical points of degree $d$.

All maps $f\in\cal F$ are tangent to the identity at the origin. For the following statement, recall that $\sclass_d$ is the class of \Cref{def:sclass}.
\begin{proposition}\label{prop:fsd}
\[\cal F\subset \sclass_d\]
\end{proposition}
\begin{proof}
This follows from \Cref{lem:sufsd}: $f\in \cal F$ has exactly the same singular values as $\cal R[B_d]$, hence $f\in \cal S_{d'}$ for some $d'$.
We also have to check that $d=d'$.
The singular value of $f$ in the immediate basin $A$ is $v$ which is not an asymptotic value hence $d'\neq\infty$ (see \Cref{def:sclass}).
Then by \Cref{def:sclass} again, $f$ has a critical point of degree $d'$ in $A$, hence $d'=d$. 
\end{proof}

We will prove the following more precise version of the main theorem:

\begin{theorem*}
The main theorem page~\pageref{thm:main} holds with $\cal B=$ the structure of maps $f\in\cal F_{\epsilon_1}$ with marked point $0$ and $\cal A=$ the substructure $\cal F_{\epsilon_0}$, for some pair $\epsilon_0>\epsilon_1$.
\end{theorem*}

The class of Schlicht maps is denoted $\sch$, thus $\cal F=\setof{\cal R[B_d]\circ \phi^{-1}}{\phi\in\sch}$.
The two steps are the following:

\begin{enumerate}
\item\label{item:crt1} Contraction: for $f\in\cal F$ denote $f=\cal R[B_d]\circ \phi_1^{-1}$, $\phi_1\in\sch$. Then by \Cref{thm:shi2b}, with an appropriate normalization, $\cal R[f]$ is of the form $\cal R[B_d] \circ \phi_2^{-1}$, $\phi_2\in\sch$. We will prove that ``the definition of $\cal R[f]$ on $\sub{\dom (\cal R[f])}{(1-\epsilon)}$ uses only iteration of $f$ on $\sub{\dom (f)}{(1-\epsilon')}$ where $\epsilon'\gg\epsilon$\ ''.
\item Perturbation: for a map $f\in \cal F$, we will define a continuous deformation $f_t \in \cal F_t$. Every map in $\cal F_t$ will be a deformation of a map in $\cal F$. We will prove that $\cal R[f_t]$ has structure at least $\cal F_{\epsilon}$, provided $t\leq \epsilon'/K$ for some $K>1$, where $\epsilon'$ is given by the first step.
\end{enumerate}

Let us give a slightly more detailed formulation of these two steps; we leave here some imprecisions; they will be fully stated and proven in details in \Cref{subsec:part1} to~\ref{subsec:part2:3}. 

Step 1: Let $E(z)=e^{2\pi i z}$, $\Phi_\at$ the extended attracting Fatou coordinate of $f$, $\Psi_\rep$ the extended repelling inverse Fatou coordinate of $f$,
and recall that $\cal R[f](z)$ can be defined (up to pre and post composition by two linear maps) as
\[ E(\Phi_\at(f^m(\Psi_\rep(u))))
,\]
where $u\in E^{-1}(z)$ is chosen so that it belongs to the image of the repelling petal by the repelling Fatou coordinates and $m\in\N$ is chosen so that $f^m(\Psi_\rep(u))$ belongs to the attracting petal. So we are following the orbit of $w=\Psi_\rep(u)$ under iteration of $f$ from the repelling petal to the attracting petal. The claim is that this orbit stays in $\sub{\dom (f)}{1-\epsilon'}$.
Now recall that by the properties of the extended repelling Fatou coordinates, we have $f^k(w) = \Psi_\rep(u+k)$ and that the domain of definition of $\Psi_\rep$ is invariant by the translation $T_1$.
Therefore, using that $E^{-1}(\sub{\dom \cal R[f]}{1-\epsilon})$ is equal to the translate by an appropriate complex constant of the domain of the horn map $h$, point \eqref{item:crt1} above can be stated as follows:
\[ \Psi_\rep(\esub{\dom (h)}{1-\epsilon})\subset \sub{\dom (f)}{1-\epsilon'}
.\]
The relation $\epsilon' \gg\epsilon$ will take the form:
\[\log \frac{1}{\epsilon'} \leq c' + c \log\left(1+\log\frac{1}{\epsilon}\right)\]
for some positive constants $c$, $c'$ (\Cref{prop:part:1}).

Step 2: In the perturbation part, given $r=1-t_0$ and $f\in \cal F_{r}$, we define an element $f_0\in \cal F$ together with a smooth interpolation $f_t$, $t\in[0,t_0]$, between $f_0$ and $f=f_{t_0}$. It has the following form:
\[ f_t = \cal R[B_d] \circ \phi_t^{-1}.\]
The map $\phi_t$ is a univalent map, defined on $B(0,1-t)$ with $\phi_t(0)=0$ and $\phi'_t(0)=1$ and is defined as follows: let $r_t=1-t$, decompose $f(z)=\cal R[B_d]\circ \wt\phi^{-1}$, let $\phi(z) = r_{t_0}^{-1}\wt\phi(r_{t_0} z)$, whence $\phi \in \sch$, and define
\[ \phi_t(z) = r_t \phi (r_t^{-1} z)
.\]
The map $\phi_t$ is an isomorphism from $B(0,1-t)$ to $r_t\cdot\dom(f)$.
In particular $\phi_t$ is \emph{not} the restriction\footnote{It would be too much to ask for an interpolation $f_t = \cal R[B_d] \circ \phi_t^{-1}$ for which $\phi_t$ is the restriction of some $\phi$ to $B(0,1-t)$.
Look for instance at $t=r$: this would mean that the initial univalent map $\wt\phi$ is the restriction to $B(0,r)$ of the univalent map $\phi$. But there are plenty of univalent maps $\wt \phi$ on $B(0,r)$ that are not the restriction of a univalent map defined on $B(0,1)$.} of $\phi$ to $B(0,1-t)$, and
\[\dom f_t = r_t\cdot \dom f,\]
and is thus usually \emph{not} equal to $\sub{\dom(f)}{r_t}$.

Now since $f_0 = \cal R[B_d] \circ \phi^{-1}$ belongs to $\cal F$, its renormalization $\cal R[f_0]$ decomposes as $\cal R[B_d]\circ \phi_2^{-1}$ for some Schlicht map $\phi_2$.
By the first step, given $\epsilon>0$ and a point $z\in \dom \sub{\cal R[f_0]}{(1-\epsilon)} =\phi_2(B(0,1-\epsilon))$, we know that the value of $\cal R[f_0]$ is obtained through iteration under $f_0$ of a point $w$ in the repelling petal of $f_0$, point whose orbit remains in $\sub{\dom f_0}{1-\epsilon'}=\phi(B(0,1-\epsilon'))$ with $\epsilon' \gg \epsilon$.
We will then vary $t$ from $0$ to $t_0$ and follow by continuity the points in the orbit of $w$, not by fixing the initial value, but instead by imposing that their attracting Fatou coordinate stays the same, where we normalize the attracting Fatou coordinates (it varies with $t$ since $f_t$ does) by putting its critical values at the nonnegative integers. In particular, $w$ moves with $t$.
A local study shows that the tail of the orbit will not move much. The motion of the other points will be bounded from above inductively by iterating backwards along the orbit, until we reach $w$. We will measure the motion in terms of the hyperbolic metric on the complement in $\C$ of the post-critical orbit of $f_0$. 
The study will show (\Cref{prop:surv}) that there is some $K>0$ independent of $f$ (necessarily $K>1$) such that, provided $\epsilon'$ is small enough, an orbit that is initially completely contained in $\sub{\dom (f_0)}{1-\epsilon'}$ survives all the way as $t$ varies from $0$ to $\epsilon'/K$.
Thus $\cal R[f_t]$ has at least structure $\cal F_\epsilon$ provided $t\leq \epsilon'/K$.
The main theorem thus holds for $\cal A=\ov{(0,f\in\cal F_{\epsilon_0})}$ and $\cal B=\ov{(0,f\in\cal F_{\epsilon_1})}$ with
$\epsilon_0=\epsilon'/K$ and $\epsilon_1=\epsilon$ with $\epsilon$ small enough, as $\epsilon' \gg \epsilon$ will imply $\epsilon_0 >\epsilon_1$.

\subsection{Normalizations}\label{subsec:nor}

In the rest of \Cref{sec:pf}, i.e.\ in the proof of the main theorem,
\begin{itemize}
\item normalized Fatou coordinates refer to the normalization by the the asymptotic expansion at infinity, convention numbered~\ref{item:nor:2} on page~\pageref{item:nor:2},
\item $\Phi_\at$ will refer to extended attracting Fatou coordinates, normalized according to the same convention,
\item $\Psi_\rep$ will refer to extended inverse of the repelling Fatou coordinates that are normalized according to the same convention,
\item $h[f] = \Phi_\at \circ \Psi_\rep$,\nomenclature[hf]{$h[f]$}{normalized extended horn maps, $h[f] = \Phi_\at \circ \Psi_\rep$\nomrefpage}
\item $\cal R[f]$ is the parabolic renormalization, normalized by the critical value 
(convention numbered~\ref{item:nor:3} on page~\pageref{item:nor:3}); see details below,
\item in the second step, we will use the notation $\Psi_t$ and $\Phi_t$ to denote the extended repelling inverse Fatou coordinate and the extended attracting Fatou coordinate of $f_t$, normalized \emph{not} by their asymptotic expansion but according to a convention analog to number~\ref{item:nor:3}.
\end{itemize}

Let $f$ satisfy the hypotheses of \Cref{thm:shi2b}.
Let us call (only in this paragraph) $U$ the connected component of $\dom(h[f])$ that contains an upper half plane and $\Xi$ the map such that
\[E\circ h[f]\big|_U = \Xi \circ E.\]
Then $\cal R[f] = M_a \circ \Xi \circ M_b^{-1}$ for a pair of linear maps $M_a : z\mapsto a z$ and $M_b: z\mapsto b z$ that depend on $f$, hence
\begin{equation*}
M_a\circ E\circ h[f]\big|_U = \cal R[f] \circ M_b \circ E
.
\end{equation*}
The constants $a$ and $b$ depend on $f$.

By \Cref{thm:shi2b} there exists a choice of $a$ and $b$ in~the equation above, such that
\[ \cal R[f]=\cal R[B_d]\circ\phi^{-1}
\]
i.e.\ such that $\cal R[f]\in\cal F$. This is the normalization that we choose for $\cal R[f]$. We have $\cal R[f]'(0)=1$ and
$\cal R[f]$ and $\cal R[B_d]$ have the same (unique) critical value, and these two conditions characterize this choice of normalization. It coincides with convention numbered~\ref{item:nor:3} on page~\pageref{item:nor:3}. The class $\cal F$ is stable by renormalization with this convention:
\[ \cal R : \cal F\to \cal F
.\]

%
%
%

For reference, let us state here the following version of universality
\begin{lemma}\label{lem:phiatu}
For $f\in\cal F$, let $v_f$ denote the critical value of $f$ and $v'_f=\Phi_\at(v_f)$.
There is a conformal map $\phi$ from the upper component $U[B_d]$ of $\dom(h[B_d])$ to the upper component $U[f]$ of $\dom(h[f])$ that commutes with $T_1$ and such that
\[T_{\tau}\circ h[f]\big|_{U[f]} = h[B_d]\circ \phi^{-1}\]
with $\tau=v'_{B_d}-v'_f$ and $T_\tau(z)=z+\tau$.
\end{lemma}
\begin{proof}
By \Cref{cor:univ}, $\Phi_\at[B_d] \circ \zeta = \tau+\Phi_\at[f]\big|_A$ for some $\tau\in\C$ and $\zeta: A\to \D$ the conjugacy from $f$ on its immediate parabolic basin to $B_d$. By applying to the unique critical value of $f\big| A$ we get $\tau = v'_{B_d}-v'_f$.
By the complement after \Cref{thm:shi2b}, $\Psi_\rep[B_d] \circ \phi^{-1} = \zeta \circ \Psi_\rep[f]$ for some conformal isomorphism $\phi: U[B_d] \to U[f]$ that commutes with $T_1$. We conclude using $h=\Phi_\at\circ\Psi_\rep$.
\end{proof}

\subsection{Chessboards}\label{subsec:sdc}

Just before we begin the proofs, let us recall that maps $f\in \cal F$ have a \keyw{structural chessboard} and a \keyw{dynamical chessboard}.
The first is a partition of $\dom f$ that is a pre-image by $f$ of the partition of $\C^*$ cut by the circle of center $0$ and passing through the critical value of $f$. The second is a partition of the basin (or of the immediate basin) of the parabolic point $z=0$ of $f$, and is $f$-invariant. The second is also a structural object w.r.t.\ $\Phi_\at[f]$.
See \Cref{subsec:viz} for more details.

We defined a chessboard for the horn maps $h$ associated to parabolic points of maps $f\in\cal F$ (more generally to maps $f$ in the class $\sclass_d$ of \Cref{def:sclass}). It is the preimage in repelling Fatou coordinates of the dynamical chessboard of $f$ and it is also the preimage by $h$ of the partition of its range cut by a horizontal line. There is a box that contains an upper half plane, we call it the \keyw{main upper box} of $h$. Similarly the box that contains a lower half plane is called the \keyw{main lower box} of $h$. 

The map $\phi$ introduced in \Cref{lem:phiatu} maps the chessboard decomposition of $h[B_d]$ to the chessboard decomposition of $h[f]$.

\subsection{Toolkit}\label{subsec:toolkit}

In this section we redo classical computations on Fatou coordinates and first terms of their expansion. We add dependence on a map staying in a compact class and put the emphasis on uniformity of the bounds obtained. The section mainly serves as a reference for the rest of the text. The trusting reader may skip it.

\subsubsection{Compact classes of parabolic maps with one attracting petal}\label{subsub:tg}

\begin{proposition}\label{prop:cf}
Assume $\cal G$ is a set of holomorphic maps $g:\D\to \C$ with $g(z)=z+c_g z^2+\ldots$, that $\cal G$ is compact for the topology of local uniform convergence and that $c_g$ is never $0$, i.e.\ that $g$ has one attracting petal. Denote $\gamma_g$ the iterative residue of $g$. Let $\log_p$ be the principal branch of the complex logarithm. Then there exists $r_0$ such that $\forall g\in\cal G$
\begin{itemize}
\item the disk $D_\at$ of diameter $[0,r_0e^{i\alpha}]$ where $\alpha$ is the direction of the attracting axis of $g$, is contained in the parabolic basin of $g$; $g(D_\at) \subset D_\at$ and every orbit in the parabolic basin eventually enters $D_\at$;
\item the extended attracting Fatou coordinate of $g$ is injective on $D_\at$ and maps $D_\at$ to a set of the form $\setof{z\in\C}{\Re(z)>\zeta(\im(z))}$ with $\zeta:\R\to\R$ an analytic function (that depends on $g$) satisfying $\zeta(x)/x\tend 0$ when $x\tend\pm\infty$;
\item on $D_\at$, the normalized attracting Fatou coordinates $\Phi$ of $g$ and the map $\wt\Phi:z\mapsto \frac{-1}{c_g z} - \gamma_g \log_p \frac{-1}{c_gz}$ have a difference uniformly bounded by a quantity that is independent of $g$.
\end{itemize}
The above points also hold if $r_0$ is replaced by any smaller positive real.
\end{proposition}
\begin{proof}
The techniques in this proof are standard (see \cite{L}, \cite{DH}, \cite{S}, \cite{Che2}). We will insist here on providing uniformity of the bounds as $g$ varies in $\cal G$.

By compactness, uniformly on $\cal G$:
\begin{itemize}
\item $c_g$ is bounded away from $0$: $\exists \epsilon>0$ such that $\forall g\in \cal G$, $|c_g| \geq \epsilon$;
\item $g$
 is bounded on $B(0,1/2)$: $\exists K>0$ such that $\forall g\in\cal G$, $|g|\leq K$ on $B(0,1/2)$.
\end{itemize}
Also, by Cauchy's inequality,
\[|c_g|\leq 4K.\]
Since $|g(z)-z|\leq K+1/2$ on $B(0,1/2)$, we get $|g(z)-z|\leq K'|z^2|$ with $K'=4K+2$, and in particular $g$ does not vanish on $B(0,1/K')$ except at the origin.

We will make a series of change of variables $z\mapsto u\mapsto w\mapsto \xi$ with
\[u=\frac{-1}{c_g z},\quad w=u-\gamma_g \log_p u,\quad Z =\Phi(z)\]
Where $\log_p$ denotes the principal branch of the logarithm.
We will denote $z'=f(z)$ and use the notation $u\mapsto u'$, \ldots, $Z\mapsto \Z'$ for the dynamical systems $z\mapsto z'$ will be conjugated to.

The first change of variable is injective on $\C^*$. It maps $D_\at$ to the half plane
\[H_\at: \Re(u)>U_0(g)=1/r_0|c_g|.\]
We have the following asymptotic expansion
\[u'\underset{\infty}=u+1+\frac{\gamma_g}{u} + \cal O(u^{-2}).\]
The condition $z\in B(0,1/K')$ is equivalent to $|u|>K'/|c_g|$. Under this condition the map $u\mapsto u'$ is holomorphic, and depends continuously on $g$. From compactness of $\cal G$, it follows that these restrictions form a compact family too. In particular, if we further restrict to $|u|>1+K'/|c_g|$, we get by a simple application of the maximum principle that
\bEA
 | u'-(u+1) | &\leq& M_1/u
\\
| u'-(u+1+\frac{\gamma_g}{u}) | &\leq& M_2/u^2
\eEA
for some constants $M_1,M_2$ independent of $g\in\cal G$.
Thus for $r_0\leq 1/(|c_g|\max(1+K'/|c_g|,4/M_1))$, we have 
\[\frac{M_1}{|u|}\leq \frac{1}{4}\]
thus
\[| u'-(u+1) | \leq 1/4\] thus the set $H_\at$ is invariant under the dynamics of $u\mapsto u'$, so $D_\at$ is invariant under $z\mapsto z'$. It is also easy to see that in the $u$-coordinate, an orbit tending to $\infty$ must eventually get into $H_\at$.
The right hand side of the condition $r_0\leq 1/(|c_g|\max(1+K'/|c_g|,4/M_1))$ depends continuously on $g$ and reaches thus a positive minimum: it is satisfied as soon as $r_0\leq r_1$ where $r_1$ is independent of $g$.

The constant $\gamma_g$ is finite and depends continuously\footnote{because it is equal to $1-a_3/c_g^2$ if we denote $g(z)\underset{0}=z+c_g z^2+a_3z^3 +\ldots$} on $g$. Thus it is bounded over $\cal G$, say by $\Gamma$: 
\[|\gamma_g|\leq\Gamma.\]
The change of variable $w=u-\gamma_g\log_p u$ has derivative $1-\gamma_g/u$. It is thus injective on the convex set $\Re(u)>2|\gamma_g|$.
Thus when $r_0\leq r_2$ where $r_2 = \min_{g\in\cal G} (1/2|\gamma_g c_g|) >0$, then $\forall g\in\cal G$, the map $u\mapsto w$ is injective on $H_\at$. We will require in fact a bit more: $r_0\leq r'_0=r_2/2$, so that $\left|\frac{\partial w}{\partial u} -1 \right|\leq \frac14$.
This implies that the image of $H_\at$ by $u\mapsto w$ is a set that is of the form $\Re(w)>\zeta(\Im(w))$ for some analytic function $\zeta:\R\to\R$ that depends on $g$ and satisfies $|\zeta'(y)|<1/\sqrt{15}$. Moreover, $\zeta(y)/y\tend 0$ when $y\tend\pm\infty$ because $w\sim u$ when $|u|\to\infty$. In this new coordinates, we get
\[w'-w=\int_{[u,u']} \left(1-\frac{\gamma_g}{a}\right)da\]
whence
\[w'-w = 1+\frac{\gamma_g}{u} + \frac{\leq M_2}{u^2} -\gamma_g\log\left(1+\frac{1}{u} + \frac{\leq M_1}{u^2}\right)\]
where $\leq M_2$ means a complex number that depends on $u$ but whose module is at most $M_2$; we require $r_0\leq r_3$ where $r_3$ is chosen independent of $g$ and so that the quantity $\frac{1}{u} + \frac{\leq M_1}{u^2}$ has necessarily modulus $<1/2$: recall that $1/u = -c_g z$ and that $|c_g|\leq 4K$.
We can then apply the following estimate: $|a|<1/2 \implies |\log_p(1+a)-a|\leq L_0 |a|^2$ for some $L_0>0$.
Hence (thanks to a cancellation of the term $\gamma_g/u$)
\[w'-w = 1+\frac{\leq M_2}{u^2} + \gamma_g \frac{\leq M_1}{u^2} + \gamma_g\frac{\leq (1+1/4)^2L_0}{u^2}\]
(recall that $M_1/|u|<1/4$). 
Thus for some constant $M_3$ independent of $g$:
\[|w' - (w + 1)| \leq \frac{M_3}{u^2}.\]

The Fatou coordinates can be defined by
\[\Phi(z) = \mu+\lim (w_n-n) \]
where $\mu$ is a constant (that depends on the normalization) and 
$w_n$ is the $n$-th iterate of $w$ under the dynamics.
Since $\Re(u_n)>\Re(u_0)+\frac{3}{4}n$ and $\Re(u_0)\geq \frac{1}{r_0|c_g|}$,
using $r_0\leq r_4=\min(r_1,r'_2, r_3)$ we thus get
\[\lim |w_n-(w_0+n)| \leq \sum \frac{M_3}{|u_n|^2} \leq \sum \frac{M_3}{\left(\frac{1}{4K r_4}+\frac{3/4}{n}\right)^2} = M_4.\]
Thus $|\Phi(z)-(\mu+w)|\leq M_4$ holds on $D_\at$ for all $g$. The normalizing constant $\mu$ is so that $\Phi(z)=w+o(1)$ as $z\to0$ (iff.\ $w\to \infty$) and therefore $|\mu|\leq M_4$ whence: $\forall g\in \cal G$, $\forall z\in D_\at$, 
\[|\Phi(z)-w|\leq 2M_4.\] 

Recall that $H_\at$ is the image of $D_\at$ in the $u$-coordinate and that it is equal to the half plane $\Re(u) > U_0(g)=1/r_0|c_g|$. Let $U_4(g)=1/r_4|c_g|$ and $H_4$ be defined by $\Re(u) > U_4(g)$. Let $\Theta : H_4 \to \C$, $u\mapsto \Phi(z)$.
Then $|\Theta(u)-(u-\gamma_g\log_p(u))| \leq 2M_4$ and by Cauchy's inequality, $|\Theta'(u)-(1-\gamma_g/u)| \leq 2M_4/(\Re(u)-u_4)$.
In particular, the image of $H_\at$ by $u\mapsto Z=\Phi(z)$ is of the form $\Re(Z)>\zeta(\Im(Z))$ for some function $\zeta:\R\to\R$ provided $r_0\leq r_5=r_4/(1+8M_4)$ so that $2M_4/(\Re(u)-u_4) \leq 1/4$ and provided $r_0 \leq r_6 = 1/16K\Gamma$ so that $|\gamma_g/u| \leq 1/4$.
The fact that $\zeta(y)/y \tend 0$ as $y\tend\pm\infty$ follows again from $|\Theta(u)-(u-\gamma_g\log_p(u))| \leq 2M_4$.

We can now fix the value of $r_0$ to $\on{min}(r_5,r_6)$ (or any smaller value) and this gives us a set $D_\at$ that satisfies all points stated in the proposition.
\end{proof}

For reference, let us extract the following point (far from being optimal) from the proof:
\begin{lemma}\label{lem:upt1}
Under the assumptions of \Cref{prop:cf}, the change of variable $u=-1/c_g z$ conjugates $z\mapsto z'=g(z)$ to $u\mapsto u'$ satisfying:
\[\forall z\in B(0,r_0),\ \forall g\in\cal G,\ |u'-(u+1)|\leq \frac{1}{4}
.\]
\end{lemma}

Similar arguments provide:

\begin{proposition}\label{prop:cfr}
 Under the same assumptions as in the \Cref{prop:cf}, let
\[D_\rep=-D_\at.\]
Then for $r_0$ small enough the following holds: $\forall g\in \cal G$,
\begin{itemize}
\item there is a branch $\ell$ of $g^{-1}$ defined on a neighborhood of $0$ containing $D_\rep$ such that $\ell(D_\rep)\subset D_\rep$, $D_\rep$ is contained in the parabolic basin of $\ell$, every orbit in the parabolic basin of $\ell$ eventually enters $D_\rep$;
\item a normalized repelling Fatou coordinate $\Phi_\rep$ for $g$ is defined on $D_\rep$; it is injective on this set and maps it to a domain of the form $\Re(z)<\zeta(\Im(z))$ for some analytic function $\zeta$;
\item $\Phi_\rep-\wt{\Phi}_\rep$ is uniformly bounded on $D_\rep$ by a constant $M_\rep$ independent of $g$, where $\wt{\Phi}_\rep=\frac{-1}{c_g z} - \gamma_g \log_p\frac{1}{c_g z}$ (notice the change of sign inside the log compared to attracting Fatou coordinates);
\end{itemize}
\end{proposition}

We will also need a control on the inverse Fatou coordinates, that we easily deduce from the control on the Fatou coordinates:

\begin{proposition}\label{prop:cf2}
Using the notations of \Cref{prop:cf}, provided $r_0$ was chosen small enough, then for all $g\in \cal G$:
\begin{itemize}
\item Let $\Psi=\Phi^{-1}$. Then the difference between $-1/c_g \Psi(Z)$ and $Z +\gamma_g \log Z$ is bounded by a quantity independent of $g$ and of $Z\in\Phi(D_\at)$.
\item The domain of definition of $\Phi^{-1}$, i.e.\ $\Phi(D_\at)$, contains the set ``$\Re Z>\xi(\Im Z)$'' where $\xi$ is a function \emph{independent} of $g$ and satisfying $\xi(y) = \cal O( \log |y|)$ as $y\tend \pm\infty$.
\end{itemize}
\end{proposition}
\begin{proof}
We will use the notations of the proof of \Cref{prop:cf}. There was a change of variables $u=s(z) =-1/c_g z$ and a bound
\[|Z-(u-\gamma_g\log_p u)|\leq M\]
for some constant $M$ independent of $g$, where $Z$.
\[ |Z-(u-\gamma_g\log_p u)| \leq M.\]
There exists $C>0$ such that for $|z|>C$ then $\Gamma |\log_p z| + M<|z|/4$ (recall $\Gamma=\ds\sup_{g\in\cal G} |\gamma_g|$), whence if $r_0<1/C\sup |c_g|$  then $H_\at$ is contained in $|u|>C$ and thus: $|\Theta(u)-u| < |u|/4$ i.e.\ $|\Theta(u)/u -1| <1/4$, i.e.\ 
\[\forall Z \in \Phi(D_\at),\ |Z/u -1| <1/4.\]
 Now
\bEA
|u-(Z+\gamma_g\log_p Z)| & \leq & |Z-(u-\gamma_g\log_p u)| + |\gamma_g||\log_p u-\log _pZ|
\\
& \leq & M + \sup |\gamma_g| \left|\log_p\frac{u}{Z}\right|
\\
& \leq & M + \sup |\gamma_g| \log\frac{3}{2}.
\eEA
The proof of the second point is similar. 
Recall that $H_\at$ depends on $g$, and is defined by $\Re z>a_g$ where
\[a_g=1/r_0|c_g|.\]
The image $\Phi(D_\at$ is of the form $\setof{z\in\C}{\Re z>\zeta(|\Im z|)}$ where $\zeta:\R\to\R$ is an analytic function that depends on $g$. 
Let $\Theta(u)=Z=\Phi(s^{-1}(u))$. Then $\Phi(D_\at) = \Theta(H_\at)$.
The map $\Theta$ extends to a neighborhood of the closure of $H_\at$ and still satisfies $|\Theta(u)-(u-\gamma_g \log_p u)| \leq M$ on this closure. The curve $\zeta(\R)$ is the image of $\partial H_\at$ under this extension of $\Theta$. Let $b\in\R$ parameterize a point $u=a_g+ib$ varying on $\partial H_\at$ and and denote $x+iy = \Theta(a_g+ib)$.
Then $\log_p u = \log|u|+i\arg_p(u)$
and $\arg_p(u)<\pi/2$ thus the bound $|\Theta(u)-(u-\gamma_g\log_p u)|<M$ yields for the real and imaginary parts:
\bEA
 |x - (a_g -\Re(\gamma_g) \log|a_g+ib|)| &\leq& M' := M+\Gamma \pi/2,
\\
|y - (b -\Im(\gamma_g) \log|a_g+ib|)| &\leq& M'.
\eEA
There exists $C'>0$, independent of $g$, such that for all $b\in\R$, $|\Im(\gamma) \log_p|a_g+ib| | \leq |b|/2 + C'$. The second line thus yields $|b| \leq |b|/2 + C'+ |y|+M'$ i.e.\ $|b|\leq 2|y|+M''$ for some $M''$. Whence $x \leq \xi(y):=\sup(a_g) +M'+\Gamma \log |\sup(a_g)+i(2|y| +M'')|$, which is independent of $g$ and has the right order of growth w.r.t.\ $y$.
\end{proof}

\begin{proposition}\label{prop:butterfly}
Under the same assumptions, there exists $h>0$ such that for all $g\in\cal G$, the normalized extended repelling inverse Fatou coordinate $\Psi_\rep$ and the normalized extended horn map $h[g]$ are defined on a set containing the half planes $\Im(z)>h$ and $\Im(z)<-h$, and injective on the union of those half planes. Moreover, for all $r>0$, there exists $h>0$ such that for all $g\in\cal G$, $\Psi_\rep$ maps these half planes inside the disk $B(0,r)$.
\end{proposition}
\begin{proof}
Let us continue with the notations of the proof of \Cref{prop:cf}.
Note that, decreasing the value of $r_0$, we can assume that the maps $g\in \cal G$ are all injective on $B(0,r_0)$.
Without loss of generality we assume $r<r_0$.
Let us again work in the coordinates
\[u=s(z)=-1/c_g z.\]
Let $D'(r)$ be the disk of diameter $[0,re^{i\alpha}]$ where $\alpha$ is the repelling direction of $f$. (In particular $D_\rep=D'(r_0)$.)
The set $D'(r)$ is transformed by $s$ into the half plane $H':\Re(z)<-1/r |c_g|$. Let us also denote
\[H'_0:\Re(z)<-1/r_0 |c_g|.\]
Recall that if $|u|>1/r_0|c_g|$ then $|u'-(u+1)|<1/4$.
To shorten formulas, we will work with $\Phi_\rep^u(u) = \Phi_\rep \circ s^{-1}(u)$, $\wt\Phi_\rep^u(u) = \wt\Phi_\rep \circ s^{-1}(u)	=u-\gamma_g\log_p(-u)$ and $\Psi_\rep^u(Z) = s\circ \Psi_\rep(Z)$.
Consider the line of slope $-1/\sqrt{15}$ that is tangent to the disk $B(0,1/r|c_g|)$. Consider open half plane $H''$ above this line: it does not contain this disk.
In particular $|u'-(u+1)|<1/4$ holds on $U$ and thus $H''$ is stable: $u\in H'' \implies u' \in H''$.
Consider now the vertical bi-infinite strip $S$ of width $5/4$ whose rightmost bounding line is the boundary of $H'_0$.
Its image in repelling Fatou coordinates contains a fundamental domain for the translation $z\mapsto z+1$.
The intersection of $S$ with $H''$ contains all points $u\in S$ with $\Im(u)>h_1$ for some $h_1$ that depends on $r$ and $r_0$ and the lower bound $\epsilon$ on $|c_g|$ mentioned at the beginning of the proof of \Cref{prop:cf}.
Using $|\Phi^u_\rep-\wt{\Phi}^u_\rep|<M_\rep$ and the upper bound $|\gamma_g|\leq\Gamma$, we deduce that $\Phi^u_\rep (S\cap H'')$ contains every point of $\Phi^u_\rep(S)$ with imaginary part $\geq h$, where $h$ depends on $r$ and on the other constants but not on $g$.
\parpic{\begin{tikzpicture}
\node at (0,0.2) {\includegraphics{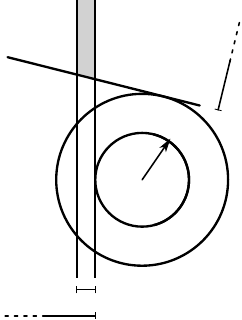}};
\node at (-0.57,-2.25) {\scriptsize $5/4$};
\node at (-1.4,-2.2) {$H'_0$};
\node at (0.35,-0.45) {$\frac{1}{r_0|c_g|}$};
\node at (1.2,-1.8) {$\frac{1}{r|c_g|}$};
\node at (1.5,1.7) {$H''$};
\end{tikzpicture}}
Recall that $\Phi^u_\rep$ maps the vertical line bounding $H'_0$ to a $y$-graph, i.e.\ a curve which crosses each horizontal line exactly once. The translate by $-1$ of this curve is the image by $\Phi^u_\rep$ of a curve $C$, preimage in $H'_0$ of $\partial H'_0$ by $u\mapsto u'$. Because of the inequality $|u'-(u+1)|<1/4$, we get $C\subset S$. Thus $\Phi_\rep^u(S)$ contains a domain bounded by a $y$-graph and and its translate by $-1$, i.e.\ a fundamental domain for the translation by $-1$.

Let us prove that the domain of the extended normalized inverse repelling Fatou coordinate $\Psi_\rep$ contains all points at height $>h$. Recall $\Psi_\rep$ is defined by extending $\Phi_\rep^{-1}$, which is defined only on $\Phi^u_\rep(H')$, by setting $\Psi_\rep(Z) = g^n(\Phi_\rep^{-1}(Z-n))$ for all $n\geq 0$ and all $Z\in\C$ such that the right hand side is defined.
Consider now $Z\in\C$. By the fundamental domain property proved above, there exists $m\in\Z$ such that $Z-m\in\Phi^u_\rep(S)$. If $m\leq 0$ then $Z\in\Phi^u_\rep(H')=\dom(\Phi_\rep^{-1})$ hence $Z\in\dom\Psi_\rep$.
If $m\geq0$ and $\im(Z)>h$ then $\im(Z-m)=\im Z>h$ and thus we have seen that $u_{-m}:=(\Phi^u_\rep)^{-1}(Z-m)$ belongs to $H''\cap S$. Since $H''$ is stable, the whole forward orbit of $u_{-m}$ belongs to $H''$. In particular $g^m(\Phi_\rep^{-1}(Z-m))$ is defined, hence $Z\in\dom\Psi_\rep$.
We have proven that the half plane ``$\Im(Z)>h$'' is contained in $\dom \Psi_\rep$.

Let now $Z\in\C$ with $\im(z)>h$ and let us prove that $\Psi_\rep(Z)\in B(0,r)$ and to the parabolic basin. Again consider $m\in\Z$ such that $Z-m\in\Phi^u_\rep(S)$.
Then in the case $m\geq0$ we just saw that the whole orbit of $u_{-m}$ is in $H''$, in particular the $m$-th iterate, which is equal to $\Psi^u_\rep(Z)$. Thus the point $\Psi_\rep(Z) = s^{-1}(\Psi^u_\rep(Z))$ belongs to $B(0,r)$.
Also, the orbit of $u$ tends to $\infty$ hence $\Psi_\rep(Z)$ belongs to the basin of the parabolic point of $g$.
In the case $m\leq 0$, then $Z\in\Phi_\rep(D'(r))$ and thus $\Psi_\rep(Z)=\Phi_\rep^{-1}(Z)\in D'(r) \subset B(0,r)$.
Since moreover $Z-m$ satisfies the first case and thus the point $\Psi_\rep(Z-m)$ belongs to the parabolic basin, we get that $\Psi_\rep(Z)$ also belongs to the basin, as it is mapped to the former point by the $|m|$-th iterate of $g$.

The proofs for the lower half plane ``$\Im z<-h$'' are similar.
Let us prove injectivity of $\Psi_\rep$ on the union $V$ of ``$\im z<-h$'' and ``$\im z>h$''. First, it is injective on $U = \Phi_\rep(D'(r))$, because it is equal to $\Phi_\rep^{-1}$ there. The map $g$ is injective on $\Psi_\rep(V)$ because the latter is contained in $B(0,r_0)$. The set $\Psi_\rep(V)$ is also stable by $g$, thus $g^n$ is also injective on it. Then, for each $n$, the map $g^n\circ \Phi_\rep^{-1} \circ T_{-n}$ is a composition of injective maps on $T_n(U) \cap V$, and coincides there with $\Psi_\rep$. Since the union over $n$ of $T^n(U)$ is the whole complex plane, the claim follows.

Injectivity of $h[g]$ on $V$ is similar, since $h[g]$ is the union over $n\geq 0$ of the maps $T^{-n}\circ \Phi_{\at}\big|_{D_\at} \circ g^n \circ \Psi_\rep$, which are injective when restricted to $V$.
\end{proof}

\medskip

Let us introduce a weak notion of convergence of analytic maps: let $X$, $Y$ be connected Riemann surfaces and let $f_n : U_n  \to Y$ and $f: U \to Y$ be analytic with $U$ and $U_n$ open subsets of $X$.
Endow $Y$ with any metric compatible with its topology.
Let us say that $f_n$ tends to $f$ if for all compact subset $K$ of $U$, $K$ is eventually contained in $U_n$ and $f_n$ tends to $f$ uniformly on $K$.
This does not depend on the choice of the metric.\footnote{This definition has the following equivalent topological formulation. Let $X'=\{0,1,1/2,1/3,1/4,\ldots\}\times X \subset \R\times X$ and embed $X'$ with the topology induced by $\R\times X$. Let $W\subset X'$ be defined by $(0,z)\in W \iff z\in U$ and $(1/n,z)\in W \iff z\in U_n$. Let $F : W \to Y$ defined by $F(0,z) = f(z)$ and $F(1/n,z) = f_n(z)$. Then $f_n \st f$ $\iff$ [$W$ is open relative to $X$ and $F$ is continuous].} Note that this does not prevent $U_n$ to have a bigger limit than $U$. In particular, limits are not unique.
We will use the following notation:
\[ f_n\st f
,\]
which is chosen so to express the fact that $f$ can be contained in limits with a bigger domain. 
We do not define an associated topology but we will use the notion of sequential continuity with respect to that notion of convergence, as illustrated by the following two properties, whose proofs are left to the reader:
\begin{enumerate}
\item The composition $f\circ g$ depends continuously on the pair $f,g$: if $f_k \st f$ and $g_k \st g$ then $f_k\circ g_k \st f\circ g$.
\item For a fixed $n$, $f^n$ depends continuously on $f$: if $f_k \st f$ then $f_k^n \st f^n$.
\end{enumerate}

For the next statement, recall that $\Phi_\at$ and $\psi_\rep$ denote the \emph{extended} Fatou functions. 

\begin{proposition}[continuous dependence]\label{prop:conti}
Assume $g_n : U_n \to\C$ is a sequence of holomorphic maps defined on an open subset $U_n$ of $\C$ containing the origin, with expansion $g_n(z)=z+c_n z^2 + \ldots$ at $0$, and with $c_{n}\neq 0$. Assume $g$ is also of this form with $c_g\neq 0$ and that $g_n \st g$. Then $\Phi_\at[g_n] \st \Phi_\at[g]$, $\Psi_\rep[g_n]\st\Psi_\rep[g]$ and $h[g_n] \st h[g]$.
\end{proposition}
\begin{proof}The claim on $h=\Phi_\at\circ\Psi_\rep$ follows from the claims on $\Phi_\at$ and $\Psi_\rep$.

Recall that $D_\at[g]$ is the disk of diameter $[0,r_0e^{i\alpha[g]}]$ where $\alpha[g]$ is the direction of the attracting axis of $g$, and that $r_0$ is independent of $g$. Hence $D_\at[g]$ depends continuously on $g$. A compact set $K$ contained in the parabolic basin of $g$ is mapped in $D_\at[g]$ by an iterate $g^k$. The latter depends continuously on $g$ when $k$ is fixed. Since the center and radius of $D_\at[g]$ depend continuously on $g$, $g_n^k(K)\subset D_\at[g_n]$ for all $n$ big enough. Continuity, as a function of $g$, of the restriction of $\Phi_\at$ to $D_\at$, follows for instance from the third point of \Cref{prop:cf} combined with uniqueness of Fatou coordinates: the sequence $\Phi_\at[g_n]$ forms a normal family, and any extracted limit is a Fatou coordinate for $g$ because the functional equation $\Phi_\at[g_n] \circ g_n = T_1 \circ \Phi_\at[g_n]$ passes to the limit, and uniqueness of the normalized Fatou coordinates implies uniqueness of the extracted limit. From the convergence of $\Phi_\at[g_n]$ to $\Phi_\at[g]$ on $D_\at[g]$ we deduce the convergence of $\Phi_\at = \Phi_\at[g_n] \circ g_n^k - k$ to $\Phi_\at[g] \circ g^k - k = \Phi_\at[g]$ on $g^{-k}(D_\at[g])$, and hence on the whole parabolic basin of $g$.

The proof for $\Psi_\rep$ is similar.
\end{proof}

\subsubsection{Transferring to $\cal F$}\label{subsub:transff}

Fix some $d\in\{2,3,\ldots,\infty\}$ and recall the definition $\cal F = \setof{\cal R[B_d] \circ \phi^{-1}}{\phi\in\sch}$. 
The conclusions of the previous propositions hold for $\cal F$.
Indeed, the set of restrictions to $\D$ of maps $A \circ f \circ A^{-1}$ with $A(z)=4z$ satisfies the assumptions of the propositions.
First, the set of Schlicht maps $\sch$ is compact, and by Koebe's one quarter theorem, the domain of their reciprocal contains $B(0,1/4)$.
The restriction of these reciprocals on $B(0,1/4)$ forms a compact family.
We saw in \Cref{prop:fsd} that $\cal F\subset\cal S_d$.
In particular, maps in $\cal F$ have only one attracting petal.
This is therefore also the case for the conjugate map $A\circ f\circ A^{-1}$.
This proves the claim.

Call $\wt f$ the restriction of $A\circ f\circ A^{-1}$ to $\D$.
The conclusions of the previous propositions are easily transposed from $\wt f$ back to $f$ because they were all local (except for \Cref{prop:conti}, which directly applies): for instance, normalized Fatou coordinates satisfy $\Phi_\at[\wt f](z)=\Phi_\at (A(z))$ for all $z$ in the domain of the left hand side (it is contained in the domain of the right hand side but not necessarily equal to it, because $\wt f$ is a restriction).
\Cref{prop:conti} did not assume that the maps are defined on the unit disk, it applies directly, so continuous dependence of $\Phi_\at[f]$ and $\Psi_\rep[f]$ holds without restricting the domain. 

Recall $h[f]$ has the following expansion:
\[h[f](z) = z + a_{\on{up}/\on{down}} + o(1)\]
as $\Im(z)\tend\pm\infty$, where $a_{\on{up}}$ and $a_{\on{down}}$ are two complex constants.
For any map in the class $\cal F$, denote $\{0,v_f,\infty\}$ its  singular values.\footnote{It turns out that $v_f$ is independent of $f$ for a fixed $d$, but we will not use that fact.}
The corresponding map $h[f]$ has a set of singular values of the form $v_h+\Z$ where
\[ v_h=v_h[f]=\Phi_\at[f](v_f)
. \]
By \Cref{prop:butterfly} there is a uniform $h>0$ such that for all $f\in\cal F$, the domain of $h[f]$ contains the half planes $\Im z>h$ and $\Im z<-h$.

For any $d\in\{2,3,\ldots,\infty\}$, the set $\cal F$ is sequentially compact, for the notion of convergence defined above.
By this we mean that every sequence $f_n\in\cal F$ has a subsequence $f_k$ such that $f_k \st f$.
(\footnote{Note that if we restrict our notion of convergence to $\cal F$, we recover uniqueness of the limit.})
A sequentially continuous real valued function over a sequentially compact set is bounded. 
This implies the following proposition.

\begin{proposition}\label{prop:cor}
For any $d\in\{2,3,\ldots,\infty\}$ over the class $\cal F$, the following holds:
\begin{enumerate}
\item (bound in the normalized attracting Fatou coordinates)\\ $\exists M$ such that $\forall f\in\cal F$, $|\Im(v_h)|\leq M$.
\item (bound on the horn map at the ends of the cylinder)\\ $\exists M$ such that $\forall f\in\cal F$, $|a_{\on{up}}[f]|\leq M$ and  $|a_{\on{down}}[f]|\leq M$.
\item (bound in the normalized repelling Fatou coordinates)\\ $\exists M$ such that $\forall f\in\cal F$, the main\footnote{terminology introduced in \Cref{subsec:sdc}} upper and lower chessboard boxes of $h[f]$ respectively contain the half planes $\Im(z)>M$ and $\Im(z)<-M$. 
\end{enumerate}
\end{proposition}
\begin{proof}
The map $f\in\cal F\mapsto v_h\in \C$ is sequentially continuous by \Cref{prop:conti}. 
The set $\cal F$ being sequentially compact, its image by $f\mapsto v_h$ is sequentially compact in $\C$ (i.e.\ compact) thus bounded. The first point follows.

For the second, by periodicity and the maximum principle and according to the expansion, the distance $|h[f](z)-z|$ is bounded over $\Im(z)>h+1$ by its supremum over a segment of length $1$ inside the line $\Im(z)=h+1$, for instance the segment $[i(h+1),1+i(h+1)]$. Continuous dependence implies the distance is uniformly bounded as $f$ varies in $\cal F$. Since $a_{\on{up}}$ is the limit of this difference as $\Im(z)\tend +\infty$, this implies the bound on $a_{\on{up}}$. (Alternatively one can use the fact that $a_{up} = \int_{ih+i}^{ih+i+1} (h(z)-z)dz$.) The proof is similar for $a_{\on{down}}$.

For the third, we will use the following trick: first $h[f]$ is an analytic isomorphism commuting with $T_1(z)=z+1$ from the upper and the lower structural boxes to one of the half plane delimited by $v_h+\R$. By Koebe's one quarter theorem, the upper box must contain $\Im(z)>\frac{\log 4}{2\pi} + \Im(v_h) -\Im(a_{\on{up}})$. The previous bounds allows to conclude. The proof is similar for the other half plane.
\end{proof}

Let us now prove an independent proposition. Let $f$ be a map in $\cal S_d$.
Then we can apply \Cref{lem:phiatu} about universality and we know that $\Phi_\at : A \to \C$ is structurally equivalent to $\Phi_\at[B_d]$ for some $d\in\{2,3,\ldots,\infty\}$. The singular values of $\Phi_\at$ are $\infty$ and the points of the form $\Phi_\at(v)-n$ with $n>0$ (see for instance Proposition~2 in \cite{BE}, where a notion of \keyw{ramified cover} is used: their proposition implies that $\Phi_\at$ is a cover outside $\infty$ and the critical values).
\begin{proposition}\label{prop:ppalcurv}
Under these conditions, the preimage $\Gamma$ by $\Phi_\at$ of the horizontal half line $\Phi_\at(v)+[0,+\infty[$ has a connected component $\cal C$ that is a curve starting from the singular value of $f$ in $A$ and ending at the parabolic point. It is stable: $f(\cal C)\subset \cal C$,\nomenclature[C1]{$\cal C$}{a curve through the orbit of the critical value,  \Cref{prop:ppalcurv}\nomrefpage} and contained in the common boundary of the two principal dynamical chessboard boxes of $f$.
\end{proposition}
This curve will be called the \keyw{principal curve}. It contains in particular the orbit of the singular value of $f$. Note that all connected components of $\Gamma$ are curves since the horizontal half line considered contains no singular value of $\Phi_\at$.

\begin{proof}
It is enough to prove the proposition for $B_d$, which is easy because the latter map is real preserving and its singular value is on the real line. Then it transfers to $f$ by universality: all claims are immediate except the statement that $\cal C$ tends to the parabolic point. The latter follows for instance from $\cal C$ being the concatenation of the sucessive images by $f$ of its part from $v_f$ to $f(v_f)$.
\end{proof}

Let us go back to maps $f\in \cal F$. As we remarked before, convergence of maps $f_n\st f$ where $f_n$ and $f$ belong to $\cal F$ is well behaved: limits are unique and in fact it is equivalent to the classical notion of convergence of a sequence with respect to a (metrizable) topology making $\cal F$ compact: Indeed,
let $f_n,f\in\cal F$.
Write $f_n=\cal R[B_d]\circ \phi_n^{-1}$ and $f=\cal R[B_d]\circ \phi^{-1}$ with $\phi_n$ and $\phi\in\sch$ (uniquely determined).
Then the following are equivalent:
\begin{enumerate}
\item $f_n\st f$,
\item for some $\epsilon>0$, the map $f_n$ tends to $f$ uniformly on $B(0,\epsilon)$,
\item for some $\epsilon>0$, the map $\phi_n$ tends to $\phi$ uniformly on $B(0,\epsilon)$,
\item $\phi_n$ tends to $\phi$ uniformly on every compact subsets of $\D$.
\end{enumerate}
A proof of (3)$\implies$(4) is for instance given by compactness of $\sch$ together with analytic continuation of equalities.
The last three notions of convergence are easily metrized and all endow $\cal F$ with the \emph{same} topology. It is Hausdorff and compact for this topology. The map $\sch\to \cal F$, $\phi\mapsto \cal R[B_d]\circ \phi^{-1}$ is hence a homeomorphism.

Recall that we denote $0$, $v_f$, $\infty$ the singular values of $f$ over $\wh{\C}$. It turns out that the class $\cal F$ has been defined so that $v_f$ does not depend on $f$, but let us temporarily ignore that.

\begin{lemma}[uniform bound on the trapping time]\label{lem:traptime}
For any $r>0$, denote $D_r[f]$ the disk of diameter $[0,re^{i\alpha}]$ where $\alpha$ is the direction of the attracting axis of $f$. There exists $n_0\in\N$ such that $\forall f\in\cal F$, $f^{n_0}(v_f)\in D_r[f]$.
\end{lemma}
\begin{proof}
Consider $r'=\min(r,r_0)$ where $r_0$ is provided by \Cref{prop:cf}. The set $D_{r'}[f]$ is an attracting petal for $f$ and is contained in $D_r[f]$.
For each $f\in \cal F$ it takes a finite number of iterates for $v_f$ to be trapped by $D_{r'}[f]$. The same number of iterates is enough for nearby\footnote{We may use the topology on $\cal F$, in which case it means that the same iterate is enough for all maps in a neighborhood. Or we may use the notion $\st$, in which case it means that for all sequence $f_n\st f$, this iterate is eventually enough.} maps in $\cal F$. By compactness\footnote{cover argument or sequence argument} of $\cal F$, it follows that there is $n_0\in\N$ such that $\forall f\in\cal F$, $\exists n\leq n_0$ such that $f_n(v_f)\in D_{r'}[f]$. Since $D_{r'}[f]$ is a trap this implies $f^{n_0}(z)\in D_{r'}[f]$ and thus $\in D_r[f]$.
\end{proof}

We will also use a slightly stronger statement:

\begin{lemma}\label{lem:traptime:2}
There exists $n_0\in\N$ and $\eta_0>0$ such that $\forall f\in\cal F$, $f^{n_0}(B(v_f,\eta_0))\in D_r[f]$.
\end{lemma}
\begin{proof}
Done by compactness as above, using the following modification of the local statement, which is immediate by continuity for $\st$ of $f\mapsto f^n$ for a fixed $n$: for each $f\in \cal F$ and each $n$ such that $f^n(v_f)\in D_r[f]$, there is $\eta>0$ such that for all maps $g\in \cal F$ close enough to $f$ the $n$-th iterate of $g$ sends $B(v_g,\eta)$ in $D_r[g]$.
\end{proof}

We have not checked if all compactness arguments in the rest of the article can be reformulated using $\st$ only. This is not the main point, however. Moreover, since there is on $\cal F$ a topology for which convergence of sequences is equivalent to $\st$, in the sequel we will use compactness of $\cal F$ for this topology and convergence of sequences in $\cal F$ w.r.t.\ this topology. Recall it is a metrizable topology for which $\cal F$ is compact.

Below, $d_\C$ refers to the Euclidean distance on $\C$ and if $U$ is a open subset of $\C$ whose complement has at least two points, $d_U$ denotes the hyperbolic distance on $U$. Let $\cal C=\cal C[f]$ be the curve introduced in \Cref{prop:ppalcurv}.

\begin{lemma}\label{lem:PSdist}
For $f\in\cal F$, let 
$PS(f)$ the orbit of the singular value $v_f$ of $f$. The following holds:
\begin{enumerate}
\item\label{item:psd:1} The sets $\cal C[f]$ and $\ov{PS}(f)$ depend continuously on $f$ for the Hausdorff topology on compact subsets of $\C$.
\item\label{item:psd:2} $\sup\setof{|z|}{z\in PS(f),\ f\in\cal F}<+\infty$
\item\label{item:psd:3} $\sup\setof{d_{\dom(f)}(0,z)}{z\in PS(f),\ f\in\cal F}<+\infty$
\item\label{item:psd:4} $\inf\setof{d_\C(z,\partial \dom(f))}{z\in PS(f),\ f\in\cal F}>0$
\end{enumerate}
\end{lemma}
\begin{proof} 
Let us use the same notations as in \Cref{lem:traptime}.
For any $r\leq r_0$, denote $D_r=D_r[f]$: it is an attracting petal for $f$.
Let $n_0(r)=n_0$ be provided by \Cref{lem:traptime}.
For a fixed $m<n_0(r)$, $f^m(v_f)$ depends continuously on $f$. The rest of the orbit of $v_f$ is contained in $D_r$. Continuity of $\ov{PS}(f)=PS(f)\cup\{0\}$ follows, as well as the point~\ref{item:psd:2}.
For point~\ref{item:psd:4}, note that $B(0,1/4)\subset \dom(f)$ (this follows from Koebe's $1/4$ theorem). Choose now $r=\min(r_0,1/8)$.For each fixed $m< n_0=n_0(r)$, the distance from $f^m(v_f)$ to $\partial \dom(f)$ reaches a positive minimum as $f$ varies in $\cal F$, again by continuity and compactness. For $m\geq n_0$, this distance is $\geq 1/8$.
For point~\ref{item:psd:3} first note that, on one hand for $m \geq n_0$, $f^m(v)\in B(0,1/8)$ and thus $d_{\dom(f)}(0,f^m(v))\leq 1$ (better constants can be computed but that is not the point here). Let us now use the sets $O$ and $O_\at$ introduced in \Cref{prop:ppalcurv}.
The map $\Phi_\at$ is a holomorphic bijection from $O_\at$ to $O=\C\setminus\,]-\infty,v'-1]$ and the set $X=\setof{f^m(v)}{0\leq m<n_0}$ is the preimage by this map of $v' + \{0,1,\ldots,n_0-1\})$. Therefore the $\dom(f)$ hyperbolic distance from $X$ to $B(0,1/8)$ is $\leq$ the hyperbolic distance in $O$ from $v'$ to $v'+n_0$, which is itself $<n_0$.
\end{proof}

\begin{figure}
\begin{tikzpicture}
\node at (0,0) {\includegraphics[width=8cm]{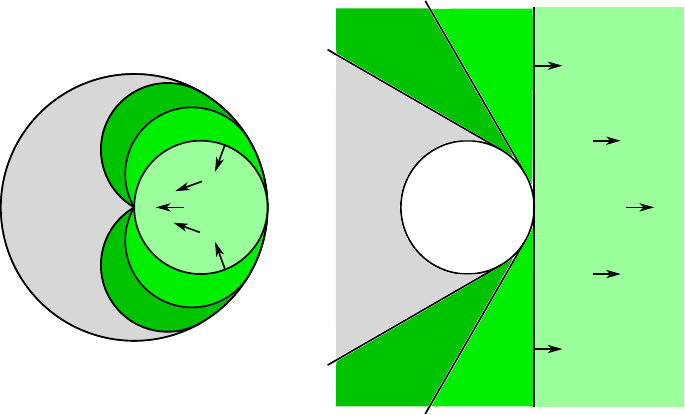}};
\node at (-2.5,-3) {$z$-coordinate, $D_\theta(r_0)$};
\node at (2.2,-3) {$u$-coordinate, $\Omega_\theta(r_0)$};
\end{tikzpicture}
\caption{Bigger domains for Fatou coordinates. On the left : $z$-coordinate and different domains $D_\theta(r_0)$ (in this example the attracting axis is the positive reals), $\theta-90\degree=0,30\degree,60\degree$, difference between these regions are highlighted in different colors. On the right, the $u$-coordinate, with $u=-1/c_f z$, and the corresponding domains $\Omega_\theta(R_0)$. In light green are $D_\at$ and $H_\at$.}
\label{fig:bigger_doms}
\end{figure}

\subsubsection{Lemmas for the second step}\label{lem:2ndstep}

Consider again a class of maps $\cal G$ as in \Cref{prop:cf}, i.e.\ $\cal G$ is a set of holomorphic maps $g:\D\to \C$ with $g(z)=z+c_g z^2+\ldots$, that is compact for the topology of local uniform convergence and such that $c_g$ is never $0$, i.e.\ $g$ has one attracting petal. In the second step of the proof of the main theorem, we will need some control on the variation of Fatou coordinates in terms of the variation of the map. For this we first need to extend Fatou coordinates to bigger petals, as in~\cite{S}.

Let $\theta\in[\pi/2,\pi]$ and $\Omega_\theta(r)$ denote the following domain: it contains a right half plane and is bounded by the arc of circle of center $0$, radius $r$ and argument ranging from $-(\theta-\pi/2)$ to $\theta-\pi/2$, and by the two half lines continuing this arc tangentially to the circle (see \Cref{fig:bigger_doms}). For $\theta=\pi/2$ and $r=R_0[g]=1/|c_g|r_0$, this domain is exactly the half plane, image of $D_\at$ in the $u$-coordinates of $g$.\nomenclature[Ztr]{$\Omega_\theta(R)$}{some domain in the coordinates $u=-1/c_f z$, extending the half plane on which we control the Fatou coordinates\nomrefpage}

\begin{lemma}\label{lem:wW}
There is $r_0$ such that for all $g\in \cal G$, the change of variable $w=u-\gamma_g \log_p u$ is injective on the non-convex set $\Omega_\pi(R_0[g])$.
\end{lemma}
\begin{proof}Let us give a computational but elementary proof of this fact.
Write $u=re^{i\alpha}$ and $u'= r'e^{i\alpha'}$ with $\alpha,\alpha'\in\,]-\pi,\pi[$ and note that $r,r'\geq R_0[g]\geq 1/r_0\inf_{g\in \cal G}|c_g|$. Then
$|r-r'|\leq |u-u'|$ and $|e^{i\alpha}-e^{i\alpha'}| \leq |u-u'|/\min(r,r')$. If $|\alpha-\alpha'|\leq \pi$ (case 1) then $|\alpha-\alpha'|\leq (\pi/2)|e^{i\alpha}-e^{i\alpha'}|$. Otherwise (case 2), let us just use that $|\alpha-\alpha'|\leq 2\pi$.
Now $w=w'$ means $u-u'= \gamma_g(\log r' -\log r) + \gamma_g i(\alpha'-\alpha)$ whence (case 1) $|u-u'| \leq \frac{|\gamma_g|(1+\pi/2)}{\min(r,r')} |u-u'|$ therefore $u-u'= 0$ provided $r_0$ was chosen big enough (independently of $g$).
Or (case 2) $|u-u'| \leq \frac{|\gamma_g|}{\min(r,r')} |u-u'| +  2\pi|\gamma_g|$. In the second case, choose $r_0$ small enough (independently of $g$) so that $\frac{|\gamma_g|}{\min(r,r')} \leq 1/2$.
Then $|u-u'|\leq 4\pi|\gamma_g|$. Since $\alpha-\alpha'>\pi$ the points $u$ and $u'$ must have opposite imaginary part and one of them at least has negative real part. Since they belong to $\Omega_\pi(R_0[g])$, which does not contain the half strip of equation ``$\Re z\leq 0$ and $-R_0[g]\leq \Im z \leq R_0[g]$'', we get in particular that $|u-u'|>R_0[g]$. So if we choose $r_0$ small enough so that, $\forall g\in\cal G$, $R_0[g]>4\pi|\gamma_g|$, this is impossible.
\end{proof}

\begin{proposition}\label{prop:bigger} Let $\theta\in[\pi/2,\pi[$.
  \Cref{prop:cf} still holds if we replace $D_\at$ with the domain $D_\theta(r_0)[g]$ whose image in the $u$-coordinate is $\Omega_{\theta}(R_0[g])$ where $R_0[g]=1/|c_g|r_0$, and if we replace the condition on $\zeta$ by $\zeta(x)=-|x\tan(\theta-\pi/2)|+ o(x)$. Similar statements hold for repelling Fatou coordinates.
\end{proposition}
\begin{proof}
The proof carries over with little modification. The constant $1/4$ has to be replaced by a smaller constant (by $\sin \theta$) when $\theta$ is too close to $\pi$. Injectivity of the change of variable $w=u-\gamma_g \log_p u$ on the non-convex set $\Omega_\theta(R_0)$ follows from the previous lemma.
For the uniform bound on $\sum M_3/|u_n|^2$: divide the orbit of $u_n$ into three parts, according to $\Re(u_n)$ being in $]-\infty,-R_0[$, in $[-R_0,R_0]$, or in $]R_0,+\infty[$. In the central part, there is a uniformly bounded number of $u_n$. The two other parts are bounded exactly like before.
\end{proof}

In particular $\Phi_\at[g]$ is defined on a set containing the image of $\Omega_\theta(R_0[g])$ by $u\mapsto -1/c_g z$. Recall that $R_0[g]$ is bounded, hence the sets $\Omega_\theta(R_0[g])$ all contain $\Omega_\theta(R_0)$ for some $R_0$ independent from $g$.
 Choose any $\theta$ with $\pi/2<\theta<\pi$. Let
\[\Xi[g](w) = \Phi(z),\]
where $u=-1/c_g z$, $w=u-\gamma_g\log_p(u)$, and $\Phi$ is the attracting Fatou coordinate: we take $\Xi$ defined on the image of $\Omega_\theta(R_0[g])$ by $u\mapsto w$.
Note that this change of coordinates depends on $g$, but if one chooses any other $\theta'<\theta$, there exists $R'_2>0$ such that for all $g\in\cal G$, it contains $\Omega_{\theta'}(R'_2)$.
By \Cref{prop:bigger}, $\Xi[g](w)-w$ is bounded by a constant independent of $g$. Hence there exists $R_2>0$ such that for all $g\in\cal G$, the domain \emph{and} the range of $\Xi[g]$ contains $\Omega_{\theta'}(R_2)$.
Recall that maps in $\cal G$ are assumed to be defined on $\D$.

\medskip

The next three lemmas express a form of Lipschitz dependence with respect to $g$ for $\Xi[g]$, $\Xi[g]^{-1}$ and $\Phi[g]$.

\begin{proposition}\label{prop:dep}
Let $R_2$ as above. Let $r'\in ]0,1[$.
There exists $M>0$, $R_1>R_2$ and $\epsilon_0>0$ such that for all $f,g\in \cal G$ with $\sup_{B(0,r')}|f-g|\leq \epsilon_0$ then $\forall w\in\C$ with $w\in \Omega_{\theta'}(R_1)$,
$|\Xi[f](w)-\Xi[g](w)| \leq M \sup_{B(0,r')}|f-g|.$ (\footnote{A better bound holds, that decays when $w$ tends to infinity, but it will not be used here.})
\end{proposition}

\begin{proof}
A trick to shorten the proof is to use holomorphic dependence of Fatou coordinates w.r.t.\ the map.
Let $\|f-g\| = \sup_{B(0,r')} |f-g|$.
Let $c_0=\inf|c_g|$ over all $g\in\cal G$.
Let first $\epsilon_0$ be such that the sum $h$ of a map in $\cal G$ with a holomorphic map defined on $B(0,r')$ and with a double root at the origin and sup norm $\leq 2\epsilon_0$, satisfies $|c_h|>c_0/2$.
Let $\cal G'$ be the union of $\cal G$ and of all the maps of the form $h_t=f+ t \frac{2\epsilon_0}{\|f-g\|} (g-f)$ where $t\in\D$, $f,g\in \cal G$ and $\|f-g\| \leq \epsilon_0$.
Then $\cal G'$ is compact (for the topology associated to uniform convergence on compact subsets of $B(0,r')$) and, conjugating its members by $z\mapsto z/r'$ and restricting to $\D$, gives a family satisfying the hypotheses of \Cref{prop:cf,prop:bigger}.
Using the latter and the same analysis as in the paragraph that follows it 
we see that maps $h\in\cal G'$ all have a function $\Xi[h]$ that is defined on a set containing $\Omega_{\theta'}(R_1)$ for some $R_1$ independent of $h$.
Moreover this function depends holomorphically on $t\in\D$ (recall the definition of $\Phi$ as a limit of $w_n-n$ and realize that $w_n$ depends holomorphically on $w_n$) and its difference with $w\mapsto w$ is uniformly bounded, hence $\Xi[h_t](w)-\Xi[f](w)$ is also bounded.
The proposition follows by Schwarz's inequality\footnote{We mean: if $f:\D\to \C$ is holomorphic and satisfies $f(0)=0$ and $\sup |f|<+\infty$ then $|f(z)|\leq |z|\sup |f|$.} applied to $t\mapsto \Xi[h_t](w)-\Xi[f](w)$, specialized to $t=\|f-g\|/2\epsilon_0$.
\end{proof}

Similarly, \Cref{prop:dep} holds word for word with $\Xi$ replaced by $\Xi^{-1}$, i.e.:

\begin{proposition}\label{prop:dep2}
There exists $M>0$, $R_1>R_2$ and $\epsilon_0>0$ such that for all $f,g\in \cal G$ with $\sup_{B(0,r')}|f-g|\leq \epsilon_0$ then $\forall Z\in\C$ with $Z\in \Omega_{\theta'}(R_1)$,
$|\Xi^{-1}[f](Z)-\Xi^{-1}[g](Z)| \leq M \sup_{B(0,r')}|f-g|.$
\end{proposition}

\begin{proof}This follows from the above proposition applied to some $\theta''$ between $\theta$ and $\theta'$, and from the fact that, the derivative of $\Xi_g$ is uniformly bounded\footnote{increase $R_2$ by $1$ and use Cauchy's formula and the uniform bound on $\Xi-\on{id}$} over $g\in\cal G$. Computations are left to the reader.
\end{proof}

\begin{remark}
Note that since the maps $\Xi_g$ and $\Xi_g^{-1}$ all differ from identity by a bounded amount that is independent of $g\in\cal G$, it follows that in both propositions, by increasing the value of $M$, we can remove the assumption $\sup_{B(0,r')}|f-g|\leq \epsilon_0$.
\end{remark}

From \Cref{prop:dep}, we deduce the following control, which is somewhat weaker:
\begin{proposition}[variation of Fatou coordinates]\label{prop:mvtphirep}
Let $r'\in\,]0,1[$. Let $R_1$ be given by \Cref{prop:dep}.
Let $\theta''<\theta'$.
Then there exists $M>0$, $R_3>R_1$ and $\epsilon_0>0$ such that for all $f,g\in \cal G$ with $\sup_{B(0,r')}|f-g|\leq \epsilon_0$ and $\forall z\in\C$ with $-1/c_f z \in \Omega_{\theta''}(R_3)$, then $-1/c_g z\in \Omega_{\theta'}(R_1)$ and
\[\big|\Phi_\at[f](z)-\Phi_\at[g](z)\big| \leq \frac{M}{|z|}\sup_{B(0,r')}|f-g|.\]
The same holds for repelling Fatou coordinates.
\end{proposition}
\begin{proof}
Let $d=\sup_{B(0,r')}|f-g|$.
The claim $-1/c_g z\in \Omega_{\theta'}(R_1)$ follows from continuity of $g\mapsto c_g$ and its non-vanishing:
given any $R_3>1$ and $\theta''<\theta'$, a small enough $d$ will ensure that the quotient $c_g/c_f$ is close enough to $1$ so that an element of $\Omega_{\theta''}(R_3)$ multiplied by $c_f/c_g$ is still in $\Omega_{\theta'}(R_1)$.
Now $\Phi_\at[f](z)=\Xi[f](w_1)$ and $\Phi_\at[g](z)=\Xi[g](w_2)$ with $w_1 = u_1 -\gamma[f] \log_p(u_1)$ and $w_2 = u_2 -\gamma[g] \log_p(u_2)$ with $u_1 = -1/c_f z$ and $u_2 = -1/c_g z$. The constants $c$, $1/c$ and $\gamma$ are Lipschitz functions of $f\in \cal G$ w.r.t.\ the distance $d$.
Now, under the assumption $d\leq \epsilon_0$, we successively get $|u_1-u_2|\leq M_1 d/|z|$, $|w_1-w_2|\leq M_2 d/|z|$ (use that $|\log_p u| \leq M_2'/|z|$ for some constant and that $u_2/u_1 = c_f/c_g$ is close enough to $1$), then we decompose $|\Xi[f](w_1)-\Xi[g](w_2)| \leq |\Xi[f](w_1)-\Xi[g](w_1)| + |\Xi[g](w_1)-\Xi[g](w_2)|$
 The first term is dealt with using \Cref{prop:dep} and the second term using the fact that there is a uniform bound on $\Xi'$.
\end{proof}

\begin{remark}
\begin{itemize}
\item Here, the condition $\sup_{B(0,r')}|f-g|\leq \epsilon_0$
cannot be removed.
\item Also, in the conclusion $|\Phi_\at[f](z)-\Phi_\at[g](z)\big|\leq \frac{M}{|z|}\sup_{B(0,r')}|f-g|$, the factor $1/|z|$ cannot be removed because $\Phi_\at[f](z)\sim -1/c_f z$ and $c_f $ varies with $f$.
\end{itemize}
\end{remark}

Let us stress again that, though maps in $\cal F$ are not defined on the unit disk, they are all defined in $B(0,1/4)$ and the results above easily transfer to $\cal F$ by a homothety. (See \Cref{subsub:transff}.)

\subsection{Step 1: contraction argument (i.e.\ there is a lot of room).}\label{subsec:part1}

Fix $d\in\N$ with $2\leq d<\infty$: we now exclude $d=+\infty$. In this section we will define constants $c_1$, $c_2$, \ldots\ 
They all depend on $d$ but not on $f\in \cal F$.

Recall that $\cal R[B_d]$ is defined on the unit disk and has derivative one at the origin. 
Recall the definition of the set $\sch$ of Schlicht maps: univalent holomorphic maps $\phi:\D\to \C$ such that $\phi(z)=z+\cal O(z^2)$.
Recall that $\cal F= \setof{\cal R[B_d]\circ\phi^{-1}}{\phi\in\sch}$, and that for all $f\in \cal F$, the map $\cal R[f]$ is again in $\cal F$, for an appropriate normalization. Since all maps in $\cal F$ have the same unique critical value, this normalization coincides with the one numbered~\ref{item:nor:3} on page~\pageref{item:nor:3}, which we called ``by the critical value''.

\begin{figure}
\begin{tikzpicture}
\node at (0,10) {\includegraphics[width=5cm]{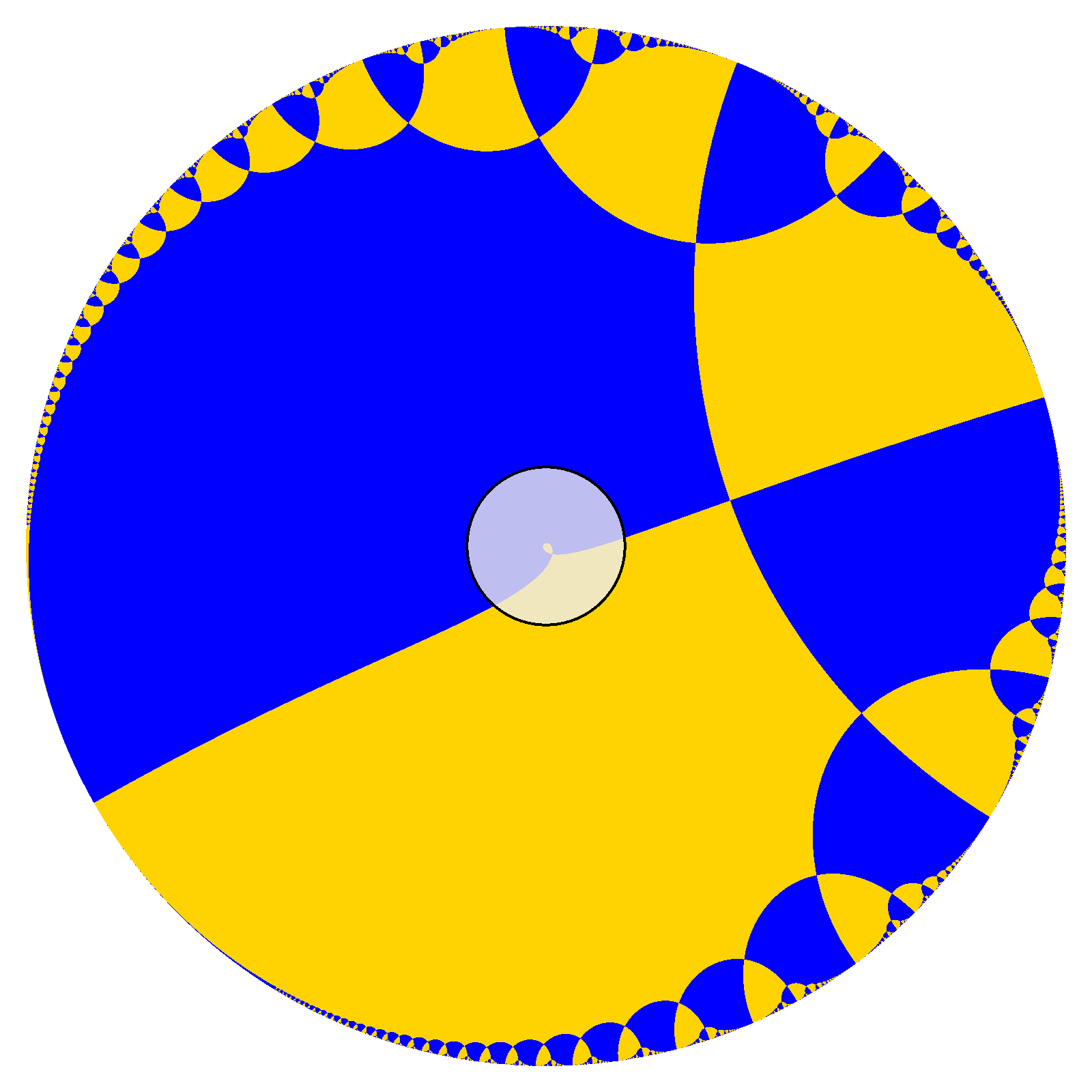}};
\node at (6,10) {\includegraphics[width=5cm]{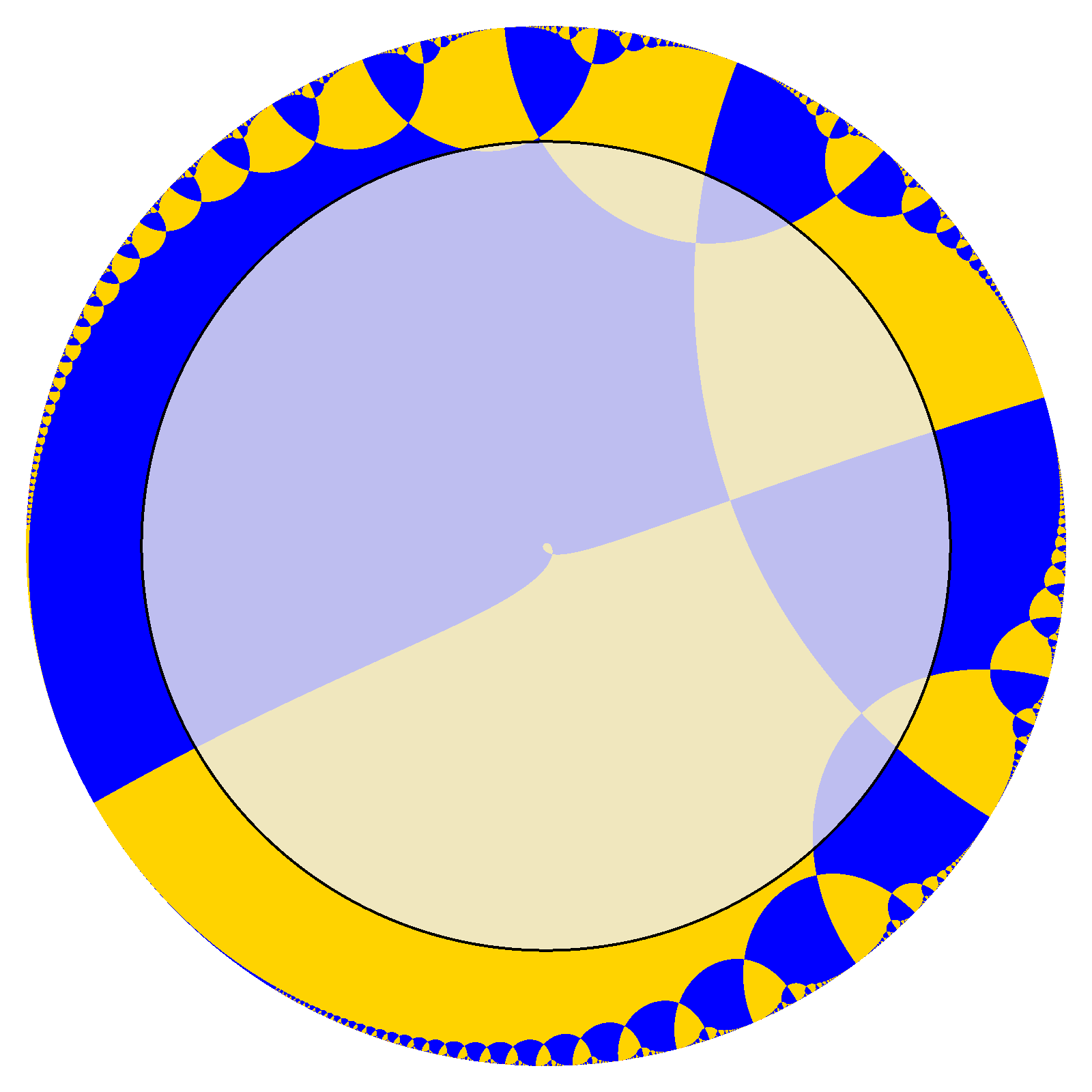}};
\node at (0,5) {\includegraphics[width=5cm]{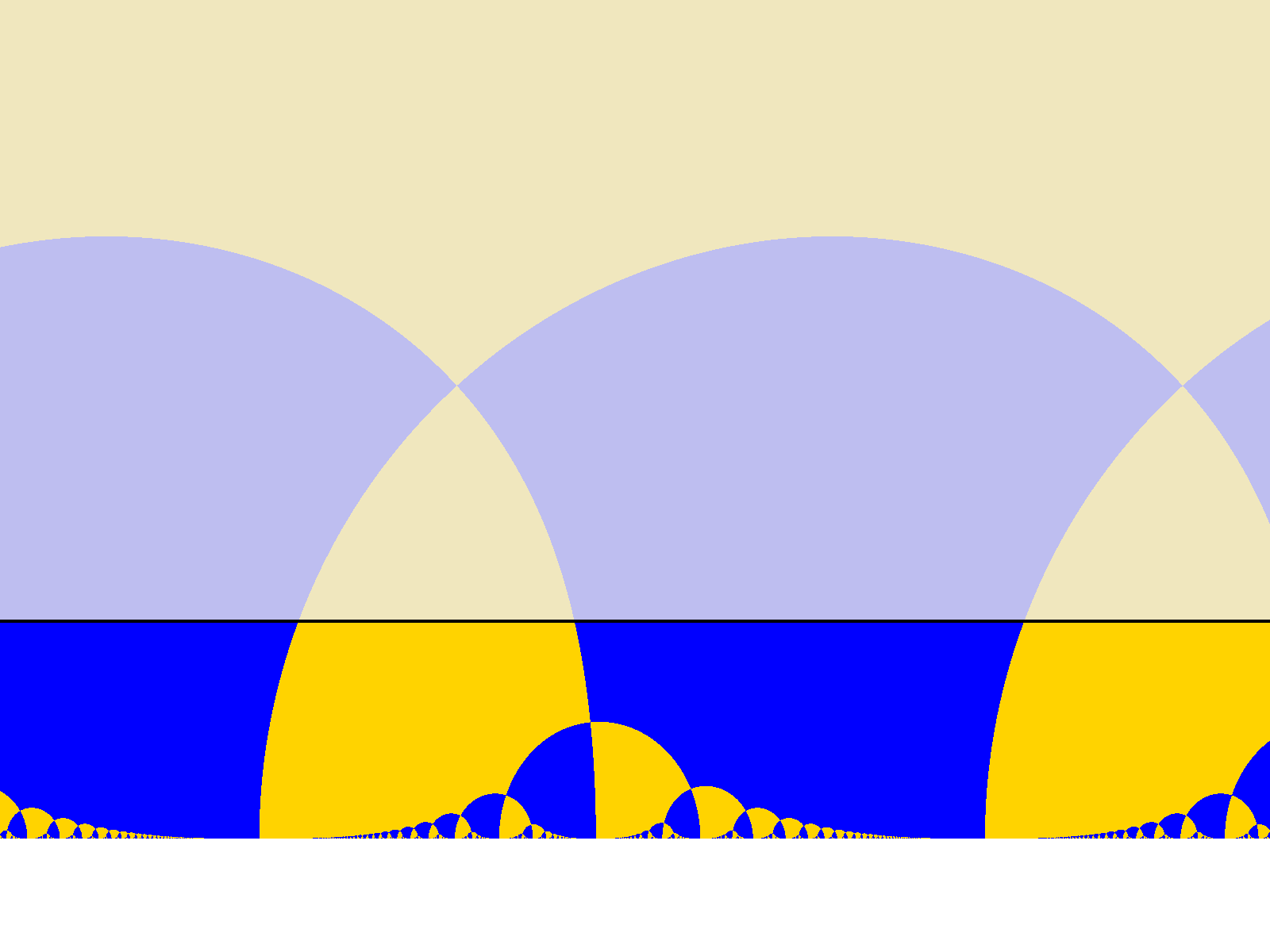}};
\node at (6,5) {\includegraphics[width=5cm]{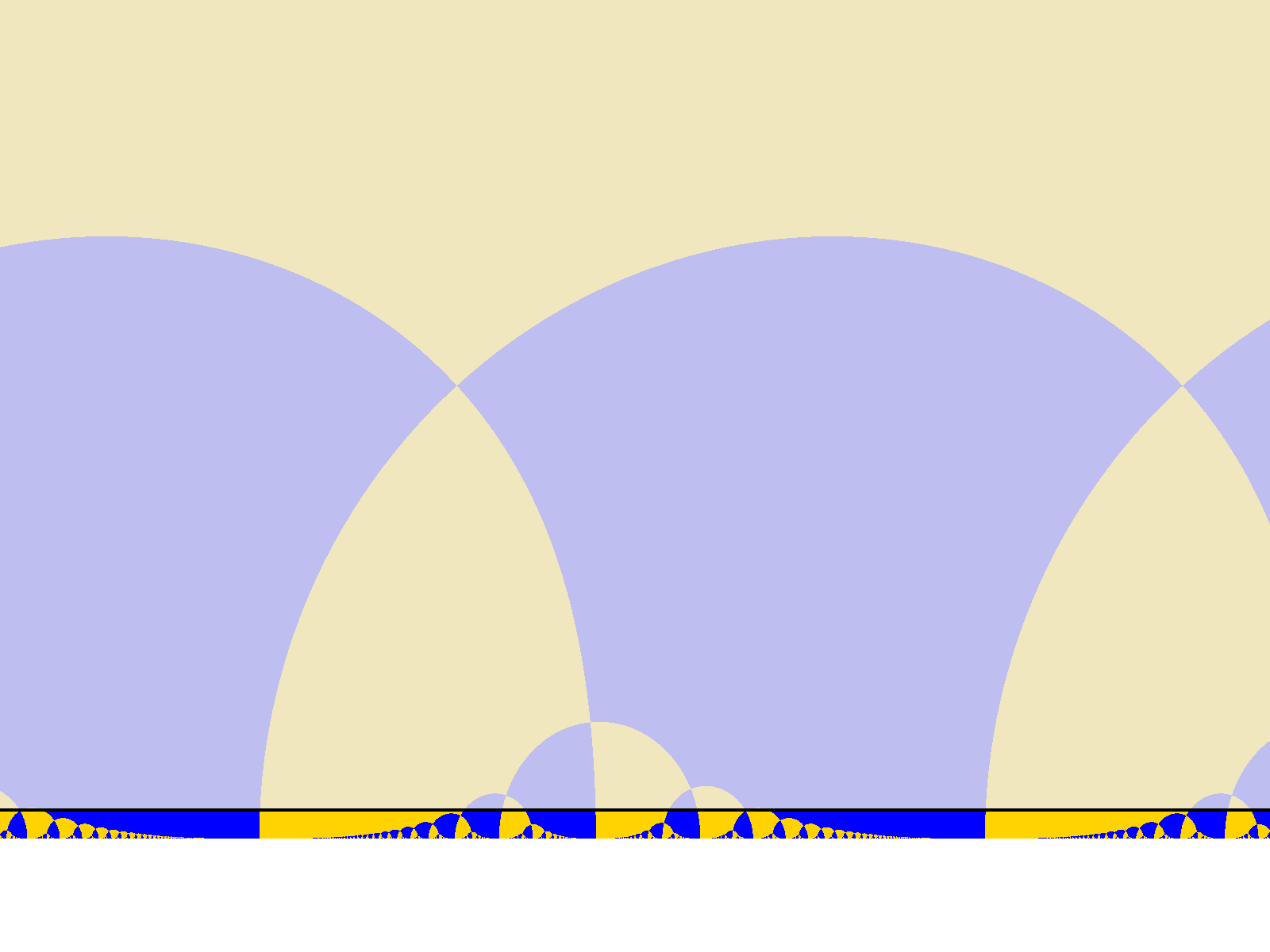}};
\node at (0,0) {\includegraphics[width=5cm]{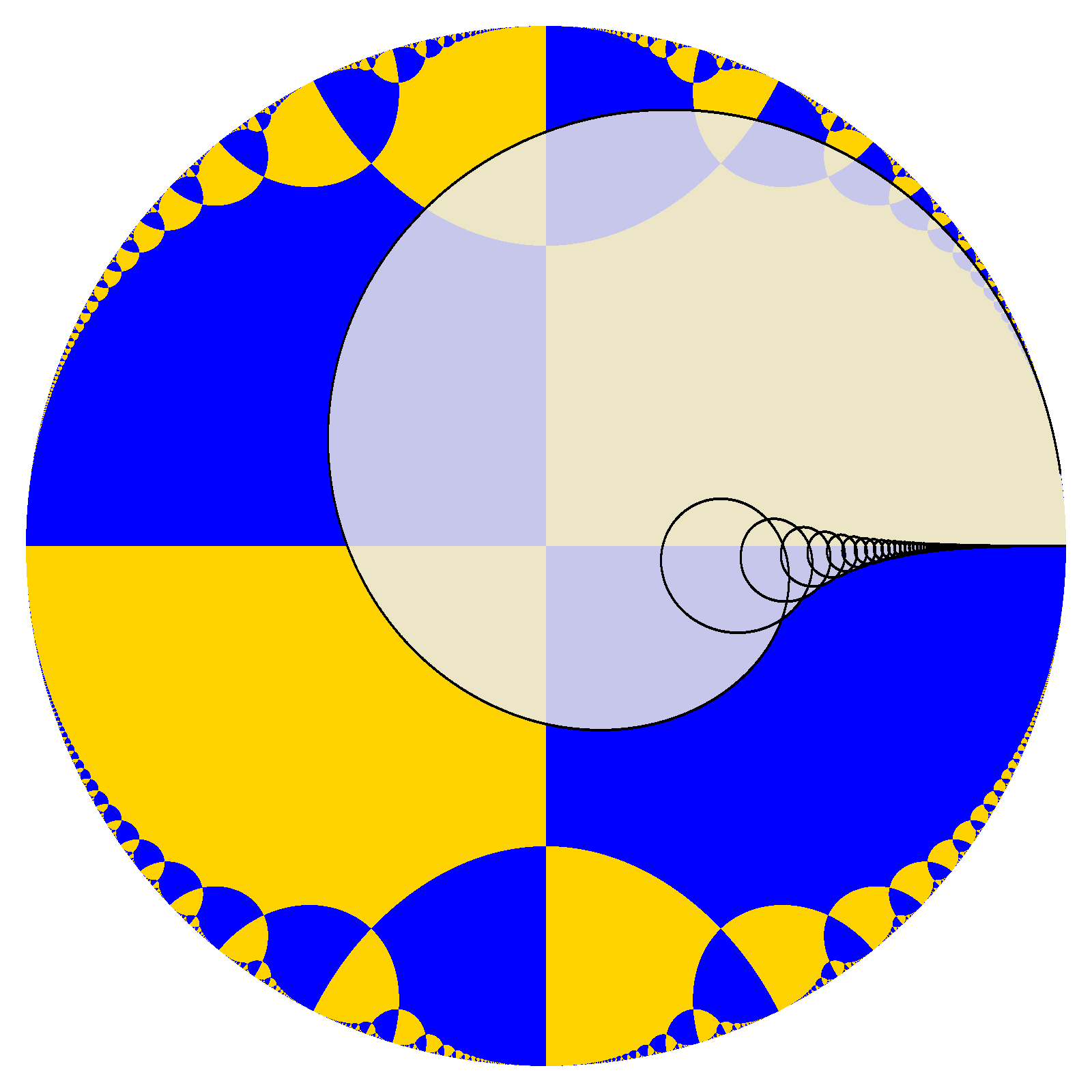}};
\node at (6,0) {\includegraphics[width=5cm]{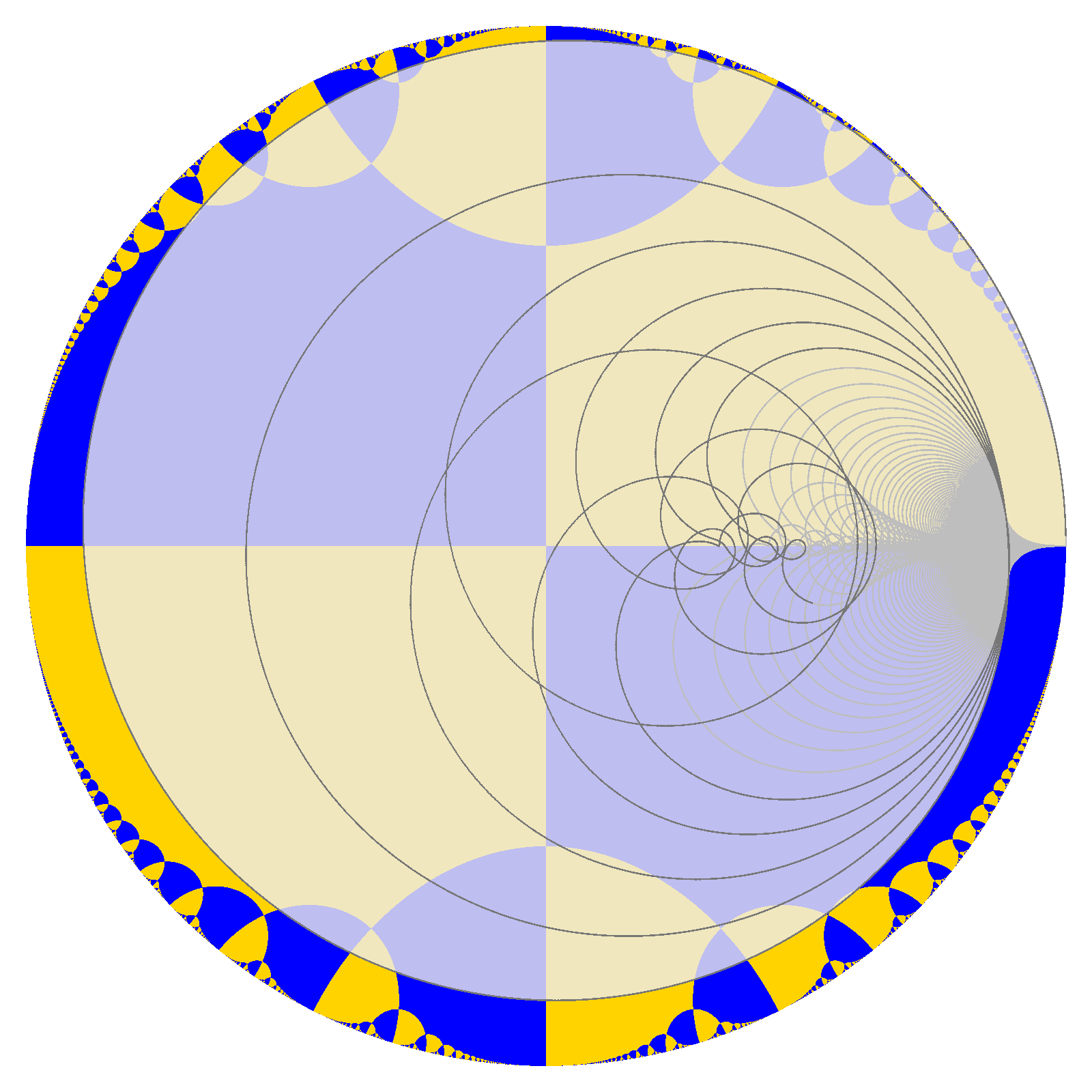}};
\path[->] (0,3.5) edge node[left] {$\Psi_\rep[B_d]$} (0,2.5);
\path[->] (6,3.5) edge node[left] {$\Psi_\rep[B_d]$} (6,2.5);
\path[->] (0,6.8) edge node[left] {$E$} (0,8);
\path[->] (6,6.8) edge node[left] {$E$} (6,8);
\end{tikzpicture}
\caption{The set $C[B_d]$ (bottom row, light tones) for $d=2$ and (left) $\epsilon \approx 0.85$ (this is quite high a value for an $\epsilon$) and (right) $\epsilon\approx 0.22$.}
\label{fig:D1}
\end{figure}

Let $f\in\cal F$: 
\[f=\cal R[B_d]\circ \phi_1^{-1}\]
where $\phi_1\in\sch$.
Denote\nomenclature[U1]{$U_1$}{domain of $f\in\cal F$\nomrefpage}
\[U_1 = \phi_1(\D) = \dom(f).\]
Let $L(\epsilon)$ be the hyperbolic radius in $\D$ of the Euclidean ball $B(0,1-\epsilon)$: 
\begin{equation}\label{eq:L}
L(\epsilon)=\tanh^{-1}(1-\epsilon) = \frac12\log\frac{2-\epsilon}{\epsilon}
.
\end{equation}
In particular
\[\frac{1}{2}\log\frac{1}{\epsilon} \leq L(\epsilon) \leq \frac{\log 2}2+\frac{1}{2}\log\frac{1}{\epsilon}
. \]
Since $\cal R[f]$ belongs to $\cal F$ (\Cref{thm:s1}), there exists $\phi_2\in\sch$ such that:
\[\cal R[f] = \cal R[B_d] \circ \phi_2^{-1}.\]
The map $\phi_2$ is an isomorphism from $\D$ to the domain of definition of $\cal R[f]$.

Denote by $A\subset U_1$ the immediate basin of the parabolic fixed point $0$ of $f$.
\nomenclature[Ab]{$A$}{immediate basin of the parabolic point of $f\in \cal F$\nomrefpage}
Let $U_u$ denote the connected component of $\dom(h[f])$ that contains an upper half plane. It is also equal to the connected component of $\Psi_\rep^{-1}(A)$ that contains an upper half plane.\nomenclature[Uu]{$U_u$}{upper component of $\Psi_\rep^{-1}(A)$, also of $\dom(h[f])$\nomrefpage}
Denote by $C\subset A$ the following set, which is the object under study in the present section:\phantomsection\label{here:C}
\[C=C[f]=\Psi_\rep\Big(\esub{U_u}{(1-\epsilon)}\Big)
\]
(The notation $\esub{}{}$ has been introduced in \Cref{subsec:morenot}).
We claim it can be rewritten as 
\[ C=\phi_3(C[B_d])\]
where $\phi_3:\D\to A$ is the conformal isomorphism conjugating $B_d$ to $f\big|_A$.
Indeed, according to the complement after \Cref{thm:shi2b}, $\phi_3 \circ \Psi_\rep[B_d] = \Psi_\rep[f] \circ \phi_4$ for some conformal map $\phi_4$ from $\H = U_u[B_d]$ to $U_u[f]$, commuting with $T_1$, thus $\phi_4(\esub{U_u[B_d]}{(1-\epsilon)}) = \esub{U_u[f]}{(1-\epsilon)}$.\nomenclature[C2]{$C$}{main object of study of \Cref{subsec:part1}\nomrefpage}

The set $C[B_d]$, which is equal to $\Psi_\rep[B_d](H(\epsilon))$ where $H(\epsilon)=E^{-1}(B(0,1-\epsilon))$ is the half plane defined by ``$\Im z>\frac{1}{2\pi}\log\left(\frac1{1-\epsilon}\right)$'', depends only on $d$ and $\epsilon$, not on $f$, and is forward invariant under $B_d$. 
\Cref{fig:D1} shows examples of sets $C[B_d]$.

In this section we will prove:
\begin{proposition}\label{prop:part:1}
 There exists $c, c'$ and $\xi>0$ (these constants depend on $d$) such that for all $\epsilon<\xi$, there exists $\epsilon'>0$ satisfying
\[\log \frac{1}{\epsilon'} \leq c' + c \log\left(1+\log\frac{1}{\epsilon}\right)\]
such that for all $f\in\cal F$,
\[ C\subset \sub{\dom(f)}{(1-\epsilon')}.
\]
\end{proposition}
Above, the set $C=C[f]$ depends also on $\epsilon$ but we did not figure it in the notation, to avoid clutter. The notation $\sub{U}{r}$ has been introduced in \Cref{subsec:morenot}.

We begin with an easy lemma (recall $L$ was defined near \Cref{eq:L}):
\begin{lemma}\label{lem:d1}
The set $C[B_d]$ is contained within hyperbolic $\D$-distance $\leq c_2+L(\epsilon)$ of the upper main chessboard box of $B_d$.
\end{lemma}
\begin{proof} The upper chessboard box of $B_d$ is the image by $\Psi_\rep[B_d]$ of an open set that contains a half plane $``\Im(z)>M_d"$ and is contained in another half plane strictly smaller that $\H$. 
Recall that $C[B_d]=\Psi_\rep[B_d](H(\epsilon))$ with $H(\epsilon) = ``\Im(z)>\frac{1}{2\pi}\log\left(\frac{1}{1-\epsilon}\right)"$.
For $\epsilon$ big, $H(\epsilon)\subset``\Im(z)>M_d"$. For other values of $\epsilon$, every point in $H(\epsilon)$ can be joined to $``\Im(z)>M_d"$ by a vertical segment of hyperbolic length in $\H$ at most $\frac{1}{2}\left(\log M_d-\log \frac{\log\frac{1}{1-\epsilon}}{2\pi}\right)$.
Since $\Psi_\rep[B_d]:\H\to\D$ contracts hyperbolic metrics and $\frac{1}{2}\log\frac{1}{\log\frac{1}{1-\epsilon}} \leq \frac{1}{2}\log \frac{1}{\epsilon}\leq L(\epsilon)$, the lemma follows.
\end{proof}

\begin{figure}
\begin{tikzpicture}
 \node at (0,11) {\includegraphics[width=5cm]{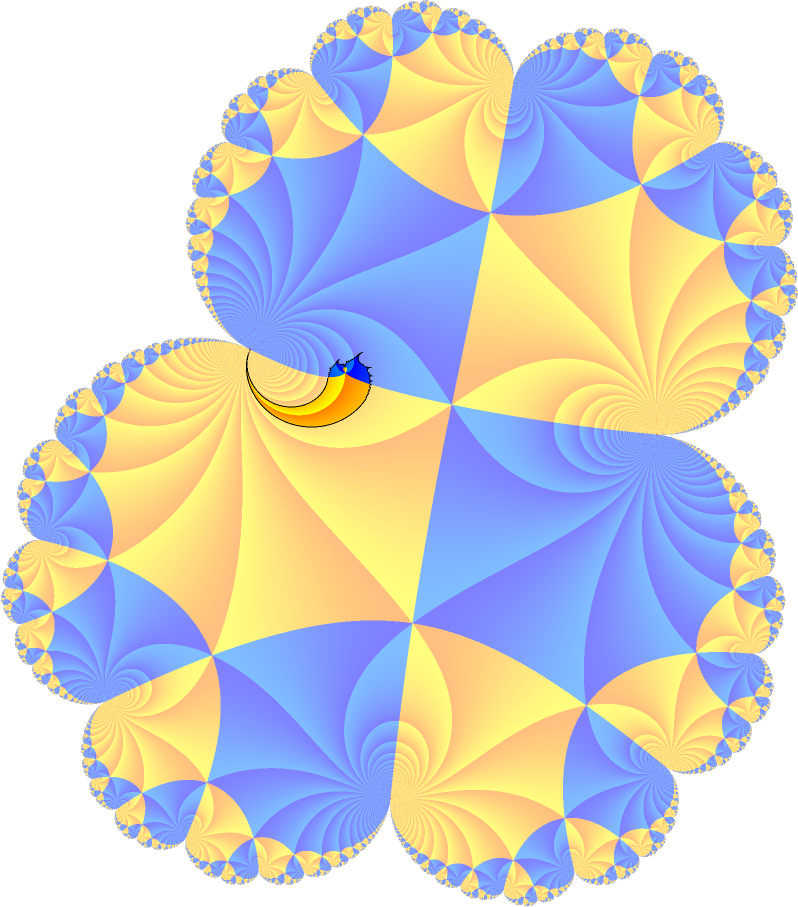}};
\node at (5.5,11) {\includegraphics[width=5cm]{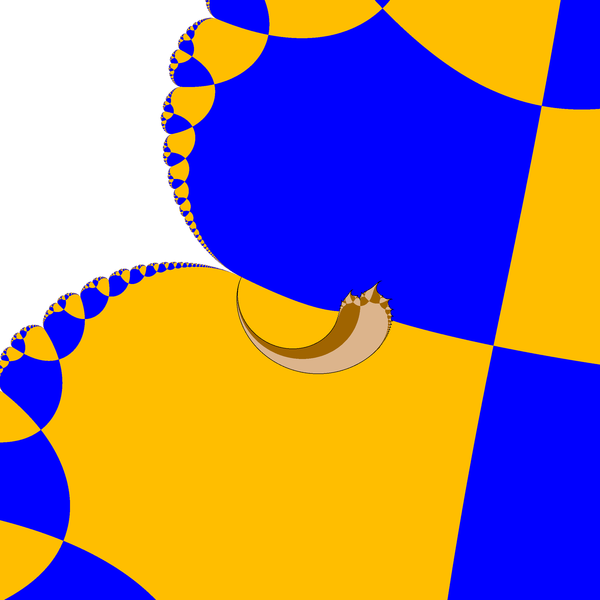}};
\node at (0,5.5) {\includegraphics[width=5cm]{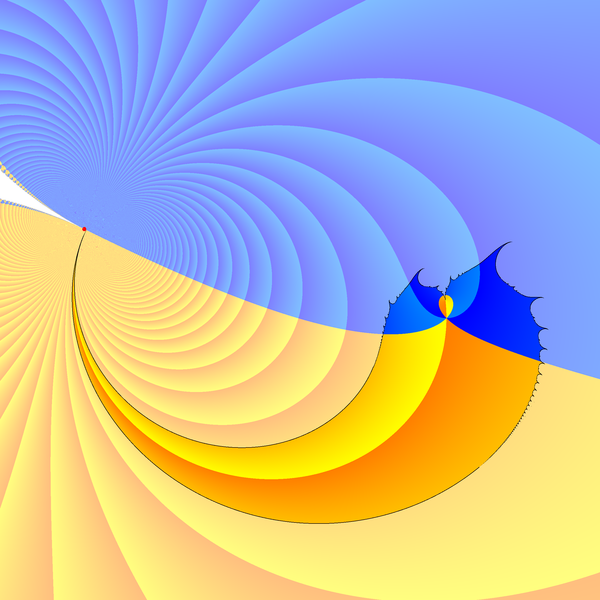}};
\node at (5.5,5.5) {\includegraphics[width=5cm]{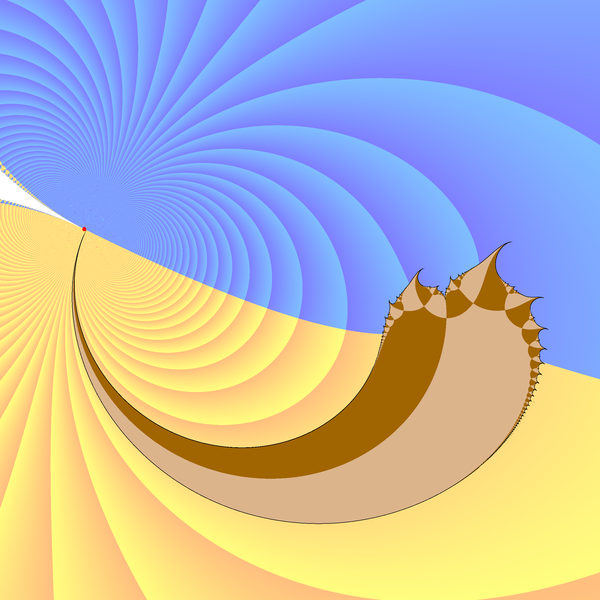}};
\node at (0,0) {\includegraphics[width=5cm]{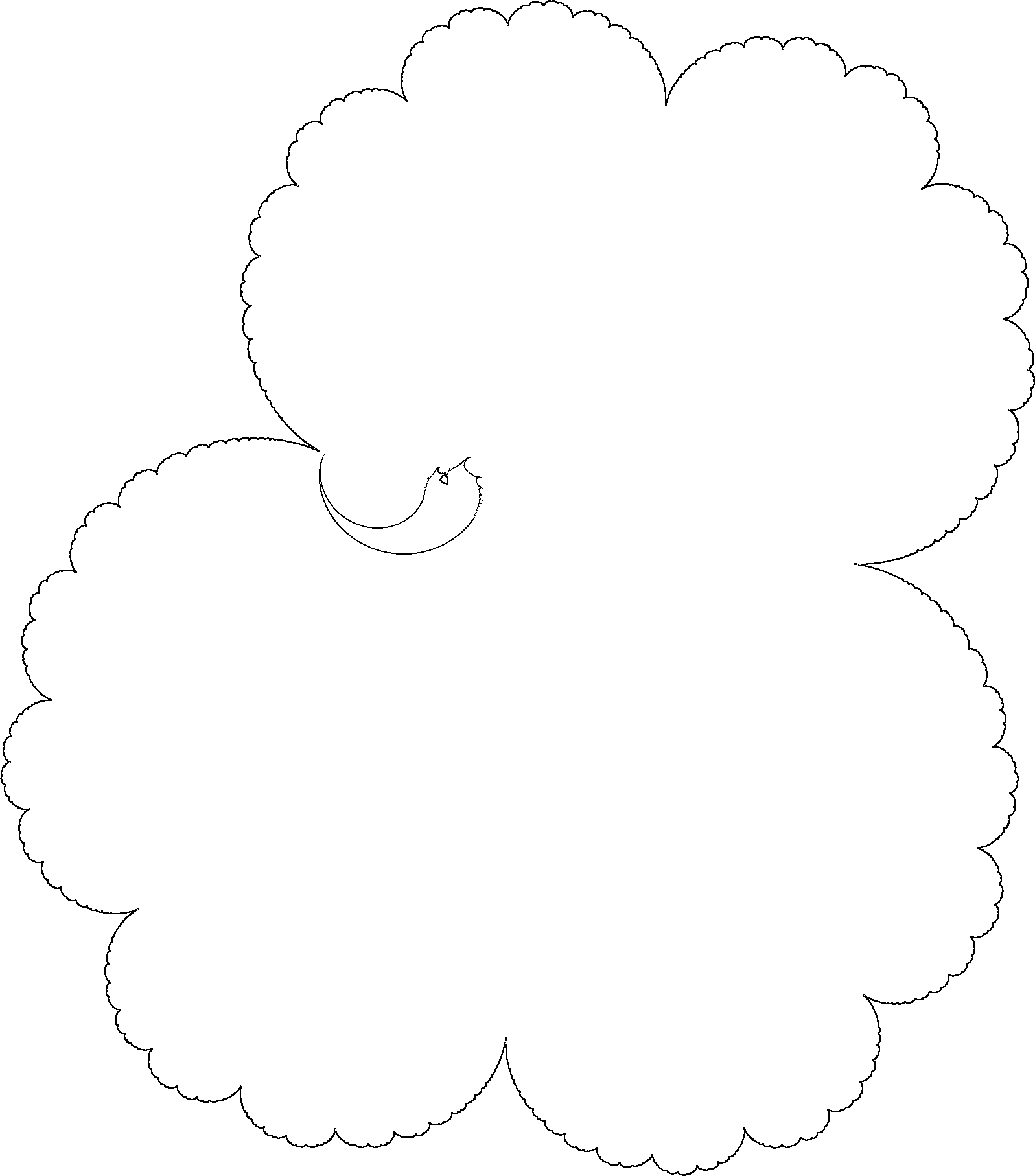}};
\node at (5.5,0) {\includegraphics[width=5cm]{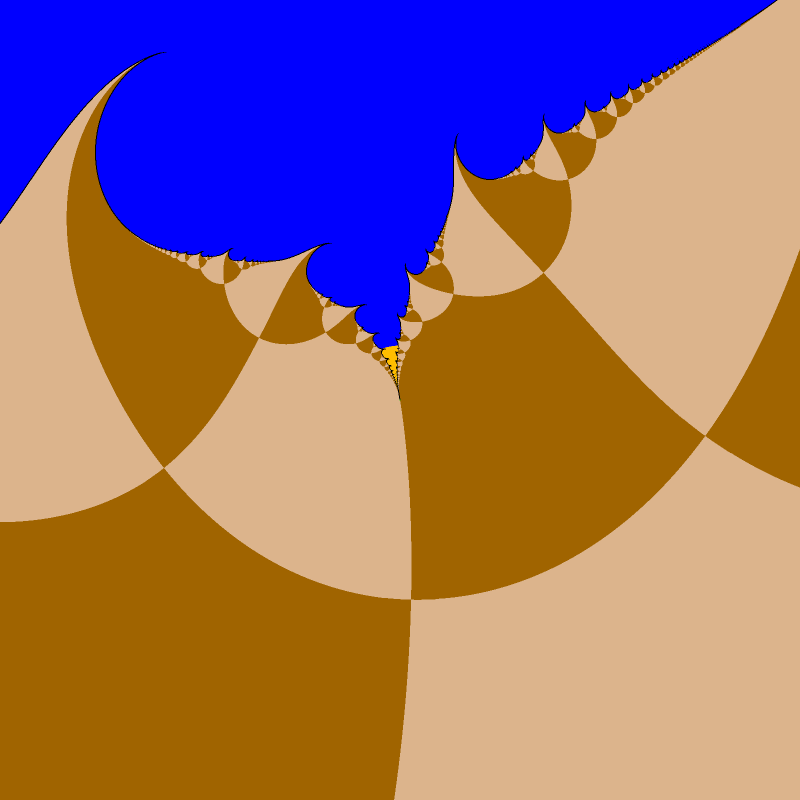}};
\path[draw] (-.32,0.52) -- (0.2,1.3) node[anchor=south] {Mudba};
\path[draw] (-0.19,0.37) -- (0.3,0.37) node[anchor=west] {$A$};
\node at (-1,-1.8) {$U_1$};
\node at (4.9,-.4) {Mudba};
\draw[fill=black] (5.5,0) circle (1.3pt);
\path[draw] (5.5,0) node[anchor=east] {} -- (6.2,0) node[anchor=west] {$0$};
\end{tikzpicture}
\caption{Some open sets associated to $\cal R[P]$ with $P:z\mapsto z+z^2$: its domain $U_1$, its parabolic immediate basin $A$, and the latter's main upper dynamical box Mudba. The rightmost column features the dynamical chessboard of $A$ in shades of brown. The blue and yellow shades depict the structural chessboard of $U_1$.}
\label{fig:choux}
\end{figure}

Note that $\phi_3:\D\to A$ is an isometry for the respective hyperbolic metrics,
and that the upper main chessboard box of $B_d$ is mapped by $\phi_3$ to the main upper dynamical chessboard box of $A$, call it Mudba:
\[\text{Mudba}=\phi_3(\text{Mudba}[B_d]).\]
See \Cref{fig:choux}. From the lemma above, it follows that the set $C=C[f]$ under study is contained within $A$-hyperbolic distance $c_2+L(\epsilon)$ of Mudba.
In order to prove an estimate concerning the latter set, we first need the following easy consequence of the compactness of $\cal F$:
\begin{lemma}\label{lem:alemma}
For all $M>0$ there exists $c'>0$ such that for all $f\in\cal F$,
the upper main and the lower main chessboard boxes of $h[f]$ are both at hyperbolic $\dom(h[f])$-distance $\leq c'$ from respectively the half planes $\Im(z)>M$ and $\Im(z)<-M$ (intersected with $\dom(h[f])$ if necessary).
\end{lemma}
\begin{proof} The extended normalized horn map of $B_d$ is defined on $\C\setminus\R$.
The upper/lower main chessboard boxes of $h[B_d]$ are at positive Euclidean distance from $\R$.
Recall (see \Cref{subsec:nor}, in particular \Cref{lem:phiatu}) that we have the following: $h[f] \circ \phi= T_{w[f]} \circ h[B_d]$ where $w[f]=v_{h[f]}-v_{h[b_d]}$ 
and $\phi$ is an isomorphism commuting with $T_1$ from $\H$, which is the upper connected component of $\dom h[B_d]$, to the upper connected component of $\dom h[f]$, and that $\phi$ maps the chessboard graph of $h[B_d]$ to that of $h[f]$.
Therefore, it is enough to prove that $\phi^{-1}(``\Im(z)>M")$ contains an upper half plane independent of $f$, and a similar statement for the lower part.
Let us write, as $\Im(z)\tend+\infty$:
\[\phi(z)  = z+\tau_f+o(1).\]
From the first point of \Cref{prop:cor} if follows that $|\im (w[f])|$ is bounded over $\cal F$.
From this and the second point, it follows that $|\im(\tau_f)|$ is bounded over $\cal F$. 
Now one of Koebe's inequalities states that $\forall f \in\sch$, $|f(z)|\leq \frac{|z|}{(1-|z|)^2}$. Equivalently, $\forall r\in\,]0,1[$, $f^{-1}\big(B\big(0,r/(1-r)^2\big)\big) \supset B(0,r)$. The map $T_{-\tau_f}\circ \phi$ is semi-conjugate by $E$ to a Schlicht map thus: $\phi^{-1}(``\Im(z)>M")$ contains the half plane ``$\Im(z)>M'$'' where $M'=M'[f]>0$ is related to $M\in\R$ by $e^{2\pi (M+\im \tau_f)}=e^{2\pi M'}+e^{-2\pi M'}-2=2(\cosh(2\pi M')-1)$. Since $\tau_f$ is bounded, the constant $M'[f]$ is bounded too.
The proof for the lower box is similar.
\end{proof}

Recall $U_1$ denotes the domain of $f$.

\begin{lemma}\label{lem:mudba}
 Mudba is contained in a hyperbolic $U_1$-ball of uniform diameter $c_7$.
\end{lemma}
\begin{proof} Choose $r$ small enough so that $B(0,2r)\subset U_1$ for all $f\in\cal F$. By \Cref{prop:butterfly}, there is some $h>0$ such that for all $f\in\cal F$, the half planes $\Im(z)>h$ and $\im(z)<-h$ are mapped by $\Psi_\rep[f]$ inside $B(0,r)$.
From \Cref{lem:alemma} the upper box is at distance $\leq c_7$ from ``$\Im(z)>h$'' for the hyperbolic metric of $\dom(h[f])$.
The map $\Psi_\rep: \dom(h[f]) \to A$ is holomorphic thus a contraction for hyperbolic metrics, thus the image by $\Psi_\rep$ of the upper chessboard box is at bounded $A$-hyperbolic distance of $B(0,r)$ (the latter is not contained in $A$ but it does not matter) and thus at $U_1$-hyperbolic distance even smaller, since the inclusion of $A$ in $U_1$ is a contraction too.
\end{proof}

By \Cref{lem:d1,lem:mudba}, to fulfill the objectives of Step 1, it is enough to prove that a path starting from Mudba, contained in $A$ and of $A$-hyperbolic length $\leq c_2+L(\epsilon)$ has a $U_1$-hyperbolic length much smaller than $c_2+L(\epsilon)$. The precise bound obtained will yield \Cref{prop:part:1}. Note that we will in fact bound the $U_1^*$-hyperbolic length, which is bigger that the $U_1$-hyperbolic length, where
\[ U_1^*=U_1\setminus\{0\}
.\]

Let us make the following change of coordinates: $w=\log(z)/2i\pi$. Let $\wt A$ be a lift of $A$: it is a connected and simply connected subset of $\C$ that does not intersect its translates $\wt A+k$ when $k\in\Z$ is non-zero. As a consequence, each horizontal intersects this open set along a union of open segments of length at most $1$ (in fact the sum of lengths is at most $1$). Thus the Euclidean distance from any $z\in \wt A$ to the boundary of $\wt A$ is $\leq 1/2$. This implies by Koebe's $1/4$ theorem:
\[\rho_{\wt A}(z)\geq 1/2\]
(a better bound holds but we do not need it; recall $\rho_U(z)|dz|$ designates the infinitesimal element of hyperbolic metrics on $U$).

\begin{figure}
\begin{tikzpicture}
\node at (0,7) {\includegraphics[width=8cm]{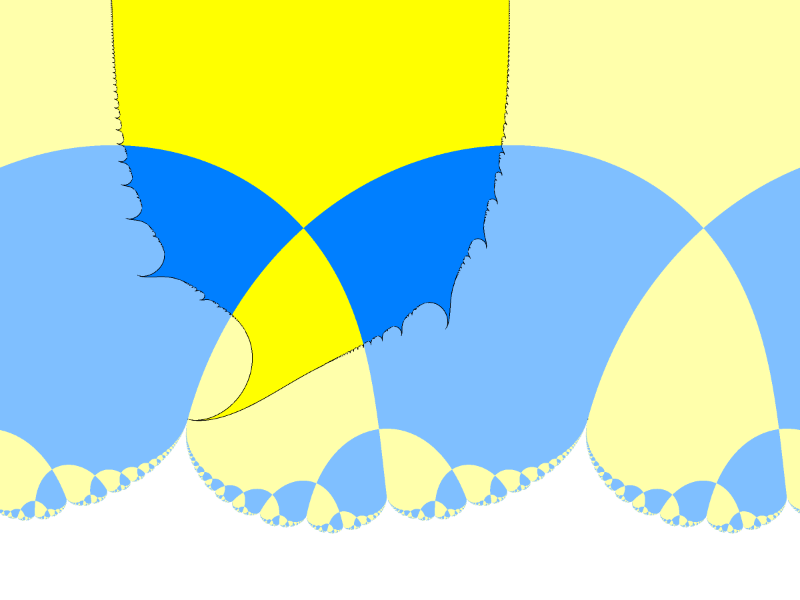}};
\node at (0,0) {\includegraphics[width=8cm]{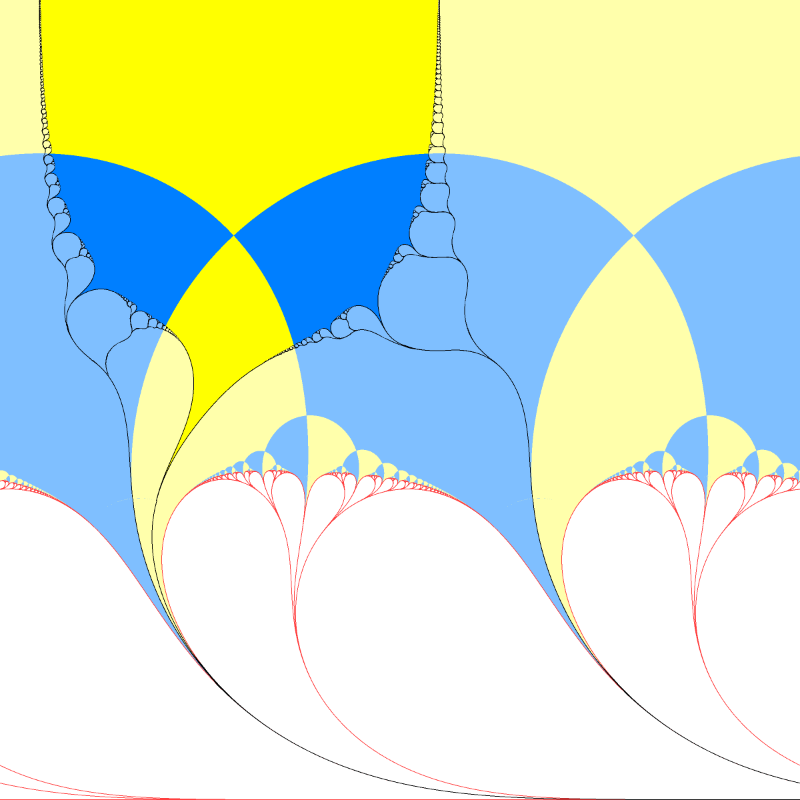}};
\end{tikzpicture}
\caption{Two examples of lifted immediate parabolic basins $\wt A$ for maps $f\in\cal F$. Left: $f=\cal R(z\mapsto z+z^2)$, Right: $f=\cal R(z\mapsto ze^z)$.}
\label{fig:Atildes}
\end{figure}

\remark The set $\wt A$ is unbounded upwards, since the image in $\wt A$ of an attracting petal in $A$ is an infinite finger-shaped domain extending upwards. See \Cref{fig:Atildes} for examples.
One should not expect $\wt A$ to be bounded in the other directions either.
Recall that $f\in\cal F$ is characterized by the choice of its domain $U_1$, which can be any simply connected domain containing the origin with conformal radius $1$ w.r.t.\ the origin.
For well chosen unbounded $U_1$, the set $\wt A$ is unbounded downwards.
One could object that since in the applications, the renormalization operator is iterated, we could restrict to maps in $\cal R[\cal F]$ instead of $\cal F$, and that maps in $\cal R[\cal F]$ all have a uniformly bounded domain of definition, as follows for instance from \Cref{prop:cor}.
But this will not prevent unboundedness in the horizontal direction: even for bounded $U_1$, provided its boundary swirls infinitely many times around $0$, carefully chosen $U_1$ will yield a set $\wt A$ whose projection on the real line is unbounded.
The latter case is not just a curiosity but does happen for $f=\cal R[z\mapsto ze^z]$, i.e.\ the first renormalization of the map $g(z)=ze^z$ which has a non-linearizable parabolic point at the origin, and whose set of singular values are the two asymptotic values $\infty$, $0$ and the image $g(-1)$ of the unique critical point $-1$.
Its immediate basin must contains a singular value, and the only possible one is $g(-1)$.
Hence the map $g$ satisfies the hypotheses of \Cref{thm:s1}, thus $f=\cal R[g] \in \cal F$.
A careful study shows that the domain of definition of $\cal R[g]$ swirls like above, more precisely that its lifted immediate basin $\wt A$ has infinitely many accesses to infinity by curves asymptotic to some common horizontal line.
The map $g$ does not belong to $\cal F$ but we believe that for all $n>0$, $\cal R^n[g]$, that belongs to $\cal F$, will have a set $\wt A$ with the same properties.
To prove this, one may try and see if there is invariance by $\cal R$ of the following property for $f\in \cal F$: let $c$ be the main critical point of $f$ (the one on the boundary of the main upper structural box); let $\wt f$ be a lift of $f$ and let $\gamma$ be the lift by $\wt f$ starting from $c$, of the horizontal half line $\wt f(c)+[0,+\infty[$, such that $\gamma$ intersects the boundary of the upper box only at $c$; then $\Re(\gamma)$ tends to infinity.
\endremark

Now consider a point $z_0\in C=\phi_3(C[B_d])$ and consider a path $\gamma$ of $A$-length at most $c_2+L(\epsilon)$ from Mudba to $z_0$.
Let us apply $f$ once.
Then $A$ is mapped to itself and so are $C$ and Mudba. The path $\gamma$ is mapped to a path $f(\gamma)$ contained in $A$, from Mudba to $z_1=f(z_0)$, and by the Schwarz-Pick inequality, the $A$ hyperbolic length of $f(\gamma)$ is $\leq$ that of $\gamma$. Consider a lift $\gamma_2$ of $f\circ \gamma$ by $E$ (the path $f(\gamma)$ is contained in $A$, thus does not meet the origin). The Euclidean length of $\gamma_2$ is equal to
\begin{equation}\label{eq:elg2}
\int_{\gamma_2} |dz| = \int_{\gamma_2}\frac{\rho_{\wt A}(z)|dz|}{\rho_{\wt A}(z)} \leq 2 \int_{\gamma_2} \rho_{\wt A}(z)|dz| \leq 2(c_2+L(\epsilon)).
\end{equation}

Let us now relate the element of length $\rho_{U_1^*}(z)|dz|$ to $|d\log f(z)/2\pi|$.
Let $\wt f$ be the continuous lift of $f$ by $E$ that fixes $\wt A$: $\wt f: \wt U_1 \overset{\text{def}}= E^{-1}(U_1) \to \C$ and $E\circ \wt f = f \circ E$. The inverse of $E$ is the multivalued function $E^{-1} (z)=\frac{1}{2\pi i}\log z$. Let $\wt v+\Z$ be the set of critical values of $\wt f$. The map $\wt f$ has no asymptotic value over $\C$.
Denote by $\C^{\pm}$ the upper half plane and the lower half plane delimited by the horizontal line through these critical values.
For every point $z$ mapped to $\C^\pm$ by any branch of $\frac{1}{2\pi i}\log f$, the latter map has inverse branches defined in $\C^\pm$, with image the $f$-structural chessboard box containing $z$. This inverse branch is univalent, except for $z$ in the little loop around $0$ where it is infinite-to one. In all cases, these inverse branches map in $U_1^*$ and are non-expanding for the respective hyperbolic metrics as follows:
\begin{equation}\label{eq:euc}
\rho_{U_1^*}(z)|dz| \leq \rho_{\C^{\pm}}(\zeta)|d\log f(z)/2\pi|
\end{equation}
where $\zeta$ is the image of $z$ by the considered branch of $\frac{1}{2\pi i}\log f$.

Near the boundary of $\C^\pm$, better estimates hold.
For instance:
\begin{lemma}\label{lem:diam}
There exists $c_3>0$ such that for all $f\in\cal F$, the following holds. Let $\wt v$ be a critical value of $\wt f$ and $V$ be any connected component of the pre-image by $\wt f$ of the square $\wt v+I+iI$ where $I=[-1/2,1/2]$. Then the hyperbolic diameter in $U_1^*$ of $E(V)$ is $\leq c_3$.
\end{lemma}
\begin{proof} Recall the critical values of $\wt f$, are the elements of $\wt v+\Z$ and that its only asymptotic value over $\wh \C$ is $\infty$. Consider the disk $\wt v+\D$ and the component $U$ of $\wt f^{-1}$ that contains $V$. Then $\wt f$ factors on $U$ as $a\circ\on{pow}\circ b$ where $\on{pow} : \D\to \D$ is either the identity or the map $z\mapsto z^d$, where $a(z)=\wt v+z$ and where $b$ is an isomorphism from $U$ to $\D$. Then $a^{-1}(V)=I+iI \subset B(0,1/\sqrt{2})$ thus $(a\circ \on{pow})^{-1}(V)$ is contained in the Euclidean ball $B(0,\left(\frac{1}{\sqrt 2}\right)^{1/d})$. The map $b^{-1}:\D\to E^{-1}(U_1)$ is non-expanding for the respective hyperbolic metrics, and $E: E^{-1}(U_1)\to U_1^*$ also is, thus the lemma holds with $c_3=$ the hyperbolic distance in $\D$ from $0$ to the $d$-th root of $1/\sqrt{2}$.
\end{proof}

Another easy lemma: 

\begin{lemma}\label{lem:hypeuc}
Let $a$, $b$ be two points in the hyperbolic plane $\H$:
\[\Im(a)\geq \frac{1}{2}\text{ and }\Im(b)\geq \frac{1}{2} \implies d_\H(a,b)\leq \log (1+2|a-b|).\]
\end{lemma}
\begin{proof}Use the following formula for the hyperbolic distance in $\H$:
\[d_\H(a,b) = 
   \on{argsh}\frac{|b-a|}{2\sqrt{\Im a\,\Im b}},\]
and the inequality $\on{argsh} t\leq \log(1+2t)$.
\end{proof}
So for instance, the hyperbolic distance from $i$ to $i+x$ is a $\cal O(\log x)$ when $x\tend+\infty$, thus much smaller than $x$. Recall that the geodesic between $a$ and $b$ in $\H$ is an arc of Euclidean circle. For the hyperbolic metric, this arc turns out to be much shorter than the straight euclidean line.

Let $\beta_0$ be the structural $U_1$ chessboard box that is a punctured neighborhood of the origin. Recall that we denote $U_1^* = U_1\setminus\{0\}$.
Consider any structural $U_1$ chessboard box $\beta$. Let us call \keyw{cubox} the set $\ov{\beta}\cap U_1^*$.
Let us endow $U_1^*\setminus f^{-1}(v)$ with the infinitesimal metric induced by pulling back the Euclidean metric by $\frac{1}{2\pi i}\log f$.
We call this the \keyw{flat metric}.
It has a regular and locally flat extension to a neighborhood of the non-critical preimages of $v$ and is singular precisely at the critical preimages of $v$, where it has a conical point of angle $2\pi d$.
Let us call \keyw{box-Euclidean distance} the distance induced on $U_1^*$ by this flat metric.
Recall that if $\beta\neq \beta_0$, then $\frac{1}{2i\pi}\log f$ is well defined on $\beta$ and maps it to a half plane $\C^\pm$. It also maps the cubox $\ov{\beta}\cap U_1^*$ to the closure of this half plane.

In the sequel, we call $b_*$\nomenclature[bs]{$b_*$}{the cubox that contains a punctured neighborhood of the origin\nomrefpage} the cubox that contains a punctured neighborhood of the origin: $b_*=\ov{\beta_0}\cap U_1^*$.

\begin{corollary}\label{lem:arcs}
Consider two points in a cubox $b$. Denote $d_e$ the distance between these two points for the metric induced by the flat metric restricted to $b$ and $d_h$ the distance between these two points for hyperbolic metric on $U_1^*$. Then
\[d_h \leq c'_5+\log (1+c_5d_e).\]
\end{corollary}
\begin{proof}Let us apply $\frac{1}{2\pi i}\log f$ so as to work in a half plane, and to fix ideas, let us assume it is the half plane $\C^+$.
If the cubox $b$ is $b_*$ then when we lift the two points we choose these lifts so that their euclidean distance is minimal, so as to coincides with $d_e$.
If any of the two points is at distance $\leq 1/2$ from the boundary of $\C^+$ then move it up so that it is at distance $1/2$: we get a new pair of points in $\C^+$ that corresponds to a new pair of points in $b$. By \Cref{lem:diam}, each new point is at $U_1^*$-hyperbolic distance $\leq c_3$ from the former so the $U_1^*$-hyperbolic distance between the the points in the pair has changed by at most $c_3$, and by at most $2c_3$ if we needed to move both points.
Similarly, the Euclidean distance between the points in $\C^+$ has changed by at most $1$. By Equation~\eqref{eq:euc} the $U_1^*$-hyperbolic distance between the two (possibly) new points will be at most their $\C^+$-hyperbolic distance. Using \Cref{lem:hypeuc}, on the latter we get
$d_h\leq 2c_3+\log(1+2(d_e+1)) = (2c_3+\log 3) + \log(1+\frac{2}{3}d_e)$.
\end{proof}

Let \keyw{$\wt f$-cuboxes} be defined similarly: these are sets of the form $\ov{b}\cap \wt U _1$ where $b$ is a structural chessboard box of $\wt f$. The map $\wt f$ is a bijection from such a set to the closed upper or lower half plane.
We can endow $\wt U_1$ with an infinitesimal box-Euclidean metric, by pulling-back by $\wt f$ the canonical Euclidean metric element $|dz|$ on the complex plane.
Recall that $f\circ E= E\circ \wt f$, thus we get the following compatibility statements. The projection by $E$ of an $\wt f$-cubox is a cubox.\footnote{The connected components of the preimage of a cubox by $\wt f$ are $\wt f$-cuboxes with one notable exception where we get a chain of cuboxes that meat at corners.}
 The box-Euclidean metric element on $\wt U_1$ is the pull-back by $E$ of the box-Euclidean metric element on $U_1^*$.

The following result is not used here, but we find it interesting:
\begin{lemma}\label{lem:sccubox}
A connected union of cuboxes that includes $b_*$ is simply connected if we add $\{0\}$ to the union.
\end{lemma}
\begin{proof} Remove the loop from the parabolic structural chessboard graph of $U_1$. Then we get a tree (an infinite tree), on which the union retracts to a connected subset, which is thus simply connected and homotopically equivalent to the union.
\end{proof}

\begin{figure}
\begin{tikzpicture}
\node at (0,0) {\includegraphics[width=10cm]{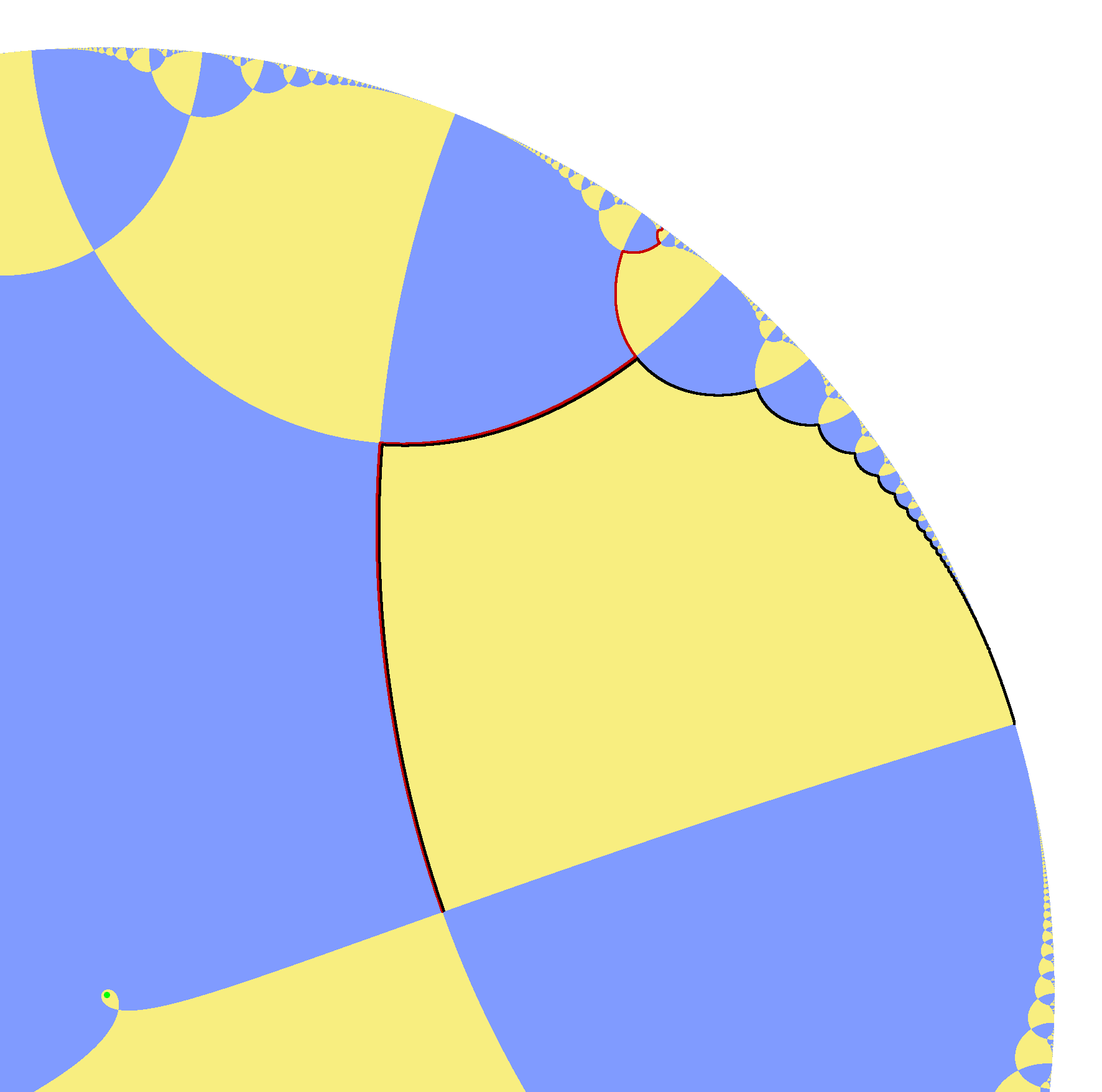}};
\end{tikzpicture}
\caption{A slow path in black, a quick path in red. The first one stays on the boundary of a single cubox. The other one turns alternately left and right at every corner. Here speed is to be understood as the order of magnitude of the hyperbolic distance from the origin, when the curve is followed at constant box-Euclidean speed (on this picture, it takes the same time to get from a corner to the next one): in the first case it is logarithmic, in the second case linear.}
\label{fig:path-path}
\end{figure}
Note that there are paths in $U_1$ reaching the boundary, and whose compact subsets are of hyperbolic diameter comparable to their box euclidean length: see \Cref{fig:path-path}. An important task is thus to formulate and prove a combinatorial statement (\Cref{lem:bddchainlenght}) about the cuboxes that the immediate basin $A$ may cross, and that prevents this kind of behaviour for paths contained in $A$.

Define a chain of boxes to be a finite sequence $b_0$, $b_1$, \ldots, $b_n$ of cuboxes such that two consecutive elements have non empty intersection, i.e.\ consecutive boxes are equal or share a side or a corner within $U_1$. The integer $n$ is called the length of the chain. With our convention there are $n+1$ cuboxes in a chain of length $n$. Define the \keyw{combinatorial distance} between cuboxes as the minimal length of chains from one to the other. With our convention, this is a distance.

\begin{lemma}\label{lem:Lc} Let $b$, $b'$ be cuboxes and consider points $x\in b$ and $x'\in b'$. Then the box-Euclidean distance $L$ between $x$ and $x'$ and the combinatorial distance $n$ between $b$ and $b'$ satisfy:
\[ n\leq \lfloor L\rfloor+1
.\]
\end{lemma}
\begin{proof}
First case: $L<1$. Recall that the set of critical values of $\wt f$ is of the form $\wt v+\Z$ for some $\wt v$ and that $\wt f$ has no asymptotic value over $\C$.
Consider a path $\gamma$ from $x$ to $x'$ and of box-Euclidean length $<1$.
Let $\wt \gamma$ be a lift of $\gamma$ by $E$. 
The image of $\wt \gamma$ by $\wt f$ has Euclidean length $<1$ in the plane. There will therefore exist $k\in \Z$ such that $\wt \gamma$ is completely contained in the plane minus the translate of $]-\infty,-1]\cup[1,+\infty[$ by $\wt v+k$.
The connected components of the pre-image by $\wt f$ of such a slit plane are contained in unions of $2$ or $2d$ of $\wt f$-cuboxes that touch at a common point: this is because there is at most one critical value (and no asymptotic value) of $\wt f$ in the slit plane.
Now $\wt\gamma$ is contained in such a component hence $n\leq 1$.

In the general case, we could use that there is a shortest path from $x$ to $x'$ (see \Cref{lem:am}) but we can do here without that information: consider a path $\gamma$ from $x$ to $x'$ of length close enough to $L$ so as to have the same integer part as $L$.
Let $\epsilon>0$ and cut the path into pieces of length $1-\epsilon$, except maybe for the last piece for which we require length $\leq 1-\epsilon$.
Let $k$ be the number of pieces obtained: if $\epsilon$ small enough, $k=\lfloor L\rfloor+1$.
Let $x_0$, \ldots, $x_k$ denote the sequence of starting and end points of these pieces.
Let $b_0=b$, $b_k=b'$ and for $0<n<k$ let $b_n$ be a cubox containing $x_n$.
From the first case we get that the combinatorial distance between $b_n$ and $b_{n+1}$ is $\leq 1$ for $0\leq n<k$.
The combinatorial distance between $b$ and $b'$ is thus $\leq k$.
\end{proof}

For $n\geq 0$, consider the set $\cal B_n$ of cuboxes at combinatorial distance $\leq n$ of the cubox $b_*$. 
Note that for $n\geq 2$, the set $\cal B_n$ is a infinite union of cuboxes.
The next lemma is illustrated by \Cref{fig:situ}.

\begin{figure}
\includegraphics[width=11.9cm]{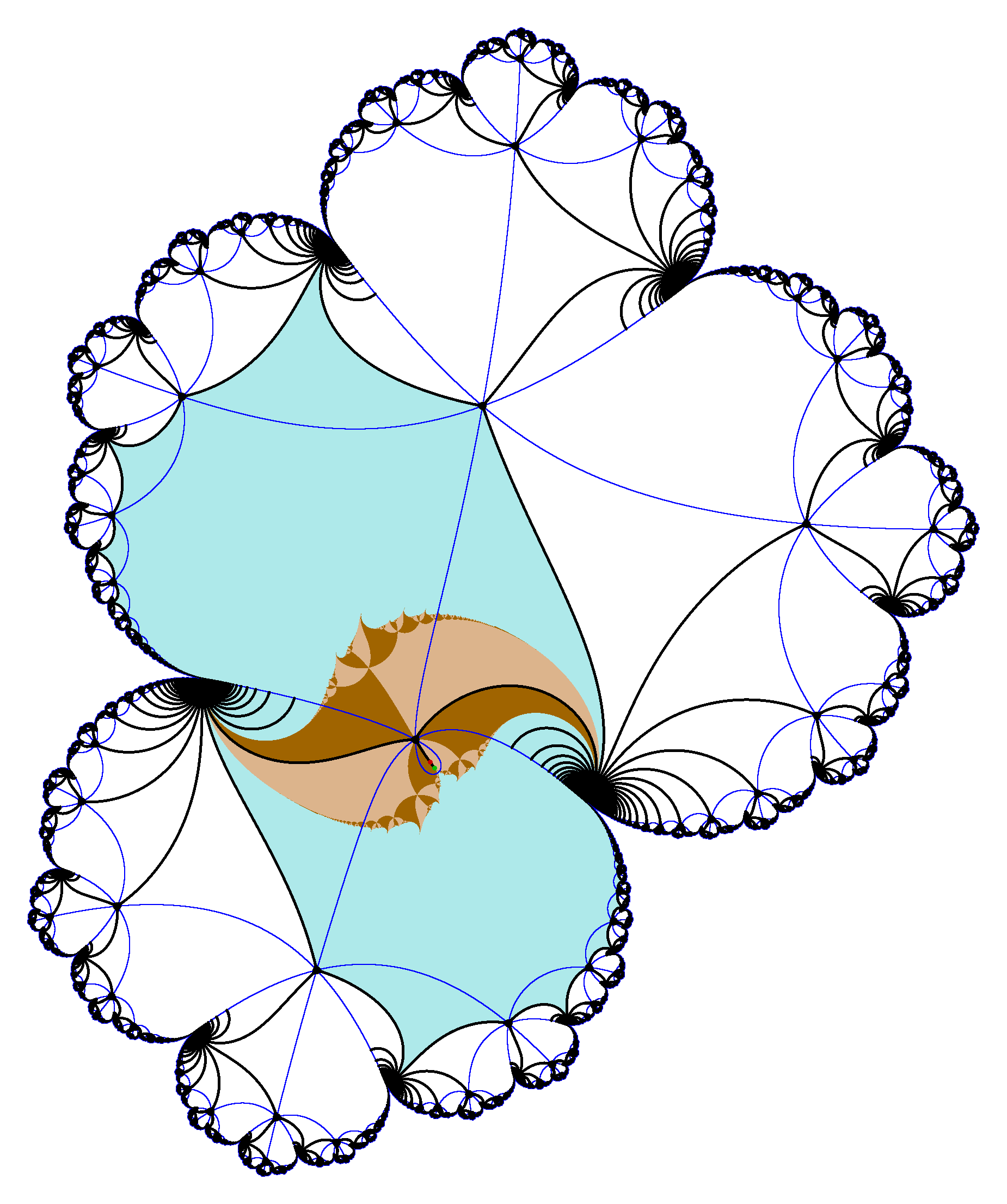}
\caption{Example for $d=3$.
We chose some $f\in \cal F$ (more precisely we took the first renormalization of $z\mapsto z^3+c$ with $c$ so that there is a fixed point tangent to the identity).
The blue graph is the structural chessboard of $f$.
The origin is marked by a tiny green dot and the critical value of $f$ by a red one.
It is a non-linearizable parabolic point of $f$.
We drew in brown shades the dynamical chessboard of $f$ in the immediate basin $A$ of this point. The dark lines are the set $f^{-1}(\cal C)$ where $\cal C$ is the principal curve (see the text).
The light blue set is the component containing $A$ of $\dom(f)$ minus the all the dark lines that are not contained in $A$.
The picture has been accurately drawn, the curve $\cal C$ is a small edge part in the black graph, from the green dot to the red one.
It is very close to be a segment. As a consequence, $f^{-1}(\cal C)$ is formed of curves that are very close to intrinsic verticals of cuboxes.
It seems therefore that the light blue is completely contained in $\cal B_2$. This is probably the case for all maps in $\cal F$ for $d=3$ because the loop is very small.
It may still hold when $d$ gets close to $\infty$, but that would require a more detailed specific analysis as in~\cite{IS}, starting from the fact that $v_f$ is close to $0$ ($|v_f|\in [\sim\, 1/140,\sim\, 1/20]$, see page~\pageref{here:gamma} in \Cref{subsec:viz}).
We decided instead to resort to general arguments instead: in the proof of \Cref{lem:bddchainlenght} we consider cases where $\cal C$ may be very far from a segment.
}
\label{fig:situ}
\end{figure}

\begin{lemma}\label{lem:bddchainlenght}
There exists $c_4\in\N$ such that $\forall f\in \cal F$, $A\subset \cal B_{c_4}$.
\end{lemma}
\begin{proof}
Let us consider the principal curve $\cal C=\cal C[f]$ defined in \Cref{prop:ppalcurv}, starting from $v=v_f$ and ending at $0$.
We recall it is a connected component of the preimage by $\Phi_\at=\Phi_\at[f]$ of $\Phi_\at(v)+[0,+\infty[$, that starts from $v$ and ends on the parabolic point (
and is contained in the common boundary in $A$ of the two principal dynamical chessboard boxes of $f$), that $f(\cal C)\subset \cal C$ and hence $\cal C$ contains the orbit of the critical value.
Let us use \Cref{prop:cf}, that provides a disk $D_\at=D_\at[f]$ of uniform diameter $r_0$ contained in the basin of $f$ and which eventually traps any orbit in the parabolic basin.
By \Cref{lem:traptime}, the number of iterates needed for the critical value $v$ to enter $D_\at$ is bounded over $\cal F$: $\exists n_0\geq 0$ such that $\forall f\in\cal F$, $f^{n_0}(v)\in D_\at[f]$.
The second point of \Cref{prop:cf} implies that the subset $\cal C'$ of $\cal C$ corresponding to $\Phi_\at(v)+[n_0,+\infty[$ satisfies $\cal C'\subset D_\at$.\footnote{The restriction of $\Phi_\at$ to the union $W$ of the principal chessboard boxes of $A$ and their common boundary (minus endpoints) maps $W$ univalently to the complement $V$ of $\Phi_\at(v) + ]-\infty,-1]$. The normalized attracting Fatou coordinate $\Phi$ of \Cref{prop:cf} necessarily coincides with $\Phi_\at$ on some smaller petal. Because of the shape of $U=\Phi(D_\at)$, we necessarily have $U\subset V$, for otherwise one can prove there would be a critical point of $f$ in $U$, leading to a contradiction with injectivity of $\Phi$ claimed in the second point of \Cref{prop:cf}. It follows that $D_\at\subset W$, and then that $\Phi = \Phi_\at$ on $D_\at$.}


We will also require $r_0<|v|$.
Then the set $\cal C'\subset D_\at$ does not cross the circle of equation $|z|=|v|$.
Now $\cal C$ is the union of $\cal C'$ and of a connected component of the preimage by $\Phi_\at$ of the segment $\Phi_\at(v)+[0,n_0]$.
As $f$ varies in $\cal F$, the maps $\Phi_\at-\Phi_\at(v)$ all have an inverse branch defined on a common open connected neighborhood $V$ of the segment $S=[0,n_0]$, mapping $0=\Phi_\at(v)-\Phi_\at(v)$ back to $v$.
This family is normal.\footnote{There are many reasons for this;  for instance one can use continuity of $f\mapsto \Phi_\at[f]$ together with compactness of $\cal F$; or remark that it is $1$-Lipschitz, hence equicontinuous, from the hyperbolic norm on the chosen neighborhood $V$ of the segment $S$ to the metric $|dz|/4|z|$, because it maps in the simply connected set $A$ that avoids $0$ so one can use the Schwarz-Pick inequality and Koebe's one quarter theorem.}
It also avoids $0$. Take a lift $\wt{\cal C}$ of $\cal C$ by $E:z\mapsto e^{2\pi i z}$.
This is a curve starting from a preimage $\wt{v}$  of $v$ and ending at $\infty$ tangentially to a vertical line.
The part corresponding to $\cal C'$ lives in the upper half plane ``$\Im(z)>\Im(\wt v)$'' because we took $r_0<|v|$.
The rest is the image of $S$ by a normal family defined in $V$.
In particular it has bounded Euclidean length. Let $L_1$ be a bound, independent of $f\in \cal F$.

There are infinitely many connected component of $f^{-1}(\cal C)$. Consider any. It consists either in a single curve or in a union of $d$ curves starting from a common critical point of $f$. Each of these curves has a part mapped in $\cal C\setminus \cal C'$ by $f$ that has box-Euclidean length $\leq L_1$, and a part mapped to $\cal C'$ by $f$ that is completely contained in one box.
By \Cref{lem:Lc}, the union of cuboxes visited by the full curve has  
\begin{equation}\label{eq:comb_diam}
\text{combinatorial diameter}\leq \lfloor L_1\rfloor+1.
\end{equation}

The lifted immediate basin $\wt A$ contains exactly one component of $E^{-1}f^{-1}(\cal C)$ and is disjoint from all other components.
We claim that $A$ is contained in $\cal B_{\lfloor L_1\rfloor+2}$: indeed consider the union $G_1$ of the $2d-1$ cuboxes which contain the critical point in $A$. It is contained in $\cal B_1$. The immediate basin $A$ contains exactly one component of $f^{-1}(\cal C)$.
Let $G_2$ be the component containing $A$ of the complement in $U_1^*$ of the union of all other components of $f^{-1}(\cal C)$ (\Cref{fig:situ} may help).
It is enough to prove that $G_2$ is contained in $\cal B_{\lfloor L_1\rfloor+2}$.

The boundary of $G_2$ in $U_1^*$ consists in curves all of whose starting points $s$ are preimages of $v$. We claim that they all belong to $\cal B_1$. Indeed the curve $\cal C$ is isotopic in $\C$ to the straight segment from $v$ to $0$ by an isotopy that does not move its endpoints. This isotopy extends to the whole Riemann sphere into an isotopy fixing $\infty$. The singular values of $f$ are $\{0,v,\infty\}$ and thus the isotopy does not move the singular values of $f$. Hence the extended isotopy lifts by $f$ to an isotopy of $U_1$. This lifted isotopy does not move the points in $f^{-1}(v)$. Now a starting point $s$ as above can be linked to the unique critical point $c_0\in b_*$ by a path within $G_2$ (except at its starting point $s\in\partial G_2$). The lifted isotopy deforms this path into a path with the same endpoints and that is completely contained in the complement of $f^{-1}([0,v])$. The image by $f$ of the new path is contained in $\C\setminus [0,v]$ and goes from $v$ to $v$. It is homotopic to a path completely contained $|z|>1$. The homotopy lifts by $f$. Hence $s$ and $c_0$ are linked by a path contained in a cubox. Whence the claim.

Consider any point $z\in G_2$. If $z$ belongs to $f^{-1}(\cal C)$ then it belongs to the unique component of $f^{-1}(\cal C)$ in $G_2$, which is the one attached to the critical point in $A$, which belongs to $b_*$. Hence $z\in \cal B_{\lfloor L_1\rfloor+1}$ by the bound~\eqref{eq:comb_diam} above. Otherwise, $f(z)\notin \cal C$. Then $f(z)\in H$ for $H=\ov{\D}\setminus\{0\}$ or $H=\C\setminus \D$ (if $|f(z)|=1$ then either can be chosen). 
There is a path $\gamma\subset H$ from $f(z)$ to a point of $\cal C\setminus\{0\}$ (which may be its endpoint $v$). Let $b$ be the (unique) cubox containing $z$ and such that $f(b)=H$.
The path $\gamma$ lifts by $f$ to a path within $b$ from $z$ to a point $w$ in $f^{-1}(\cal C)$, and $w$ is either in $G_2$ or in $\partial G_2$.
We saw that the  component of $f^{-1}(\cal C)$ that $w$ belong to is attached to a point in $f^{-1}(v)$ that belongs to $\cal B_1$. By the bound~\eqref{eq:comb_diam}, we get that $b\in \cal B_{\lfloor L_1\rfloor+2}$. This ends the proof that $G_2\subset\cal B_{\lfloor L_1\rfloor+2}$.
\end{proof}

Let the combinatorial distance between two points of $\dom \wt f$ be the smallest combinatorial distance of boxes containing them. Two important facts used in the proof of the following lemma are that the chessboard graph of $\dom \wt f$ is a tree and that the boundary in $\dom\wt f$ of a $\wt f$-cubox is a connected subset of this graph.
\begin{lemma}\label{lem:am}
For any two points $w,z\in \dom \wt f$ then there is a unique shortest path $\gamma'$ from $w$ to $z$ for the box-Euclidean distance.
If $m$ denotes the combinatorial distance from $w$ to $z$ then $\gamma'$ can be cut into $\leq m+1$ connected pieces, each of which stays in some cubox.
\end{lemma}
\begin{proof}\footnote{Special thanks to Arnaud Mortier for a great help in this proof.}
Let us define a projection from $\dom \wt f$ to the chessboard graph of $\wt f$ as follows.
Recall that each $\wt f$-cubox $b$ is homeomorphically mapped by $\wt f$ to a closed half plane and that the box-Euclidean metric element is sent to the canonical Euclidean element $|dz|$ of $\C$.
The vertical projection on this half plane to its boundary is $1$-Lipschitz and can be conjugated back to a projection from $b$ to its boundary in $\dom \wt f$.
The union of all these projections for all cuboxes $b$ is easily seen to match at the boundary points and corners, and yields a projection function from $\dom\wt f$ to the chessboard graph, which is locally $1$-Lipschitz for the box-Euclidean metric (the only place where checking this claim is not trivial is at corners). In particular it is $1$-Lipschitz for path-length. 

Given any path $\gamma$ from $w$ to $z$, if the path meets the chessboard graph then the part from its first intersection with the graph to its last can be projected as above.
The new path is strictly shorter unless the part was already contained in the graph.
This part can be further simplified into an injective path within the graph, strictly shorter unless it was already injective.

If moreover both $w$ and $z$ belong to a given cubox $c$, the first and last point in the graph are also in $c$, and since the graph is a tree and the boundary of a cubox is a connected subset of this tree, the simplified part is necessarily contained in this boundary, hence the simplified path is contained in $c$. 
We have thus in particular proved that for any path that is not completely contained in $c$ there is a strictly shorter path contained in $c$. 
Hence the straight segment $\gamma''$ from $w$ to $z$ for the Euclidean structure on $c$ is the unique shortest box-Euclidean path from $w$ to $z$ within $\dom \wt f$. 
The other conclusions of the lemma are trivial in this case: $m=0$ and $\gamma''$ does not need to be cut.

In the rest of the proof of the lemma, we assume that there is no cubox containing both $w$ and $z$.

Then, given the simplification of path constructed above, it follows that the infimum of box-Euclidean lengths of paths between $w$ and $z$ is the same as the infimum over the set $\cal A$ of paths defined below, and that a path that is not in $\cal A$ cannot be minimal.
The set $\cal A$ consist in paths that are a straight box-Euclidean line from $w$ to the boundary of its box if $w$ is in the interior of a box, then an injective path within the graph, then similarly a straight box-Euclidean line to $z$ if $z$ is in the interior of a box.
From the form of $\cal A$ and the fact that the distance along the graph between two points $a$ and $b$ of the graph is a continuous function of the pair $(a,b)$, the fact that a minimal distance is reached on $\cal A$ easily follows. 
Let us sum up what we have proved so far: there is at least one shortest path, all shortest paths are in $\cal A$.

Let $I_w$ be defined as follows: if $w$ is in the graph we let $I_w=\{w\}$; otherwise we let $I_w$ be the boundary in $\dom \wt f$ of the unique cubox containing $w$. 
In the latter case, $I_w$ is an infinite curve in the graph.
The set $I_w \cap I_z$ is either empty or a point or a connected curve, of finite or infinite box-Euclidean length.

First case: $I_w \cap I_z$ is empty or a point. Then there is a unique shortest path $\gamma''$ within the graph from the set $I_w$ to the set $I_z$; we allow $\gamma''$ to be reduced to a single point to include the case when $I_w \cap I_z$ is a single point. We call $w'$ the initial point of $\gamma''$ and $z'$ the endpoint; as we explained, $z'$ may be equal to $w'$. It is also possible that $w=w'$, similarly $z=z'$ is possible. If $w\neq w'$ then there is a unique cubox containing both. Similarly for $z$ and $z'$.
In all cases, the box-straight path from $w$ to $w'$, followed by $\gamma''$, followed by the box-straight path from $z'$ to $z$ is the unique shortest path in $\cal A$, and thus the unique shortest path within $\dom \wt f$. Call it $\gamma'$.

Consider now any cubox chain $b_0$, \ldots, $b_m$ with $w\in b_0$ and $z\in b_m$.
We claim that this chain necessarily covers $\gamma'$.
Let us prove this claim.
Note that the part of $\gamma'$ from $w$ to $w'$ is contained in $b_0$ and the part from $z'$ to $z$ in $b_m$.
Each $b_m$ is path connected, so by definition of a chain, the union of the $b_m$ is path connected.
Consider path from $w'$ to $z'$ within this union. Project it on the graph and simplify it as above.
This leads to an injective path from $w'$ to $z'$ contained in the graph.
By uniqueness of injective paths in a tree, this path is equal to $\gamma''$. 
This proves the claim

The intersection of a cubox with the graph is a connected subset of this tree (it is a curve, infinite in both directions). It follows that the intersection of $\gamma'$ with a cubox is necessarily a connected portion of $\gamma'$.
Let us now split $\gamma'$ as follows: choose any cubox $b_{i}$ containing $w$, define $i_1=i$ and cut $\gamma'$ at the last point where it is contained in $b_{i_1}$. Note that the part before the cut is entirely contained in $b_{i_1}$.
If this cutpoint is not the endpoint of $\gamma'$, then a non-trivial sub-part of the path starting from $b_{i_1}$ belongs to another cubox $b_{i'}$. 
Define $i_2=i'$ and cut the remaining part of the path at the last point where it is contained in $b_{i_2}$. 
And so on. 
This process necessarily ends (because, for instance, the cut points are contained in a discrete set, because they are either branch points of the graph or the first or the last intersection of $\gamma'$ with the graph). 
So we get a finite sequence of cuboxes $b_{i_1}$, $b_{i_2}$, \ldots, $b_{i_{m'}}$ for some $m'\in \N^*$ and a splitting $\gamma'_1$, \ldots, $\gamma'_{m'}$ of $\gamma$ into connected pieces with $\gamma'_j\subset b_{i_j}$ for all $j\leq m'$. 
By construction $b_{i_{j+1}}\neq b_{i_j}$.
Now by the property that the intersection of a cubox with $\gamma'$ is necessarily connected, and the it follows that no two cuboxes $b_{i_j}$ and $b_{i_k}$ can be equal for $j\geq k+2$, for otherwise the whole part of the path between $\gamma'_{j}$ and $\gamma'_{k}$ (included) would be contained in $b_{i_k}$, contradicting the way we built the splitting.
Hence $m'\leq m+1$.


Second case: $I_w \cap I_z$ is a connected curve in the graph. Let $b$ be the unique cubox containing $w$ and $b'$ be the same for $z$. Note that $b$ and $b'$ are adjacent, and $m=1$.
The union $b\cup b'$ is connected. It consists in the interior of $b$, the interior of $b'$, the common curve, and at most four disjoint pieces of curves in the boundaries of $b$ or $b'$, attached to an end point of the common curve. Because the graph is a tree, all paths in $\cal A$ are contained in $b\cup b'$ and all paths in $\cal A$ must meet the common curve, possibly at an end thereof. It follows that the shortest path in $\cal A$ from $w$ to $z$ is a straight segment to a point in the common curve, followed by a straight segment. We have thus cut the shortest path in two pieces satisfying the conclusion of the lemma, since $m+1=2$.
\end{proof}

Let us now go back to the situation we were studying: recall $L(\epsilon)$ was defined at the beginning of \Cref{subsec:part1}; for convenience we denote $L=L(\epsilon)$; we had a path $\gamma$ of $A$-length at most $c_2+L$, starting from Mudba and going to some point $z_0$.
We are ready to prove that:
\begin{equation}\label{eq:log}
 d_{U_1^*}(\gamma(0),\gamma(1)) \leq c'_6+c_6\log (1+L).
\end{equation}

Recall that there is a special cubox $b_*$ that is a punctured neighborhood of the origin. Note that $b_*$ is the only cubox that has a unique lift by $E$, which we denote $\wt b_*$.
Let $\wt \gamma$ be a lift of $\gamma$ by $E$ and $\gamma_2 = \wt f\circ \wt \gamma$. Then $\gamma_2$ is also a lift of $f\circ \gamma$ by $E$.
The $U_1^*$-hyperbolic length of $\gamma$ is equal to the $\dom(\wt f)$-hyperbolic length of $\wt\gamma$.
The box-Euclidean length of $\wt \gamma$ is equal to Euclidean length of $\gamma_2$ and is thus $\leq 2(c_2+L)$ by Equation~\eqref{eq:elg2}.
By \Cref{lem:bddchainlenght}, the path $\gamma$ is contained in $\cal B_{c_4}$.
There is thus for any $t\in[0,1]$ a chain of cuboxes of length at most $c_4$ from some cubox containing $\gamma(t)$ to $b_*$. 
This chain lifts by $E$ into a chain of cuboxes from $\wt \gamma(t)$ to $\wt b_*$.
Applying this to $t=0$ and $t=1$, we get that the combinatorial distance\footnote{notion defined just before \Cref{lem:am}} from $\wt \gamma(0)$ to $\wt \gamma(1)$ is $\leq 2c_4$.
Consider the path $\gamma'$ provided by \Cref{lem:am}, from $\wt \gamma(0)$ to $\wt \gamma(1)$, 
of box-Euclidean length at most that of $\wt \gamma$, and consisting in $p\leq 2c_4+1$ parts $\gamma'_i$ each contained in some cubox. Denote $L_i$ the box-Euclidean length of $\gamma'_i$. Then $\sum_{i=1}^{p} L_i \leq 2(c_2+L)$ and in particular $L_i\leq 2(c_2+L)$.
By \Cref{lem:arcs}, the endpoints of $\gamma'_i$ sit at $U_1^*$-hyperbolic distance $\leq c'_5+\log(1+c_5 L_i)$ from each other. Thus, putting it all together:
\bEA
  d_{U_1^*}(\gamma(0),\gamma(1)) & \leq & d_{\dom\wt f}(\wt \gamma(0),\wt \gamma(1))\\
  & \leq & \sum_{i=1}^{p} c'_5 + \log (1+c_5 L_i)\\
  & \leq & p c'_5 +  p\log ( 1+ 2 c_5 (L+c_2)) \\
  & \leq & (2c_4+1) c'_5 + (2c_4+1) \log \left( 1+2 c_5 L+2c_5c_2\right) \\
  & \leq & c'_6 + c_6 \log \left( 1+L \right)
\eEA
for some constants $c_6$, $c_6'$ that depend only on $c_2$, $c_4$, $c_5$ and $c_5'$,
which proves \eqref{eq:log}.

Because of the inclusion $U_1^*\subset U_1$, the $U_1$-hyperbolic distance between $\gamma(0)$ and $\gamma(1)$ will be even shorter. Using \Cref{lem:mudba}, we get that $z_0$ belongs to the $U_1$-hyperbolic ball of center $0$ and radius \[L'=c_7+c'_6+c_6\log(1+L).\]
Hence the set $C$ object of \Cref{prop:part:1}, which we are proving, is contained in this hyperbolic ball (see the discussion after \Cref{lem:mudba}).
Recall that $L=L(\epsilon)=\tanh^{-1}(1-\epsilon)$. Introduce $\epsilon'\in\,]0,1[$ such that $\tanh^{-1}
(1-\epsilon') = L'$. Then
\[ C \subset\phi_1(B(0,1-\epsilon')).
\]
Now $\epsilon'=2/(e^{2L'}+1) \geq e^{-2L'}$ and $L'=c_7+c'_6+c_6\log(1+L)$ and $L \leq c_1+ \frac12\log\frac1\epsilon$ so $L'\leq c'_8+c_8\log(1+\log(1/\epsilon))$, thus
\begin{equation}\label{eq:ee}
\framebox{$\ds \log \frac{1}{\epsilon'} \leq c'_9 + c_9 \log\left(1+\log\frac{1}{\epsilon}\right)$}
\end{equation}
In particular, as $\epsilon\tend 0$, $\epsilon'$ also tends to $0$ but remains much bigger than $\epsilon$. 
This proves \Cref{prop:part:1}.

\subsection{Step 2, I: Perturbation argument.}\label{subsec:part2:1}

Let us recall the notations introduced in \Cref{subsec:outline}:
\[\cal F = \setof{\cal R[B_d] \circ \phi^{-1}}{\phi:\D\to \C \text{ is univalent and } \phi(z)=z+\cal O(z^2)}
\]
and
\[\cal F_\epsilon = \setof{\cal R[B_d] \circ \phi^{-1}}{\phi:B(0,1-\epsilon)\to \C \text{ is univalent and } \phi(z)=z+\cal O(z^2)}
\]
where $\cal R[B_d]$ is the (upper) parabolic renormalization of the Blaschke product, normalized to be defined on the unit disk. In particular,
\[\cal F_0 = \cal F
.\]
Last, for $X\subset [0,1]$, we will denote 
\[ \cal F_X =\bigcup_{x\in X} \cal F_x
.\]

\subsubsection{An interpolation}\label{subsub:outline}

Let $\epsilon_1>0$ and $f\in\cal F_{\epsilon_1}$:
\[ f=\cal R[B_d]\circ\wt\phi^{-1}\]
For convenience, we will denote
\[r'=1-\epsilon_1
\]
and $\phi(z) = \frac{1}{r'}\wt\phi(r' z)$. Then $\phi\in\sch$ (the class of Schlicht maps) and 
\[f(z) = \cal R[B_d](r'\phi^{-1}(z/r'))
.\]
Let $U=\phi(\D)$.
We will interpolate smoothly between $f$, which belongs to $\cal F_{\epsilon_1}$, and an element of $\cal F$ as follows:
for $t\in[0,1[$, let
\[\phi_t(z)=r_t\phi(z/r_t)\,\text{ with }\ r_t = 1-t.\]
Then the map $\phi_t$ is an isomorphism from $B(0,r_t)$ to $r_t U$. Let
\[ f_t(z) = \cal R[B_d] \circ \phi_t^{-1}: r_t U \to \C.
\]
Then
\[f_{\epsilon_1} = f,\]
and\nomenclature[f0]{$f_0$}{an element of $\cal F$\nomrefpage}
\[f_t\in\cal F_{t} \text{ thus }f_0 \in \cal F.\]
\nomenclature[ft]{$f_t$}{a deformation of $f_0$, element of $\cal F_t$\nomrefpage}%
In the sequel, we will start from knowledge about $f_0$ and transfer it to $f_{\epsilon_1}$, by continuously increasing $t$ from $0$ to $\epsilon_1$.

Using the language of structures that we introduced in \Cref{subsec:structeq}, let us stress that maps in $\cal F_t$ are all $(I,\wh{\C})$-structurally equivalent ($I$ being a singleton and the origin being the marked point). For $t'>t$, the structure of maps in $\cal F_{t'}$ is a sub-structure of that of maps in $\cal F_{t}$.

\begin{remark}Though, for $t'>t$, $f_{t'}$ is a sub-structure of $f_t$, it is very unlikely that the map $f_{t'}$ would be conjugate to a restriction of $f_{t}$.
\end{remark}

Let us show a non-commuting diagram that the reader may find useful in order to follow the arguments.
\[\xymatrix@R=30pt@C=30pt{\ar[d]_{\ds \phi_0^{-1}} & \ar[l]_{\ds \cdot / r_t} \\ \ar[r]_{\ds r_t \times \cdot} & \ar[u]_{\ds \cal R[B_d]}}\]
The map $f_t$ consists in turning once around this diagram, starting from the upper right corner.\footnote{It may at first seem to be better to start from the upper left corner, since the corresponding composition has a domain $U$ that does not depend on $t$. However, when we iterate these maps, we basically go in round circles along a non-commuting diagram again and again, and the author thinks that it would not simplify the proof that much.}

\subsubsection{About the critical value}\label{subsub:harcrit}

Let $T_0$\nomenclature[T0]{$T_0$}{for $f\in\cal F_{[0,T_0[}$ have a (unique) critical value\nomrefpage} be one minus the absolute value of the critical point of $\cal R[B_d]$ that is closest to $0$. Then for all $t \in [0,T_0[$, maps in $\cal F_t$ have a unique critical value.

\begin{lemma}\label{lem:keep}
There exists\nomenclature[T1p]{$T_1'$}{for $f\in\cal F_{[0,T_1']}$, the critical value is attracted to $0$\nomrefpage} $T_1'\in\,]0,T_0[$ such for all maps $f\in\cal F_{[0,T_1']}$,  the critical value is attracted to $0$.
\end{lemma}
\begin{proof}By Fatou's theorem (\Cref{thm:fat}), this is the case for all maps in $\cal F_0$. The existence of $T'_1$ then follows from compactness of $\cal F_0$ and the fact that for a parabolic map with one petal attracting a given point, nearby parabolic maps will attract nearby points.
\end{proof}

A consequence of the uniqueness of the critical value is that the extended attracting Fatou coordinate $\Phi_\at[f_t]$ has a set of critical values contained in $\setof{v'-n}{n>0}$
where $v'=\Phi_\at[f_t](v)$ and $v$ is the critical value of $f_t$.
Unlike the case $t=0$, when $t>0$ the map $\Phi_\at[f_t]$ probably has a big set of asymptotic values (it is likely that it contains curves).

\subsection{Step 2, II: Following fibers.}\label{subsec:part2:2}

\subsubsection{A motion of the fibers of the Fatou coordinates and of the renormalized map}\label{subsub:motion}

The point of view outlined in \Cref{subsub:outline} can be reversed and we may start from any map $f_0 = \cal R[B_d]\circ \phi_0^{-1} \in\cal F$, which has the full structure of $\cal R[B_d]$ and perturb it into the map $f_t \in \cal F_t$ as before, which has less and less structure as $t\in[0,1[$ increases.
Let us recall how $f_t$ is defined:
\[f_t(z)=\cal R[B_d]\circ \phi_t^{-1}\ \text{ with }\ \phi_t(z)=r_t\phi_0(z/r_t)\ \text{ and }\ r_t = 1-t
.\]
Studying the survival of (part of) the structure of the parabolic renormalization $\cal R[f_t]$ as $t$ increases means following fibers of $\cal R[f_t]$.

Recall that $\cal R[f_t]$ is defined by
\[\cal (a^{-1} \circ \cal R[f_t] \circ b) \circ E = E \circ (\Phi_\at[f_t]\circ\Psi_\rep[f_t])\big|_{W_t}
\]
with $E(z)=e^{2\pi iz}$, $W_t$ is some domain, and $a$ and $b$ are linear maps that depend on $f_t$ and on normalization conventions. Recall that we chose to normalize Fatou coordinates by their expansion at infinity, and to normalize $\cal R[f_t]$ by fixing its critical value. See \Cref{subsec:nor} for more details.

To lighten the expressions, let us abbreviate $R_t=\cal R[f_t]$ and introduce extended Fatou coordinates $\Phi_t$ and $\Psi_t$ of $f_t$, normalized differently from $\Phi_\at[f_t]$ and $\Psi_\rep[f_t]$, and so that
\[ R_t \circ E = E \circ \Phi_t \circ \Psi_t \big|_{W_t}.
\]
We defined in  \Cref{subsub:harcrit} two constants $T_0$ and $T_1'< T_0$ such that:
\begin{itemize}
\item For $t\leq T_0$, for all $f\in \cal F$, $f_t$ has a unique critical value.
Let us denote it by $v_t$.
\item For $t\leq T_1'$, this point $v_t$ is in the domain of definition of $\Phi_\at[f_t]$.
\end{itemize}
Let
\[\Phi_t(z)=\Phi_\at[f_t](z)+\beta_t
\]
where $\beta_t = \sigma_d -\Phi_\at[f_t](v_t)$, so that $\Phi_t(v_t)$ does not depend on $t$ and where $\sigma_d$ is a constant that depends only on $d$ and is chosen so that $E(\sigma_d)$ is the critical value of $\cal R[b_d]$.\footnote{One can for instance take $\sigma_d = i\pi\gamma[B_d]$ but we will not use this fact.}
\nomenclature[Wz3]{$\Phi_t$}{$\Phi_t=\Phi_\at[f_t]+\beta_t$ with $\beta_t$ a constant so that the critical value of $\Phi_t$ is independent of $t$\nomrefpage}%
\nomenclature[Wz4]{$\Psi_t$}{$\Psi_t(z)=\Psi_\rep[f_t](z-\beta'_t)$ for $\beta'_t=\beta_t- i\pi \gamma[f_t]$\nomrefpage\ with $\gamma$ the iterative residue}%
\nomenclature[bt]{$\beta_t$}{constant so that $\Phi_t(z)=\Phi_\at[f_t](z)+\beta_t$ has a critical value independent of $t$; $\beta_t = \sigma_d-\Phi_\at[f_t](v_t)$\nomrefpage}%
\nomenclature[sigmad]{$\sigma_d$}{complex number chosen so that $E(\sigma_d)$ is the critical value of $\cal R[B_d]$\nomrefpage}%
For the repelling inverse Fatou coordinate (whose normalization is less important) we let
\[\Psi_t(z)=\Psi_\rep[f_t](z-\beta'_t),
\]
for $\beta'_t=\beta_t- i\pi \gamma[f_t]$ (recall $\gamma$ is the iterative residue, see \Cref{subsec:parabopt}).
Let $\Phi : (t,z)\mapsto \Phi_t(z)$, that we define on
\[\dom \Phi =\setof{(t,z)\in[0,T'_1[\,\times\C}{z\in\dom(\Phi_t)}.\]
It is an open subset of $[0,T'_1[\times\C$, and $\Phi$ is a continuous function of $(t,z)$ by \Cref{prop:conti} (in fact, it is analytic, see~\cite{T}).
Similarly, let
\[R: \left\{ \begin{array}{rcl} \dom R &\to& \C \\ (t,z)  &\mapsto& R_t(z) \end{array}\right.
\]
The domain of $R$ is an open subset of $[0,T'_1[\times\C$ and $R$ is continuous, analytic w.r.t.\ $z$ for fixed values of $t$. (It is also analytic w.r.t.\ $(t,z)$ but we will not use this fact.)

The critical values of $\Phi_t$ and $R_t$ do not move when $t$ varies (even when some critical points vanish). It has the following consequence:

\begin{proposition}[following part of the structure]\label{prop:fibers}
Let $F=\Phi$ or $F=R$. 
Then
\begin{itemize}
\item (\Cref{lem:loc_triv}) fibers of $\Phi$ form a foliation that is locally parallelizable over the first coordinate.
\end{itemize}
It follows that there exists a function $\tau : \dom F_0 \to \,]0,T'_1]$ (survival time) and function $\zeta(t,z)$ (fiber follower)
such that:
\begin{itemize}
\item $\dom \zeta = \setof{(t,z)\in [0,T'_1[\,\times \dom(F_0) }{t\in[0,\tau(z)[\,}$
\item the map $\tau$ is lower semi continuous, i.e.\ for all $t\in[0,T'_1[\,$, the set $U_t = \tau^{-1}(\,]t,T'_1]) \subset \C$ is open
\item the above two points imply that $\dom \zeta$ is an open subset of $[0,T'_1[\,\times\C$ and $\dom \zeta = \setof{(t,z)\in[0,T'_1[\,\times\C}{z\in U_t}$
\item the map $\zeta$ is continuous
\item for each fixed $t\in [0,T'_1[\,$, the map $z\in U_t \mapsto \zeta(t,z)$ is holomorphic and injective
\item $\forall (t,z) \in \dom \zeta$, $F_0(z) = F_t(\zeta(t,z))$, i.e.\ the map $t\in [0,\tau(z)[\mapsto (t,\zeta(t,z))$ follows a fiber of $F$
\item (maximality and uniqueness) consider any continuous map following a fiber of $F$ as $t$ varies from $0$ to some $t_0$, starting from $(0,z)\in\dom F$; then $\tau(z)> t_0$ and the continuous map must coincide with $t\in[0,t_0]\mapsto \zeta(t,z)$.
\end{itemize}
\end{proposition}

The rest of the present section (\Cref{subsub:motion}) is devoted to the proof of the above proposition.
The proof is written for $\Phi$ but is the same, word for word, for $R$.

\begin{lemma}\label{lem:cmv}
Let $(t_0,z_0)\in\dom\Phi$ and assume that $z_0$ is a critical point of $\Phi_{t_0}$. Then there exists a connected neighborhood $I$ of $t_0$ in $[0,T'_1[$, and $r_0>0$ such that for all $t\in I$, $\Phi_t$ has a unique critical point in $B(z_0,r_0)$, it moves continuously with $t$ and its multiplicity does not change.
\end{lemma}
\begin{proof}
We apply Hurwitz's theorem\footnote{There seems to be several statements called Hurwitz's theorem. We are referring to the following: for a sequence of holomorphic functions $f_n$ converging uniformly on compact subsets of an open subset $U$ of $\C$, call its limit $f$. If $D$ is a disk compactly contained in $U$ and $f$ does not vanish on the boundary of $D$ then for all $n$ big enough, $f$ and $f_n$ have the same number of zeroes in $D$, counted with multiplicity.} to $\Phi_t'$ and to $\Phi_t$ (note that $\Phi'_t$ also depends continuously on $t$, by Cauchy's estimates): let $u=\Phi_{t_0}(z_0)$. Take $r_0>0$ small enough so that $z_0$ is the only critical point of $\Phi_{t_0}$ in $\ov B:=\ov{B}(z_0,r_0)$, the only solution of $\Phi_{t_0}(z) = u$ in $\ov B$, and such that $\Phi_{t_0}$ maps this disk in $B(u,1/2)$; there exists $\epsilon_0$ such that for all $t\in[0,T'_1[$ with $|t-t_0|<\epsilon_0$, $\Phi_t$ is defined on $\ov B$ and maps it in $B(u,1/2)$; then by Hurwitz's theorem, there exists $0<\epsilon<\epsilon_0$ such that for $|t-t_0|<\epsilon$, $\Phi_t-u$ has $d-1$ critical points counted with multiplicity in $B$ and $d$ roots in $B$.
Now recall we normalized the maps $\Phi_t$ so that all critical values belong to $\Z+u$ and $u$ does not depend on $t$. Since $\Phi_t(\ov B) \subset B(u,1/2)$, this implies that all critical points of $\Phi_t$ in $\ov B$ map to $u$. 
Thus the sum of local degrees of $\Phi_t$ at preimages of $u$ in $\ov{B}$ equals $d$, and the sum of local degrees minus one equals $d-1$: there is exactly one preimage of $u$, thus exactly one critical point. Moreover, its local degree is $d$, thus its multiplicity is constant.
Continuous dependence is a classical application of Hurwitz's theorem and is left to the reader.\footnote{There is a more direct proof, with Hurwitz's theorem used only at the end to deduce continuity.
From the fact that $z_0$ is in a parabolic basin and that all critical points of $f_t$ map to the same point, it follows that the orbit of $z_0$ hits the set of critical points only once. Then one uses that $\Phi_\at = -n+\Phi_\at\circ f^n$, and that $\Phi_\at$ is injective in the petal $D_\at[f_t]$ and that the latter moves continuously with $t$.
Similar arguments can be carried out for $R$ in place of $\Phi$.}
\end{proof}

Now consider the fibers of $\Phi$: $X_c = \setof{(t,z)\in\dom\Phi}{\Phi(t,z)=c}$. They form a collection of disjoint closed subsets of $\dom\Phi$. We will prove that this collection is a locally trivial foliation, in the following precise sense:

\begin{lemma}[local trivialization]\label{lem:loc_triv}
 All $(t_0,z_0)\in\dom\Phi$ has an open neighborhood $V$ in $\dom\Phi$ on which a change of variable $U: V\to V'\underset{\text{open}}{\subset} [0,T'_1[\times \C$ of the form
\[U: (t,z) \mapsto (t,u(t,z))\]
is defined, 
\begin{enumerate}
\item $U$ is a homeomorphism to $V'$,
\item for all $t$, $z\mapsto u(t,z)$ is holomorphic,
\item $\forall c\in\C$, $U(X_c)$ is the intersection of a horizontal with $V'$: it is of the form $V' \cap ([0,T'_1[\times \{w\})$ for some $w\in \C$.
\end{enumerate}
\end{lemma}
\begin{proof}
Case 1: $z_0$ is not a critical point of $\Phi_{t_0}$. It is an application of Hurwitz's theorem. Since the family $\Phi_t$ depends continuously on $t$ and $\Phi_{t_0}$ is not locally constant near $z_0$, one can deduce from Hurwitz's theorem that the map $U=\Phi$ itself, restricted to an appropriate neighborhood $V$, will be a local trivialization. Details are left to the reader.
\\ Case 2: $z_0$ is a critical point of $\Phi_{t_0}$. A consequence of \Cref{lem:cmv}, is that we can factor $\Phi_t(z) = (z-c_t)^d h_t(z)$ where $h_t(z)$ is a holomorphic function in $z$, continuous in $(t,z)$, defined locally and non-vanishing. The map $g(t,z)= \sqrt[d]{h_t(z)}$ is defined locally, and we leave to the reader to check that the map $(t,z)\mapsto(t,(z-c_t)g(t,z))$ is a local trivialization.
\end{proof}

Hence connected components of fibers are graphs of continuous functions $t\mapsto z(t)$ defined on connected open subsets of $[0,T'_1[$.
Now, given any $z\in\dom(\Phi_0)$, we follow its fiber as $t$ increases from $0$ as long as possible: this gives a maximal continuous function $t\in[0,\tau(z)[\,\mapsto \zeta_z(t)$ such that $\zeta_z(0)=z$ and $\Phi_t(\zeta_z(t))$ is constant.
The real number $\tau(z)$ belongs to $]0,T'_1]$.
Uniqueness and maximality (last point of \Cref{prop:fibers}) follow easily.
In the lemma below, the point $\zeta_z(t)$ is denoted
\[ z\ag{t}
.\]

\begin{lemma}\label{lem:sch}
The following holds:
\begin{enumerate}
\item The function $\tau$ is lower semi-continuous, i.e.\ for all $t\in[0,T'_1[$, the set $U_t =\tau^{-1}(]t,T'_1]) \subset \C$ is open.
\item On $U_t$, the function $z\mapsto z\ag{t}$ is holomorphic.
\end{enumerate}
\end{lemma}
\begin{proof} For a given $z\in U_t$, since $t<\tau(z)$, cover the compact set $[0,t]$ by open subsets on which there is a local trivialization of the fiber $z$ belongs to. Extract a finite cover. From it, one can build a trivialization like in the previous lemma, but in a whole neighborhood of $z\ag{[0,t]}$ relative to $[0,t]\times\C$. The lemma follows.
\end{proof}

This ends the proof of the \Cref{prop:fibers}.

\subsubsection{Objectives}

Let $f\in\cal F$ and denote by $\tau_R[f]$ the $\tau$ function corresponding to $R$ in \Cref{prop:fibers}: i.e.\ $\tau_R[f](z)$ is the time up to which the fiber of $(t,z)\mapsto R_t(z)$ that contains $(0,z)$ can be followed.
Recall that $\cal R[f_t]$ denotes the parabolic renormalization of $f_t$, normalized so that the critical value does not move as $t$ varies, and recall that $f_t$ is a specific perturbation of $f_0=f$.
Consider the parabolic renormalization $\cal R[f_0]$ of $f_0$. 
\begin{lemma}
If $\forall z\in\sub{\dom(\cal R[f_0])}{(1-\epsilon_1)}$, $\tau_R[f_0](z)>\epsilon_0$  then $\cal R[f_{\epsilon_0}]$ has a restriction that belongs to $\cal F_{\epsilon_1}$.
\end{lemma}
\begin{proof}
The map $\cal R[f_0]$ belongs to $\cal F$, thus it can be written as
$\cal R[f_0]=\cal R[B_d]\circ \phi_2^{-1}$ where $\phi_2:\D\to\C$ is univalent and $\phi_2(z)=z+\cal O(z^2)$.
By hypothesis, the set $U_{\epsilon_0}$ contains $\sub{\dom(\cal R[f_0])}{(1-\epsilon_1)} = \phi_2(B(0,1-\epsilon_1))$ (the sets $U_t$ were defined in \Cref{prop:fibers,lem:sch}). According to \Cref{prop:fibers}, the map $\zeta^t : z\in U_t\mapsto \zeta(t,z)$ is a holomorphic bijection to its image, and $\cal R[f_t](\zeta^t(z))=\cal R[f_0](z)$ holds on $U_t$.
Apply this to $t=\epsilon_0$: let $V = \zeta^{\epsilon_0}(\phi_2(B(0,1-\epsilon_1)))$, then $\zeta^{\epsilon_0}\circ\phi_2$ is a structural equivalence, with $0$ as a marked point, between the restriction of $\cal R[f_{\epsilon_0}]$ to $V$ and the restriction of $\cal R[B_d]$ to $B(0,1-\epsilon_1)$.
\end{proof}

So the Main theorem will be proved if we can prove the following claim:

\begin{assertion}[survival of fibers of $R$]\label{ass:1}
There exists a pair $\epsilon_1<\epsilon_0$ with $\epsilon_0<T_1'$
such that for all $f_0\in\cal F$, 
for all $z\in \sub{\dom(\cal R[f_0])}{(1-\epsilon_1)}$,
\[\tau_R[f_0](z)>\epsilon_0
.\]
\end{assertion}

The constant $T'_1$ was defined in \Cref{subsub:harcrit}. We will in fact prove more: for all $\epsilon_0$ small enough, there exists $\epsilon_1<\epsilon_0$ such that the conclusion of the assertion holds. Better: we can take $\epsilon_1 \ll \epsilon_0$ (see details in \Cref{subsec:ccl}).

\subsubsection{Restatement of the objectives}\label{subsub:restate}

Let $\epsilon>0$ and consider some 
\[z\in \sub{\dom(\cal R[f_0])}{(1-\epsilon)}
.\]
The map $R_t = \cal R[f_t]$ is the semi-conjugate by $E$ of the composition $\Phi_t \circ \Psi_t$, but it can also be viewed differently: recall that the extended Fatou coordinates $\Phi_t$ and extended inverse $\Psi_t$ are defined via iteration of $f_t$, using bijective Fatou coordinates in petals as a starting point. Let $\cal P_\rep$ be a repelling petal and $\Phi_\rep$ be a repelling Fatou coordinate such that $\Psi_t = \Phi_\rep^{-1}$ holds on $\Phi_\rep(\cal P_\rep)$.
The value $R_t(z)$ thus decomposes as follows (see \Cref{fig:decomp} in \Cref{subsec:parabopt}):
\[ R_t(z)=E(\Phi_t(f_t^{m_0}(\Psi_t(u))))\]
where $E(z)=e^{2\pi i z}$, $u\in E^{-1}(z) \cap \Phi_\rep(\cal P_\rep)$ and $m_0=m_0(z)\in\N$ is chosen so that $f_t^{m_0}(\Psi_t(u))$ belongs to the attracting petal. 
Let us now focus on the initial situation, at $t=0$: consider the $f_0$ bilateral orbit
\[(n\in\Z)\quad \omega_n:=\Psi_0(u+n).\]
It depends on $z$ and on the choice of $u\in E^{-1}(z) \cap \Phi_\rep(\cal P_\rep)$.
Interestingly, if one chooses another $u\in E^{-1}(z) \cap \Phi_\rep(\cal P_\rep)$, we get the same orbit, but with the index $n$ shifted.
According to the first step, if $z\in \sub{\dom(\cal R[f_0])}{(1-\epsilon)}
$ then the orbit $\omega_n$ is contained in $\sub{\dom(f_0)}{(1-\epsilon')}=\phi_0(B(0,1-\epsilon'))$ with $\epsilon' \gg \epsilon$:
\[\forall n\in\Z,\ \omega_n \in \sub{\dom(f_0)}{(1-\epsilon')}
.\ \footnote{Let us again insist on our interpretation of this fact, that is the central idea of the whole machinery: given $f_0\in\cal F$, the restriction of its renormalized map $R_0$ to a map with substructure $\cal F_\epsilon$, can be defined using a restriction of the map $f_0$ that has structure $\cal F_{\epsilon'}$, i.e.\ much less structure. If all maps with structure $\cal F_{\epsilon'}$ were restrictions of maps in $\cal F$ we would be done (the main theorem would follow at once), but this is of course not the case, and this is the reason why we introduced the interpolation $f_t$. The idea is then the following: since $\epsilon\ll\epsilon'$, for $t$ at most $\epsilon$ or just slightly bigger, the map $f_t$ will be extremely close to $f_0$ on a set slightly bigger than $\sub{\dom(f_0)}{(1-\epsilon')}$. The task is then to check that this is close enough so that the fibers attached to the orbits $\omega_n$ survive and thus the $\cal F_\epsilon$-structure of the parabolic renormalization survives.}
\]
Let us now denote $\tau_\Phi[f]$ the $\tau$ function corresponding to $\Phi$ in \Cref{prop:fibers}. This proposition also provides a map $(t,z)\mapsto \zeta(t,z)$, to be interpreted as a motion of $z$ as $t$ varies. For convenience, in the sequel we will use the notation
\[z\ag t =\zeta(t,z)
.\]
\begin{lemma}[the motion is compatible with the dynamics]\label{lem:compat}
$\forall z\in U_0$, $\tau_\Phi(f_0(z))\geq \tau_\Phi(z)$ and $\forall t<\tau_\Phi(z)$, $f_t(z\ag t) = f_0(z)\ag t$.
\end{lemma}
\begin{proof}

By construction of the extended Fatou coordinates, if $(t,z)\in\dom\Phi$ then $(t,f_t(z)) \in \dom \Phi$ and $\Phi(t,f_t(z))=1+\Phi(t,z)$. By hypothesis, the graph of $t\in[0,\tau_\Phi(z))\mapsto z\ag t$ is contained in $\dom\Phi$ hence so is the graph of $t\in[0,\tau_\Phi(z))\mapsto f_t(z\ag t)$ and $\Phi(t,f_t(z\ag t)) = 1+\Phi(t,z\ag t)$, and thus remains constant as $t$ varies, by construction of the motion $z\ag t$. This means that $t\in[0,\tau_\Phi(z))\mapsto f_t(z\ag t)$ is in the unique fiber of $\Phi$ containing $f_0(z)$: hence
$f_t(z\ag t) = f_0(z)\ag t$.
\end{proof}
Now for a given $t$ consider the sequence
\[\omega_n\ag{t}
.\]
It is an orbit of $f_t$, though, depending on $t$, it may not be defined for all $n$:
\begin{lemma}\label{lem:compat2}
For all $t\in[0,T'_1[$:
\begin{itemize}
\item if $\omega_n\ag t$ is defined (i.e.\ $\tau_\Phi(\omega_n)>t$) then $\omega_{n+1}\ag t$ is defined and $\omega_{n+1}\ag t=f_t(\omega_n\ag{t})$,
\item $\omega_n\ag t$ is defined when $n$ is big enough.
\end{itemize}
\end{lemma}
\begin{proof}
Since $\omega_n\ag 0=\omega_n$ is an orbit for $f_0$: $\omega_{n+1}\ag 0 = f_0(\omega_n\ag 0)$. The first point follows from the previous lemma. Informally, the second point states that points deep enough in the attracting petal can be followed for a long time. 
Let us apply \Cref{prop:cf2} and its companion \Cref{prop:cf} to the family of maps $\cal G=\setof{f_s}{s\in[0,t]}$.
The Fatou coordinates in this proposition are normalized by the expansion. They thus differ from $\Phi_s$ by the constant $\beta_s = \sigma_d-\Phi_\at[f_t](v_t)$ of \Cref{subsub:motion}, which is bounded for $s\in[0,t]$.
Hence there is a map $\xi$, independent of $s\in[0,t]$, such that the domain of equation $\Re(z)>\xi(\Im(z))$ is contained in the image by $\Phi_s$ of the attracting petal $D_\at[f_s]$ (defined in \Cref{prop:cf}).
Choose $N_1$ so that $\omega_{N_1}\ag 0\in D_\at[f_0]$. For $N=N_1+k \geq N_1$, we have $\omega_{N}\ag 0 \in D_\at[f_0]$ and $\Phi_0(\omega_N\ag 0) = \Phi_0(\omega_{N_1})+k$ hence there is some $N_2\geq N_1$ such that for all $n\geq N_2$, $\omega_n\ag 0$ is in the domain of equation $\Re(z)>\xi(\Im(z))$. Let us call $\Psi_{\at,s}$ the inverse of the restriction of $\Phi_s$ to the petal.
The function $s\mapsto \Psi_{\at,s}(\Phi_0(\omega_n\ag 0))$ then defines a motion of $\omega_n\ag 0$ within a fiber of $\Phi$, whence the conclusion by the uniqueness point of \Cref{prop:fibers}.
\end{proof}

The sequence $\omega_n\ag t$ is thus defined either for all $n\in\Z$ or for all $n \geq N\in\Z$, where $N$ depends both on $t$ and on the orbit $\omega_n=\omega_n\ag 0$.
\Cref{prop:cfr} provides a repelling petal $D_\rep[f_t]$ of diameter $r_0$ that varies continuously with $f_t$. Here $r_0$ can be any small enough constant independent of $f_t$.
Assertion~\ref{ass:1}, and thus the main theorem, will follow from:

\newcommand{\ass}{%
There exists $r_0'<r_0$ and a pair $\epsilon_1<\epsilon_0$ with $\epsilon_0<T_1'$ such that for all $f_0\in\cal F$,
for all $z\in\sub{\dom(\cal R[f_0])}{(1-\epsilon_1)}$, if we consider the orbit $\omega_n$ associated to $z$, then
\begin{itemize}
\item for all $n\in\Z$
\[\tau_\Phi[f_0](\omega_n)>\epsilon_0
,\]
\item there exists $M\in\Z$ such that $(t\leq \epsilon_0$ and $n\leq M) \implies \omega_n\ag{t}\in D_\rep[f_t](r_0')$.
\end{itemize}%
}
\begin{assertion}[survival of orbits as fibers of $\Phi$, and control]\label{ass:2}
\ass
\end{assertion}

Indeed, 
let $\Phi_{+,t}$ be the repelling Fatou coordinates on $D_\rep[f_t]$ such that $\Psi_t\circ\Phi_{+,t}(z)=z$ holds on $D_\rep[f_t]$.
It depends continuously on $t$. 
Consider a point $z\in \sub{\dom(R_0)}{1-\epsilon_1}$ and the $f_0$-orbit $\omega_n$ associated to $z$.
Let then $z(t)=E(\Phi_{+,t}(\omega_M\ag t))$.
Then $z(t)\in \dom R_t$ and $\forall n\geq M$, $R_t(z(t))=E(\Phi_{+,t}(\omega_M\ag t)) = E(\Phi_t(\omega_n \ag t) + M-n) = E(\Phi_t(\omega_n \ag t)) = E(\Phi_0(\omega_n\ag 0))$ (the last equality because we follow a fiber of $\Phi$), i.e.\ $R_t(z(t))$ is constant as $t$ varies.
Since $z(0)=z\ag 0$, we have followed the $R$-fiber associated to $z$: $z(t)=z\ag t$.
In particular $\tau_R(z)>\epsilon_0$.

Again, we will get slightly stronger information on the valid pairs $(\epsilon_0,\epsilon_1)$ for Assertion~\ref{ass:2}, see \Cref{subsec:ccl}.

\subsection{Step 2, III: Survival of fibers.}\label{subsec:part2:3}

In this section, we will prove the following proposition (see the paragraph just before Assertion~\ref{ass:2} for information about the constant $r_0$):
\begin{proposition}\label{prop:surv}
There exists $K>0$, $r_0'<r_0$ and $\epsilon'_0$ such that
for all $\epsilon'<\epsilon'_0$, for all $f_0\in\cal F_0$, for all $f_0$-orbit $\omega_n$ indexed by $I=\Z$ that tends to $0$ in the future (in an attracting petal) and in the past (in a repelling petal), if the orbit $(\omega_n)$ is completely contained in $\sub{\dom(f)}{(1-\epsilon')}$ then its survival time is at least $\epsilon'/K$:
\[ \ds \forall n\in\Z,\ \tau_\Phi(\omega_n)> \epsilon'/K
.\]
Moreover\footnote{This constant $M$ will of course \emph{not} be independent of the orbit $(\omega_n)$.} there is some $M\in\Z$ such that $\forall n\in \Z$ with $n\leq M$ and $\forall t\leq \epsilon'/K$, $\omega_n\ag{t}\in D_\rep[f_t](r_0')$.
\end{proposition}

Here we do \emph{not} need to assume that $\epsilon'$ is related to some $\epsilon>0$ like in \Cref{prop:part:1}.

\subsubsection{Local orbits}

We first consider those orbits that stay near the parabolic point, and prove their survival for some uniform time.

\begin{lemma}[Survival of local orbits]\label{lem:survlocorb}
 For all $T_3<T'_1$ there exists $r_1>0$ such that for all $f_0\in\cal F_0$ and for all $f_0$-orbit $\omega_n$ indexed by $I=\Z$ or $I=\N$, if the sequence $(\omega_n)$ is contained in $B(0,r_1)$, then
\begin{itemize}
\item for all $n\in I$, $\tau_\Phi[f_0](\omega_n)> T_3$,
\item if $I=\Z$, then there exists $N\in\Z$ such that $\forall n\in\Z$ with $n\leq N$ and $\forall t\in[0,T_3]\,$, $\omega_n\ag{t}\in D_\rep[f_t](r_0)$.
\end{itemize}
\end{lemma}
\begin{proof} Recall the statements and notations of \Cref{prop:cf,prop:cf2} and apply them to the compact set of maps $\cal F_{[0,T_3]}$, which yields a value $r_0$. In their proofs, we introduced the right half plane $H_\at[f]$, image of the disk $D_\at[f]$ by
$z\mapsto s(z)=-1/c_f z$.
The boundary of $H_\at$ is a vertical line of abscissa $1/r_0|c_f|$. Call $R_0$ the supremum of $1/r_0|c_f|$ when $f$ varies over $\cal F_0$. The function $\Psi_\at$ was the inverse of $\Phi_\at : D_\at \to \Psi_\at(D_\at)$.
It is important to note a difference: the Fatou coordinates were normalized by they asymptotic expansion in these propositions, whereas here they are normalized using the critical value $v[f_t]$: $\Phi_t(z)=\Phi_\at[f_t](z)+\beta_t$ where $\beta_t=\beta[f_t]=\sigma_d-\Phi_\at[f_t](v[f_t])$.
Let
\[ z\mapsto s_t(z)=-1/c[f_t] z
.\]
Let $\wt\Psi_{t}=\Phi_t^{-1}$ defined on $\Phi_t(D_\at[f_t])$.
Choose any $T_3'\in\,]T_3,T_1'[$. The following three bounds are finite: 
\[ B=\ds\sup_{f\in\cal F_{[0,T_3']}} \big|c_f\big|
,\quad
B'=\sup_{f\in\cal F_{[0,T_3']}}|\beta_t| 
\quad\text{and}\quad
\Gamma = \ds\sup_{f\in\cal F_{[0,T_3']}} \big|\gamma[f]\big|
. \]
Since $B'<+\infty$, one can translate the estimates given in \Cref{prop:cf,prop:cf2} into estimates on $\wt\Psi_{t}$ and $\Phi_t$ as follows: 
\bEA
|s_t(f_t(z))-(s_t(z)+1)| & \leq & 1/4 \qquad (\forall z\in B(0,r_0))
\\
|\Phi_t(s_t^{-1}(u))-(u-\gamma\log_p u) | &\leq& M_1
\\
|s_t\circ\wt\Psi_{t}(Z)-(Z+\gamma\log_p Z) | &\leq& M_2
\\
\dom(\wt\Psi_{t}) & \supset & \setof{Z\in\C}{\Re Z>\xi(\Im Z)}
\\
\xi(y) &\underset{y\to \pm\infty}{=}& \cal O(\log|y|)
\eEA
where $s_t$, $\gamma=\gamma[f_t]$, $\Phi_t$ and $\wt\Psi_{t}$ all depend on $f_t$, but the function $\xi$ and the constants $M_1$, $M_2$ are independent of $f_0$ and of $t$.
Consider now a real number $a>R_0$ and the sector $S\subset H_\at$ defined by $\on{arg}(z-a)<\pi/3$.
By the first estimate above, $s_t^{-1}(S)$ is stable by $f_t$.
By the other estimates, if $a$ is big enough, for all $f_t\in \cal F_{[0,T_3']}$, for all $z\neq 0$, if $s_0( z) \in S$, then 
$\wt\Psi_t(\Phi_0(z))$ is defined. It follows a fiber of $\Phi$ hence by uniqueness in \Cref{prop:fibers},
$\tau_\Phi(z) \geq T_3'$ and
\[z\ag t = \wt\Psi_t(\Phi_0(z))
.\]
Using the estimate above on $\wt\Psi_t$, we get
$\forall t\in[0,T_3'[\,$, $s_t(z\ag{t}) \in H_\at[f_t]$ provided $a\geq A'$ for some $A'$ independent of $f_0$, $t$ and $z$.
Let
\[u(t) := s_t(z\ag{t})=s_t\circ \wt\Psi_t(\Phi_0(z))
.\]
In particular $u(0)=s_0(z)$.
We then get the following bound on the motion:
\[\big| u(t)-u(0)\big| \leq M_4\log (M_4'+|u(0)|)\]
where $M_4$ and $M_4'$ are independent of $t$, $f_0$ and $z$.
Indeed, we start from $|\log_p(x)|\leq \pi+\log |x|$ when $\log |x|>0$. We then use the estimates above to first get $|\Phi_0(z)|\leq M_1 + |u(0)| +\Gamma \pi+\Gamma \log|u(0)| $ (we can ensure $\log|u(0)|>0$ by taking $a>1$) and $|\Phi_0(z)|>1$ (take $a$ big enough).
Then $|u(t)|\leq M_2 + |\Phi_0(z)| + \Gamma \pi + \Gamma \log |\Phi_0(z)| \leq M+M'|u(0)|$ for a pair $(M,M')$ independent of $t$, $f_0$, $z$.
Then $|u(t)-u(0)|\leq |u(t)-\Phi_0(z)|+|\Phi_0(z)-u(0)|$. Last we use for $t'=t$ and $t'=0$ that $|u(t')-\Phi_0(z)| \leq M_1+\Gamma \log|u(t')|$.

So far, we have proved survival of points $z$ with in $s_0(z)\in S$, i.e.\ $\tau_\Phi(z)\geq T_3'>T_3$. \Cref{fig:cfp} illustrates the next step of the proof. 
\begin{figure}
\begin{tikzpicture}
\node at (0,0) {\includegraphics[height=7cm]{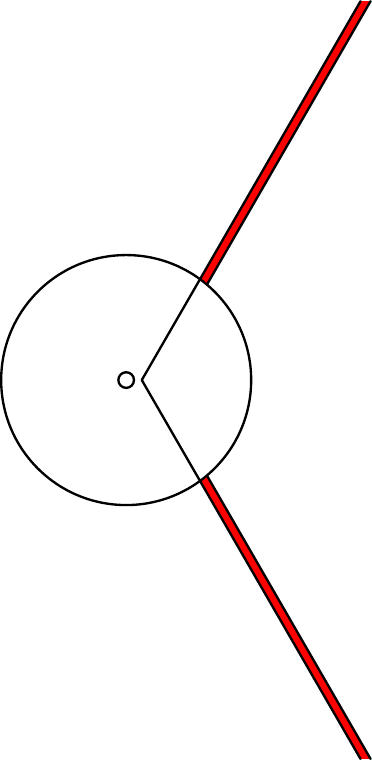}};
\node at (5,0) {\includegraphics[height=7cm]{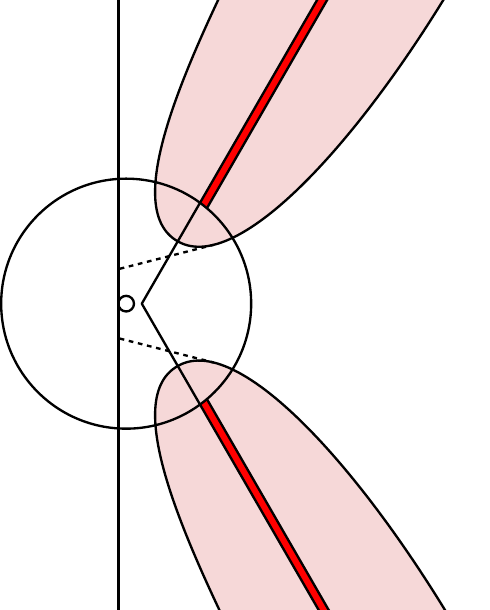}};
\end{tikzpicture}
\caption{Illustration of the proof of \Cref{lem:survlocorb}. Both pictures live in the $u$-plane. The small circle has radius $R_0$, the big circle radius $R_1$, both are centered on the origin. The sector $S$ has apex having some real affix $a>R_0$, which we depicted closer to $R_0$ than to $R_1$. See the text for further description.}
\label{fig:cfp}
\end{figure}
Let $r_1$ to be chosen later, with $r_1<r_0$. Let $R_1=\inf(1/|c_{f} r_1|) = 1/(r_1\sup|c_f|)$ where the extrema are taken over $f\in\cal F_{[0,T_3']}$. 
Assume $z_n$ is an orbit of $f_0$ indexed by $\N$ that is contained in $B(0,r_1)$. Then the sequence $u_n = s_0(z_n)$ is contained in ``$|u|>R_1$''.
If $u_0\in S$ then $\forall n\geq 0$,
$\tau_\Phi(z)\geq T_3'$.
If $u_0\notin S$, let $n_0$ be the smallest positive integer such that $u_{n_0} \in S$ (there is one, by the first estimate in the list).
Since $u_{n_0-1}\in ``|u|>R_1" \setminus S$ and $u_{n_0}\in ``|u|>R_1" \cap S$, the first estimate in the list gives, again, that $u_{n_0}$ must belong to the set $\Lambda$, depicted in red in \Cref{fig:cfp}, intersection of $``|u|>R_1"$ with the set of points in $S$ at distance $\leq 5/4$ from $\partial S$.
By the bound on the motion, $\forall t\in[0,T_3'[$, $u_{n_0}(t)$ belongs to the set $\Lambda'$, depicted in light red, union of balls of center $u\in\Lambda$ and of radius $M_4\log(M'_4+|u|)$.
The map $f_t$ still satisfies the first inequality in the list, hence, provided $R_1$ is big enough then for all $f_0$ and for all $t\in[0,T_3'[$, and for all sequence $u_n$ as above, there is an inverse orbit of the conjugate of $f_t$ by $s_t$, starting from $u_{n_0}(t)$ and remaining in ``$|u|>R_0$'', in fact remaining above or below a domain delimited by the dotted line on the figure (on which we interrupted the dotted line when it reaches the repelling petal, delimited by the vertical plain line).
By continuity, this orbit is equal to $s_t(z_n\ag{t})$ and $\tau_\Phi(z_n)\geq T_3'>T_3$, for all $n\in I=\Z$ or $\N$. If $I=\Z$, let as above $n_0$ be the smallest relative integer such that $u_{n_0}\in S$. By the first inequality it exists, and moreover the inverse orbit $u_n(t)$, $n$ negative, must enter the repelling petal (and stay there) as soon as $|u_{n_0}|+M_4\log(M'_4+|u_{n_0}|)+\frac34 (n-n_0) < -R_0$.
\end{proof}

We can in fact bound their motion.

\begin{lemma}[Bound on the motion of local orbits]\label{lem:bmslo}
The following can be added to the conclusions of \Cref{lem:survlocorb}:
\begin{itemize}
\item $\forall t\in[0,T_3]$, $\forall n\in I$, let $z=\omega_n$: $|z\ag{t}-z|\leq K_1 |z|t$.
\end{itemize}
The constant $K_1$ is independent of $f_0$, $t$ and $z$ but may depend on $T_3$.
\end{lemma}
\begin{proof}
The lazy way uses holomorphic motions:\footnote{It is possible to avoid holomorphic motions completely, by using \Cref{prop:dep,prop:dep2} and the remark that follows, which can themselves be proved without holomorphic motions. However, that is much longer.}
let us extend the deformations $f_t$ to complex values of $t$ in an open neighborhood $V$ of $[0,T_3]$ that does not depend on $f_0\in \cal F$.
The hyperbolic length of $[0,T_3]$ in $V$ is thus independent of $f_0$.
For those values of $t$ such that $|r_t|>1$, where $r_t=1-t$, the map $f_t$ is only defined on $r_t\phi_0(r_t^{-1}\D)$ instead of $r_t\phi_0(\D)$ when $|r_t|\leq 1$. Those sets contain a common ball $B(0,r)$ for some $r$ independent of $f$.
By compactness, an analog of \Cref{lem:survlocorb} still holds. The function $t\mapsto z\ag t$ is defined on $V$ and holomorphic.\footnote{Hence we have a holomorphic motion, because it is injective w.r.t.\ $z$, but we will not use that fact.} Consider the cone of vertex $0$, axis $\R_+$ and angle $3\pi$: this is a Riemann surface over $\C^*$ that is bijectively parameterized in polar coordinates $(r,\theta)$ by $]0,+\infty[\,\times\,]-3\pi/2,3\pi/2[$.
The study made in the previous lemma shows that, for $r_1$ small enough, the points $\omega_n$ satisfying the assumptions of the theorem have a motion $\omega_n\ag t$ such that $u_n(t) := -1/c[f_t] \omega_n\ag{t}$ stays in this cone when $t$ varies.
The element of hyperbolic metric on the cone has expression $c(\theta)|du|/r$ where $c(\theta)\geq c(0)>0$.
The movement of $u$ is holomorphic, hence bounded in this metric by the hyperbolic length of $[0,T_3]$ in $V$.
In Euclidean terms, $u_n(t)$ has moved by at most $K t|u_n(0)|$ for some $K$ independent of $f_0$.
Moreover, $|u_n(t)|$ and $|u_n(0)|$ are of comparable size.
Going back to $\omega_n\ag t =-1/c[f_t] u_n(t)$, we get 
$|\omega_n\ag t -\omega_n\ag 0| \leq |1/c[f_t] u_n(t) - 1/c[f_t] u_n(0)| + |1/c[f_t] u_n(0) - 1/c[f_0] u_n(0)| \leq |u_n(0)-u_n(t)|/|c[f_t]u_n(0)u_n(t)| + |1/c[f_t]-1/c[f_0]|/|u_n(0)|$.
One concludes recalling $c[f_t]$ is not too close to $0$ and depends holomorphically on $t$.
\end{proof}

\subsubsection{Contraction}

Arguments in this section are standard in holomorphic dynamics in complex dimension one.

Let $PC(f_0)$\nomenclature[PC]{$PC(f)$}{the post critical set of $f$} denote the post critical set of $f_0$, i.e.\ the orbit of the (unique) critical value. Since this orbit tends to $0$, the closure $\ov{PC}(f_0)$ equals $PC(f_0)\cup\{0\}$. Let\nomenclature[W0]{$W_0$}{the complement in $\C$ of the closure of the post critical set of $f_0$\nomrefpage}
\[W_0=\C\setminus\ov{PC}(f_0)
.\]
It is well known that inverse branches of $f_0$ are locally contracting for the hyperbolic metric of $W_0$. Let us recall the argument: $f_0$ is a cover from $W_0':=f_0^{-1}(W_0)$ to $W_0$.\nomenclature[W0p]{$W_0'$}{$W_0'=f_0^{-1}(W_0)$\nomrefpage} As such, it is an isometry, at the infinitesimal level, from the hyperbolic metric of $W_0'$ to that of $W_0$. Now $W_0'\subset W_0$, and strict inclusion maps are locally contracting. Recall that for a hyperbolic domain $U$ of $\C$ we denote $\rho_U(z)|dz|$ the element of hyperbolic metric of $U$. For $z\in W_0'$, let us denote $\lambda(z)$
\nomenclature[lambda]{$\lambda[f_0](z)$}{some contraction factor in $W_0$\nomrefpage}
the contraction factor of $f_0^{-1}$ from $f_0(z)$ to $z$, measured with the hyperbolic metric element of $W_0$:
\[\lambda(z) = \frac{\rho_{W_0}(z)}{\rho_{W_0}(f_0(z))}\left|\frac{dz}{df_0(z)}\right|;\]
it is also equal to the contraction factor at $z$ of the inclusion map from $W_0'$ to $W_0$:
\[\lambda(z) = \frac{\rho_{W_0}(z)}{\rho_{W'_0}(z)}.\]
The function $\lambda$ is continuous and takes values in $]0,1[$.

Let us recall that a hyperbolic open subset of the Riemann sphere with an isolated point $a$ in its complement has a hyperbolic metric coefficient $\rho(z) \sim \frac{1}{2|z-a|\log\frac{1}{|z-a|}}$ as $z\tend a\neq\infty$, or  $\rho(z)\sim \frac{1}{2|z|\log|z|}$ as $z\tend a=\infty$.

\begin{lemma}\label{lem:ls}
Let $z_n\in W'_0$ be a sequence.
\begin{enumerate}
\item\label{item:l1} If $z_n$ leaves every compact subset of the open set $W'_0\cup\ov{PC}(f_0)$, then $\lambda(z_n) \tend 0$.
\item\label{item:l2} If $\lambda(z_n) \tend 1$ then $z_n \tend \ov{PC}(f_0)$.
\end{enumerate}
\end{lemma}
\begin{proof} We may extract a subsequence and assume $z_n$ convergent in the Riemann sphere.
\\Point \eqref{item:l1}: If $z_n \tend \infty$ then $\rho_{W_0}(z)\sim \frac{1}{2|z|\log|z|}$ whereas $\rho_{W'_0}(z) \geq \rho_{\dom(f_0)}(z)$ and the latter is $\geq \frac{1}{4d_{\C}(z,\partial \dom(f_0))}$ by Koebe's one quarter theorem. Now since the domain of $f_0$ is the image of $\D$ by a Schlicht map, there is at least one point in its complement that is at distance at most $1$ from $0$. Hence $d_{\C}(z,\partial \dom(f_0)) \leq 1+|z|$. Putting it all together, we get that $\rho_{W_0}(z)/\rho_{W'_0}(z) \tend 0$ as $|z|\tend+\infty$.
In the remaining case: $\lim z_n\neq \infty$ so $\rho_{W_0}(z)$ converges to a constant whereas $\rho_{W'_0}(z) \tend +\infty$.
\\Point \eqref{item:l2}: The function $\lambda$ is continuous and $\lambda(z)<1$ thus if $\lambda(z_n)$ tends to $1$ then $z_n$ leaves every compact subset of $W'_0$, and we conclude by the previous point.
\end{proof}

\begin{lemma}[Definite contraction factor at definite distance of $PC$]\label{lem:definite}
For all $\delta>0$, there exists $\Lambda(\delta)<1$ such that $\forall f\in \cal F$, $\forall z\in W'_0$, if $d_\C(z,\ov{PC}(f))\geq \delta$ then $\lambda(z)\leq \Lambda(\delta)$.
\end{lemma}
\begin{proof}If not, there would be sequences $f_n=\cal R[B_d] \circ \phi_n^{-1} \in\cal F$ and $z_n\in W'_0[f_n]$ such that $d_\C(z_n,\ov{PC}(f_n))\geq \delta$ but $\lambda[f_n](z_n) \tend 1$. Let us extract convergent subsequences and assume that $z_n\tend z'\in\wh \C$ and $\phi_n\tend \phi$, thus $f_n\tend f=\cal R[B_d]\circ\phi^{-1}$.
Since $\ov{PC}(f)$ is contained in a ball $B(0,R)$ with $R$ independent of $f\in\cal F$ (Point~\ref{item:psd:2} of \Cref{lem:PSdist}),
$W_0(f)$ contains $V:=\C\setminus\ov B(0,R)$, hence $\rho_{W_0}(z)\leq \rho_V(z)\sim 1/2|z|\log|z|$ as $z\tend\infty$. This gives an upper bound like in Point~\eqref{item:l1} of \Cref{lem:ls}, but moreover independent of $f\in\cal F$.
It follows that $z'\neq\infty$.
By \Cref{lem:PSdist}, $\ov{PC}(f)$ depends continuously on $\phi$ thus $d_\C(z',\ov{PC}(f))\geq \delta$. Hence $z'\in W_0[f]$.
Now the marked domains $(W_0[f_n],z_n)$ converge for the Caratheodory topology on marked domains.
Hence their universal cover from $(\D,0)$ with real positive derivative at the origin converge, and the coefficient of the hyperbolic metric converges locally uniformly: $\rho_{W_0[f_n]}(z_n)\tend \rho_{W_0[f]}(z')$.
Concerning the marked domains $(W'_0[f_n],z_n)$, there are two cases: either $z'\in W_0'[f]$ in which case there is Caratheodory convergence to $(W'_0[f],z')$ and thus $\rho_{W_0'[f_n]}(z_n)\tend \rho_{W_0'[f]}(z')$; or $z'\notin W_0'[f]$ in which
case\footnote{Indeed, there exists then a point $x_n\in\C\setminus W_0'[f_n]$ such that $x_n\tend z'$.
Let $r'=|z'|$ and let $r''\geq 1$ be any real such that $r''\neq r'$, for instance $r''=r'+1$.
Since the conformal radius w.r.t.\ $0$ of the simply connected set $\dom f_n$ is $1$, there exists a point in $\C\setminus \dom f_n$ of any modulus $\geq 1$, in particular a point $y_n$ of modulus $r''$.
Let $V_n=\C\setminus\{x_n,y_n\}$. Then $\rho_{W'_0[f_n]}(z_n)\geq \rho_{V_n}(z_n)$.
Let $\phi_n$ be the unique $\C$-affine map sending $0$ to $x_n$ and $1$ to $y_n$ and let $u_n=\phi_n^{-1}(z_n)$.
Then $\phi'_n=y_n-x_n$ and $\rho_{\C\setminus\{0,1\}} = \phi_n^* (\rho_{V_n}) = |\phi'_n|\times\rho_{V_n}\circ \phi_n$.
For $n$ big enough, the sequence $x_n-y_n$ is bounded away from $0$ (and $\infty$) thus $u_n\tend 0$ thus $\rho_{\C\setminus\{0,1\}}(u_n)\tend +\infty$ and also $\rho_{V_n}(z_n)=\rho_{\C\setminus\{0,1\}}(u_n)/|y_n-x_n|\tend +\infty$.} $\rho_{W_0'[f_n]}(z_n)\tend+\infty$.
In the first case $\lambda[f_n](z_n)\tend \lambda[f](z')<1$.
In the second case $\lambda[f_n](z_n)\tend 0$.
Both cases lead to a contradiction.
\end{proof}

\subsubsection{Putting back the post critical set}\label{subsub:putback}

The following easy lemma will be useful in several places.

\begin{lemma}\label{lem:back}
There exists a function $\delta>0\mapsto M(\delta)>0$ such that the following holds.
For all $f\in\cal F$, for all $z\in \dom(f)$, if $d_\C(z,PC(f))\geq \delta$ then
\[\frac{\rho_{W_0(f)}(z)}{\rho_{\dom(f)} (z)} \leq M (\delta).\]
\end{lemma}
\begin{proof} In this proof, the notation $B(z,r)$ denotes the euclidean ball and $PC=PC(f)$.
By \Cref{lem:PSdist}, there is $R>0$ such that for all $f\in\cal F$, $PC\subset \ov B(0,R)$. Let $U=\C\setminus \ov B(0,R)$. Then for $|z|>R$:
\[ \rho_{W_0}(z) \leq \rho_U(z) = \frac{1}{2|z| \log \frac {|z|}R}
.\]
For any $z\in W_0$, since the disk $D$ of center $z$ and radius $d_\C(z,PC)$ is contained in $W_0$, we get
\[\rho_{W_0}(z) \leq \rho_{D}(z) = \frac{1}{d_\C(z,PC)}
.\]
By the theory of univalent functions,
\[\rho_{\dom(f)} (z) \geq \frac{1}{4(1+|z|)}.\]
The lemma follows.
\end{proof}

\subsubsection{Homotopic length and decomposition}\label{subsub:homotopic}

I was introduced to the notion of \keyw{homotopic length} by reading \cite{CPT}.

For $\gamma$ a path defined on an interval $I$ containing $[a,b]$, let us denote its restriction to $[a,b]$ by
\[ \gamma\big|[a,b] 
.\]
Let us similarly denote
\[ \omega_n\ag{[0,t]}: s\in[0,t]\mapsto \omega_n\ag s
\]
where $\omega_n$ is an orbit of $f_0$ as in \Cref{subsub:restate}.

To bound the motion of $\omega_n\ag t$ we will look at the \keyw{homotopic length} of the path $\omega_n\ag{[0,t]}$ for the hyperbolic metric on $W_0 = \C\setminus \ov{PC}(f_0)$. Homotopic length of a path $\gamma$ refers to the infimum of $W_0$-hyperbolic lengths of paths homotopic to $\gamma$ in $W_0$, where the ends of the path are fixed. It will be denoted
\[ \hl{W_0}{\gamma}
.\]
By contrast, we denote as follows the usual length of a rectifiable path for the hyperbolic metric of $W_0$:
\[ \len{W_0}{\gamma}
.\]
Last, we will call \keyw{extent} of a path $\gamma$ defined on $[0,t]$ the quantity\phantomsection\label{here:extent}
\[ \extent{W_0}{\gamma}=\sup_{t'\in[0,t]} \hl{W_0}{\gamma\big|[0,t']}
.\]

\begin{remark} Homotopic length is also the hyperbolic distance between the starting point and the end point of a lift of the curve to the universal cover. There are in particular shortest homotopic paths.
The extent of a curve is the smallest radius of a ball in the universal cover containing a lift of the curve and centered on the initial point of this lift.
If $U$ is connected and $\gamma\subset U\subsetneq V$ then the $V$-homotopic length of $\gamma$ is strictly smaller than its $U$-homotopic length: consider for instance the shortest homotopic path for $V$; its $U$-length is strictly shorter. If $U$ and $V$ are hyperbolic Riemann surfaces and $f:U\to V$ is a cover then $\hl{U}{\gamma}=\hl{V}{f\circ\gamma}$. 
\end{remark}

\begin{remark} The sequence $(\omega_n\ag{t})_{n\in\N}$ is an orbit of $f_t$, not $f_0$. It may therefore seem unnatural to measure the motion of $t\mapsto \omega_n\ag{t}$ using the hyperbolic metric on the complement of $\ov{PC}(f_0)$. 
However, we found the proof simpler to write that way. Note that the motion will be evaluated only at some distance from the post critical points, and in the end it will be small.
\end{remark}

Recall that $f_0\in\cal F$ decomposes as
\[ f_0=\cal R[B_d]\circ \phi_0^{-1}
\]
with $\phi_0:\D\to U_0$ a Schlicht map.
Let us decompose the map $f_t$ as follows:
\[ f_t = f_0 \circ \sigma_t
\]
where\nomenclature[sigmat]{$\sigma_t$}{a motion appearing in the decomposition $f_t = f_0 \circ \sigma_t$\nomrefpage}
\[ \sigma_t(z) = \phi_0 \circ r_t \circ \phi_0^{-1} \circ r_t^{-1}
\]
with the notations of \Cref{subsub:outline} and letting $r_t$ denote the multiplication by \[r_t = 1-t
.\]
The map $\sigma_0$ is the identity restricted to $\dom f_0$.
If we interpret $\sigma_t(z)$ as a motion of $z$ as $t$ varies, then it can be viewed as the composition of two motions:
$(t,z)\mapsto (t,r_t^{-1} z)$ followed by the conjugate by $\phi_0$ of the radial motion $(t,z)\mapsto (t,r_t z)$ on the unit disk:
\[\sigma_t=\mu_t\circ r_t^{-1}\]
with
\[\mu_t=\phi_0 \circ r_t\circ \phi_0^{-1}.\]
The domain of definition of the reciprocal $\sigma_t^{-1}$ equals $\phi_0(B(0,r_t))=\sub{\dom(f_0)}{r_t}$ and thus
as $t$ varies away from $0$, it shrinks.

One way to get a control $\omega_{n-1} \ag s$ is to do it inductively from a control on $\omega_n\ag s$, using the relation $f_s (\omega_{n-1}\ag s) = \omega_n\ag s$ of \Cref{lem:compat2}.
Consider the case where $\omega_n\ag 0$ is not equal to $0$ nor to the singular value $v$ of $f_0$.
Then $\omega_n\ag s \notin\{0,v\}$, because $0$ and $v$ do not move under the fiberwise motion, and $\Phi$-fibers are disjoint. Recall that the singular values of $f_0$ are precisely $0$, $\infty$ and $v$.
We claim that, under some condition stated below, the path $s\in[0,t]\mapsto \omega_{n-1}\ag{s}$ is homotopic (with endpoints fixed) in $W_0$ to the concatenation of the following two paths (see \Cref{fig:homotopy}):\phantomsection\label{here:gamma12}
\begin{itemize}
\item The first path, denoted $\gamma_1=f_0^*\omega_n$ by a slight abuse of notation, is parameterised by $s\in[0,t]$ and is defined by continuity by $\gamma_1(0)=\omega_{n-1}\ag 0$ and $f_0 (\gamma_1(s)) = \omega_n\ag s$, i.e.\ we replaced $f_s$ by $f_0$ in $f_s(\omega_{n-1}\ag s)=\omega_n\ag s$. Existence of this path follows from $\omega_n\ag s$ never hitting the singular values of $f_0$. It ends at some point $w'$ (which depends on $t$);
\item The second path is $\gamma_2 : s\in[0,t]\mapsto \sigma_s^{-1}(w')$. For it to be defined up to $s=t$, we need to assume that $w'\in \sub{\dom(f_0)}{(1-t)} = \phi_0(B(0,1-t))$.
\end{itemize}
The homotopy will be defined by means of a map $h$ defined on the set of $(x,y)\in [0,t]^2$ such that $y\leq x$ by
\[h(x,y) = \sigma_y^{-1}(\gamma_1(x)).\]
For it to be well defined, we will make assumptions on $t$, on the length of $\omega_n$ and on the $\epsilon$ such that $\omega_{n-1}\ag 0\in \sub{\dom(f_0)}{(1-\epsilon)}$. For it to be a homotopy in $W_0$, we need to prove that its support does not intersect $\ov{PC}(f_0)$ and for this we will make further assumptions on $t$, on the length of $\omega_n$ and on the Euclidean distance from $\omega_{n-1}\ag 0$ to $\ov{PC}(f_0)$.

\begin{figure}
\begin{tikzpicture}
\node at (-6,0) {\includegraphics[scale=1.5]{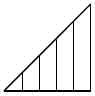}};
\draw[->] (-4.5,0) -- node[above]{$(x,y)\mapsto h(x,y)$} (-1.5,0);
\node at (0,0) {\includegraphics[scale=1.5]{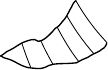}};
\node at (-1.7,-0.9) {$\omega_{n-1}\ag 0$};
\node at (.7,1.2) {$\omega_{n-1}\ag t$};
\draw (-.5,.5) node[above left] {$\omega_{n-1}$} -- (0,0.15);
\draw (.6,-1.2) node[below] {$f_0^*\omega_{n}$} -- (0.2,-0.75);
\node at (1.6,-0.2) {$w'$};
\draw (1.4,.5) node[right] {$\sigma_s^{-1}(w')$} -- (1,.3);
\node at (-7.3,-1.1) {$0$};
\node at (-4.4,-1.1) {$(0,t)$};
\node at (-4.85,1.3) {$(t,t)$};
\end{tikzpicture}
\caption{The map $h(x,y)=\sigma_y^{-1}(f_0^*\omega_n(x))$ defined on the triangle ``$x\in[0,t]$, $y\in[0,t]$, $y\leq x$'' induces a homotopy between $\omega_{n-1}$ on $[0,t]$ and the concatenation of $f_0^*\omega_{n}$ and $s\in[0,t]\mapsto \sigma_s^{-1}(w')$.}
\label{fig:homotopy}
\end{figure}

To state these sufficient conditions, we will introduce the following objects and quantities.
For $\delta>0$ let $V_\delta[f]$ \nomenclature[V]{$V_\delta[f]$}{the $\delta$-neighborhood of $PC(f)$}
denote the $\delta$-neighborhood of $PC(f)$, i.e.\ the set of points whose Euclidean distance to $PC(f)$ is $<\delta$ (see \Cref{fig:Vdelta}).
According to \Cref{lem:PSdist}, the following quantity is positive:
\[\delta_1 := \inf_{f_0\in\cal F_0} d_{\C}(PC(f_0),\C\setminus \dom f_0)
\]
where $d_{\C}$ refers to the Euclidean distance, and the following are finite:
\[R_1 := \sup\setof{|z|}{z\in PC(f_0),\ f_0\in\cal F_0},
\]
\[R_2 := \sup\setof{d_{\dom f_0}(0,z)}{z\in PC(f_0),\ f_0\in\cal F_0}.
\]

\begin{figure}
\begin{tikzpicture}
\node at (0,0) {\includegraphics[width=10cm]{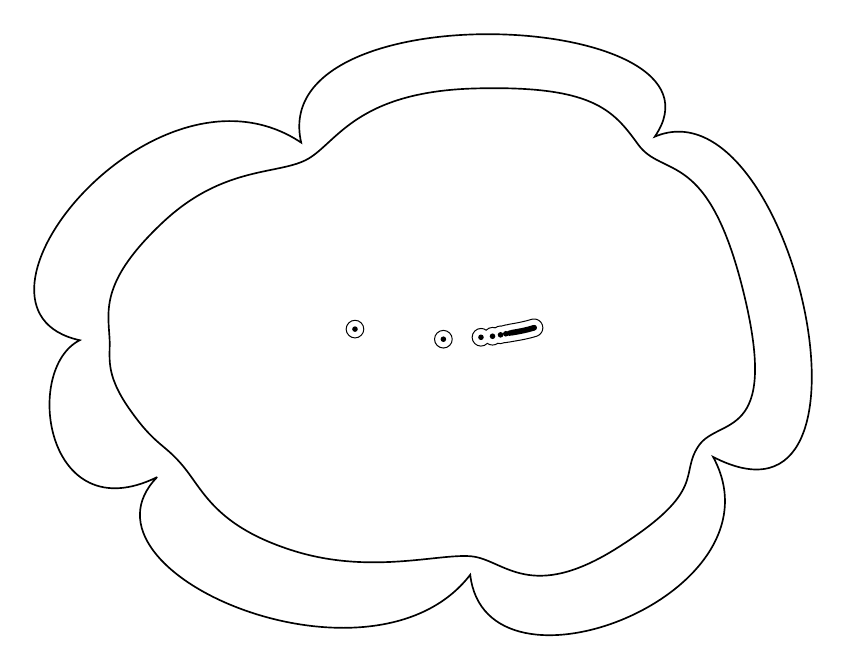}};
\node at (-0.77,0.38) {$v_f$};
\node at (1.39,0.35) {$0$};
\node at (0.5,-0.65) {$V_\delta[f]$};
\end{tikzpicture}
\caption{A schematic illustration of $\dom(f)$, $\sub{\dom(f)}{r}$ and $V_\delta[f]$. Scales are not respected. The outer curve represents the boundary of the domain of some $f\in \cal F$, the nearby smooth curve the boundary of the sub-domain $\sub{\dom(f)}{1-\epsilon}$.
The post-critical set is indicated by dots, its $\delta$-neighborhood for the Euclidean metric is $V_\delta[f]$ and its boundary is indicated by thin curves.}
\label{fig:Vdelta}
\end{figure}

\begin{lemma}\label{lem:stayfar}
For all $(\delta,\delta')$ with $\delta'<\delta<\delta_1$, there exists $T=T(\delta,\delta')>0$ such that $\forall f_0\in\cal F_0$, $\forall t<T$:
\begin{itemize}
\item $\mu_t^{-1}\big(\C \setminus V_{\delta}[f_0]\big) \cap V_{\delta'}[f_0] = \emptyset$, 
\item $r_t \big(\C \setminus V_{\delta}[f_0]\big) \cap V_{\delta'}[f_0]  = \emptyset$,
\item $\sigma_t^{-1} \big(\C \setminus V_{\delta}[f_0]\big) \cap V_{\delta'}[f_0]  = \emptyset$.
\end{itemize}
\end{lemma}
\begin{proof} We can deduce the third point from the first two, using an intermediary value $\delta''$. This may not be optimal\footnote{Near $z=0$, the Euclidean motion of $\sigma_t$ is of order $|z|^2$, thus smaller than the sum of the motions of $\mu_t$ and of $r_t$, which are both of order $|z|$.} but it is not the point here.
For the second point, an explicit valid value of $T$ can easily be computed using \Cref{lem:PSdist}: assume $z\in r_t \big(\C \setminus V_{\delta}[f_0]\big) \cap V_{\delta'}[f_0]$. Then there exists $z'\in PC(f_0)$ such that $|z-z'|< \delta'$, thus $|z|<R_1+\delta'$. Then $|z-r_t^{-1}z| \leq (R_1+\delta')(r_T^{-1}-1)$. If $T$ is chosen so that $(R_1+\delta')(r_T^{-1}-1)<\delta-\delta'$ then $r_t^{-1}z$ cannot belong to $\C\setminus V_{\delta}[f_0]$.
For the first point, let us work by contradiction and assume there is $f_n\in \cal F_0$, $a_n\in\C\setminus V_\delta[f_n]$, $b_n\in V_{\delta'}[f_n]$ and $t_n>0$ such that $t_n\tend 0$ and $a_n=\mu_{t_n}(b_n)$. We may assume that $f_n\tend f\in \cal F_0$, $a_n\tend a\in\wh\C$ and $b_n\tend b\in\C$. From $|a_n-b_n|>\delta-\delta'$ we get $|a-b|\geq \delta-\delta'$. Write $f_n=\cal R[B_d]\circ \phi_n^{-1}$ and $f=\cal R[B_d]\circ \phi^{-1}$ .
From $\delta'<\delta_1$ and $R_2<+\infty$ we deduce that $\phi_n^{-1}(b_n)$ remains in a compact subset of $\D$ thus $b$ belongs to $\dom(f)$, but then $a=\mu_0(b)=b$, a contradiction.
\end{proof}

We will later choose some
\[\delta<\delta_1.\]
Let then\nomenclature[d1]{$d_1$}{infimum over $\cal F$ of some hyperbolic distance\nomrefpage}
\[d_1 = d_1(\delta) = \inf_{f_0\in\cal F_0} d_{W_0'}\big(V_{\delta/3}[f_0]\ ,\, \C\setminus V_{\delta/2}[f_0]\big)
\]
with $W_0':=f_0^{-1}(W_0)$.
Let also
\[d_1'' = d_1''(\delta) = \inf_{f_0\in\cal F_0} d_{f^{-1}(\C\setminus\{0,v\})}\big(V_{\delta/3}[f_0]\ ,\, \C\setminus V_{\delta/2}[f_0]\big),
\]
$v$ being the critical value of $f_0$,
and note that $d_1''<d_1$.
Using the notation of \Cref{lem:stayfar} let
\[T_4(\delta) = T\big(\delta/3\ ,\, \delta/4\big)
\]
so that $\forall f_0\in\cal F_0$, $\forall t<T_4(\delta)$, $\sigma_t^{-1}\big(\C \setminus V_{\delta/3}[f_0]\big) \cap V_{\delta/4}[f_0] = \emptyset$.
Let $\ell(x)$ denote the hyperbolic distance from $0$ to $x$ in $\D$:\nomenclature[l]{$\ell(x)$}{the hyperbolic distance from $0$ to $x$ in $\D$\nomrefpage}
\[\ell(x)=d_\D(0,x) = \on{argth}(x)
.\]
It is a bijection from $[0,1[$ to $[0,+\infty[$.
For a given $\epsilon'>0$, let $T_5=T_5(\delta,\epsilon')\in\,]0,1[$ be the unique solution to\newcommand{\excl}{!}\nomenclature[T5]{$T_5$}{$\exists\excl\, T_5\in\,]0,1[$ s.t.\ $\ell(1-T_5) = d_1(\delta) +\ell(1-\epsilon')$\nomrefpage}
\[\ell(1-T_5) = d_1(\delta) +\ell(1-\epsilon')
.\]
Note that the solution $T_5''$ of $\ell(1-T_5'') = d_1''(\delta) +\ell(1-\epsilon')$ satisfies $T_5''>T_5$.
We will later look at how $T_5(\delta,\epsilon')$ varies as $\epsilon'\tend 0$ for a fixed $\delta$. Recall the definition of extent given at the beginning of the present section on page~\pageref{here:extent}.

\begin{proposition}\label{prop:hl}
Let $t>0$. If we assume that 
\begin{enumerate}
\item\label{item:hl:1} $\tau_\Phi(\omega_n\ag 0)>t$,
\item\label{item:hl:2} the path $s\in[0,t]\mapsto \omega_n\ag s$ is contained in $W_0$, 
\item\label{item:hl:3} $\extent{W_0}{\omega_n\ag{[0,t]}}\leq d_1(\delta)$,
\item\label{item:hl:4} $\omega_{n-1}\ag 0 \in \sub{\dom(f_0)}{(1-\epsilon')}$,
\item\label{item:hl:5} $\omega_{n-1}\ag 0 \notin V_{\delta/2}[f_0]$,
\item\label{item:hl:6} $t\leq T_4(\delta)$,
\item\label{item:hl:7} $t\leq T_5(\delta,\epsilon')$,
\end{enumerate}
then $\tau_\Phi(\omega_{n-1}\ag 0) > t$ and the function $h$ mentioned above is well defined and has support in $W_0$ (even better: it avoids $V_{\delta/4}[f]$). In particular $s\in[0,t]\mapsto \omega_{n-1}(s)$ is homotopic in $W_0$ to the concatenation $\gamma_1\cdot\gamma_2$, of the two paths defined earlier, page~\pageref{here:gamma12}.
We also have $\gamma_1\subset\sub{\dom(f_0)}{(1-T_5)}.$
\end{proposition}
\begin{proof} By~\eqref{item:hl:2} the path $\omega_n$ is contained in $\C\setminus\{0,v\}$ thus the path $\gamma_1$, defined as the pull-back by $f_0$ of $\omega_n\ag\cdot$ starting from $\omega_{n-1}\ag{0}$, is well defined. Let $t'\in[0,t]$: 
\[ \hl{\dom f_0}{\gamma_1\big|_{[0,t']}}<\hl{W_0'}{\gamma_1\big|_{[0,t']}} = \hl{W_0}{\omega_n\ag\cdot\big|_{[0,t']}} \leq d_1
\]
(the first inequality comes from the strict inclusion $W_0'\subset\dom f_0$, the equality follows from $f_0$ being a cover from $W_0'$ to $W_0$, the second inequality comes from point~\eqref{item:hl:3}).
In particular the $\dom f_0$-hyperbolic distance from $\gamma_1(0)$ to $\gamma_1(t')$ is $\leq d_1$.
Since moreover by~\eqref{item:hl:4}, $d_{\dom f_0}(0, \omega_{n-1}\ag 0)\leq\ell(1-\epsilon')$ we get that $\gamma_1$ is contained in the $\dom f_0$ hyperbolic ball of center $0$ and radius $d_1+\ell(1-\epsilon') = \ell(1-T_5)$.
Hence $\gamma_1\subset\sub{\dom(f_0)}{(1-T_5)}$.
Hence by~\eqref{item:hl:7}, $\gamma_2$ and the map $h$ defined at the same place are well defined. Let us check that $h$ takes values in $W_0$, i.e.\ that it avoids $\ov{PC}(f_0)$. Note that we have already proved that 
$\hl{W_0'}{\gamma_1\big|_{[0,t']}} \leq d_1$.
In particular the $W_0'$-hyperbolic distance from $\gamma_1(0)$ to $\gamma_1(t')$ is $\leq d_1$.
Together with point~\eqref{item:hl:5} and the definition of $d_1$, it implies that $\gamma_1$ is contained in $\C\setminus V_{\delta/3}[f_0]$. Point~\eqref{item:hl:6} then implies that $\gamma_2$ and $h$ take value in $\C\setminus V_{\delta/4}[f_0]$, which is contained in $W_0$. 
The points $h(s,s)$ and $\omega_{n-1} \ag s$ are both mapped by $f_s$ to the same point: $\omega_n\ag s$, for which we recall that $\Phi_s(\omega_n\ag s)$ stays constant when $s$ varies. The uniqueness statement in \Cref{prop:fibers} applied to $\Phi$ then implies that $\tau_\Phi(\omega_{n-1})>t$ and that the functions defined on $[0,t]$, $s\mapsto h(s,s)$ and $s\mapsto \omega_{n-1} \ag s$, are in fact equal.
\end{proof}

We have the following variation with $W_0$ replaced by $\C\setminus\{0,v\}$ in the hypotheses, but not in the conclusion:

\begin{proposition}\label{prop:hl2}
Let $t>0$. If we assume that 
\begin{enumerate}
\item\label{item:hl2:1} $\tau_\Phi(\omega_n\ag 0)>t$,
\item\label{item:hl2:2} the path $s\in[0,t]\mapsto \omega_n\ag s$ is contained in $\C\setminus\{0,v\}$, 
\item\label{item:hl2:3} $\extent{\C\setminus\{0,v\}}{\omega_n\ag{[0,t]}}\leq d_1''(\delta)$,
\item\label{item:hl2:4} $\omega_{n-1}\ag 0 \in \sub{\dom(f_0)}{(1-\epsilon')}$,
\item\label{item:hl2:5} $\omega_{n-1}\ag 0 \notin V_{\delta/2}[f_0]$,
\item\label{item:hl2:6} $t\leq T_4(\delta)$,
\item\label{item:hl2:7} $t\leq T_5''(\delta,\epsilon')$,
\end{enumerate}
then $\tau_\Phi(\omega_{n-1}\ag 0) > t$ and the function $h$ is well defined and avoids $V_{\delta/4}[f]$. In particular it has support in $W_0$ and the path $s\in[0,t]\mapsto \omega_{n-1}(s)$ is homotopic in $W_0$ to $\gamma_1\cdot\gamma_2$.
We also have $\gamma_1\subset\sub{\dom(f_0)}{(1-T''_5)}.$
\end{proposition}
\begin{proof}As in the previous proof.
\end{proof}

\begin{lemma}\label{lem:hl:1}
Under the conditions of \Cref{prop:hl}, the $W_0$-homotopic length of $\gamma_1$ is at most $\Lambda(\delta/3)$ times the $W_0$-homotopic length of $\omega_n$, where $\Lambda(\delta/3)<1$ is given by \Cref{lem:definite}.
\end{lemma}
\begin{proof}
We have seen that $\hl{W_0'}{\gamma_1}\leq d_1$.
Consider a shortest path $\gamma$ homotopic to $\gamma_1$ in $W_0'$: $\len{W_0'}{\gamma}=\hl{W_0'}{\gamma_1}$. It is a geodesic for the hyperbolic metric, in particular all its points are at $W_0'$-hyperbolic distance $\leq d_1$ from its starting point.
By the definition of $d_1$ this implies that $\gamma$ is disjoint from $V_{\delta/3}[f_0]$. By \Cref{lem:definite}, we have $\lambda(z) \leq \Lambda(\delta/3)$ for $z$ in the support of $\gamma$, with $\lambda(z)=\rho_{W_0}(z)/\rho_{W_0'}(z)$. Thus $\hl{W_0}{\gamma_1}\leq\len{W_0}{\gamma}\leq \Lambda(\delta/3) \len{W_0'}{\gamma}= \Lambda(\delta/3)\hl{W_0} {\omega_n}$.
\end{proof}

\begin{lemma}\label{lem:hl2:1}
Under the conditions of \Cref{prop:hl2}, the $W_0$-homotopic length of $\gamma_1$ is at most $M(\delta/3)$ times the $\C\setminus\{0,v\}$-homotopic length of $\omega_n$, where $M(\cdots)$ is given in \Cref{lem:back}.
\end{lemma}
\begin{proof}
This done as in the previous lemma, with $W_0'$ replaced by $f^{-1}(\C\setminus\{0,v\})$, $d_1$ by $d_1''$ and $\lambda(z)$ by $\rho_{W_0}(z)/\rho_{f^{-1}(\C\setminus\{0,v\})}(z)$. By inclusion, the latter quantity is $\leq \rho_{W_0}(z)/\rho_{\dom f}(z)$ thus $\leq M(\delta/3)$.
\end{proof}

The $W_0$-homotopic length of $\gamma_2$ will be controlled using \Cref{lem:t5p} below. To state it we need to introduce another quantity.
By \Cref{lem:stayfar} there exists $T_6=T_6(\delta)$ such that $\forall f_0\in\cal F_0$, $\forall t<T_6$, $\mu_t^{-1}\big(\C \setminus V_{\delta/4}[f_0]\big) \cap V_{\delta/5}[f_0] = \emptyset$ and
$r_t \big(\C \setminus V_{\delta/5}[f_0]\big) \cap V_{\delta/6}[f_0]  = \emptyset$.

\begin{lemma}\label{lem:t5p}
For all $\delta<\delta_1$,  
there exists $K_0=K_0(\delta)$ such that under the conditions of \Cref{prop:hl} or~\ref{prop:hl2}, and assuming moreover 
\begin{itemize}
\item $t\leq T_6(\delta)$ and $t\leq T_5(\delta,\epsilon')/2$
\end{itemize}
then the $W_0$-homotopic length of $\gamma_2$ is $\leq K_0 t/T_5(\delta,\epsilon')$.
\end{lemma}
\begin{proof}
Similarly to the proof of the propositions, the condition $t<T_6$ ensures that there is a homotopy in $W_0$ between $\gamma_2$ and $\gamma_3.\gamma_4$ where $\gamma_3(s)=\mu_s^{-1}(\gamma_2(0))$ and $\gamma_4(s)=r_s w''$ where $w''$ is the endpoint of $\gamma_3$.
The motion $\mu_s$ is the conjugate by $\phi_0$ of the radial motion and we have seen that $x:=|\phi_0^{-1}(\gamma_2(0))| \leq 1-T_5$ (in the case of \Cref{prop:hl2} we have $x\leq 1-T_5'' <1-T_5$) thus the length of $\gamma_3$ for the hyperbolic metric of $\phi_0(\D)=\dom f_0$ is $\leq d_{\D}(x,\frac{x}{1-t}) \leq d_\D(1-T_5,\frac{1-T_5}{1-t})= {\frac12\log\left(\frac{1-\frac{t}{2-T_5}}{1-\frac{t}{T_5}}\right)} \leq \frac12\times - \log(1-\frac{t}{T_5}) \leq t\log(2)/T_5$, the latter because $t/T_5\leq 1/2$.
\Cref{lem:back} implies that its $W_0$-length is at most $M(\delta/5)$ times this quantity.
To bound the $W_0$-length of $\gamma_4$, note that it is contained in the complement of $V_{\delta/6}$, thus $\forall z\in \gamma_4$, $\rho_{W_0}(z)\leq 6/\delta$.
Also, $\rho_{W_0}(z)\leq 1/(|z|-R_1)$ where $R_1=\sup\setof{|z|}{f\in\cal F,\ z\in PC(f)}$. If $|w''|>4R_1$ then since $t\leq 1/2$, the whole path $\gamma_4$ is contained in the complement of $B(0,2R_1)$ and thus $\rho_{W_0} \leq 2/|z|$ whence a $W_0$-length of $\gamma_4$ that is $\leq \int_{|z|}^{\frac{|z|}{1-t}}\frac2xdx=2\log(1/(1-t))\leq 4t\log 2$ because $t\leq 1/2$.
If $|w''|\leq 4R_1$ then the whole euclidean length of $\gamma_4$ is $\leq 4R_1 t$ hence the $W_0$-length is $\leq (6/\delta)4R_1 t$. Recall that $T_5< 1$.
\end{proof}

\begin{remark}
The linearity of the bound w.r.t.\ $t$ is not crucial for this article: weaker orders of convergence to $0$ would work for our purpose, thanks to the fact that in \Cref{prop:part:1}, $\epsilon'$ is much bigger than $\epsilon$. What will be important is that values of $t$ for which the bound is a given small constant are much bigger than $\epsilon$.
So how $T_5$ depends on $\epsilon$ will be important too (recall $\delta$ will be fixed).
\end{remark}

\begin{remark}%
\Cref{lem:hl:1,lem:hl2:1,lem:t5p} give an upper bound on the $W_0$-homotopic length of the curve $s\in[0,t]\mapsto \omega_{n-1}\ag{s}$ on each of its subsegments $s\in[0,t']$ for $t'<t$, by applying \Cref{prop:hl} to $t'$ instead of $t$. So we get in fact bounds on the extent. This allows for induction.
\end{remark}

\subsubsection{Visits in the repelling petal}\label{subsub:reppass}

\begin{lemma}\label{lem:cmpmet}
There exists $K_4>0$ such that $\forall f\in\cal F$, $\forall z\in W_0$, if $|z|\leq 1$ then $\rho_{W_0}(z)\geq 1/K_4|z|$.
\end{lemma}
\begin{proof}
Let us work by contradiction and assume that there exist $f_n\in \cal F$ and $z_n\in \D$ such that $z_n\in W_0[f_n]$ and $\rho_{W_0[f_n]}(z_n)|z_n|\tend 0$.
Consider the dilatation by $1/z_n$: the set $z_n^{-1}W_0[f_n]\subset\C$ does not contain $0$, but it contains the point $1$ and has a hyperbolic metric coefficient at this point tending to $0$ as $n\to+\infty$.
There would thus exist $R_n,r_n>0$ such that $R_n>|z_n|>r_n$, such that $W_0[f_n]$ contains the annulus ``$r_n<|z|<R_n$'' and such that $R_n/|z_n|\tend+\infty$ and $r_n/|z_n|\tend 0$ (this can be proved by contradiction, using that $\C$ minus two points is hyperbolic, and that inclusion is non-expanding for the hyperbolic metric).

Let us apply \Cref{lem:traptime} to $r=r_0$ where $r_0$ is provided by \Cref{prop:cf}.
The point $f^{n_0}(v_f)$ belongs to $D_{r_0}[f]$. It depends continuously on $f$ and thus it remains in a compact subset of $\C\setminus \{0\}$.
Let $a_n[f] = -1/c_f f^{n}(v_f)$ where $c_f$ is defined in \Cref{prop:cf} and is bounded away from $0$ and $\infty$ as $f$ varies in $\cal F$.
Then $a_{n_0}[f]$ also belongs to a compact set.
By \Cref{lem:upt1}, $\forall n\geq 0$, $3n/4-A\leq |a_{n_0+n}[f]| \leq A+5n/4$ for a constant $A>0$ that is independent of $f$. It follows that there is $A', A''>0$ and $n_1\geq n_0$ such that for all $f\in\cal F$ and for all $n\geq n_1$, $\frac{A'}{n}\leq |f^{n}(v_f)|\leq \frac{A''}{n}$.

Hence the aforementioned sequence of annuli cannot exist, which yields a contradiction.
\end{proof}

In coordinates $u=-1/c_f z$ this reads $\rho_{-1/(c_fW_0)}(u)\geq 1/K_4|u|$.

\begin{proposition}\label{prop:homotopicrep}
There exists $r_2$, $T_8$ and $d_1'$, positive reals, such that for all $f_0\in\cal F$, for all $n_0,n_1\in\Z$ with $n_0<n_1$, for all $f_0$-orbit $\omega_n$ indexed by $\Z\cap[n_0,+\infty[$, and for all $t>0$, if
\nomenclature[r2]{$r_2$}{defined in \Cref{prop:homotopicrep}}
\begin{enumerate}
\item $\omega_{n_0}\ag{0},\ldots,\omega_{n_1}\ag{0} \in D_\rep[f_0](r_2)$,
\item $\tau (\omega_{n_1}) > t$,
\item $t\leq T_8$,
\item $\extent{W_0}{\omega_{n_1}\ag{[0,t]}}<d_1'$,
\end{enumerate}
then $\tau (\omega_{n_0})>t$, 
the paths $\gamma_1$ and $\gamma_2$ defined below are well defined, and $s\in[0,t]\mapsto \omega_{n_0}\ag s$ is homotopic in $W_0[f_0]$ to their concatenation $\gamma_1\cdot \gamma_2$.
The path $\gamma_1$ is the pull-back of $s\in[0,t]\mapsto \omega_{n_1}\ag s$ by $f_0^{n_1-n_0}$ that starts from $\omega_{n_0}\ag{0}$; the path $\gamma_2: [0,t]\to \C$ is the continuous solution, starting from $\gamma_1(t)$, of $f_s^{n_1-n_0}(\gamma_2(s)) = \on{const} = f_0^{n_1-n_0}(\gamma_1(t))=\omega_{n_1}\ag{t}$.
\end{proposition}
\begin{proof}
Consider the domains $\Omega_\theta(R)$ and $D_\theta(r_0)[g]$ introduced in \Cref{lem:2ndstep}, with $-1/c_g D_\theta(r_0)=\Omega_\theta(1/|c_g|r_0)$.
We will take some $T_8\leq 1/2$. 
The class of maps $\cal F_{[0,1/2]}$ is compact and the domain of its members all contain $B(0,1/8)$, so we can apply \Cref{prop:bigger,prop:mvtphirep} to the restriction to $\D$ of the conjugates of maps in this class by $z\mapsto 8z$.
Choose $\theta=3\pi/4$, $\theta'=(\theta+\frac\pi2)/2$, $\theta''=(\theta'+\frac\pi2)/2$, so that $\pi/2<\theta''<\theta'<\theta$.
It was proved in \Cref{prop:bigger} that for $r_0>0$ small enough, the (invertible) repelling Fatou coordinates of $g\in \cal F_{[0,1/2]}$ extend to $-D_\theta(r_0)[g]$ 
for some $r_0>0$, and that $-D_\theta(r_0)[g]$, $-D_{\theta'}(r_0)[g]$ and $-D_{\theta''}(r_0)[g]$ are all invariant by a branch of $g^{-1}$. 
By compactness, for $r_0$ small enough, there is only one such branch.
Also, provided $r_0$ has been chosen small enough, it can be checked using \Cref{prop:cf} and \Cref{lem:traptime} that $PC[g]$ does not intersect $-D_\theta(r_0)[g]$.

Now choose any $r_1<r_0$, for instance $r_1=r_0/2$ and impose $r_2\leq r_1$.
Let
\[\gamma_0:[0,t]\to\C,\ s\mapsto \omega_{n_1}\ag s\]
By assumption its initial point is contained in $D_{\pi/2}[f_0](r_2)$.
By \Cref{lem:cmpmet}, for $d_1'$ small enough, we are ensured that $\gamma_0$ is contained in $-D_{\theta''}(r_0)[f_0]$ (this is more easily seen in coordinates $u=-1/c_{f_0} z$: the path stays in a ball of center its initial point $u_0$ and radius $\cal O(d'_1|u_0|)$).
Since $-D_{\theta''}(r_0)[f_0]$ is stable by a branch of $f_0^{-1}$, the path $\gamma_1$ is well defined and contained in $-D_{\theta''}(r_0)[f_0]$.
Now, as in the proof of \Cref{prop:hl}, we set up a triangular homotopy $h(x,y)$ for $y\leq x \leq t$ with $f_y^{n_1-n_0}(h(x,y))=f_0^{n_1-n_0}(\gamma_1(x))=\gamma_0(x)$ and $h(x,0)=\gamma_1(x)$. 
Taking $T_8$ small enough, we get $-D_{\theta''}(r_0)[f_0]\subset-D_{\theta'}(r_0)[f_y]$ for all $y\leq T_8$ and all $f_0\in\cal F$. In particular $\gamma_0\subset -D_{\theta'}(r_0)[f_y]$.
Since the latter is invariant by a branch of $f_y^{-1}$, unique and continuously depending on $y$, it follows that $h$ is well-defined, continuous, and has support in $-D_{\theta'}(r_0)[f_y]$. For $T_8$ small enough, $-D_{\theta'}(r_0)[f_y]\subset -D_{\theta}(r_0)[f_0]$, hence $h$ takes values in $W_0[f_0]$.
\end{proof}

This proof yields more:

\begin{complem} There is some $r_4>0$ and $\theta'>\pi/2$ such that under the conditions of the proposition above, and $\forall s\in[0,t]$,  $-D_{\theta'}[f_s](r_4)$ is a repelling petal for $f_s$ and for all $k$ with $n_0\leq k\leq n_1$, $\omega_{k}\ag s\in -D_{\theta'}[f_s](r_4)$.
\end{complem}
\begin{proof}
Change the value of $n_0$ to that of $k$ in the previous proposition. Its
proof provided some quantities called $r_0$ and $\theta'$, and proved the claim of the complement for $r_4=r_0$ and the same value of $\theta'$.
\end{proof}

Note that by infinitesimal contraction of $f_0^{-1}$ for the hyperbolic metric of $W_0$,
\[ \hl{W_0}{\gamma_1}<\hl{W_0}{s\in[0,t]\mapsto \omega_{n_0}\ag s}
.\]
Since $\gamma_2$ stays far from the boundary of $W_0$, the control we get on its homotopic length is better than in \Cref{lem:t5p}:

\begin{lemma}\label{lem:rpb}
We can add the following conclusion to the previous lemma
\[ \hl{W_0}{\gamma_2}\leq K_5 t
.\]
\end{lemma}
\begin{proof}
In this proof we will say that a constant is independent if it is independent of $f$, of $t$, of the chosen orbit $\omega_n$ and of the length $n_1-n_0$.
We will use $=\cal O(\text{expression})$ to express a quantity that is at most the expression times a constant that is independent.
We will write that two quantities are comparable when their quotient is bounded away from $0$ and $\infty$ independently of $f$, of $t$, of the chosen orbit $\omega_n$ and of the length $n_1-n_0$.
Let us continue with the notations of the previous proof. In particular, $\theta = 3\pi/4$. Note that $\gamma_2(y)=h(t,y)$ and $\gamma_2(t) \in -D_\theta(r_0)[f_0]$.
Since there are sectors $-D_{\theta_3}(r_3)[f_0]$ contained in $W_0[f_0]$ for $\theta_3 =(\theta+\pi)/2>\theta$ with $r_3$ independent of $f_0$, by imposing $r_0<r_3$, we have $\forall z\in -D_\theta(r_0)[f_0]$, $B(z,|z|/K)\subset W_0[f_0]$ for some $K>1$.
Hence it is enough to prove that for $y\leq t$,
\[|\gamma_2(y)-\gamma_2(0)|= \cal O(K' t |\gamma_2(0)|),
\]
in which case, for $t<T_8$, $T_8$ is small enough, the euclidean ball $B(\gamma_2(0), K' t |\gamma_2(0)|)$ is contained in $W_0[f_0]$ and contains $\gamma_2$ thus $\gamma_2$ is homotopic in $W_0[f_0]$ to the straight segment from $\gamma_2(0)$ to $\gamma_2(t)$
and the latter has a $W_0[f_0]$-hyperbolic length at most its $B(\gamma_2(0),|\gamma_2(0)|/K)$-hyperbolic length thus at most $K_5t$ for $T_8$ small enough.
As in the proof of \Cref{prop:homotopicrep}, let
\[\gamma_0(s) = \omega_{n_1}\ag s\]
and let $\Phi_\rep[f_y]$ be a repelling Fatou coordinate on $-D_\theta(r_0)[f_y]$, normalized by the expansion.. We have seen in this former proof that $\gamma_1$ and $\gamma_2$ and $\gamma_0$ are contained in $-D_{\theta'}(r_0)[f_y]$ for all $y\leq T_8$, for some constant $\theta'=5\pi/8<\theta=3\pi/4$.
Then
\[ \Phi_\rep[f_y](\gamma_2(y))=\Phi_\rep[f_y](\gamma_0(t))-(n_1-n_0)
. \]
Let us denote $\on{Rep}_yz=\Phi_\rep[f_y](z)$. By taking $r_0$ small enough we can ensure that for all $z\in -D_{\theta}(r_0)[f_y]$, the quantity $\on{Rep}_yz$ is comparable to $1/z$ and the quantity $\on{Rep}'_y(z)$ is comparable to $1/z^2$ (use the bound on $\wt\Phi$ given in \Cref{prop:cf} that extends to $\Omega_\theta$ according to \Cref{prop:bigger}).
For $y\leq 1/2$, we have $\sup_{|z|<1/16}|f_0(z)-f_y(z)|\leq K y$ for some $K$ independent of $f$.
We have $\big|\on{Rep}_y\gamma_2(y)-\on{Rep}_0\gamma_2(0)\big|
= \big|\on{Rep}_y\gamma_0(t)-\on{Rep}_0\gamma_0(t)\big|$.
Provided $r_2$ has been chosen small enough, \Cref{prop:mvtphirep} gives
$\big|\on{Rep}_y\gamma_0(t)-\on{Rep}_0\gamma_0(t)\big| = \cal O(y/|\gamma_0(t)|)$, as $\Phi_\rep$ is normalized by the expansion.
Let $u_x=-1/(c[f_0]\gamma_0(x))$ and $Z_x=\on{Rep}_0\gamma_0(x)$.
The size of the quantities $Z_x$, $Z_0$, $u_x$, $u_0$, $1/\gamma_0(x)$ and $1/\gamma_0(0)$ are all comparable.
Similarly, $|1/\gamma_2(0)|$ is comparable to $|\on{Rep}_0\gamma_2(0)|=|Z_t-(n_1-n_0)|$.
Note that the positive integer $n_1-n_0$ can be arbitrarily large. However since $Z_t$ is contained in $-\Omega_{3\pi/4}(10)$ (provided $r_2$ is small enough), there is an independent lower bound on $|Z_t-(n_1-n_0)|/|Z_t|$ thus
$y/|\gamma_0(t)| = \cal O(y|Z_0|) = \cal O(y|Z_t|)
\leq \cal O(y |Z_t-(n_1-n_0)|) = \cal O( y |\on{Rep}_0\gamma_2(0)|)
= \cal O( y/|\gamma_2(0)|)$: for some $M>0$
\[ \big|\on{Rep}_y\gamma_2(y)-\on{Rep}_0\gamma_2(0)\big| \leq M y / |\gamma_2(0)|
.\]
Then by \Cref{prop:mvtphirep} again we get $\big|\on{Rep}_y\gamma_2(0)-\on{Rep}_0\gamma_2(0)\big| \leq M'y/|\gamma_2(0)|$
thus $\big|\on{Rep}_y\gamma_2(y)-\on{Rep}_y\gamma_2(0)\big|
\leq \big|\on{Rep}_y\gamma_2(y)-\on{Rep}_0\gamma_2(0)\big|
+ \big|\on{Rep}_0\gamma_2(0)-\on{Rep}_y\gamma_2(0)\big| \leq (M+M') y/|\gamma_2(0)|$.
The straight segment from $\on{Rep}_y\gamma_2(y)$ to $\on{Rep}_y\gamma_2(0)$ is contained in the subset $-\Omega_{\theta'}(R_2)$ of the domain of $\on{Rep}_y^{-1}$ and $|(\on{Rep}_y^{-1})'(Z)|$ is comparable to $1/|Z|^{2}$ for $Z\in -\Omega_{\theta'}(R_2)$.
Using moreover that $\on{Rep}_y(Z)$ is comparable to $1/Z$, we get: provided $T_8$ was chosen small enough, for all $y\leq t$, $|\gamma_2(y)-\gamma_2(0)|\leq K' y |\gamma_2(0)|$.
\end{proof}

\begin{lemma}\label{lem:replimsec}
If in \Cref{prop:homotopicrep} we take $n_0=-\infty$, i.e.\ start from an orbit indexed by $\Z$ such that $\omega_n\ag 0\in D_\rep[f_0](r_2)$ for all $n\leq n_1$, and leave the other three assumptions unchanged, then for all $\alpha>0$ and all $r>0$, $\exists n'\in \Z$ such that $\forall n\leq n'$, $\forall s\in[0,t]$, $\omega_{n_0}\ag s$ belongs to the sector of apex $0$, radius $r$, and angle $\alpha$ around the repelling axis of $f_s$. 
\end{lemma}
\begin{proof}
In the course of the proof of \Cref{prop:homotopicrep} we proved that $\gamma_0: s\in[0,t]\mapsto\omega_n\ag s$ has a support contained in $-D_{\theta'}[f_y](r_0)$ for all $y\leq t$. In particular the function $\chi: s\mapsto -1/c_{f_s}\gamma_0(s)$ takes values in $-\Omega_{\theta'}(1/|c_{f_s}r_0|)$. Recall that on this set, the dynamics differs from the translation by $1$ by at most $1/4$. The path $\chi$ has compact image. The lemma follows.
\end{proof}

\subsubsection{Bounding the motion of orbits (putting it all together).}\label{subsub:bddm}

We now have the tools to prove \Cref{prop:surv}.

Recall that we are considering an orbit $\omega_n$ indexed by $\Z$ of a map $f_0\in\cal F$, eventually captured by an attracting petal in the future, by a repelling petal in the past, and defined a movement $\omega_n\ag t$ of this sequence, for which it remains an orbit of $f_t$ and so that its attracting Fatou coordinate, normalized by immobilizing the image of the critical value, remains constant.
The starting hypothesis is that $\omega_n$ is entirely contained in $\sub{\dom(f_0)}{(1-\epsilon')}=\phi_0(B(0,1-\epsilon'))$. In particular condition~\eqref{item:hl:4} of \Cref{prop:hl} and its analog in \Cref{prop:hl2} are satisfied for all $n\in \Z$ by the assumption.

We will now compute a lower bound for the survival time $\tau(\omega_n)$, that depends only on $\epsilon'$.

This will be done by decreasing induction on $n$, using \Cref{prop:hl,prop:hl2,prop:homotopicrep} and their complements \Cref{lem:hl:1,lem:hl2:1,lem:t5p,lem:rpb}. The induction hypothesis will be that
the motion of $t\mapsto \omega_{n}\ag t$, measured with the hyperbolic metric of the set $W_0[f_0]$, more precisely what we called the extent at the beginning of \Cref{subsub:homotopic}, is smaller than the constants $d_1$, $d_1'$ and $d_1''$ appearing in the propositions.
The complements 
then give a upper bound on the motion of $t\mapsto \omega_{n-1}\ag t$.
We will show that for $t$ small enough, this bound is also less than $d_1$, $d_1'$ and $d_1''$, so that the induction can go on, and we will give a lower bound on how small $t$ needs to be.

\medskip

Recall $r_0$ is a small enough constant provided by \Cref{prop:cf} to~\ref{prop:butterfly}, and~\ref{prop:bigger}.

By \Cref{lem:survlocorb}, we know the survival of local orbits. More precisely
let us choose $T_3=T_1'/2$.\nomenclature[T3]{$T_3$}{some parameter in \Cref{lem:survlocorb}, later chosen to be $=T_1'/2$\nomrefpage}
\Cref{lem:survlocorb} yields a value $r_1$. 
If the whole orbit $(\omega_n\ag 0)_{n\in\Z}$ is contained in $B(0,r_1)$ then we get the lower bound $\tau(\omega_n)\geq T_3$. In this simple case, the lower bound is independent of $\epsilon'$, so it is even better.
In the sequel, we assume that we are not in this case, i.e.\ that the orbit $(\omega_n\ag 0)_{n\in\Z}$ leaves $B(0,r_1)$ at least once.

Recall that maps $ f\in \cal F$ all have the same critical value $v$.
We have already remarked that by compactness of $\cal F$ and \Cref{prop:conti} (see also \Cref{lem:traptime:2}), there exists $\eta_0$ such that $\forall f\in\cal F$, $B(v,\eta_0)$ is contained in the basin of the parabolic point. Recall $D_\rep(r)=D_\rep[f](r)$ denotes the disk of diameter $[0,re^{i\theta}]$ where $e^{i\theta}$ points in the direction of the repelling axis of $f$.
Let $f^\N(B(v,r))$ denote the union of $B(v,r)$ and of all its images by iteration of $f$.

\begin{lemma}\label{lem:disj}
There exists $r_3>0$ and $\eta'_0<\eta_0$ such that $\forall r\leq r_3$, $\forall f\in \cal F$, the set $f^\N(B(v,\eta'_0))$ is disjoint from $f(B(0,r))\setminus B(0,r)$ and  from $f(D_\rep(r))$.
\end{lemma}
\begin{proof} 
Let $r_0$ be provided by \Cref{prop:cf}: for all $f\in\cal F$, and all $r\leq r_0$, $D_\at(r)$ is stable by $f$ and contained in the parabolic basin.
Note that for some $r'$ small enough, then for all $r$ small enough, then for all $f\in \cal F$, $f(B(0,r))\setminus B(0,r)$ and $f(D_\rep(r))$ are disjoint from $D_\at(r')$, as easily follows from \Cref{lem:upt1} and the fact that in the change of variable $u=-1/c_f z$, the factor $c_f$ is bounded away from $0$ and from $\infty$.
By \Cref{lem:traptime:2} there is some $n_0$ and some $\eta_0'>0$ such that $\forall f\in \cal F$, $f^{n_0}(B(v,\eta_0'))\subset D_\at(r')$, and hence $\forall n\geq n_0$, $f^{n}(B(v,\eta_0'))\subset D_\at(r')$.
By compactness of $\cal F$ again, there is a uniform lower bound on the distance from $0$ to $f^{n}(B(v,\eta_0'))$ as $n$ varies between $0$ and $n_0-1$ and $f$ varies in $\cal F$.
So the lemma will hold for $r$ small enough.
\end{proof}

Let $T_8$, $d_1'$ and $r_2$ be provided by \Cref{prop:homotopicrep}. Let
\[ r_0'=\min(r_0,r_1,r_2,r_3) \]
and denote 
\[D_\rep = D_\rep[f] = D_\rep[f](r_0')
.\]

We introduced earlier the $\delta$-neighborhood $V_\delta[f]$ of $PC(f)$.
Let $\wt B(r)=\wt B(r)[f]$ be the set of points in $B(0,r)\setminus\{0\}$ whose forward orbit by $f$ is contained in $B(0,r)$.
Let\nomenclature[Veta]{$\wt V_\eta[f]$}{some domain used in the proofs\nomrefpage}
\[ \wt V_\eta=\wt V_\eta[f]= \wt B(\eta) \cup f^\N(B(v,\eta))
.\]
By construction, $f(\wt V_\eta)\subset \wt V_\eta$ (do not forget that there is no other preimage of the origin than itself\footnote{And even if there were, it would be sufficient to assume $\eta$ small enough.}).

\begin{lemma}\label{lem:etadelta}
The following holds, where $D=D_\rep[f](\eta)$:
\begin{enumerate}
\item\label{item:ed:2} $\forall \eta>0$, $\exists \delta>0$ s.t.\ $\forall f\in\cal F$, $V_\delta[f]   \subset \wt V_\eta[f] \cup (D \cap f^{-1}(D))$, 
\item\label{item:ed:3} $\exists \eta_2>0$, $\forall \eta\leq\eta_2$, $\exists \delta>0$ s.t.\ $\forall f\in\cal F$, $V_\delta[f] \cap f^{-1}(\wt V_\eta[f]) \subset \wt V_\eta[f]$, 
\end{enumerate}
\end{lemma}
\begin{proof} 
These are again proved by compactness arguments.
Let $r_0$ be provided by \Cref{prop:cf} applied to $\cal F$. Then $D_\at[f](r)$ is an attracting petal for all $r\leq r_0$.
\begin{enumerate}
\item
The set $\wt B(\eta) \cup (D_\rep[f](\eta) \cap f^{-1}(D_\rep[f](\eta)))\subset \wt V_\eta[f] \cup (D_\rep[f](\eta) \cap f^{-1}(D_\rep[f](\eta)))$ is a neighborhood of $0$ thus contains a ball $B(0,r)$. We can take a uniform value of $r$ for maps $f\in \cal F$ (this can be seen in coordinates $u=-1/c_f z$ as in the proof of \Cref{prop:cf}: the constant $c_f$ is bounded away from $0$ and $\infty$ and the map $f$ is conjugated to a map $u\mapsto u'$ defined on a uniform neighborhood of $\infty$ and with $|u'-(u+1)|<1/4$). We impose $\delta\leq r/2$. By \Cref{lem:traptime} for some $n_0$ we have $\forall f\in \cal F$, $f^{n_0}(v_f)\in D_\at[f](r)$ and thus $\forall n\geq n_0$, $B(f^n(v_f),\delta) \subset \wt V_\eta[f] \cup (D_\rep[f](\eta) \cap f^{-1}(D_\rep[f](\eta)))$. Finally by compactness there is a lower bound on $\inf \setof{\delta>0}{\forall f\in\cal F,\ \forall k<n_0,\ B(f^k(v_f),\delta) \subset f^k(B(v_f,\eta))}$.
\item 
Let $\eta_0>0$ to be determined below and set $\eta_2=\eta_0/2$. Let us assume by contradiction that for some $\eta\leq\eta_0/2$ there exists sequences $\delta_n\tend 0$, $f_n \in \cal F$, $z_n$ such that $z_n \in V_{\delta_n}[f_n]$, $f_n(z_n)\in \wt V_\eta[f_n]$, $z_n\notin \wt V_\eta[f_n]$.
We may extract a subsequence so that $f_n\tend f_0$, and $z_n\tend z_0$. If $z_0\neq 0$ then $z_0\in PC(f_0)$ (see point~\eqref{item:psd:1} of \Cref{lem:PSdist}), a fortiori $z_0\in f_0^\N(B(v_{f_0},\eta))$ and thus for $n$ big enough $z_n\in f_0^\N(B(v_{f_n},\eta))$ by Hurwitz's theorem, thus $z_n\in\wt V_\eta[f_n]$, leading to a contradiction.
If $z_n\tend 0$ then for $n$ big enough, let us prove the statement $f_n(z_n)\in \wt V_\eta[f_n] \implies z_n\in \wt V_\eta[f_n]$, which leads to a contradiction.
Indeed either $f_n(z_n)\in \wt B(\eta)[f_n]$ but then as soon as $|z_n|<\eta$, the whole orbit of $z_n$ by $f_n$ is in $B(0,\eta)$ and thus $z_n\in \wt B(\eta)[f_n]$ thus $z_n\in\wt V_{\eta}[f_n]$.
Or $f_n(z_n)\in f_n^{\N}(B(v_{f_n},\eta))$, say $f_n(z_n)\in f^{k_n}(B(v_{f_n},\eta))$.
For a fixed $k$, by compactness there is a lower bound on the distance from $0$ to $f^{k}(B(v_f,\eta_0/2))$ for $k'<k$ and $f\in \cal F$.
So $k_n\to+\infty$.
Now $f_n$ is injective on $B(0,r)$ for some uniform $r\leq r_0$.
By \Cref{lem:traptime:2} there is some $n_0$ and $\eta_0>0$ such that $\forall f\in \cal F$, we have $f^{n_0}(B(v_f,\eta_0)) \subset D_\at[f](r)$. As soon as $k_n\geq n_0+1$, both $f_n^{k_n-1}(B(v_{f_n},\eta))$ and $f_n^{k_n}(B(v_{f_n},\eta))$ are contained in $D_\at[f](r)\subset B(0,r)$, and $z_n$ also belongs to $B(0,r)$ for $n$ big enough. Hence $f(z_n)\in f_n^{k_{n}}(B(v_{f_n},\eta))$ $\implies$ $z_n\in f_n^{k_n-1}(B(v_{f_n},\eta))$.
\end{enumerate}
\end{proof}

Let
\[\eta_1 = \min(\eta_0/2,r_0,r_1,r_2,r_3,\delta_1/2,\eta'_0,\eta_2)
\]
where $\delta_1$ was defined just before \Cref{lem:stayfar}, $r_2$ in \Cref{prop:homotopicrep}, $\eta_0$, $r_0$ and $r_0'=\min(r_0,r_1,r_2,r_3)$ at the beginning of the current section (\Cref{subsub:bddm}), $\eta_0'$ and $r_3$ in \Cref{lem:disj}, $\eta_2$ in \Cref{lem:etadelta}.

Let $\delta$ be the smallest of the two values associated to $\eta=\eta_1$ by points~\eqref{item:ed:2} and~\eqref{item:ed:3} of \Cref{lem:etadelta}. Since $\eta_1\leq r_0'$ we get $D_\rep[f](\eta_1)\subset D_\rep[f](r_0')=D_\rep[f]$ and thus: $\forall f\in\cal F$,
\bEAn
\label{eq:etadelta:1} V_\delta[f]  & \subset & \wt V_{\eta_1}[f] \cup (D_\rep[f] \cap f^{-1}(D_\rep[f])),
\\
\label{eq:etadelta:2} f^{-1}(\wt V_{\eta_1}[f]) \cap V_\delta[f]  & \subset & \wt V_{\eta_1}[f].
\eEAn

Let $d_1=d_1(\delta)$, $d_1''=d_1''(\delta)$ and $T_4=T_4(\delta)$ be the values associated to $\delta$ just before \Cref{prop:hl}, and $T_6=T_6(\delta)$ defined just before \Cref{lem:t5p}.

Just before \Cref{prop:hl} we also defined $T_5(\delta,\epsilon')$, by $\ell(1-T_5(\delta,\epsilon'))=d_1(\delta)+\ell(1-\epsilon')$ where $\ell(x)=d_\D(0,x)$. Since we just have fixed $\delta$, let us denote $T_5(\epsilon')=T_5(\delta,\epsilon')$. Then 
\[ T_5 (\epsilon') \underset{\epsilon'\to 0}{\sim} K_3\,\epsilon'
\]
with $K_3 = e^{-2d_1(\delta)}$ (the value of this constant is not important, nor is its dependence on $\delta$). 

\begin{lemma}\label{lem:t7}
There exists $K_2>0$ 
and $T_7>0$
such that for all $f_0\in\cal F$, for all $z\in \wt V_{\eta_1}[f_0]$, $\tau(z)>T_7$ and for all $t\leq T_7$, the length of the curve $x\in[0,t]\mapsto z\ag x$ is $\leq K_2 t$ when measured with the hyperbolic metric of $\C\setminus\{v,0\}$.
\end{lemma}
\begin{proof} If the starting point $z\ag 0$ belongs to the part $\wt B(\eta_1)$ of $\wt V_{\eta_1}$ of points whose orbit never leaves $B(0,\eta_1)$, this follows from \Cref{lem:bmslo} since $\eta_1\leq r_1$ and since $\rho(z):=\rho_{\C\setminus\{0,v\}}(z)=o(1/|z|)$ near $0$ thus $\rho(z)|z|$ is bounded on $B(0,\eta_1)$ (note that $\eta_1<\eta_0<|v|$).
Otherwise the starting point $z\ag 0$ belongs to $f_0^\N(B(v,\eta_1))$. Note first that only a finite number of iterates of $B(v,\eta_1)$, bounded independently of $f_0$, are not already contained in the first part.
Moreover, let $m-1$ be such a bound.
Then for all $k\leq m$, for all $z\in f_0^k(B(v,\eta_1))$, $z\ag t = f_t^{-(m-k)}\circ \Phi_t^{-1} \circ \Phi_0 \circ f_0^{m-k}(z)$ for some inverse branch of $f_t^{(m-k)}$.
Since we do not hit a critical point, everything moves differentiably w.r.t.\ the pair $(t,z)$. We thus get\footnote{Here we are \emph{not} using complex values of $t$.} the claimed bound on the hyperbolic length of the curve $z\ag t$ away from $v$, i.e.\ if $z\ag 0\notin B(v,\eta_1)$.
Last, for starting points $z\ag 0$ near $v$, i.e.\ in $B(v,\eta_1)$, note first that $v$ does not move at all: $v\ag t = v$. 
Then $|z\ag{t}-z| 
\leq K |z-v| t$ since the function $(z,t)\mapsto z\ag t-z$ is at least ${\mathrm C}^2$ and vanishes whenever $t=0$ or $z=v$.
Since $\rho(z) = o(1/|z-v|)$ near $v$, the lemma follows.
\end{proof}

Recall that we are dealing with the case where the sequence $n\in\Z\mapsto \omega_n\ag 0$ is not completely contained in $B(0,r_1)$. Together with \Cref{lem:disj} and $\eta_1\leq r_1$, this implies that the first point in this orbit that does not belong to $B(0,r_1)$ also does not belong to $\wt V_{\eta_1}[f_0]$. On the other hand the orbit tends to $0$ thus eventually stays in $B(0,\eta_1)$ hence in $\wt B(\eta_1)[f_0]\subset\wt V_{\eta_1}[f_0]$. The set $\wt V_{\eta_1}[f_0]$ is mapped in itself by $f_0$.
Therefore there is a unique $n_+\in\Z$ such that
\[\omega_{n}\ag 0 \in \wt V_{\eta_1}[f_0] \iff  n\geq n_+\
.\]

If we follow the orbit in the past, it eventually stays in $D_\rep=D_\rep[f_0](r_0')$ in the past.
There is thus a maximal $n_-\in\Z$ such that $\forall n\leq n_-$,
$\omega_n\ag 0 \in D_\rep$. Moreover, $n_-+1<n_+$ because by \Cref{lem:disj}, $\omega_{n_-+1}\ag 0$ cannot belong to $f^\N(B(v,\eta_1))$ and if $\omega_{n_-+1}\ag 0$ were in $B(0,\eta_1)$ then the whole orbit would be contained in $B(0,r_1)$.

Between $n_-$ and $n_+$, the orbit may visit and leave the repelling petal  several times. Let $J$ denote the set of $n\in\Z$ with $n_-<n<n_+$ and $\omega_n\ag 0\notin D_\rep$. This set is non-empty and its extreme values are $n_-+1$ and $n_+-1$ (these two values may be equal).

Denote as follows the constant provided by \Cref{lem:definite} and used in \Cref{lem:hl:1}:
\[\Lambda:=\Lambda(\delta/3)<1 .\]

\newcommand{\tm}{t_{\max}}
Let now $\tm\leq \min(T_3,T_4,T_5/2,T_6,T_7,T_8)$ to be determined later.
Let us work with $t\in[0,\tm]$ and let us do a finite decreasing induction on $J$. In the process, more conditions will be imposed on $\tm$.

\textbf{Initialization:} By \Cref{lem:t7}, $\tau(\omega_{n_+})> \tm$ and for all $t\leq \tm$, the length of $\gamma: s\in[0,t]\mapsto \omega_{n_+} \ag s$ is $\leq K_2 t$ when measured with the hyperbolic metric on $\C\setminus\{0,v\}$.
Provided $K_2 \tm\leq d_1''$, we can apply \Cref{prop:hl2}
(in particular condition~\eqref{item:hl2:5} of this proposition follows from Equation~\eqref{eq:etadelta:2}), thus $\tau(\omega_{n_+-1}\ag 0)>\tm$ and $\forall t\in[0,\tm]$, the path $s\in[0,t]\mapsto \omega_{n_+-1} \ag s$ is homotopic in $W_0$ to $\gamma_1\cdot\gamma_2$ where $\gamma_1$ and $\gamma_2$ are defined in \Cref{prop:hl2}.
By \Cref{lem:hl2:1}, $\hl{W_0}{\gamma_1}\leq M(\delta/3)\hl{\C\setminus\{0,v\}}{\gamma}$ thus $\leq M_0K_2t$ with $M_0=M(\delta/3)$. 
And by \Cref{lem:t5p}, $\hl{W_0}{\gamma_2}\leq K_0 t/T_5$.
Let us sum up: we assumed $K_2 \tm\leq d_1''$ and got $\forall t\in[0,\tm]$,
$\hl{W_0}{\omega_{n_+-1}\big|_{[0,t]}}\leq M_0 K_2 t+K_0 t/T_5$. In particular 
\[ \extent{W_0}{\omega_{n_+-1}\ag{[0,\tm]}}\leq M_0 K_2 \tm+K_0 \tm/T_5
.\]
Let us assume moreover that
\[ M_0 K_2 \tm+K_0 \tm/T_5 \leq \min(d_1,d_1')
\]
so that $ \extent{W_0}{\omega_{n_+-1}\ag{[0,\tm]}}\leq \min(d_1,d_1')$.

\textbf{Induction:} Let $n\in \Z$ satisfying $n_-+1<n\leq n_+-1$ and either $n\in J$ or $n-1\in J$ and assume that we have proved $\tau(\omega_n\ag0)>\tm$ and $\extent{W_0}{\omega_n\ag{[0,\tm]}} \leq \min(d_1,d_1')$.

By Equation~\eqref{eq:etadelta:1}, $\omega_{n-1}\ag 0\notin V_\delta[f]$ thus condition~\eqref{item:hl2:5} of \Cref{prop:hl} is satisfied. Hence we can apply it, and its complements \Cref{lem:hl:1,lem:t5p} and we get $\hl{W_0}{\omega_{n-1}\big|_{[0,t]}}\leq \Lambda \min(d_1,d_1') + K_0 t /T_5$. Let us impose on $\tm$ that $\Lambda \min(d_1,d_1') + K_0 \tm /T_5 \leq \min(d_1,d_1')$, so that we get
$\extent{W_0}{\omega_{n-1}\ag{[0,\tm]}}\leq \min(d_1,d_1')$.
If $n-1\in J$ we can carry on the induction with $n-1$.
If $n-1\notin J$, let $n'$ be the first element of $J$ below $n$ and let $n_1=n-1$ and $n_0=n'+1$: $n_0\leq n_1$.
If $n_0=n_1$ we can also carry on the induction with $n-1$, because $(n-1)-1\in J$.
If $n_0<n_1$ we can apply \Cref{prop:homotopicrep} and its complement \Cref{lem:rpb}: 
$\hl{W_0}{\omega_{n_0}\big|_{[0,t]}} \leq \hl{W_0}{\omega_{n_1}\big|_{[0,t]}} + K_5 t$. Then we can carry on the induction with $n'$, provided we require on $\tm$ that $\Lambda \min(d_1,d_1') + K_0 \tm /T_5 + K_5 \tm \leq \min(d_1,d_1')$.

In all cases, for the induction to carry on it is enough to assume that
\[ \Lambda \min(d_1,d_1') + K_0 \tm /T_5 + K_5 \tm \leq \min(d_1,d_1')
. \]

\textbf{Post induction:} we now know that 
$\extent{W_0}{\omega_{n}\ag{[0,\tm]}} \leq \min(d_1,d_1')$
holds for $n=n_-+1$. We can apply once more \Cref{prop:homotopicrep} and we get that the rest of the orbit (for all $n\in\Z$ with $n\leq n_-$) is defined at least up to time $\tm$.
Moreover, by \Cref{lem:replimsec}, we get that for all $n$ below some relative integer, possibly much smaller\footnote{\Cref{prop:surv} claims uniformity w.r.t.\ $t$, but not w.r.t.\ $f$.} than $n_-$, the full motion takes place in the petal: one of the conclusions of \Cref{prop:surv}.

\medskip

Taking everything into account, we get that the full orbit $\omega_n$ survives for any time $t$ satisfying $t\leq \tm$ for any $\tm$ satisfying $\tm\leq \min(T_3,T_4,T_5/2,T_6,T_7,T_8)$, $\tm\leq d_1''/K_2$, $\tm\leq\min(d_1,d_1')/(M_0K_2+K_0/T_5)$ and $\tm\leq \min(d_1,d_1')(1-\Lambda)/(K_0/T_5 + K_5)$.

Recall that $\delta$ is fixed but not $\epsilon'$. All constants depend only on $\delta$ thus are fixed, except, as we saw earlier, $T_5\sim K_3\epsilon'$ ($K_3$ also depends on $\delta$ thus is fixed). As $\epsilon'\tend 0$, the biggest $\tm$ we can take is equivalent to $K_6 T_5$ where $K_6= \min(1/2,\min(d_1,d_1')(1-\Lambda)/K_0) $.

Hence, for $\epsilon'$ small enough, the survival time of the full orbit is $> K_6 \epsilon'$:

\centerline{\fbox{$\ds \forall n\in\Z,\ \tau_\Phi(\omega_n)> K_6\epsilon'$.}}
\smallskip

This completes the proof of \Cref{prop:surv} with $K=1/K_6$.

\subsection{Step 2, Conclusion}\label{subsec:ccl}

Here we will prove Assertion~\ref{ass:2} (which is what is left to prove the main theorem), whose statement we recall:

\begin{assertion*}
\ass
\end{assertion*}

Consider $\epsilon_1\in\,]0,1[$ to be determined later.
Let $f_0\in\cal F$, and $z\in\sub{\dom(\cal R[f])}{(1-\epsilon_1)}$ and
apply \Cref{prop:part:1} to $\epsilon=\epsilon_1$. For this we have to assume $\epsilon_1<\xi$ for some $\xi>0$ given by the proposition. We obtain some $\epsilon'=\epsilon'(\epsilon_1)>0$ such that the associated orbit $\omega_n \ag 0$ of $f_0$ is contained in $\sub{\dom(f_0)}{(1-\epsilon')}$. By the previous section (\Cref{prop:surv}), $\forall n\in\Z$, $\tau_\Phi(\omega_n) > \epsilon'(\epsilon_1)/K$. 
We can take $\epsilon_0=\min(T''_1,\epsilon'(\epsilon_1)/K)$ where $T_1''<T_1'$ is any chosen constant. 
Since $\epsilon'(\epsilon)\gg\epsilon$, for small enough values of $\epsilon_1$ we have $\epsilon_0>\epsilon_1$.
\Cref{prop:surv} also provides the second claim in Assertion~\ref{ass:2}.

Q.E.D.

\medskip

Now comes a final set of remarks. Let us call $(\epsilon_0,\epsilon_1)$ a valid pair whenever $\epsilon_1<\epsilon_0<T_1'$ and the assertion holds with these values.
Given $\epsilon_1$ small enough, the set of valid values for $\epsilon_0$ includes the interval $]\epsilon_1,\epsilon'(\epsilon_1)/K[$. As the right bound is $\gg \epsilon_1$, it is easy to deduce that: $\forall \epsilon_0$ there exists $\epsilon_1$ such that $(\epsilon_0,\epsilon_1)$ is a valid pair. Moreover we can take $\epsilon_1 = o(\epsilon_0)$.

This implies that if one iterates renormalization starting from a map in $\cal F_\epsilon$ with $\epsilon$ small enough, the map $\cal R^n[f]$ will have at least structure $\cal F_{\epsilon_n}$ with $1/\epsilon_n$ increasing faster than any exponential: the structure tends rapidly to the full structure $\cal F$.

Now, given the specific formula in \Cref{prop:part:1}:
\[\log \frac{1}{\epsilon'(\epsilon_1)} \leq c' + c \log\left(1+\log\frac{1}{\epsilon_1}\right)\]
and the computations above, we get that we can take $\epsilon_1 \leq \exp(\beta-\alpha/\epsilon_0)$ for some constants $\alpha,\beta>0$, i.e.\ $1/\epsilon_n$ increases at least like an \emph{iterated} exponential.

\section{Acknowledgements}

The idea for the general setup came to the author in 2010. The article was written and submitted to arXiv in 2014, \cite{Che4}, and revised in 2015, \cite{Che5}. The author would like to thank Misha Lyubich for insisting that I should submit this article to a peer reviewed journal, for otherwise it would have stayed dormant forever.
As already mentionned, the author would like to thank Arnaud Mortier for his help in \Cref{lem:am}.

\printnomenclature[1.5cm]

\bibliographystyle{alpha}
\bibliography{Parabo}

\end{document}